\documentclass[bj]{imsart}

\RequirePackage{amsthm,amsmath,amsfonts,amssymb}
\RequirePackage{natbib}
\RequirePackage[colorlinks,citecolor=blue,urlcolor=blue]{hyperref}
\RequirePackage{graphicx}
\graphicspath{ {./} }

\usepackage[utf8]{inputenc}
\usepackage{xcolor}
\usepackage{amsfonts}
\usepackage{amsmath}
\usepackage{amssymb}
\usepackage{amsthm}
\usepackage{bbm}

\usepackage{multirow}

\usepackage[ruled]{algorithm2e}
\usepackage{hyperref} % for url reference in simulation section

\usepackage{threeparttable} % for table notes

\usepackage{multibib}

\startlocaldefs

\theoremstyle{plain}
\newtheorem{thm}{Theorem}
\newtheorem{lem}{Lemma}
\newtheorem{pro}{Proposition}
\newtheorem{cor}{Corollary}

\newtheorem{manualprominner}{Proposition}
\newenvironment{manualpro}[1]{%
  \renewcommand\themanualprominner{#1}%
  \manualprominner
}{\endmanualprominner}

\newtheorem{manualremminner}{Remark}
\newenvironment{manualrem}[1]{%
  \renewcommand\themanualremminner{#1}%
  \manualremminner
}{\endmanualremminner}

\newtheorem{manualcorminner}{Corollary}
\newenvironment{manualcor}[1]{%
  \renewcommand\themanualcorminner{#1}%
  \manualcorminner
}{\endmanualcorminner}

\theoremstyle{remark}

\newtheorem{rem}{Remark}
\newtheorem{ass}{Assumption}
\SetKwProg{Repeat}{Repeat}{}{}
\SetKwProg{Init}{init}{}{}

\newcommand{\ramp}{\mathrm{ramp}}

\def\me{\mathrm{e}}

\def\dif{\mathrm{d}}

\def\E{\mathrm{E}}
\def\var{\mathrm{Var}}

\def\pr{\mathrm{P}}
\def\N{\mathrm{N}}

\def\diag{\mathrm{diag}}
\def\tr{\mathrm{tr}}
\def\T{ {\mathrm{\scriptscriptstyle T}} }

\def\bbR{\mathbb R}
\def\dom{\mathrm{dom}}
\def\sign{\mathrm{sign}}
\def\sigm{\mathrm{sigmoid}}
\def\argmax{\mathrm{argmax}}
\def\argmin{\mathrm{argmin}}

\def\vec{\mathrm{vec}}

\def\op{\mathrm{op}}
\def\fro{\mathrm{F}}

\def\HG{\mathrm{HG}}

\def\TV{\mathrm{TV}}

\def\sp{\mathrm{sp}}

\def\rp{\mathrm{rp}}

\def\rpL{\mathrm{rp1}}
\def\rpQ{\mathrm{rp2}}

\def\spL{\mathrm{sp1}}
\def\spQ{\mathrm{sp2}}

\def\err{\mathrm{Err}}

\def\erf{\mathrm{erf}}

\def\pen{\mathrm{pen}}

\def\dprime{ {\prime\prime} }

\newcites{append}{References}

\endlocaldefs

\begin{document}

\begin{frontmatter}

\title{Tractable and Near-Optimal Adversarial Algorithms for Robust Estimation in Contaminated Gaussian Models}
\runtitle{Adversarial Algorithms for Robust Estimation}

\begin{aug}
%%%%%%%%%%%%%%%%%%%%%%%%%%%%%%%%%%%%%%%%%%%%%%
%%Only one address is permitted per author. %%
%%Only division, organization and e-mail is %%
%%included in the address.                  %%
%%Additional information can be included in %%
%%the Acknowledgments section if necessary. %%
%%%%%%%%%%%%%%%%%%%%%%%%%%%%%%%%%%%%%%%%%%%%%%
\author[A]{\fnms{Ziyue} \snm{Wang}\ead[label=e1, mark]{zw245@stat.rutgers.edu}}
\and
\author[A]{\fnms{Zhiqiang} \snm{Tan}\ead[label=e2, mark]{ztan@stat.rutgers.edu}}
%%%%%%%%%%%%%%%%%%%%%%%%%%%%%%%%%%%%%%%%%%%%%%
%% Addresses                                %%
%%%%%%%%%%%%%%%%%%%%%%%%%%%%%%%%%%%%%%%%%%%%%%
\address[A]{Department of Statistics,
Rutgers University, Piscataway, New Jersey 08854, USA
\printead{e1,e2}}
\end{aug}

\begin{abstract}
Consider the problem of simultaneous estimation of location and variance matrix under Huber's contaminated Gaussian model.
First, we study minimum $f$-divergence estimation at the population level, corresponding to a generative adversarial method with a nonparametric discriminator
and establish conditions on $f$-divergences which lead to robust estimation, similarly to robustness of minimum distance estimation.
More importantly, we develop tractable adversarial algorithms with simple spline discriminators, which can be implemented
via nested optimization such that the discriminator parameters can be fully updated by maximizing a concave objective function given the current generator.
The proposed methods are shown to achieve minimax optimal rates or near-optimal rates depending on the $f$-divergence and the penalty used.
 {This is the first time such near-optimal error rates are established for adversarial algorithms with linear discriminators under Huber's contamination model.}
We present simulation studies to demonstrate advantages of the proposed methods over classic robust estimators, pairwise methods, and
a generative adversarial method with neural network discriminators.
\end{abstract}


\end{frontmatter}
%%%%%%%%%%%%%%%%%%%%%%%%%%%%%%%%%%%%%%%%%%%%%%
%%%% Main text entry area:

\section{Introduction}
Consider Huber's contaminated Gaussian model (\cite{Hub64}): independent observations $X_1,\ldots,X_n$ are obtained from
$P_\epsilon = (1-\epsilon) \N(\mu^*, \Sigma^*) + \epsilon Q$,
where $\N(\mu^*,\Sigma^*)$ is a $p$-dimensional Gaussian distribution with mean vector $\mu^*$ and variance matrix $\Sigma^*$,
$Q$ is a probability distribution for contaminated data, and $\epsilon$ is a contamination fraction.
Our goal is to estimate the Gaussian parameters $(\mu^*, \Sigma^*)$, without any restriction on $Q$
for a small $\epsilon$.
 {This allows both outliers located in areas with vanishing probabilities under $\N(\mu^*,\Sigma^*)$
and other contaminated observations in areas with non-vanishing probabilities under $\N(\mu^*,\Sigma^*)$.}
We focus on the setting where the dimension $p$ is small relative to the sample size $n$, and
no sparsity assumption is placed on $\Sigma^*$ or its inverse matrix. The latter, $\Sigma^{*-1}$, is
called the precision matrix and is of particular interest in Gaussian graphical modeling.
In the low-dimensional setting, estimation of $\Sigma^*$ and $\Sigma^{*-1}$ can be treated as being equivalent.

 {There is a vast literature on robust statistics (e.g., Huber and Ronchetti 2009; Maronna et al.~2018).}
In particular, the problem of robust estimation from contaminated Gaussian data has been extensively studied,
and various interesting methods and results have been obtained recently.
Under Huber's contamination model above,
while the bulk of the data are still Gaussian distributed, a challenge is that the contamination status of each observation is hidden,
and the contaminated data may be arbitrarily distributed. In this sense, this problem should be distinguished from various related problems, including
multivariate scatter estimation for elliptical distributions as in \cite{Tyl87} and
estimation in Gaussian copula graphical models as in \cite{LHYLW12} and \cite{XZ12}, among others.
 {For motivation and comparison, we discuss below several existing approaches directly related to our work.}

{\bf Existing work.}\;
As suggested by the definition of variance matrix $\Sigma^*$,
a numerically simple method, proposed in \cite{OC15} and \cite{TMW15}, is to apply
a robust covariance estimator for each pair of variables, for example, based on robust scale and correlation estimators,
and then assemble those estimators into an estimated variance matrix $\hat\Sigma$.
 {These pairwise methods are naturally suitable for both Huber's contamination model and the cellwise contamination model where
the components of a data vector can be contaminated independently, each with a small probability $\epsilon$.}
For various choices of the correlation estimator, such as the transformed Kendall's $\tau$ and Spearman's $\rho$ estimator, this method
is shown in \cite{LT18} to achieve, in the maximum norm $\|\hat \Sigma - \Sigma^*\|_{\max}$, the minimax error rate $\epsilon + \sqrt{\log(p)/n}$
under cellwise contamination and Huber's contamination model.
However, because a transformed correlation estimator is used,
the variance matrix estimator in \cite{LT18} may not be positive semidefinite (\cite{OC15}).
Moreover, this approach seems to rely on the availability of individual elements of $\Sigma^*$ as pairwise covariances and generalization to other multivariate models can be difficult. In our numerical experiments, such pairwise methods
have relatively poor performance when contaminated data are not easily separable from the uncontaminated marginally, especially with nonnegligible $\epsilon$.

For location and scatter estimation under Huber's contamination model, \cite{CGR18} showed that the minimax error rates in the $L_2$ and operator norm,
$\| \hat\mu-\mu^*\|_2$ and $\|\hat\Sigma-\Sigma^*\|_\op$, are
$\epsilon + \sqrt{p/n}$ and attained by maximizing Tukey's half-space depth (\cite{Tuk75}) and a matrix depth function,
which is also studied in \cite{Zha02} and \cite{PB18}.
Both depth functions, defined through minimization of certain discontinuous objective functions, are in general difficult to compute,
and maximization of these depth functions is also numerically intractable.
Subsequently, \cite{GLYZ18} and \cite{GYZ20} exploited a connection between depth-based estimators and generative adversarial nets (GANs) (\cite{GPM14}),
and proposed robust location and scatter estimators in the form of GANs.
These estimators are also proved to achieve the minimax error rates
in the $L_2$ and operator norms under Huber's contamination model.
 {More recent work in this direction includes \cite{ZJT20}, \cite{Wu20}, and \cite{LL21}.}

GANs are a popular approach for learning generative models, with numerous impressive applications (\cite{GPM14}).
In the GAN approach, a generator is defined to transform white noises into fake data,
and a discriminator is then employed to distinguish between the fake and real data.
The generator and discriminator are trained through minimax optimization with a certain objective function.
For GANs used in \cite{GLYZ18} and \cite{GYZ20},
the generator is defined by the Gaussian model and the discriminator is a multi-layer neural network
with sigmoid activations in the top and bottom layers.
Hence the discriminator can be seen as logistic regression with the ``predictors'' defined by the remaining layers of the neural network.
The GAN objective function, usually taken to the log-likelihood function in the classification of fake and real data,
is more tractable than discontinuous depth functions, but remains nonconvex in the discriminator parameters and nonconcave in the generator parameters.
Training such GANs is challenging through nonconvex-nonconcave minimax optimization (\cite{FO20}; \cite{JNJ20}).

There is also an interesting connection between GANs and minimum divergence (or distance) (MD) estimation, which has been traditionally studied for
robust estimation (\cite{DL88}; \cite{Lin94}; \cite{BL94}).
A prominent example is minimum Hellinger distance estimation (\cite{Ber77}; \cite{TB86}).
In fact, as shown in $f$-GANs (\cite{NCT16}),
various choices of the objective function in GANs can be derived from variational lower bounds of $f$-divergences
between the generator and real data distributions.
Familiar examples of $f$-divergences include
the Kullback--Leibler (KL), squared Hellinger divergences, and the total variation (TV) distance (\cite{AS66}; \cite{Csi67}).
In particular, using the log-likelihood function in optimizing the discriminator leads to a lower bound of the Jensen--Shannon (JS) divergence for the generator.
Furthermore, the lower bound becomes tight if the discriminator class is sufficiently rich (to include the nonparametrically optimal discriminator given any generator).
In this sense, $f$-GANs can be said to nearly implement minimum $f$-divergence estimation,
where the parameters are estimated by minimizing an $f$-divergence between the model and data distributions.
 {However, this relationship is {\it only approximate and suggestive}},
because even a class of neural network discriminators may not be nonparametrically rich with population data.
A similar issue can also be found in the previous studies,
where minimum Hellinger estimation and related methods require a smoothed density function of sample data. This approach
is impractical for multivariate continuous data.

In addition to MD estimation mentioned above,
two other methods of MD estimation have also been studied for robust estimation both in general parametric models and in multivariate Gaussian models.
The two methods are defined by minimization of power density divergences (also called $\beta$-divergences) (\cite{BHHJ98}; \cite{MK06})
and that of $\gamma$-divergences (\cite{Win95}; \cite{FE08}; \cite{HFS17}).
See \cite{JHHB01} for a comparison of these two methods.
In contrast with $f$-divergences, these two divergences can be evaluated without requiring smooth density estimation from sample data,
and hence the corresponding MD estimators can be computed by standard optimization algorithms.
To our knowledge, error bounds have not been formally derived for these methods under Huber's contaminated Gaussian model.

% Lin: Lindsay paper
% Ber: Beran paper
% TB: Tamura and Boos paper in Lindsay

% BHHJ: Robust and Efficient Estimation by Minimising a Density Power Divergence
% Author(s): Ayanendranath Basu, Ian R. Harris, Nils L. Hjort and M. C. Jones

% MK: Robust Gaussian graphical modeling
% Masashi Miyamura,Yutaka Kano

% Win: WINDHAM, M. P. (1995). Robustifying model fitting. J. R. Statist. Soc. B 57, 599-609.

% FE: Robust parameter estimation with a small bias against heavy contamination
% Hironori Fujisawa, Shinto Eguchi

% HSF: Robust sparse Gaussian graphical modeling
% Kei Hirose, Hironori Fujisawa, Jun Sese

% JHHB: A Comparison of Related Density-Based Minimum Divergence Estimators
% Author(s): M. C. Jones, Nils Lid Hjort, Ian R. Harris and Ayanendranath Basu

Various methods based on iterative pruning or convex programming have been studied with provable error bounds for robust estimation
in Huber's contaminated Gaussian model (\cite{LRV16}; \cite{BD15}; \cite{DKKLMS19}). These methods either handle scatter estimation after location estimation
sequentially in two stages, or resort to using normalized differences of pairs with mean zero for scatter estimation.

{\bf Our work.}\; We propose and study adversarial algorithms with linear spline discriminators, and
establish various error bounds for simultaneous location and scatter estimation under Huber's contaminated Gaussian model. Two distinct types of GANs are exploited.
The first one is logit $f$-GANs (\cite{TSO19}), which corresponds to a specific choice of $f$-GANs with the objective function
formulated as a negative loss function for logistic regression (or equivalently a density ratio model between fake and real data) when
training the discriminator.
The second is hinge GAN (\cite{LY17}; \cite{ZML17}), where the objective function is taken to be the negative hinge loss function
when training the discriminator.
The hinge objective can be derived from a variational lower bound of the total variation distance (\cite{NWJ10}; \cite{TSO19}),
but cannot be deduced as a special case of the $f$-GAN %or logit $f$-GAN
objective even though the total variation is also an $f$-divergence.
See Remark \ref{rem:tv-gan}.
In addition, we allow two-objective GANs, including the log$D$ trick in \cite{GPM14}, where
two objective functions are used, one for updating the discriminator and the other for updating the generator.

As a major departure from previous studies of GANs, our methods use a simple linear class of spline discriminators,
where the basis functions
consist of univariate truncated linear functions (or ReLUs shifted) at $5$ knots and the pairwise products of such univariate functions.
For hinge GAN and certain logit $f$-GANs including those based on the reverse KL (rKL) and JS divergences,
the objective function is concave in the discriminator. By the linearity of the spline class,
the objective function is then concave in the spline coefficients. Hence our hinge GAN and logit $f$-GAN methods involve
maximization of a concave function when training the spline discriminator for any fixed generator.
In contrast with nonconvex-nonconcave minimax optimization for GANs with neural network discriminators (\cite{GLYZ18}; \cite{GYZ20}),
our methods can be implemented through nested optimization in a numerically tractable manner.
See Remarks \ref{rem:nash}, \ref{rem:f-condition2-b} and \ref{rem:hinge-concavity} and Algorithm \ref{alg:1}.
 {While the outer minimization for updating the generator remains nonconvex in our methods,
such a single nonconvex optimization is usually more tractable than nonconvex-nonconcave minimax optimization.}

In spite of the limited capacity of the spline discriminators, we establish various error bounds for our location and scatter estimators,
depending on whether the hinge-GAN or logit $f$-GAN is used and whether an $L_1$ or $L_2$ penalty is incorporated when training the discriminator.
See Table \ref{tab:rates} for a summary of existing and our error rates in scatter estimation.
Our $L_1$ penalized hinge GAN method achieves the minimax error rate $\epsilon + \sqrt{\log(p)/n}$ in the maximum norm. Our $L_2$ penalized hinge GAN method achieves the error rate $\epsilon \sqrt{p}+ \sqrt{p/n}$,
whereas the minimax error rate is $\epsilon + \sqrt{p/n}$, in the $p^{-1/2}$-Frobenius norm.
While this might indicate the price paid for maintaining the convexity in training the discriminator,
our error rate reduces to the same order $\sqrt{p/n}$ as the minimax error rate provided that
$\epsilon$ is sufficiently small, $\epsilon = O(\sqrt{1/n})$,
such that the contamination error term $\epsilon\sqrt{p}$ is dominated by the sampling variation term $\sqrt{p/n}$
up to a constant factor.
 {To our knowledge, such near-optimal error rates were previously inconceivable for adversarial algorithms with linear discriminators in robust estimation.}
Moreover, the error rates for our logit $f$-GAN methods exhibit a square-root dependency on the contamination fraction $\epsilon$, instead of
a linear dependency for our hinge GAN methods.  {This shows, for the first time, some theoretical advantage of hinge GAN over logit $f$-GANs,
although comparative performances of these methods may vary in practice, depending on specific settings.}

\begin{table}[]
\begin{threeparttable}[]
\caption{Comparison of existing and proposed methods.}
\label{tab:rates}
\begin{tabular}{lll}
\hline
& Error rates & Computation \\
\hline
OC15, TMW15, LT18    & $\|\hat\Sigma - \Sigma^*\|_{\max} = (\epsilon + \sqrt{\log (p)/n}) O_p(1) $ & Non-iterative computation \\ \\

CGR18   & $\|\hat\Sigma - \Sigma^*\|_\op = (\epsilon + \sqrt{p/n}) O_p(1) $ & Minimax optimization \\
        &                                                                   & with zero-one discriminators \\
DKKLMS19             &  $\|\hat\Sigma - \Sigma^*\|_\op = O_p(\epsilon)$,  & convex optimization \\
                     &  provided $\epsilon \ge p/\sqrt{n}$ up to log factors  &    \\

GYZ20                & $\|\hat\Sigma - \Sigma^*\|_\op = (\epsilon + \sqrt{p/n}) O_p(1) $ & Minimax optimization with \\
                     &                                                                   & neural network discriminators  \\  \hline
$L_1$ logit $f$-GAN  & $\|\hat\Sigma - \Sigma^*\|_{\max} = (\sqrt{\epsilon} + \sqrt{\log (p)/n}) O_p(1)$ &  \\  {(Theorem~\ref{thm:logit-L1})}\\
$L_1$ hinge GAN      & $\|\hat\Sigma - \Sigma^*\|_{\max} = (\epsilon + \sqrt{\log (p)/n}) O_p(1)$ &   {Nested optimization} with an\\  {(Theorem~\ref{thm:hinge-L1})}
& & objective function concave in\\

$L_2$ logit $f$-GAN  & $p^{-\frac{1}{2}} \|\hat\Sigma - \Sigma^*\|_\fro = (\sqrt{\epsilon} + \sqrt{p/n}) O_p(1),$ &linear spline discriminators\\  {(Theorem~\ref{thm:logit-L2})}
                     &       provided $\epsilon \le 1/p$ up to a constant factor                                                                    & \\
$L_2$ hinge GAN      & $p^{-\frac{1}{2}} \|\hat\Sigma - \Sigma^*\|_\fro = (\epsilon\sqrt{p} + \sqrt{p/n}) O_p(1)$ & \\
 {(Theorem~\ref{thm:hinge-L2})}\\
\hline
\end{tabular}
Note: OC15, TMW15, LT18 refer to methods and theory in \cite{OC15}, \cite{TMW15}, \cite{LT18};
CGR18, DKKLMS19, and GYZ20 refer to, respectively, \cite{CGR18},  \cite{DKKLMS19}, and  \cite{GYZ20}.
\end{threeparttable}
\end{table}

 {To facilitate and complement our sample analysis, we provide error bounds for the population version of hinge GAN or logit $f$-GANs with nonparametric discriminators,
that is, minimization of the exact total variation or $f$-divergence at the population level.}
%Notice that we are not minimizing a lower bound but the exact $f$-divergence because of the nonparametric discriminators.
From Theorem \ref{thm:pop-robust}, population minimum TV or $f$-divergence estimation under a simple set of conditions on $f$ (Assumption~\ref{ass:f-condition})
leads to errors of order $O(\epsilon)$ or $O(\sqrt{\epsilon})$ respectively under Huber's contamination model.
Assumption~\ref{ass:f-condition} allows the reverse KL, JS, reverse $\chi^2$, and squared Hellinger divergences,
but excludes the mixed KL divergence, $\chi^2$ divergence, and, as reassurance,
the KL divergence which corresponds to maximum likelihood estimation and is known to be non-robust.
Hence certain (but not all) minimum $f$-divergence estimation achieves robustness under Huber's contamination model or an $\epsilon$ TV-contaminated neighborhood.
 {Such robustness is identified for the first time for minimum $f$-divergence estimation,}
and is related to, but distinct from, robustness of minimum distance estimation under $\epsilon$ contaminated neighborhood with respect to the same distance (\cite{DL88}).
See Remark \ref{rem:pop-robust-DL} for further discussion.
%For this reason, we refer to $f$-divergences satisfying Assumption~\ref{ass:f-condition} as robust $f$-divergences.
The population error bounds in the $L_2$ and $p^{-1/2}$-Frobenius norms are independent of $p$ and hence tighter than
the corresponding $\epsilon$ terms in our sample error bounds for both hinge GAN and logit $f$-GAN. %location and scatter estimation.
These gaps can be attributed to the use of nonparametric versus spline discriminators.

 {Remarkably, our population analysis also sheds light on the comparison of our sample results and those in \cite{GYZ20}.}
On one hand, another set of conditions (Assumption~\ref{ass:f-condition2}), in addition to Assumption~\ref{ass:f-condition},
are required in our sample analysis of logit $f$-GANs with spline discriminators.
On the other hand, GANs used in \cite{GYZ20} can be recast as logit $f$-GANs with
neural network discriminators (see Section \ref{sec:GYZ}). But minimax error rates are shown to be achieved in \cite{GYZ20} for
an $f$-divergence (for example, the mixed KL divergence)
which, let alone Assumption \ref{ass:f-condition2}, does not even satisfy Assumption~\ref{ass:f-condition} used in our analysis to show robustness of minimum $f$-divergence estimation.
The main reason for this discrepancy is that the neural network discriminator in \cite{GYZ20}
is directly constrained to be of order $\epsilon + \sqrt{p/n}$ in the log odds,
which considerably simplifies the proofs of rate-optimal robust estimation.
In contrast,  {our methods use linear spline discriminators (with penalties independent of $\epsilon$),
and our proofs of robust estimation need to carefully tackle various technical difficulties due to the simple design of our methods.}
See Figure~\ref{fig:1}(b) for an illustration of non-robustness by minimization of the mixed KL divergence,
and Section \ref{sec:GYZ} for further discussion on this subtle issue in \cite{GYZ20}.

{\bf Notation.}\; For a vector $a=(a_1, \dots, a_p)^T \in \bbR^p$, we denote by $\|a\|_1 = \sum_{i=1}^{p}|a_i|$, $\|a\|_{\infty} = \max_{1 \le i \le p} |a_i|$, and $\|a\|_2 = (\sum_{i=1}^{p}a_i^2)^{1/2}$ the $L_1$ norm, $L_{\infty}$ norm, and $L_2$ norm of $a$, respectively. For a matrix $A= (a_{ij}) \in \bbR^{m\times n}$, we define the element-wise maximum norm $\|A\|_{\max} = \max_{1\le i \le m, 1\le j \le n} |a_{ij}|$, the Frobenius norm $\|A\|_{\fro}=(\sum_{i=1}^{m}\sum_{j=1}^{n} a_{ij}^2)^{1/2}$, the vectorized $L_1$ norm $\|A\|_{1,1}=\sum_{i=1}^{m}\sum_{j=1}^{n}|a_{ij}|$, the operator norm $\|A\|_{\op} = \sup_{\|x\|_2 \le 1} \|Ax\|_2$, and the $L_{\infty}$-induced operator norm $\|A\|_{\infty} = \sup_{\|x\|_{\infty} \le 1} \|Ax\|_{\infty}$.
For a square matrix $A$, we write $A\succeq 0$ to indicate that $A$ is positive semidefinite. The tensor product of vectors $a$ and $b$ is denoted by $a \otimes b$, and the vectorization of matrix $A$ is denoted by $\vec(A)$. The cumulative distribution function of the standard normal distribution is denoted by $\Phi(x)$,
and the Gaussian error function is denoted by $\erf(x)$.

\section{Numerical illustration}

    We illustrate the performance of our rKL logit $f$-GAN and six existing methods, with two samples of size $2000$
    from a $10$-dimensional Huber's contaminated Gaussian distribution with $\epsilon = 5\%$ and $10\%$, based on
     {a Toeplitz covariance matrix} and the first contamination $Q$ described in Section \ref{sec:simulation-setting}.
    Figure~\ref{fig:Ellipsoid_TypeA} shows the 95\% Gaussian ellipsoids for the first two coordinates,
    using the estimated location vectors and variance matrices except for Tyler's M-estimator (\cite{Tyl87}), Kendalls's $\tau$ with MAD (\cite{LT18}), and
    Spearman's $\rho$ with $Q_n$-estimator (\cite{OC15}) where the locations are set to the true means.
    The performance of our JS logit $f$-GAN and hinge GAN is close to that of rKL logit $f$-GAN.
    See the Supplement for illustration based on the second contamination in Section \ref{sec:simulation-setting}.

    Among the methods shown in Figure \ref{fig:Ellipsoid_TypeA},
    {the rKL logit $f$-GAN gives an ellipsoid closest to the truth, followed with relatively small but noticeable differences by
    the JS-GAN (\cite{GYZ20}), MCD (\cite{Rou85}), and Mest (\cite{Roc96}), which are briefly described in Section \ref{sec:simulation-method}.}
    The remaining three methods, Kendall's $\tau$ with MAD, Spearman's $\rho$ with $Q_n$-estimator, and Tyler's M-estimator show much less satisfactory performance.
    The estimated distributions from these methods are dragged towards the corner contamination cluster.

    The relatively poor performance of the pairwise methods, Kendall's $\tau$ with MAD and Spearman's $\rho$ with $Q_n$-estimator, may be explained by
    the fact that as shown by the marginal histograms in Figure \ref{fig:Ellipsoid_TypeA},
    the data in each coordinate are one-sided heavy-tailed, but no obvious outliers
    can be seen marginally. The correlation estimates from Kendall's $\tau$ and Spearman's $\rho$ tend to be inaccurate even after
    sine transformations, especially with nonnegligible $\epsilon=10\%$.
    In contrast, our GAN methods as well as JS-GAN, MCD, and the Mest are capable of capturing higher dimensional information so that
    the impact of contamination is limited to various extents.

    \begin{figure}[h]
        \includegraphics[scale=0.11]{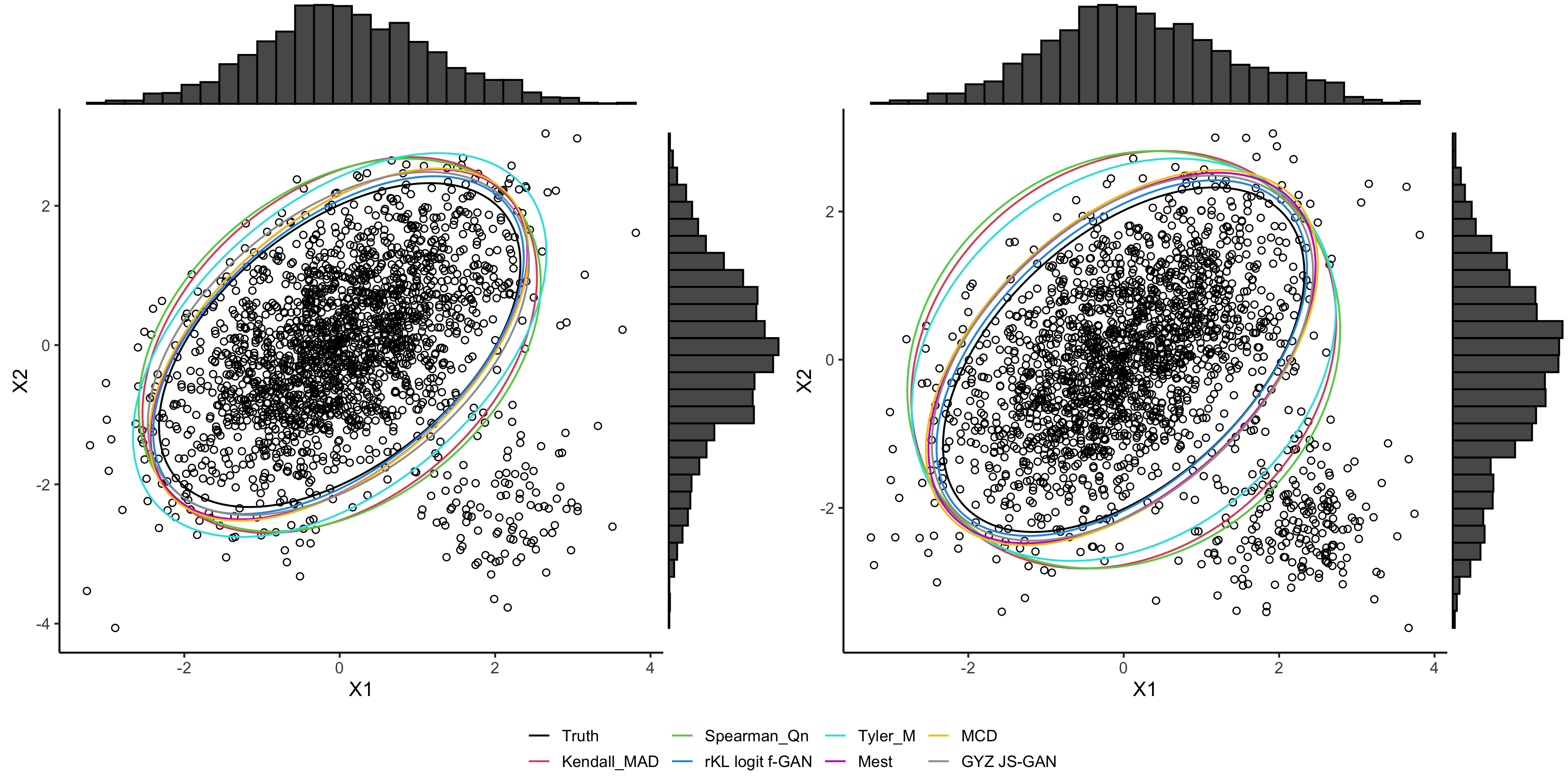}
        \caption{The $95\%$ Gaussian ellipsoids estimated for the first two coordinates and observed marginal histograms,
         from contaminated data based on the first contamination in Section \ref{sec:simulation-setting} with $\epsilon=5\%$ (left) or $10\%$ (right).}
        \label{fig:Ellipsoid_TypeA}
    \end{figure}
\vspace{-30pt}
\section{Adversarial algorithms} \label{sec:review-gan}

We review various adversarial algorithms (or GANs), which are exploited by our methods for robust location and scatter estimation.
To focus on main ideas, the algorithms are stated in their population versions, where the underlying data distribution $P_*$ is involved
instead of the empirical distribution $P_n$.
Let $\{P_\theta: \theta\in\Theta\}$ be a statistical model and
$\{h_\gamma: \gamma\in\Gamma\}$ be a function class, where $P_\theta$ is called a generator and $h_\gamma$ a discriminator.
In our study, $P_\theta$ is a multivariate Gaussian distribution $\N(\mu,\Sigma)$, and
$h_\gamma$ is a pairwise spline function which is specified later in Section \ref{sec:logit-fgan}.

For a convex function $f: (0,\infty) \to \bbR$, the $f$-divergence between the distributions $P_*$ and $P_\theta$ with
density functions $p_*$ and $p_\theta$ is
\begin{align*}
D_f( P_* \| P_\theta) = \int f\left( \frac{p_*(x)}{p_\theta(x)}\right) \,\dif P_\theta.
\end{align*}
For example, taking $f(t) = t \log t$ yields the Kullback--Liebler (KL) divergence $D_{\mathrm{KL}} (P_* \| P_\theta)$.
%Please note that $f$-divergences are not uniquely defined in the sense that given any $f_0$ that defines an $f$-divergence, functions in $\{f_c: f(t) = f_0(t) + c(t-1), c\in \bbR\}$ equivalently define the same $f$-divergence.
%This could cause confusion when discussing conditions on convex function $f$. See Remark~\ref{rem:equi-def} for details.
%Taking $f(t) = -\log t$ yields the reverse KL divergence $D_{\mathrm{rKL}}(P_* \| P_\theta) = D_{\mathrm{KL}} (P_\theta \| P_* )$.
The logit $f$-GAN (\cite{TSO19}) is defined by solving the minimax program
\begin{align}
& \min_{\theta\in\Theta} \max_{\gamma\in\Gamma}  \; K_f (P_*, P_\theta; h_\gamma),   \label{eq:logit-fgan}
\end{align}
where
\begin{align*}
K_f (P_*, P_\theta; h)
& = \E_{P_*} f^{\prime}(\me^{h (x)}) -\E_{P_\theta}  f^\# (e^{h (x)})  \\
& = \E_{P_*} f^{\prime}(\me^{h (x)}) -\E_{P_\theta} \left\{ \me^{h (x)}f^{\prime}(\me^{h (x)})-f(e^{h (x)}) \right\}  ,
\end{align*}
Throughout, $f^\#(t) = t f^\prime(t) - f(t)$ and $f^\prime$ denotes the derivative of $f$.
A motivation for this method is that the objective $K_f$ is a nonparametrically tight, lower bound of the $f$-divergence (\cite{TSO19}, Proposition S1):
for each $\theta$, it holds that for any function $h$,
\begin{align}
K_f( P_*, P_\theta; h)  \le D_f ( P_* \| P_\theta) , \label{eq:logit-fgan-lbd}
\end{align}
where the equality is attained at $h_{*\theta}(x) = \log\{ p_*(x) / p_\theta(x)\}$, the log density ratio between $P_*$ and $P_\theta$
or equivalently the log odds for classifying whether a data point $x$ is from $P_*$ or $P_\theta$.
There are two choices of $f$ of particular interest.
Taking $f(t) = t\log t -(t+1)\log(t+1) + \log 4$ leads the Jensen--Shannon (JS) divergence,
$ D_{\mathrm{JS}} ( P_* \| P_\theta ) = D_{\mathrm{KL}} ( P_* \| (P_*+P_\theta)/2) + D_{\mathrm{KL}} ( P_\theta \| (P_* + P_\theta)/2 )$,
and the objective function
\begin{align*}
K_{\mathrm{JS}} (P_*, P_\theta; h) = - \E_{P_*} \log (1+ \me^{-h(x)}) - \E_{P_\theta}\log(1+ \me^{h(x)}) + \log 4,
\end{align*}
which is, up to a constant, the expected log-likelihood for logistic regression with log odds function $h(x)$.
For $K_f = K_{\mathrm{JS}}$, program (\ref{eq:logit-fgan}) corresponds to the original GAN (\cite{GPM14}) with discrimination probability $\sigm(h(x))$.
Taking $f(t) = -\log t$ leads to
the reverse KL divergence
$D_{\mathrm{rKL}}( P_* \| P_\theta )= D_{\mathrm{KL}} (P_\theta \| P_* )$ and the objective function
\begin{align*}
K_{\mathrm{rKL}} (P_*, P_\theta; h) = 1 - \E_{P_*} \me^{-h(x)} - \E_{P_\theta} h(x) ,
\end{align*}
which is the negative calibration loss for logistic regression in \cite{Tan20}.

The objective $K_f$ with fixed $\theta$ can be seen as
a proper scoring rule reparameterized in terms of the log odds function $h(x)$ for binary classification (\cite{TZ20}).
Replacing $K_f$ in (\ref{eq:logit-fgan}) by the negative hinge loss (which is not a proper scoring rule) leads to
\begin{align}
& \min _{\theta\in\Theta} \max_{\gamma\in\Gamma}  \; K_{\HG} (P_*, P_\theta; h_\gamma), \label{eq:hinge-gan}
\end{align}
where
\vspace{-5pt}
\begin{align*}
K_{\HG} (P_*, P_\theta; h) & = - \left\{ \E_{P_*} \max(0, 1-h(x)) + \E_{P_\theta} \max(0, 1+h(x)) \right\} + 2 \\
& = \E_{P_*} \min(1, h(x)) + \E_{P_\theta} \min(1, -h(x)) .
\end{align*}
This method is related to the geometric GAN described later in (\ref{eq:geo-gan}) and  will be called hinge GAN.
By \cite{Ngu09} or Proposition 5 in \cite{TSO19}, the objective $K_{\HG}$ is a nonparametrically tight, lower bound of the total variation distance scaled by $2$:
for each $\theta$, it holds that for any function $h(x)$,
\begin{align}
K_{\HG} ( P_*, P_\theta; h)  \le 2D_{\TV} ( P_* \| P_\theta) , \label{eq:hinge-gan-lbd}
\end{align}
where the equality is attained at $h_{*\theta}(x) = \sign(p_*(x) - p_\theta(x))$, and $D_{\TV} ( P_* \| P_\theta) = \int | p_*(x) - p_\theta(x) |/2\,\dif x$.
The objectives $K_f$ and $K_{\HG}$, with fixed $\theta$, represent two types of loss functions for binary classification.
See \cite{BSS05} and \cite{Ngu09} for further discussions about loss functions and scoring rules.

The preceding programs, (\ref{eq:logit-fgan}) and (\ref{eq:hinge-gan}), are defined as minimax optimization, each with a single objective function.
There are also adversarial algorithms, which are formulated as alternating optimization with two objective functions (see Remark \ref{rem:nash}).
For example, GAN with the $\log D$ trick in \cite{GPM14} is defined by solving
\begin{align} \label{eq:gan-logD}
\left\{ \begin{array}{ll}
\max\limits_{\gamma\in\Gamma}\; K_{\mathrm{JS}} (\theta,\gamma) & \text{with $\theta$ fixed},  \\
\min\limits_{\theta\in\Theta}\; \E_{P_\theta} \log ( 1+\me^{-h_\gamma(x)}) & \text{with $\gamma$ fixed}. \\
\end{array} \right.
\end{align}
The second objective is introduced mainly to overcome vanishing gradients in $\theta$ when the discriminator is confident.
The calibrated rKL-GAN (\cite{Hus16}; \cite{TSO19}) is defined by solving
\begin{align} \label{eq:cal-rkl-gan}
\left\{ \begin{array}{ll}
\max\limits_{\gamma\in\Gamma}\; K_{\mathrm{JS}} (\theta,\gamma) & \text{with $\theta$ fixed},  \\
\min\limits_{\theta\in\Theta}\; - \E_{P_\theta} h_\gamma (x) & \text{with $\gamma$ fixed}. \\
\end{array} \right.
\end{align}
The two objectives are chosen to stabilize gradients in both $\theta$ and $\gamma$ during training.
The geometric GAN in \cite{LY17} or, equivalently, the energy-based GAN in \cite{ZML17} as shown in \cite{TSO19}, is defined by solving
\begin{align} \label{eq:geo-gan}
\left\{ \begin{array}{ll}
\max\limits_{\gamma\in\Gamma}\; K_{\HG} (\theta,\gamma) & \text{with $\theta$ fixed},  \\
\min\limits_{\theta\in\Theta}\; - \E_{P_\theta} h_\gamma(x) & \text{with $\gamma$ fixed}. \\
\end{array} \right.
\end{align}
Interestingly, the second line in (\ref{eq:cal-rkl-gan}) or (\ref{eq:geo-gan}) involves the same objective $- \E_{P_\theta} h_\gamma(x) $,
which can be equivalently replaced by $K_{\mathrm{rKL}}(P_*, P_\theta; h_\gamma)$ because $\gamma$ and hence $h_\gamma$ are fixed.

\begin{rem} \label{rem:nash}
We discuss precise definitions for a solution to a minimax problem such as (\ref{eq:logit-fgan}) or (\ref{eq:hinge-gan}),
and a solution to an alternating optimization problem such as (\ref{eq:gan-logD})--(\ref{eq:geo-gan}).
For an objective function $K(\theta,\gamma)$, we say that $(\hat\theta,\hat\gamma)$ is a solution to
\begin{align}\label{eq:nested-opt}
\min_\theta \max_\gamma \; K(\theta,\gamma),
\end{align}
if $K(\hat\theta,\hat\gamma) = \max_\gamma K(\hat\theta, \gamma) \le \max_\gamma K(\theta,\gamma)$ for any $\theta$.
In other words, we treat (\ref{eq:nested-opt}) as nested optimization:
$\hat\theta$ is a minimizer of $K(\theta, \hat\gamma_\theta)$ as a function of $\theta$ and $\hat \gamma = \hat\gamma_{\hat\theta}$,
where $\hat\gamma_\theta$ is a maximizer of $K(\theta, \gamma)$ for fixed $\theta$.
This choice is directly exploited in both numerical implementation and theoretical analysis of our methods later.
For two objective functions $K_1 (\theta,\gamma)$ and $K_2 (\theta,\gamma)$, we say that
$(\hat\theta,\hat\gamma)$ is a solution to the alternating optimization problem
\begin{align} \label{eq:alternating-opt}
\left\{ \begin{array}{ll}
\max\limits_{\gamma} \; K_1(\theta,\gamma) & \text{with $\theta$ fixed},  \\
\min\limits_{\theta} \; K_2(\theta, \gamma) & \text{with $\gamma$ fixed},\\
\end{array} \right.
\end{align}
if $ K_1 (\hat\theta,\hat\gamma) = \max_\gamma K_1(\hat\theta, \gamma)$ and
 $ K_2 (\hat\theta,\hat\gamma) = \min_\theta K_2(\theta, \hat\gamma)$.
In the special case where $K_1(\theta, \gamma)=K_2(\theta, \gamma)$, denoted as $K(\theta, \gamma)$,
a solution $(\hat\theta,\hat\gamma)$ to (\ref{eq:alternating-opt}) is also called a Nash equilibrium of $K(\theta,\gamma)$,
satisfying $ K(\hat\theta,\hat\gamma) = \max_\gamma K(\hat\theta, \gamma) = \min_\theta K (\theta, \hat\gamma)$.
A solution to minimax problem (\ref{eq:nested-opt}) and a Nash equilibrium of $K(\theta,\gamma)$ may in general differ from each other,
although they coincide by Sion's minimax theorem %(\cite{Sio58})
in the special case where $K(\theta,\gamma)$ is convex in $\theta$ for each $\gamma$
and concave in $\gamma$ for each $\theta$.
Hence our treatment of GANs in the form (\ref{eq:nested-opt}) as nested optimization should be distinguished
from existing studies where GANs are interpreted as finding  Nash equilibria (\cite{FO20}; \cite{JNJ20}).
%JNJ: What is Local Optimality in Nonconvex-Nonconcave Minimax Optimizationhttps://www.overleaf.com/project/62a3549370384484cfddc56d
%FO: Do GANs always have Nash equilibria
\end{rem}

\begin{rem} \label{rem:f-gan}
The population $f$-GAN (\cite{NCT16}) is defined by solving
\begin{align}
& \min _{\theta\in\Theta} \max_{\gamma\in\Gamma}  \; \left\{E_{P_*} T_\gamma (x) -E_{P_\theta} f^*(T_\gamma (x)) \right\} , \label{eq:fgan}
\end{align}
where $f^*$ is the Fenchel conjugate of $f$, i.e., $f^*(s) = \sup_{t\in(0,\infty)} (st - f(t))$ and $T_\gamma$ is a function taking values in the domain of $f^*$.
Typically, $T_\gamma$ is represented as $T_\gamma(x) = \tau_f ( h_\gamma(x))$, where $\tau_f: \bbR \to \dom(f^*)$ is an activation function
and $h_\gamma(x)$ take values unrestricted in $\bbR$. The logit $f$-GAN corresponds to $f$-GAN with the specific choice
$\tau_f(u ) =  f^\prime (\me^u)$ by the relationship $f^* (f^\prime (t) ) = f^\# (t)$ (\cite{TSO19}).
Nevertheless, a benefit of logit $f$-GAN is that the objective $K_f $ in (\ref{eq:logit-fgan}) takes the explicit form of a negative discrimination loss such that $h_\gamma(x)$ can be seen
to approximate the log density ratio between $P_*$ and $P_\theta$.
\end{rem}

\begin{rem} \label{rem:tv-gan}
There is an important difference between hinge GAN and logit $f$-GAN, although the total variation is also an $f$-divergence with $f(t)=|t-1|/2$.
In fact, taking this choice of $f$ in logit $f$-GAN (\ref{eq:logit-fgan}) yields
\begin{align}
& \min_{\theta\in\Theta} \max_{\gamma\in\Gamma}  \;  \left\{E_{P_*} \sign(h_\gamma(x)) -E_{P_\theta} \sign(h_\gamma(x)) \right\}. \label{eq:tv-learning}
\end{align}
This is called TV learning and is related to depth-based estimation in \cite{GLYZ18}. Compared with hinge GAN in (\ref{eq:hinge-gan}),
program (\ref{eq:tv-learning}) is computationally more difficult to solve. Such a difference also exists in the application of general $f$-GAN to the total variation.
For the total variation distance scaled by $2$ with $f(t) = |t-1|$, the conjugate is $f^*(s) = \max(-1,s)$ if $s \le 1$ or $\infty$ if $s >1$.
If $T_\gamma $ is specified as $T_\gamma=\min(1, h_\gamma(x))$, then the objective in $f$-GAN (\ref{eq:fgan}) can be shown to be
\begin{align*}
E_{P_*} \min(1, h_\gamma (x)) + \E_{P_\theta} \min(1, \max(-1, -h_\gamma (x)) ),
\end{align*}
which in general differs from the negative hinge loss in (\ref{eq:hinge-gan}) unless $h_\gamma$ is upper bounded by 1.
If $h_\gamma$ is specified as $2 \,\sigm(\tilde h_\gamma) -1$ for a function $\tilde h_\gamma$ taking values unrestricted in $\bbR$,
the resulting $f$-GAN is equivalent to TV-GAN in \cite{GLYZ18} defined by solving
\begin{align}
& \min _{\theta\in\Theta} \max_{\gamma\in\Gamma}  \; \left\{E_{P_*} \sigm(\tilde h_\gamma (x)) -E_{P_\theta} \sigm( \tilde h_\gamma(x)) \right\} . \label{eq:tv-gan}
\end{align}
However, solving program (\ref{eq:tv-gan}) is numerically intractable as discussed in \cite{GLYZ18}.
\end{rem}

\section{Theory and methods}

We propose and study various adversarial algorithms with simple spline discriminators for robust estimation in a multivariate Gaussian model.
Assume that $X_1,\ldots,X_n$ are  independent observations obtained from Huber's $\epsilon$-contamination model,
that is, the data distribution $P_*$ is of the form
\begin{align}
P_\epsilon = (1-\epsilon) P_{\theta^*} + \epsilon Q , \label{eq:huber}
\end{align}
where $P_{\theta^*}$ is $\N(\mu^*,\Sigma^*)$ with unknown $\theta^* = (\mu^*,\Sigma^*)$, $Q$ is a probability distribution for contaminated data, and $\epsilon $ is
a contamination fraction. Both $Q$ and $\epsilon$ are unknown. The dependency of $P_\epsilon$ on $(\theta^*,Q)$ is suppressed in the notation.
Equivalently, the data $(X_1, \ldots, X_n)$ can be represented in a latent model: $(U_1,X_1), \ldots, (U_n,X_n)$ are independent, and $U_i$
is Bernoulli with $\pr(U_i=1) = \epsilon$
and $X_i$ is drawn from $P_{\theta^*}$ or $Q$ given $U_i=0$ or 1 for $i=1,\ldots,n$.

For theoretical analysis, we consider two choices of the parameter space.
The first choice is
$\Theta_1 = \{ (\mu,\Sigma): \mu \in \bbR^p, \Sigma \succeq 0, \|\Sigma\|_{\max} \le M_1 \}$ for a constant $M_1 >0$.
Equivalently, the diagonal elements of $\Sigma$ is upper bounded by $M_1$ for $(\mu,\Sigma)\in \Theta_1$.
The second choice is
$\Theta_2 = \{ (\mu,\Sigma): \mu \in \bbR^p, \Sigma \succeq 0, \|\Sigma\|_{\op} \le M_2 \}$ for a constant $M_2 >0$.
For simplicity, the dependency of $\Theta_1$ on $M_1$ or $\Theta_2$ on $M_2$ is suppressed in the notation.
For the second parameter space $\Theta_2$, the minimax rates in the $L_2$ and operator norms have been shown
to be achieved using matrix depth (\cite{CGR18}) and GANs with certain neural network discriminators (\cite{GYZ20}).

Our work aims to investigate adversarial algorithms with a simple linear class of spline discriminators for computational tractability,
and establish various error bounds for the proposed estimators, including those
matching the minimax rates in the {maximum} norms for the location and scatter estimation over $\Theta_1$,
and, provided that $\epsilon\sqrt{n}$ is bounded by a constant (independent of $p$), the minimax rates in the $L_2$ and Frobenius norms over $\Theta_2$.

\subsection{Population analysis with nonparametric discriminators} \label{sec:population}

A distinctive feature of GANs is that they can be motivated as approximations to minimum divergence estimation.
For example, if the discriminator class $\{h_\gamma\}$ in (\ref{eq:logit-fgan}) is rich enough to include the nonparametrically optimal discriminator
such that $\max_{\gamma\in\Gamma} K_f( P_*, P_\theta; h_\gamma) = D_f ( P_* \| P_\theta)$ for each $\theta$,
then the (population) logit $f$-GAN amounts to minimizing the $f$-divergence $D_f ( P_* \| P_\theta)$.
Similarly, if the discriminator class $\{h_\gamma\}$ in (\ref{eq:hinge-gan}) is sufficiently rich, then the (population) hinge GAN
amounts to minimizing the total variation $D_{\TV} ( P_* \| P_\theta)$.

As a prelude to our sample analysis, Theorem~\ref{thm:pop-robust} shows that at the population level, minimization of the total variation and certain $f$-divergences
satisfying Assumption~\ref{ass:f-condition}
achieves robustness under Huber's contamination model,
in the sense that the estimation errors are respectively $O(\epsilon)$ and $O( \sqrt{\epsilon})$, uniformly over all possible $Q$.
Hence with sufficiently rich (or nonparametric) discriminators, the population versions of the hinge GAN and certain $f$-GANs
can be said to be robust under Huber's contamination.
From Table~\ref{tab:1}, Assumption~\ref{ass:f-condition} is satisfied by
the reverse KL, JS, %reverse $\chi^2$,
and squared Hellinger divergences, but violated by the KL divergence.
Minimization of the KL divergence corresponds to maximum likelihood estimation, which is known to be non-robust under Huber's contamination model.

\begin{ass} \label{ass:f-condition}
Suppose that $f: (0,\infty) \to \bbR$ is convex with $f(1)=0$ and satisfies the following conditions.
\begin{itemize}
\item[(i)] $f$ is twice differentiable with $f^\dprime(1) >0$.
\item[(ii)] $f$ is non-increasing.
\item[(iii)] $f^\prime$ is concave (i.e., $f^\dprime$ is non-increasing)
\end{itemize}
See Table \ref{tab:1} for validity of conditions (ii) and (iii) in various $f$-divergences.
\end{ass}

\begin{rem} \label{rem:equi-def}
Given a convex function $f$ with $f(1)=0$, the same $f$-divergence $D_f$ can be defined using the convex function
$f(t) + c(t-1)$ for any constant $c\in \bbR$. Hence condition (ii) in Assumption \ref{ass:f-condition}
can be relaxed such that $f^\prime$ is upper bounded by a constant.
The non-increasingness of $f$ is stated above for ease of interpretation.
The other conditions in Assumption~\ref{ass:f-condition} and Assumption~\ref{ass:f-condition2} are not affected by non-unique choices of $f$.
\end{rem}

\begin{thm}  \label{thm:pop-robust}
Let $\Theta_0 = \{ (\mu,\Sigma): \mu \in \bbR^p, \Sigma \text{ is a $p\times p$ variance matrix} \}$.

(i) Assume that $f$ satisfies Assumption \ref{ass:f-condition}. Let $\bar\theta =\argmin_{\theta \in \Theta_0} D_{f}(P_{\epsilon}||P_{\theta})$.
If $\sqrt{-2 (f^{\dprime}(1))^{-1}f^{\prime}(1/2)\epsilon}+\epsilon < 1/2$, then  {for any contamination distribution $Q$,}
\begin{align}
\|\bar\mu - \mu^*\|_2 \le C\|\Sigma^*\|_{\op}^{1/2}\sqrt{\epsilon},\quad
\|\bar\mu - \mu^*\|_{\infty} \le C\|\Sigma^*\|_{\max}^{1/2}\sqrt{\epsilon},  \label{eq:thm-pop-robust-mu}
\end{align}
and
\begin{align}
\|\bar\Sigma - \Sigma^*\|_{\op} \le  C\|\Sigma^*\|_{\op}\sqrt{\epsilon}, \quad
\|\bar\Sigma - \Sigma^*\|_{\max}  \le  C\|\Sigma^*\|_{\max}\sqrt{\epsilon},   \label{eq:thm-pop-robust-sigma}
\end{align}
where $C > 0$ is a constant depending only on $f$.
The same inequality as in (\ref{eq:thm-pop-robust-sigma}) also holds with $\|\bar\Sigma - \Sigma^*\|_{\op}$ replaced by
$ p^{-1/2} \|\bar\Sigma - \Sigma^*\|_{\fro}$.

(ii) Let $\bar\theta = \argmin_{\theta \in \Theta_0} D_\TV(P_{\epsilon}||P_{\theta})$. If $\epsilon < 1/4$ then
(\ref{eq:thm-pop-robust-mu}) and (\ref{eq:thm-pop-robust-sigma}) hold for an absolute constant $C>0$ with $\sqrt{\epsilon}$ replaced by $\epsilon$ throughout.
\end{thm}
\begin{figure}
    \includegraphics[scale=0.1]{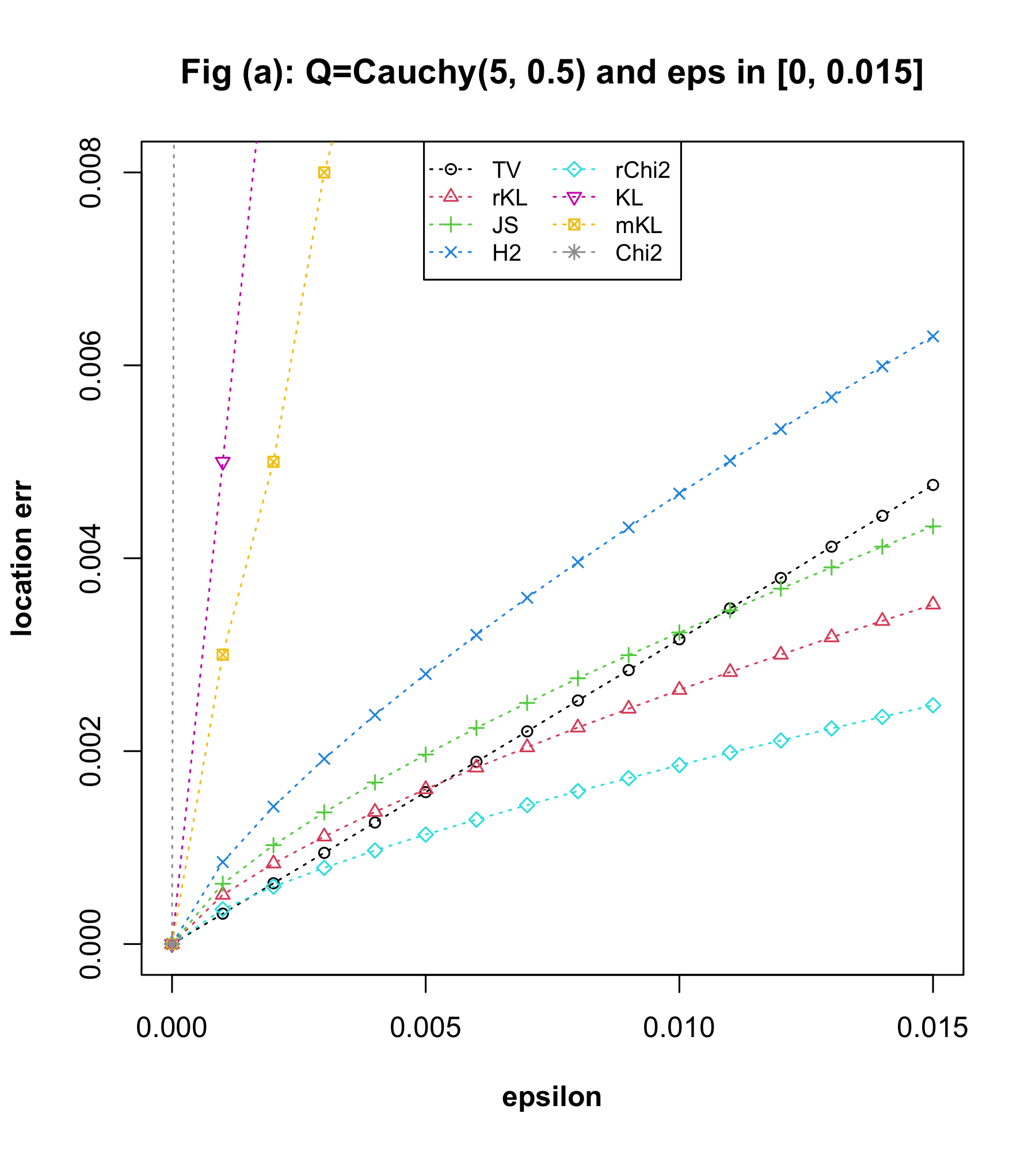}
    \includegraphics[scale=0.1]{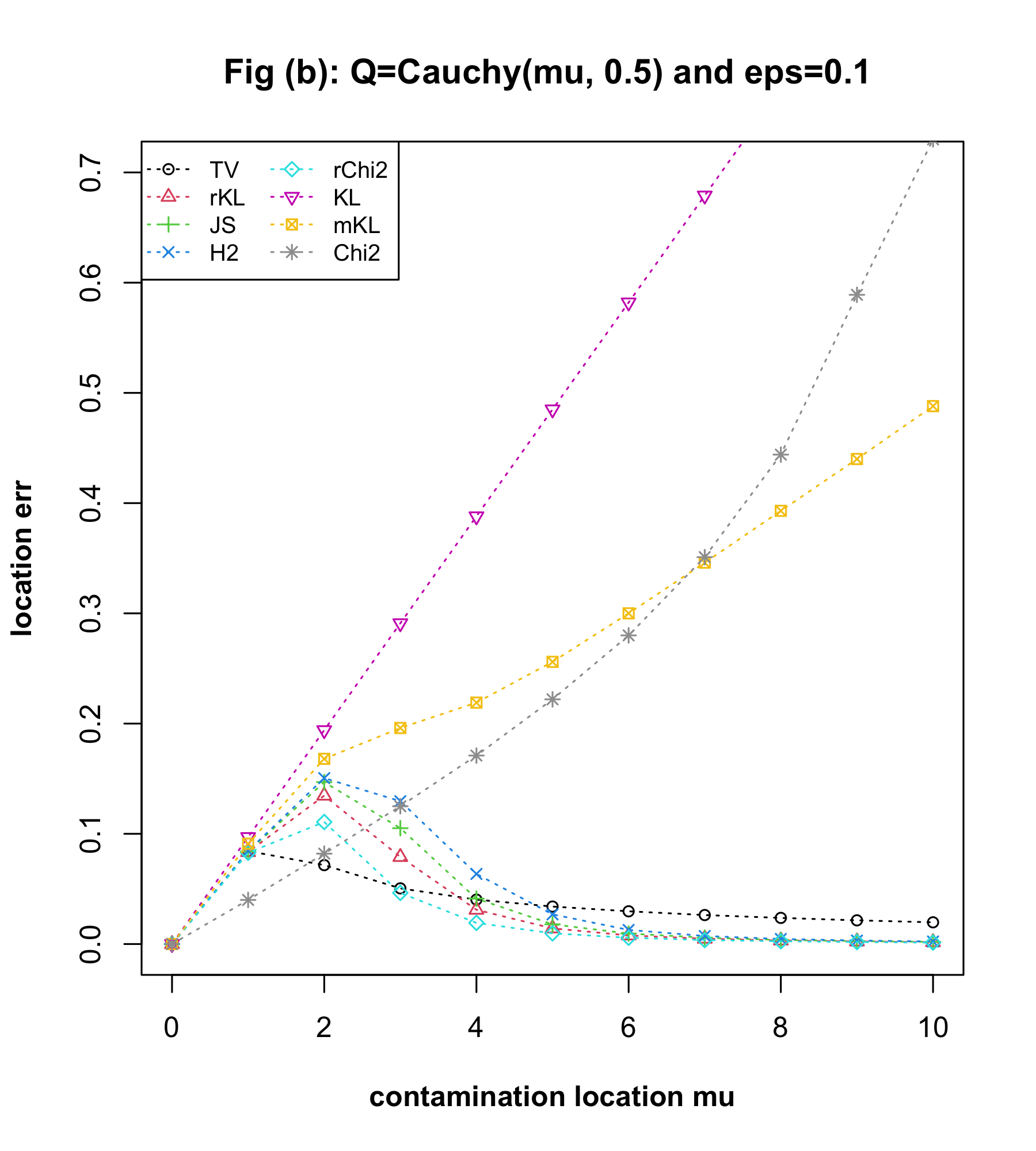}
    \caption{Illustration of robustness of minimum $f$-divergence location estimation. Figure (a): Location error $|\bar\mu - \mu^*|$ against contamination fraction $\epsilon$ from $0$ to $0.015$, with $P_{\theta^*}$ being $\N(0, 1)$ and contamination $Q$ being $\N(5, 1/4)$ fixed; Figure (b): Location error $|\bar\mu - \mu^*|$ against contamination location $\mu_Q$  from $0$ to $10$, with $P_{\theta^*}$ being $\N(0,1)$, contamination fraction $\epsilon = 0.1$ fixed, and contamination $Q$ being $\N(\mu_Q, 1/4)$.
    The squared Hellinger, reverse $\chi^2$, and mixed KL are denoted by H2, rChi2, and mKL respectively.}
    \label{fig:1}
\end{figure}

Figure \ref{fig:1} provides a simple numerical illustration.
From Figure~\ref{fig:1}(a), the location errors $|\bar\mu - \mu^*|$ of minimum divergence estimators corresponding to the four robust $f$-divergences (reverse KL, JS, squared Hellinger, and reverse $\chi^2$) satisfying Assumption \ref{ass:f-condition} are of shapes in agreement with the order $\sqrt{\epsilon}$ in Theorem \ref{thm:pop-robust},
whereas those corresponding to TV appear to be linear in $\epsilon$, for $\epsilon$ close to 0.
For the KL, mixed KL, and $\chi^2$ divergences, which do not satisfy Assumption~\ref{ass:f-condition}(ii), their corresponding errors quickly increase out of the plotting range, indicating non-robustness of the associated minimum divergence estimation.
The differences between robust and non-robust $f$-divergences are further demonstrated in Figure \ref{fig:1}(b).
As the contamination location moves farther away, the errors of the robust $f$-divergences increase initially but then decrease to near 0,
whereas those of the non-robust $f$-divergences appear to increase unboundedly.

\begin{table}[]
\begin{threeparttable}
\caption{Common $f$-divergences and validity of Assumptions \ref{ass:f-condition} (ii)--(iii) and \ref{ass:f-condition2} (i)--(ii).}
\label{tab:1}
\begin{tabular}{llllll}
\hline
\multirow{2}{*}{Name} & \multirow{2}{*}{Convex $f(t)$} & Non-incr. & Concave & Concave & Lipschitz \\
& & $f(t)$ & $f^\prime(t)$ & $f^\prime(\me^u)$ & $f^\# (\me^u)$ \\
\hline
Total variation  & $(1-t)_+, |t-1|/2$ & \checkmark & --- & --- & --- \\
Reverse KL & $-\log t$ & \checkmark & \checkmark & \checkmark & \checkmark \\
Jensen-Shannon  & $t\log t-(t+1)\log(t+1)+\log 4$ & \checkmark & \checkmark & \checkmark & \checkmark \\
Squared Hellinger & $(\sqrt{t}-1)^2$ & \checkmark & \checkmark & \checkmark & \\
Reverse $\chi^2$  & $t^{-1}-1$ & \checkmark & \checkmark & \checkmark &                 \\
\hline
KL & $t\log t$ & & \checkmark & \checkmark & \\
Mixed KL & $\{(t-1)\log t\}/2$ & & \checkmark & \checkmark & \\
$\chi^2$ & $(t-1)^2$ &  & \checkmark &  & \\ \hline
\end{tabular}
Note: The mixed KL divergence is defined as $D_{\mathrm{mKL}}(P||Q) = D_{\mathrm{KL}}(P||Q)/2 + D_{\mathrm{KL}}(Q||P)/2$.
\end{threeparttable}

\end{table}

\begin{rem} \label{rem:pop-robust-ass1}
From the proof in Section \ref{sec:prf-thm-pop-robust}, Theorem~\ref{thm:pop-robust}(i) remains valid if $f^\dprime(1)$
is replaced by $C_f =\inf_{t\in (0,1]} f^\dprime (t)$ in Assumption \ref{ass:f-condition}(i) and the definition of $\err_{f0}(\epsilon)$,
and Assumption \ref{ass:f-condition}(iii), the concavity of $f^\prime$, is removed.
On the other hand, a stronger condition than Assumption \ref{ass:f-condition}(iii) is used in our sample analysis:
for convex $f$, the concavity of $f^\prime$ is implied by Assumption \ref{ass:f-condition2}(i), as discussed in Remark \ref{rem:f-condition2-a}.
\end{rem}

\begin{rem} \label{rem:pop-robust-adaptive}
The population bounds in Theorem \ref{thm:pop-robust} are more refined than those in our sample analysis later.
The population minimizer $\bar \theta = (\bar\mu, \bar\Sigma)$ is defined by minimization over the unrestricted space $\Theta_0$
instead of $\Theta_1$ or $\Theta_2$ with the restriction $\| \Sigma \|_{\max} \le M_1$ or $\| \Sigma\|_\op \le M_2$.
The scaling factors in the population bounds also depend directly on the maximum or operator norm of
the true variance matrix $\Sigma^*$, instead of pre-specified constants $M_1$ or $M_2$.
 {Note that the parameter space is also restricted such that $\| \Sigma\|_\op \le M_2$ and the error bounds depend on $M_2$
in sample analysis of \cite{GYZ20}.}
Nevertheless, the population bounds share a similar feature as in our sample bounds later:
the error bounds in the maximum norms are governed by $\| \Sigma^* \|_{\max}$, which
can be much smaller than $\| \Sigma^* \|_\op$ involved in the error bounds in the operator norm.
\end{rem}

\begin{rem} \label{rem:pop-robust-DL}
It is interesting to connect and compare our results with \cite{DL88},
where minimum distance (MD) estimation is studied, that is, minimization of a proper distance $D(P, P_\theta)$ satisfying the triangle inequality.
For minimum TV estimation, let $\bar\theta_P =(\bar\mu_P, \bar \Sigma_P) = \argmin_\theta D_\TV ( P \| P_\theta )$.
For location estimation, define
\begin{align*}
& b (\epsilon) = \sup_{P: D_\TV (P \| P_{\theta^*}) \le \epsilon } \| \bar\mu_P - \mu^* \|_2 , \quad
& b_0 (\epsilon) = \sup_{\theta: D_\TV (P_\theta \| P_{\theta^*}) \le \epsilon } \| \mu - \mu^* \|_2 ,
\end{align*}
which are called the bias distortion curve and the gauge function.
Scatter estimation can be discussed in a similar manner.
For a general family $\{P_\theta\}$, the first half in our proof of Theorem \ref{thm:pop-robust}(ii) shows that
for any $P$ satisfying $D_\TV (P \| P_{\theta^*}) \le \epsilon$, we have
$D_\TV ( P_{\bar\theta_P} \| P_{\theta^*} ) \le 2\epsilon$.
This implies a bound similar to Proposition 5.1 in \cite{DL88}:
\begin{align}
b (\epsilon ) \le b_0 ( 2 \epsilon) . \label{eq:DL}
\end{align}
%In fact, for any $P$ satisfying $D_\TV (P \| P_{\theta^*}) \le \epsilon$, we have
%\begin{align*}
%D_\TV ( P_{\bar\theta_P} \| P_{\theta^*} ) & \le D_\TV ( P \| P_{\theta^*}) + D_\TV ( P \| P_{\bar\theta_P} ) \\
%& \le 2 D_\TV ( P \| P_{\theta^*})  \le 2 \epsilon,
%\end{align*}
%where the first line uses the triangle inequality and the second line
%uses $D_\TV ( P \| P_{\bar\theta_P} ) \le D_\TV ( P \| P_{\theta^*})$ by the definition of $\bar\theta_P$.
%Then (\ref{eq:DL}) follows by the definition of $b_0(\epsilon)$.
For the multivariate Gaussian family $\{P_\theta\}$, the second half in our proof of Theorem \ref{thm:pop-robust}(ii) derives an explicit upper bound on $b_0(\epsilon)$ provided that
$ 2\epsilon \le a $ for a constant $a \in [0, 1/2)$:
\begin{align*}
b_0 (2 \epsilon) \le S_{1,a}\|\Sigma^*\|_{\op}^{1/2} (2\epsilon),
\end{align*}
%See Proposition \ref{pro:local-linear-pop} in the Supplement.
where $S_{1,a}=\{ \Phi^\prime(\Phi^{-1}(1/2+a)) \}^{-1}$. Combining the preceding inequalities yields
$ b(\epsilon) \le C\|\Sigma^*\|_{\op}^{1/2}\epsilon$ in Theorem \ref{thm:pop-robust}(ii), with $C = 2 S_{1,a}$.
In addition, Proposition 5.1 in \cite{DL88} gives the same bound as (\ref{eq:DL}) for MD estimation using certain other distances $D ( P, P_\theta )$,
including the Hellinger distance, where the MD functional
$\bar\theta_P =(\bar\mu_P, \bar \Sigma_P)$ is defined as $ \argmin_\theta D ( P, P_\theta )$,
and $b (\epsilon)$ and $b_0 (\epsilon)$ are defined with $D_\TV (P \| P_{\theta^*})$ replaced by $D ( P, P_\theta )$.
The distances used in defining the MD functional and the contamination neighborhood are tied to each other.
Hence, except for minimum TV estimation, our setting differs from \cite{DL88} in studying different choices of minimum $f$-divergence estimation
over the same Huber's contamination neighborhood.

%\begin{align*}
%& b_2 (\epsilon) = \sup_{P: D_\TV (P \| P_{\theta^*}) \le \epsilon } \| \Sigma^{*-1} \bar\Sigma_P - I \|_\op , \quad
%& b_{02} (\epsilon) = \sup_{P_\theta: D_\TV (P_\theta \| P_{\theta^*}) \le \epsilon } \| \Sigma^{*-1} \Sigma - I \|_\op,
%\end{align*}
\end{rem}

\begin{rem} \label{rem:breakdown}
 {We briefly comment on how our result is related to breakdown points in robust statistics (\cite{Hub81}, Section 1.4).}
For estimating $\mu^*$, the population breakdown point of a functional $T = T(P)$ can be defined as
$ \sup\{ \epsilon :  b_T (\epsilon) < \infty \}$,
where $b_T (\epsilon) =  \sup_{P: D_\TV (P \| P_{\theta^*}) \le \epsilon} \| T(P)-\mu^* \|_2 $.
Scatter estimation can be discussed in a similar manner.
For $T$ defined from minimum TV estimation, Theorem \ref{thm:pop-robust}(ii) shows that if $\epsilon < 1/4$, then
$ b_T (\epsilon) \le C\|\Sigma^*\|_{\op}^{1/2}\epsilon$, as noted in Remark \ref{rem:pop-robust-DL}.
This not only provides an explicit bound on $b_T(\epsilon)$, but also implies that the population breakdown point is at least $1/4$ for minimum TV estimation.
Similar implications can be obtained from Theorem \ref{thm:pop-robust}(i) for minimum $f$-divergence estimation.
For $T$ defined from minimum rKL divergence estimation, Theorem \ref{thm:pop-robust}(i) shows that
if $2\sqrt{\epsilon} + \epsilon <1/2$, then $ b_T (\epsilon) \le C\|\Sigma^*\|_{\op}^{1/2} \sqrt\epsilon$,
and hence the population breakdown point is at least $0.051$.
While these estimates of breakdown points can potentially be improved,
our population analysis as well as sample analysis in the subsequent sections focus on deriving
quantitative error bounds in terms of sufficiently small $\epsilon$ and some scaling constants free of $\epsilon$.
\end{rem}

\subsection{Logit $f$-GAN with spline discriminators}  \label{sec:logit-fgan}

For the population analysis in Section~\ref{sec:population}, a discriminator class is assumed to be rich enough to include the nonparametrically optimal discriminator which depends on unknown $(\epsilon,Q)$.
Because $Q$ can be arbitrary, this nonparametric assumption is inappropriate for sample analysis.
 {Recently, GANs with certain neural network discriminators are shown to achieve
sample error bounds matching minimax rates (\cite{GLYZ18, GYZ20}).
It is interesting to study whether similar results can be obtained when using GANs with simpler
and computationally more tractable discriminators.}

We propose and study adversarial algorithms, including logit $f$-GAN in this section and
hinge GAN in Section {\ref{sec:hinge-gan}}, each with simple spline discriminators.
Define a linear class of pairwise spline functions, denoted as $\mathcal H_{\sp}$:
\begin{align*}
h_{\sp,\gamma}(x)=
%\gamma_0 + \sum_{j=1}^p \sum_{l=1}^5 \gamma_{1,jl} (x_j- \xi_l)_+
%+ \sum_{1 \le j_1 \not= j_2 \le p}^p \sum_{1 \le l_1, l_2 \le 5} \gamma_{1,j_1j_2l_1l_2} (x_{j_1}- \xi_{l_1} )_+ (x_{j_2}- \xi_{l_2})_+
\gamma_0 + \gamma_1^\T \varphi(x) + \gamma_2^\T \vec( \varphi(x) \otimes \varphi(x) ),
\end{align*}
where $\gamma =(\gamma_0, \gamma_1^\T, \gamma_2^\T)^\T \in \Gamma$ with $\Gamma=\bbR^{1+5p+(5p)^2}$ and $\varphi(x)=  (\varphi_1^\T(x), \ldots, \varphi_5^\T (x))^\T$.
The basis vector $\varphi_l(x) \in \bbR^p$ is obtained by applying $t \mapsto (t-\xi_l)_+$ componentwise to $x \in \bbR^p$,
with the knot $\xi_l = -2, -1, 0, 1$, or $2$ for $l =1,\ldots,5$ respectively.
For concreteness, assume that every two components of $\gamma_2$ are identical if associated with the same product of two components of $\varphi(x)$,
that is, $\gamma_2$ can be arranged to a symmetric matrix.
The preceding specification is sufficient for our theoretical analysis.
Nevertheless, similar results can also be obtained, while allowing various changes to the basis functions, for example,
adding $x$ as a subvector to $\varphi(x)$.
With this change, a function in $\mathcal H_{\sp}$ has a main effect term in each $x_j$, which
is a linear spline with knots in $\{-2,-1,0,1,2\}$, and a square or interaction term in each pair $(x_{j_1},x_{j_2})$, which is a product
of two spline functions in $x_{j_1}$ and $x_{j_2}$ for $1\le j_1, j_2 \le p$.

We consider two logit $f$-GAN methods with an $L_1$ or $L_2$ penalty on the discriminator,
which lead to meaningful error bounds over the parameter space $\Theta_1$ or $\Theta_2$ respectively
under the following conditions on $f$, in addition to Assumption \ref{ass:f-condition}.
Among the $f$-divergences in Table \ref{tab:1},
the reverse KL and JS divergences satisfy both Assumptions \ref{ass:f-condition} and \ref{ass:f-condition2},
and hence the corresponding logit $f$-GANs achieve sample robust estimation using spline discriminators.
The squared Hellinger and reverse $\chi^2$ divergences satisfy Assumption \ref{ass:f-condition},
but not the Lipschitz condition in Assumption \ref{ass:f-condition2}(ii).
For such $f$-divergences, it remains a theoretical question whether sample robust estimation can be achieved using spline discriminators.

%The robustness of logit $f$-GANs based on these two $f$-divergences may be sufficient from a practical perspective.
%On the other hand, it remains a theoretical question whether alternative methods can be derived to achieve robust estimation
%for $f$-divergences satisfying Assumption \ref{ass:f-condition}
%but not Assumption \ref{ass:f-condition2} such as the squared Hellinger and reverse $\chi^2$ divergences.

\begin{ass} \label{ass:f-condition2}
Suppose that $f: (0,\infty) \to \bbR$ is strictly convex and three-times continuously differentiable with $f(1)=0$ and satisfies the following conditions.
\begin{itemize}
\item[(i)] $f^\prime(\me^u)$ is concave in $u \in \bbR$.
\item[(ii)] $f^\# (\me^u)$ is $R_1$-Lipschitz in $u \in \bbR$ for a constant $R_1 >0$.
\end{itemize}
See Table \ref{tab:1} for validity of conditions (i) and (ii) in various $f$-divergences,
and Remarks \ref{rem:f-condition2-a} and \ref{rem:f-condition2-b} for further discussions.
\end{ass}

The first method, $L_1$ penalized logit $f$-GAN, is defined by solving
\begin{align}
& \min_{\theta\in\Theta_1} \max_{\gamma\in\Gamma}  \; \left\{ K_f (P_n, P_\theta; h_{\gamma,\mu}) - \lambda_1 \;\pen_1 (\gamma) \right\},   \label{eq:logit-fgan-L1}
\end{align}
where $K_f (P_n, P_\theta; h)$ is  $K_f (P_*, P_\theta; h)$ in {(\ref{eq:logit-fgan})} with $P_*$ replaced by $P_n$,
$ h_{\gamma,\mu}(x) = h_{\sp,\gamma} ( x-\mu)$,
$\pen_1(\gamma) = \|\gamma_1\|_1 + \|\gamma_2\|_1$, the $L_1$ norm of $\gamma$ excluding the intercept $\gamma_0$,
and $\lambda_1 \ge 0$ is a tuning parameter.
In addition to the replacement of $P_*$ by $P_n$, there are two notable modifications in (\ref{eq:logit-fgan-L1}) compared with the population version (\ref{eq:logit-fgan}).
First, a penalty term is introduced on $\gamma$, to achieve suitable control of sampling variation.
Second, the discriminator $h_{\gamma,\mu}$ is a spline function with knots depending on $\mu$, the location parameter for the generator. By a change of variables,
the non-penalized objective in (\ref{eq:logit-fgan-L1}) can be equivalently written as
\begin{align}
 K_f (P_n, P_\theta; h_{\gamma,\mu}) = \E_{P_n-\mu} f^{\prime}(\me^{h_{\sp,\gamma}(x)}) -
 \E_{P_{0,\Sigma}} f^\#(e^{h_{\sp,\gamma}(x)}),   \label{eq:logit-fgan-L1-transf}
 % \left\{ \me^{h_{\sp,\gamma}(x)}f^{\prime}(\me^{h_{\sp,\gamma}(x)})-f(e^{h_{\sp,\gamma}(x)}) \right\} ,
\end{align}
where $P_n -\mu$ denotes the empirical distribution on $\{X_1-\mu, \ldots, X_n-\mu\}$.
Hence $K_f (P_n, P_\theta; h_{\gamma,\mu})$ is a negative loss for discriminating between the shifted empirical distribution $P_n-\mu$ and the mean-zero generator $P_{0,\Sigma}$.
{The adaptive choice of knots for the spline discriminator $h_{\gamma,\mu}$
not only is numerically desirable but also facilitates the control of sampling variation in our theoretical analysis.
See Propositions~\ref{pro:logit-L1-upper}, \ref{pro:logit-L2-upper}, \ref{pro:hinge-L1-upper}, and \ref{pro:hinge-L2-upper}.}

\begin{thm} \label{thm:logit-L1}
Assume that $\| \Sigma^* \|_{\max}\le M_1$ and $f$ satisfies Assumptions~\ref{ass:f-condition}--\ref{ass:f-condition2}.
Let $\hat \theta=(\hat\mu,\hat\Sigma)$ be a solution to (\ref{eq:logit-fgan-L1}). For $\delta <1/7$, if $\lambda _1 \ge C_1\left(\sqrt{\log{p}/n} + \sqrt{\log(1/\delta)/n}\right)$ and $\sqrt{\epsilon}+\sqrt{1/(n\delta)} + \lambda_1 \le C_2 $, then with probability at least $1-7\delta$
the following bounds hold  {uniformly over contamination distribution $Q$,}
\begin{align*}
 \| \hat\mu - \mu^* \|_\infty & \le C\left( \sqrt{\epsilon} + \sqrt{1/(n\delta)} + \lambda_1 \right), \\
 \| \hat\Sigma - \Sigma^* \|_{\max} & \le  C\left( \sqrt{\epsilon} + \sqrt{1/(n\delta)} + \lambda_1 \right),
\end{align*}
where $C_1, C_2, C > 0$ are constants, depending on $M_1$ and $f$ but independent of $(n, p, \epsilon, \delta)$.
\end{thm}

For $L_1$ penalized logit $f$-GAN,
Theorem \ref{thm:logit-L1} shows that the estimator $(\hat\mu,\hat\Sigma)$
achieves error bounds in the {maximum} norms in the order $\sqrt\epsilon + \sqrt{\log(p)/n}$.
These error bounds match sampling errors of order $\sqrt{\log (p)/n}$ in the maximum norms for the standard estimators
(i.e., the sample mean and variance) in a multivariate Gaussian model in the case of $\epsilon=0$.
Moreover, up to sampling variation, the error bounds also match the population error bounds of order $\sqrt{\epsilon}$ in the maximum norms
with nonparametric discriminators in Theorem~\ref{thm:pop-robust}(i),
even though a simple, {\it linear} class of spline discriminators is used.

The second method, $L_2$ penalized logit $f$-GAN, is defined by solving
\begin{align}
& \min_{\theta\in\Theta_2} \max_{\gamma\in\Gamma}  \; \left\{ K_f (P_n, P_\theta; h_{\gamma,\mu}) - \lambda_2 \;\pen_2(\gamma_1) - \lambda_3 \;\pen_2(\gamma_2) \right\},   \label{eq:logit-fgan-L2}
\end{align}
where $K_f (P_n, P_\theta; h)$  and $ h_{\gamma,\mu}(x)$ are defined as in (\ref{eq:logit-fgan-L1}),
$\pen_2(\gamma_1) = \|\gamma_1\|_2$ and $\pen_2(\gamma_2) =\|\gamma_2\|_2$, the $L_2$ norms of $\gamma_1$ and $\gamma_2$,
and $\lambda_2 \ge 0$ and $\lambda_3 \ge 0$ are tuning parameters.
Compared with $L_1$ penalized logit $f$-GAN (\ref{eq:logit-fgan-L1}),
the $L_2$ norms of $\gamma_1$ and $\gamma_2$ are separately associated with tuning parameters $\lambda_2$ and $\lambda_3$ in (\ref{eq:logit-fgan-L2}),
in addition to the change from $L_1$ to $L_2$ penalties.
As seen from our proofs in Sections~\ref{sec:logit-L2} and \ref{sec:hinge-L2} in the Supplement,
the use of separate tuning parameters $\lambda_2$ and $\lambda_3$ is crucial for achieving meaningful error bounds in the $L_2$ and Frobenius norms
for simultaneous estimation of $(\mu^*,\Sigma^*)$. Our method does not rely on the use of normalized differences of pairs of the observations to
reduce the unknown mean to 0 for scatter estimation as in \cite{DKKLMS19}.

\begin{thm} \label{thm:logit-L2}
Assume that $\| \Sigma^* \|_\op \le M_2$, $f$ satisfies Assumptions~\ref{ass:f-condition}--\ref{ass:f-condition2}, and $p\epsilon$ is upper bounded by a constant $B$. Let $\hat \theta=(\hat\mu,\hat\Sigma)$ be a solution to (\ref{eq:logit-fgan-L2}). For $\delta <1/8$, if
$\lambda_2 \ge C_1\left(\sqrt{p/n} + \sqrt{\log(1/\delta)/n}\right)$, $\lambda_3 \ge C_1\sqrt{p}\left(\sqrt{p/n} + \sqrt{\log(1/\delta)/n}\right)$,
and $\sqrt{\epsilon} + \sqrt{1/(n\delta)} + \lambda_2 \le C_2$, then with probability at least $1- {8}\delta$
the following bounds hold  {uniformly over contamination distribution $Q$,}
\begin{align*}
 \| \hat\mu - \mu^* \|_2 & \le C\left(\sqrt{\epsilon} + \sqrt{1/(n\delta)} + \lambda_2\right), \\
 %\| \hat\sigma - \sigma^* \|_2 & \le S_{6,a}  \err_{f2} (n,p, \delta, \epsilon), \\
 p^{-1/2}\| \hat\Sigma - \Sigma^* \|_\fro
& \le C\left(\sqrt{\epsilon} + \sqrt{1/(n\delta)} + \lambda_2 + \lambda_3/\sqrt{p}\right),
\end{align*}
where $C_1, C_2, C > 0$ are constants, depending on $M_2$ and $f$ but independent of $(n, p, \epsilon, \delta)$ except through the bound $B$ on $p \epsilon$.
\end{thm}

For $L_2$ penalized logit $f$-GAN, Theorem \ref{thm:logit-L2} provides error bounds of order
$\sqrt\epsilon + \sqrt{p/n}$, in the $L_2$ and $p^{-1/2}$-Frobenius norms for location and scatter estimation.
A technical difference from Theorem \ref{thm:logit-L1} is that these bounds are derived under an extraneous condition that $p \epsilon$ is upper bounded.
Nevertheless, the error rate, $\sqrt\epsilon + \sqrt{p/n}$, matches
the population error bounds of order $\sqrt{\epsilon}$ in Theorem \ref{thm:pop-robust}(i), up to sampling variation of order $\sqrt{p/n}$
in the $L_2$ and $p^{-1/2}$-Frobenius norms.
We defer to Section \ref{sec:hinge-gan} further discussion about the error bounds in Theorems \ref{thm:logit-L1}--\ref{thm:logit-L2}
compared with minimax error rates.

\begin{rem} \label{rem:f-condition2-a}
There are important implications of Assumption \ref{ass:f-condition2}(i) together with Assumption \ref{ass:f-condition}(ii), based on the
fact (``composition rule'') that the composition of a non-decreasing concave function and a concave function is concave.
First, for convex $f$, concavity of $f^\prime (\me^u)$ in $u \in \bbR$ implies Assumption \ref{ass:f-condition}(iii), that is, concavity of $f^\prime (t)$ in $t \in (0,\infty)$.
This follows by writing $f^\prime (t) = g( \log t)$ and applying the composition rule, where $g(u) = f^\prime (\me^u)$, in addition to being concave,
 is non-decreasing by convexity of $f$, and $\log t$ is concave in $t$.
Note that  concavity of $f^\prime (t)$ in $t$ may not imply concavity of $f^\prime (\me^u)$ in $u$,
as shown by the Pearson $\chi^2$  in Table \ref{tab:1}.
Second, for convex and non-increasing $f$, concavity of $f^\prime (\me^u)$ in $u \in \bbR$ also implies concavity of $-f^\# (\me^u) $ in $u \in \bbR$.
In fact, as mentioned in Remark \ref{rem:f-gan}, $f^\# (t)$ can be equivalently obtained as $f^\# (t) = f^*( f^\prime (t))$,
where $f^*$ is the Fenchel conjugate of $f$ (\cite{TSO19}).
By the composition rule, $-f^\# (\me^u) = g( f^\prime(\me^u))$ is concave,
where $g  = - f^*$ is concave and non-decreasing by non-increasingness of $f$.
\end{rem}

\begin{rem} \label{rem:f-condition2-b}
The concavity of $f^\prime (\me^u)$ and $- f^\# (\me^u)$ in $u$
from Assumptions \ref{ass:f-condition}(ii) and \ref{ass:f-condition2}(i), as discussed in Remark \ref{rem:f-condition2-a},
is instrumental from both theoretical and computational perspectives.
These concavity properties are crucial to our proofs of Theorems \ref{thm:logit-L1}--\ref{thm:logit-L2} and Corollary \ref{cor:two-obj}(i).
See Lemmas \ref{lem:logit-upper} and \ref{lem:two-obj-centered-upr-bound} in the Supplement.
Moreover, the concavity of $f^\prime (\me^u)$ and $- f^\# (\me^u)$ in $u$,
in conjunction with the linearity of the spline discriminator $h_{\gamma,\mu}$ in $\gamma$,
indicates that the objective function $K_f (P_n, P_\theta; h_{\gamma,\mu})$
is concave in $\gamma$ for any fixed $\theta$.
Hence our penalized logit $f$-GAN (\ref{eq:logit-fgan-L1}) or (\ref{eq:logit-fgan-L2}) under Assumptions \ref{ass:f-condition}--\ref{ass:f-condition2}
can be implemented through nested optimization as shown in Algorithm \ref{alg:1},
where a concave optimizer is used in the inner stage to train the spline discriminators.
%Such convexity makes it numerically tractable to implement our methods through nested optimization.
See Remark \ref{rem:nash} for further discussion.
\end{rem}
\vspace{-20pt}
\begin{algorithm} \label{alg:1}
    \caption{Penalized logit $f$-GAN or hinge GAN}
    \KwSty{Require} A penalized GAN objective function $K(\theta,\gamma; \lambda)$ as in (\ref{eq:logit-fgan-L1}), (\ref{eq:logit-fgan-L2}), (\ref{eq:hinge-gan-L1}), (\ref{eq:hinge-gan-L2}), initial value $\theta_0$ for the generator, and a decaying learning rate $\alpha_t$.\\
    \Repeat{}{
        \KwSty{Sampling:} Draw $(Z_1,\ldots,Z_n)$ from $\N(0,I)$ and approximate
        $P_{\theta_{t-1}}$ by \\
        the empirical distribution on the fake data $\xi_i = \mu_{t-1} + \Sigma_{t-1}^{1/2} z_i$, $i=1,\ldots,n$. \\
        \KwSty{Updating:} Compute $\gamma_{t} =\argmax_{\gamma} K(\theta_{t-1}, \gamma; \lambda)$ by a concave optimizer,\\
        and compute $\theta_t = \theta_{t-1} - \alpha_t \nabla_{\theta} K (\theta, \gamma_t; \lambda) |_{\theta_{t-1}}$ by gradient descent.
    }(\KwSty{until} convergence.)
\end{algorithm}
\vspace{-20pt}
\subsection{Hinge GAN with spline discriminators}  \label{sec:hinge-gan}

We consider two hinge GAN methods with an $L_1$ or $L_2$ penalty on the spline discriminator,
which leads to theoretically improved error bounds in terms of dependency on $(\epsilon,p)$ over the parameter space $\Theta_1$ or $\Theta_2$ respectively,
compared with the corresponding logit $f$-GAN methods in Section~\ref{sec:logit-fgan}.

The first method, $L_1$ penalized hinge GAN, is defined by solving
\begin{align}
& \min_{\theta\in\Theta_1} \max_{\gamma\in\Gamma}  \;
\left\{ K_\HG (P_n, P_\theta; h_{\gamma,\mu}) - \lambda_1 \;\pen_1 (\gamma) \right\},   \label{eq:hinge-gan-L1}
\end{align}
where $K_\HG (P_n, P_\theta; h)$ is the hinge objective $K_\HG (P_*, P_\theta; h)$ in (\ref{eq:hinge-gan}) with $P_*$ replaced by $P_n$ and,
similarly as in $L_1$ penalized logit $f$-GAN (\ref{eq:logit-fgan-L1}),
$ h_{\gamma,\mu}(x) = h_{\sp,\gamma} ( x-\mu)$, $\pen_1(\gamma) = \|\gamma_1\|_1 + \|\gamma_2\|_1$, and $\lambda_1 \ge 0$ is a tuning parameter.

\begin{thm} \label{thm:hinge-L1}
Assume that $\|\Sigma^{*}\|_{\max} \le M_1$. Let $\hat \theta=(\hat\mu,\hat\Sigma)$ be a solution to (\ref{eq:hinge-gan-L1}). For $\delta <1/7$,
if $\lambda _1 \ge C_{1} \left( \sqrt{\log{p}/n} + \sqrt{\log{(1/\delta)}/n}\right) $ and $\epsilon + \sqrt{\epsilon/(n\delta)} + \lambda_{1} \le C_2 $,
then with probability at least $1-7\delta$ the following bounds hold  {uniformly over contamination distribution $Q$,}
\begin{align*}
 \| \hat\mu - \mu^* \|_\infty & {\le} C \left(\epsilon + \sqrt{\epsilon/(n\delta)} + \lambda_1\right), \\
 \| \hat\Sigma - \Sigma^*\|_{\max} & {\le} C\left(\epsilon + \sqrt{\epsilon/(n\delta)} + \lambda_1\right),
\end{align*}
where $C, C_1, C_2 > 0$ are constants, depending on $M_1$ but independent of $(n, p, \epsilon,\delta)$.
\end{thm}

For $L_1$ penalized hinge GAN,
Theorem \ref{thm:hinge-L1} shows that the estimator $(\hat\mu,\hat\Sigma)$
achieves error bounds in the maximum norms in the order $\epsilon + \sqrt{\log(p)/n}$,
which improve upon the error rate $\sqrt\epsilon + \sqrt{\log(p)/n}$ in terms of dependency on $\epsilon$ for $L_1$ penalized logit $f$-GAN.
This difference can be traced to that in the population error bounds in Theorem \ref{thm:pop-robust}.
Moreover, Theorem 5.1 in \cite{CGR18} indicates that a minimax lower bound on estimator errors $\| \hat\mu-\mu^*\|_\infty$ or
$\| \hat\Sigma-\Sigma^* \|_{\max}$ is also of order $\epsilon + \sqrt{\log(p)/n}$ in Huber's contaminated Gaussian model,
where $\sqrt{\log (p)/n}$ is a minimax lower bound in the maximum norms in the case of $\epsilon=0$.
Therefore, our $L_1$ penalized hinge GAN achieves the minimax rates in the maximum norms for Gaussian
location and scatter estimation over $\Theta_1$.

The second method, $L_2$ penalized hinge GAN, is defined by solving
\begin{align}
& \min_{\theta\in\Theta_2} \max_{\gamma\in\Gamma}  \; \left\{ K_\HG (P_n, P_\theta; h_{\gamma,\mu}) - \lambda_2 \;\pen_2(\gamma_1) - \lambda_3 \;\pen_2(\gamma_2) \right\},   \label{eq:hinge-gan-L2}
\end{align}
where, similarly as in $L_2$ penalized logit $f$-GAN (\ref{eq:logit-fgan-L2}),
$ h_{\gamma,\mu}(x) = h_{\sp,\gamma} ( x-\mu)$, $\pen_2(\gamma_1) = \|\gamma_1\|_2$ and $\pen_2(\gamma_2) =\|\gamma_2\|_2$,
and $\lambda_2 \ge 0$ and $\lambda_3 \ge 0$ are tuning parameters.

\begin{thm} \label{thm:hinge-L2}
Assume that $\| \Sigma^* \|_\op \le M_2$ and $f$ satisfies Assumptions~\ref{ass:f-condition}--\ref{ass:f-condition2}. Let $\hat \theta=(\hat\mu,\hat\Sigma)$ be a solution to (\ref{eq:hinge-gan-L2}). For $\delta < 1/8$, if
$\lambda_2 \ge C_1\left(\sqrt{p/n} + \sqrt{\log(1/\delta)/n}\right)$, $\lambda_3 \ge C_1\sqrt{p}\left(\sqrt{p/n} + \sqrt{\log(1/\delta)/n}\right)$,
and $\sqrt{p}\left(\epsilon + \sqrt{\epsilon/(n\delta)}\right) + \lambda_2 \le C_2 $, then with probability at least $1- {8}\delta$
the following bounds hold  {uniformly over contamination distribution $Q$,}
\begin{align*}
 \| \hat\mu - \mu^* \|_2 & \le C\left(\sqrt{p}\left(\epsilon + \sqrt{\epsilon/(n\delta)}\right) + \lambda_2\right), \\
 %\| \hat\sigma - \sigma^* \|_2 & \le S_{6,a}  \err_{f2} (n,p, \delta, \epsilon), \\
 p^{-1/2}\| \hat\Sigma - \Sigma^* \|_\fro
& \le C\left(\sqrt{p}\left(\epsilon + \sqrt{\epsilon/(n\delta)}\right) + \lambda_2 + \lambda_3/\sqrt{p}\right),
\end{align*}
where $C_1, C_2, C > 0$ are constants, depending on $M_2$ but independent of $(n, p, \epsilon, \delta)$.
\end{thm}

For $L_2$ penalized hinge GAN,
Theorem \ref{thm:hinge-L2} shows that the estimator $(\hat\mu,\hat\Sigma)$
achieves error bounds in the $L_2$ and $p^{-1/2}$-Frobenius norms in the order $\epsilon \sqrt{p} + \sqrt{p/n}$.
On one hand, these error bounds reduce to the same order, $\sqrt{\epsilon} + \sqrt{p/n}$, as those for
$L_2$ penalized logit $f$-GAN, under the condition that $p \epsilon$ is upper bounded by a constant.
On the other hand, when compared with the minimax rates,
there remain nontrivial differences between $L_2$ penalized hinge GAN and logit $f$-GAN.
In fact, the minimax rates in the $L_2$ and operator norms for location and scatter estimation over $\Theta_2$ is known to be $\epsilon + \sqrt{p/n}$
in Huber's contaminated Gaussian model (\cite{CGR18}). The same minimax rate can also be shown in the $p^{-1/2}$-Frobenius norm for scatter estimation.
Then the error rate for $L_2$ penalized hinge GAN in Theorem \ref{thm:hinge-L2} matches the minimax rate, and both reduce to
the contamination-free error rate $\sqrt{p/n}$,
provided that $\epsilon \sqrt{n}$ is bounded by a constant, i.e., $\epsilon = O( \sqrt{1/n} )$, independently of $p$.
For $L_2$ penalized logit $f$-GAN associated with the reverse KL or JS divergence (satisfying Assumptions \ref{ass:f-condition}--\ref{ass:f-condition2}),
the error bounds from Theorem \ref{thm:logit-L2} match the minimax rate provided both $\epsilon = O(p/n)$ and
$\epsilon = O(1/p)$. The latter condition can be restrictive when $p$ is large.

\begin{rem} \label{rem:hinge-concavity}
The two functionals, $\min(1,h)$ and $\min(-1,h)$, are concave in $h$ in the hinge objective $K_\HG (P_n, P_\theta; h)$.
This is reminiscent of the concavity of $f^\prime (\me^h)$ and $- f^\# (\me^h)$ in $h$ in the logit $f$-GAN objective $K_f (P_n, P_\theta; h)$ under
Assumptions \ref{ass:f-condition}(ii) and \ref{ass:f-condition2}(i) as discussed in Remark \ref{rem:f-condition2-b}. These concavity properties are crucial to our proofs of Theorems \ref{thm:hinge-L1}--\ref{thm:hinge-L2} and Corollary \ref{cor:two-obj}(ii) .
See Lemmas \ref{lem:hinge-upr-bound} and \ref{lem:hinge-two-obj-upr-bound} in the Supplement.
Moreover, the concavity of $K_\HG (P_n, P_\theta; h)$ in $h$,
together with the linearity of the spline discriminator $h_{\gamma,\mu}$ in $\gamma$,
implies that the objective function $K_\HG (P_n, P_\theta; h_{\gamma,\mu})$
is concave in $\gamma$ for any fixed $\theta$.
Hence similarly to penalized logit $f$-GAN, our penalized hinge GAN (\ref{eq:hinge-gan-L1}) or (\ref{eq:hinge-gan-L2})
can also be implemented through nested optimization with a concave inner stage in training spline discriminators as shown in Algorithm \ref{alg:1}.
\end{rem}

\subsection{Two-objective GAN with spline discriminators}  \label{sec:two-obj-gan}

We study two-objective GANs, where the spline discriminator is trained
using the objective function in logit $f$-GAN or hinge GAN, but the generator is trained using a different objective function.

Consider the following two-objective GAN related to logit $f$-GANs (\ref{eq:logit-fgan-L1}) and (\ref{eq:logit-fgan-L2}):
\begin{align} \label{eq:two-obj-gan-centered}
\left\{ \begin{array}{ll}
\max\limits_{\gamma\in\Gamma}\; K_{f}({P_n}, P_{\theta} ; h_{\gamma, \mu}) - \pen(\gamma; \lambda) & \quad \text{with $\theta$ fixed},  \\
\min\limits_{\theta\in\Theta}\; \E_{ {P_n} }f^{\prime}(\me^{h_{\gamma, \mu}(x)})-\E_{P_{\theta}}G (h_{\gamma, \mu}(x)) & \quad \text{with $\gamma$ fixed}. \\
\end{array} \right.
\end{align}
Similarly, consider the two-objective GAN related to the hinge GAN (\ref{eq:hinge-gan-L1}) and (\ref{eq:hinge-gan-L2}):
\begin{align} \label{eq:two-obj-hinge-centered}
\left\{ \begin{array}{ll}
\max\limits_{\gamma\in\Gamma}\; K_{\HG}( {P_n}, P_{\theta} ; h_{\gamma, \mu}) - \pen(\gamma; \lambda) & \quad \text{with $\theta$ fixed},  \\
\min\limits_{\theta\in\Theta}\; \E_{ {P_n} }\min(h_{\gamma, \mu}(x), 1)-\E_{P_{\theta}}G (h_{\gamma, \mu}(x)) & \quad \text{with $\gamma$ fixed}. \\
\end{array} \right.
\end{align}
Here $\pen(\gamma; \lambda)$ is an $L_1$ penalty, $\lambda_1 ( \| \gamma_1 \|_1 + \|\gamma_2\|_1) $
and $\Theta$ is $\Theta_1 = \{ (\mu,\Sigma): \mu \in \bbR^p, \| \Sigma\|_{\max} \le M_1\}$ as in (\ref{eq:logit-fgan-L1}),
or $\pen(\gamma; \lambda)$ is an $L_2$ penalty $\lambda_2 \| \gamma_1\|_2 + \lambda_3 \| \gamma_2\|_2$
and $\Theta$ is $\Theta_2 = \{ (\mu,\Sigma): \mu \in \bbR^p, \|\Sigma\|_\op \le M_2\}$ as in (\ref{eq:logit-fgan-L2}),
and $G$ is a function satisfying Assumption \ref{ass:generator-condition}.
Note that the discriminator $h_{\gamma, \mu}$ is a spline function
with knots depending on $\mu$, so that $\E_{ {P_n} }f^{\prime}(\me^{h_{\gamma, \mu}(x)})$ cannot be dropped
in the optimization over $\theta$ in (\ref{eq:two-obj-gan-centered}) or (\ref{eq:two-obj-hinge-centered}).
We show that the two-objective logit $f$-GAN and hinge GAN achieve similar error bounds as the
corresponding one-objective versions in Theorems \ref{thm:logit-L1}--\ref{thm:hinge-L2}.

\begin{ass} \label{ass:generator-condition}
Function $G$ in (\ref{eq:two-obj-gan-centered}) or (\ref{eq:two-obj-hinge-centered})
is convex and strictly increasing. Hence the inverse function $G^{-1}$ exists and is concave and strictly increasing.
\end{ass}

\begin{cor} \label{cor:two-obj}
(i) If $\hat\theta$ is replaced by a solution to the alternating optimization problem (\ref{eq:two-obj-gan-centered}) with
the $L_1$ or $L_2$ penalty on $\gamma$ as in (\ref{eq:logit-fgan-L1}) or (\ref{eq:logit-fgan-L2}) and
the corresponding choice of $\Theta$,
then the results in Theorem \ref{thm:logit-L1} or \ref{thm:logit-L2} remains valid respectively.

(ii) If $\hat\theta$ is replaced by a solution to the alternating optimization problem (\ref{eq:two-obj-hinge-centered}) with
the $L_1$ or $L_2$ penalty on $\gamma$ as in (\ref{eq:hinge-gan-L1}) or (\ref{eq:hinge-gan-L2}) and
the corresponding choice of $\Theta$,
then the results in Theorem \ref{thm:hinge-L1} or \ref{thm:hinge-L2} remains valid respectively.
\end{cor}

The two-objective GANs studied in Corollary~\ref{cor:two-obj} differ slightly from existing ones
as described in (\ref{eq:gan-logD})--(\ref{eq:geo-gan}), due to the use of the discriminator $h_{\gamma,\mu}$ depending on $\mu$
to facilitate theoretical analysis as mentioned in Section \ref{sec:logit-fgan}.
If $h_{\gamma,\mu}$ were replaced by a discriminator $h_\gamma$ defined independently of $\theta$, then
taking $K_f = K_{\mathrm{JS}}$ and $G(h) = -  \log ( 1+\me^{-h})$ or $G(h) = -h$ in (\ref{eq:two-obj-gan-centered})
reduces to GAN with log$D$ trick (\ref{eq:gan-logD}) or calibrated rKL-GAN (\ref{eq:cal-rkl-gan}) respectively,
and taking $K_f = K_\HG$ and $G(h) = -h$ in (\ref{eq:two-obj-hinge-centered}) reduces to geometric GAN (\ref{eq:geo-gan}).

\section{Discussion} \label{sec:discussion}

\subsection{GANs with data transformation} \label{sec:transform-GAN}

Compared with the usual formulations (\ref{eq:logit-fgan}) and (\ref{eq:hinge-gan}),
our logit $f$-GAN and hinge GAN methods in Sections \ref{sec:logit-fgan}--\ref{sec:hinge-gan} involve a notable modification that
both the real and fake data are discriminated against each other after being shifted by the current location parameter.
%As mentioned after (\ref{eq:logit-fgan-L1}), this location transformation effectively induces an adaptive choice of knots for the spline discriminator,
%which not only is numerically desirable but also facilitates our theoretical analysis.
Without the modification, a direct approach based on logit $f$-GAN would use the objective function
\begin{align}
 K_f (P_n, P_\theta; h_{\sp,\gamma}) = \E_{P_n} f^{\prime}(\me^{h_{\sp,\gamma}(x)}) -
 \E_{P_{\mu,\Sigma}} f^\#(e^{h_{\sp,\gamma}(x)}),   \label{eq:logit-fgan-no-transf}
\end{align}
where the real data and the Gaussian fake data generated from standard noises are discriminated again each other given the parameters $(\mu,\Sigma)$.
The idea behind our modification can be extended by allowing both location and scatter transformation.
For example, consider logit $f$-GAN with full transformation:
\begin{align}
& \min_{\theta\in\Theta} \max_{\gamma\in\Gamma}  \;
\left\{ K_f (P_n, P_\theta; h_{\gamma,\mu,\Sigma}) - \pen(\gamma; \lambda) \right\},   \label{eq:logit-fgan-transf}
\end{align}
where $K_f$ is the logit $f$-GAN objective as in (\ref{eq:logit-fgan-L1}) and (\ref{eq:logit-fgan-L2}),
$h_{\gamma,\mu,\Sigma}(x) = h_{\sp,\gamma} ( \Sigma^{-1/2} (x-\mu))$
and $\pen(\gamma; \lambda)$ is an $L_1$ or $L_2$ penalty term.
The discriminator $h_{\gamma,\mu,\Sigma}(x)$ is obtained by applying $h_{\sp,\gamma}(\cdot)$ with fixed knots
to the transformed data $\Sigma^{-1/2} (x-\mu)$.
Similarly to (\ref{eq:logit-fgan-L1-transf}), the non-penalized objective in (\ref{eq:logit-fgan-transf}) can be equivalently written as
\begin{align}
 K_f (P_n, P_\theta; h_{\gamma,\mu,\Sigma}) = \E_{\Sigma^{-1/2}(P_n-\mu)} f^{\prime}(\me^{h_{\sp,\gamma}(x)}) -
 \E_{P_{0,I}} f^\#(e^{h_{\sp,\gamma}(x)}),   \label{eq:logit-fgan-transf2}
\end{align}
where $\Sigma^{-1/2} (P_n -\mu)$ denotes the empirical distribution on $\{\Sigma^{-1/2} (X_1-\mu), \ldots, \Sigma^{-1/2} ( X_n-\mu) \}$.
 {Compared with (\ref{eq:logit-fgan-L1-transf}) and (\ref{eq:logit-fgan-no-transf}), there are
two advantages of using (\ref{eq:logit-fgan-transf2}) with full transformation.
First, due to both location and scatter transformation, logit $f$-GAN (\ref{eq:logit-fgan-transf}), but not (\ref{eq:logit-fgan-L1}) or (\ref{eq:logit-fgan-L2}),
can be shown to be affine equivariant.}
Second, the transformed real data and the standard Gaussian noises in (\ref{eq:logit-fgan-transf2}) are discriminated against each other given the current parameters $(\mu,\Sigma)$,
while employing the spline discriminators $h_{\sp,\gamma}(x)$ with knots fixed at $\{-2,-1,0,1,2\}$.
Because standard Gaussian data are well covered by the grid formed from these marginal knots,
the discrimination involved in (\ref{eq:logit-fgan-transf2}) can be informative even when the parameters $(\mu,\Sigma)$ are updated.
The discrimination involved in (\ref{eq:logit-fgan-no-transf}) may be problematic when employing the fixed-knot spline discriminators, because
both the real and fake data may not be adequately covered by the grid formed from the knots.

From the preceding discussion, it can be more desirable to incorporate both location and scatter transformation as in (\ref{eq:logit-fgan-transf2})
than just location transformation as in (\ref{eq:logit-fgan-L1-transf}), which only aligns the centers, but not the scales and correlations, of the Gaussian fake data with the knots
in the spline discriminators.
%In numerical experiments (Section \ref{sec:simulation}), we use the fully transformed versions of our methods such as
%logit $f$-GAN defined in (\ref{eq:logit-fgan-transf}).
%Nevertheless, our current theory only handles the location transformation.
%It is desired in future work to develop theoretical analysis of our methods while allowing both location and scatter transformation.
As mentioned in Section \ref{sec:logit-fgan}, our sample analysis exploits the location transformation in
establishing certain concentration properties in the proofs.
On the other hand, our current proofs are not directly applicable while allowing both location and scatter transformation.
It is desired in future work to extend our theoretical analysis in this direction.

\subsection{Comparison with \cite{GYZ20}} \label{sec:GYZ}

We first point out a connection between logit $f$-GANs and the GANs based on proper scoring rules in \cite{GYZ20}.
For a convex function $g: (0,1) \to \bbR$, a proper scoring rule can be defined as (\cite{Sav71}; \cite{BSS05}; \cite{GR07})
\begin{align*}
S_g (\eta, 1) = g(\eta) + (1-\eta) g^\prime (\eta),\quad S_g(\eta,0) = g(\eta) - \eta g^\prime(\eta).
\end{align*}
The population verion of the GAN studied in \cite{GYZ20} is defined as
\begin{align}
& \min_{\theta\in\Theta} \max_{\gamma\in\Gamma}  \; L_g (P_*, P_\theta; q_\gamma),   \label{eq:GYZ-gan}
\end{align}
where $q_\gamma(x) \in [0,1]$, also called a discriminator, represents the probability that an observation $x$ comes from $P_*$ rather than $P_\theta$, and
\begin{align*}
L_g (P_*, P_\theta; q)
& = (1/2) \left\{\E_{P_*} S_g( q(x), 1) + \E_{P_\theta} S_g( q(x), 0) \right\}- g(1/2).
\end{align*}
The objective $L_g (P_*, P_\theta; q)$ is shown to be a lower bound, being tight if $q= 2 \,\dif P_* /\dif (P_*+P_\theta)$, for the divergence
$ D_{g_0} ( P_* \| (P_* + P_\theta)/2 )$, where $g_0 (t) = g(t/2) - g(1/2)$ for $t\in (0,2)$.
For example, taking $g(\eta) = \eta\log \eta + (1-\eta)\log(1-\eta)$ leads to the log score,
$S_g(\eta,1)= \log\eta$ and $S_g(\eta,0) = \log(1-\eta)$. The corresponding objective function $L_g(P_*, P_\theta; q_\gamma)$
reduces to the expected log-likelihood with discrimination probability $q_\gamma(x)$ as used in \cite{GPM14}.
We show that if $q_\gamma(x)$ is specified as a sigmoid probability, then
$L_g (P_*, P_\theta; q_\gamma)$ can be equivalently obtained as a logit $f$-GAN objective for a suitable choice of $f$.

\begin{pro} \label{pro:GYZ}
Suppose that the discriminator is specified as $q_\gamma(x) = \sigm (h_\gamma(x))$. Then
$ L_g (P_*, P_\theta; q_\gamma) = K_f ( P_*, P_\theta; h_\gamma)$ for $K_f$ defined in (\ref{eq:logit-fgan}) and
$ f(t) = \frac{1+t}{2} g_0 ( \frac{2t}{1+t} )$ satisfying that
$ D_{g_0} ( P_* \| (P_* + P_\theta)/2 ) = D_f ( P_* \| P_\theta )$.
\end{pro}

In contrast with $h_\gamma(x)$ parameterized as a pairwise spline function, \cite{GYZ20} studied robust estimation in Huber's contaminated Gaussian model,
where $q_\gamma(x)$ is parameterized as a neural network with two or more layers and sigmoid activations in the top and bottom layers.
In the case of two layers, the neural network in \cite{GYZ20}, Section 4, is defined as
\begin{align}
q_\gamma(x) = \sigm ( h_\gamma(x) ), \quad h_\gamma(x) = \sum_{j=1}^J \gamma_j^{(1)} \sigm ( \gamma^{(2) \T}_j x + \gamma_{0j}^{(2)} ),
\label{eq:GYZ-discriminator}
\end{align}
where $ (\gamma_j^{(2)}, \gamma_{j0}^{(2)})$, $j=1,\ldots,J$, are the weights and intercepts in the bottom layer,
and $ \gamma_j^{(1)}$, $j=1,\ldots,J$, are the weights in the top layer constrained such that
$\sum_{j=1}^J |\gamma_j^{(1)}| \le \kappa$ for a tuning parameter $\kappa$.
Assume that $g(\eta)$ is three-times continuously differentiable at $\eta=1/2$, $g^\dprime(1/2) >0$, and
for a universal constant $c_0 >0$,
\begin{align}
 2 g^\dprime(1/2) \ge g^{\prime\prime\prime}(1/2) + c_0, \label{eq:GYZ-condition}
\end{align}
Then \cite{GYZ20} showed that the location and scatter estimators from the sample version of (\ref{eq:GYZ-gan})
with discriminator (\ref{eq:GYZ-discriminator})
achieve the minimax error rates, $O( \epsilon + \sqrt{p/n})$, in the $L_2$ and operator norms,
provided that $\kappa = O (\epsilon+ \sqrt{p/n})$ among other conditions.
However, with sigmoid activations used inside $h_\gamma(x)$,
the sample objective $L_g (P_n, P_\theta; q_\gamma)$
may exhibit a complex, non-concave landscape in $\gamma$, which makes minimax optimization difficult.

There is also a subtle issue in how the above result from \cite{GYZ20} can be compared with even our population analysis
for minimum $f$-divergence estimation, i.e., population versions of GANs with nonparametric discriminators.
In fact, condition (\ref{eq:GYZ-condition}) can be directly shown to be equivalent to saying that
$\frac{\dif^2}{\dif u^2} f^\prime (\me^u) | _{u=0} \ge c_0$ for $f$ associated with $g$ in Proposition \ref{pro:GYZ}.
This condition can be satisfied, while Assumption \ref{ass:f-condition} is violated,
for example, by the choice $g(\eta) = (\eta-1) \log(\eta/(2-\eta))$ and $f(t)= \{(t-1)\log t\}/2$, corresponding to the mixed KL divergence
$ D_{\mathrm{KL}} (P || Q)/2 + D_{\mathrm{KL}} (Q || P)/2$.
As shown in Figure \ref{fig:1}, minimization of the mixed KL does not in general lead to robust estimation.
Hence it seems paradoxical that minimax error rates can be achieved by the GAN in \cite{GYZ20} with
its objective function derived from the mixed KL.
On the other hand, a possible explanation can be seen as follows.
By the sigmoid activation and the constraint $\sum_{j=1}^J |\gamma_j^{(1)}| \le \kappa $,
the log-odds discriminator $h_\gamma(x)$ in (\ref{eq:GYZ-discriminator}) is forced to be bounded, $ | h_\gamma(x) | \le \kappa$,
where $\kappa$ is further assumed to small, of the same order as the minimax rate $O( \epsilon + \sqrt{p/n})$.
As a result, maximization of the population objective $L_g (P_*, P_\theta; q_\gamma)$ over such
constrained discriminators may produce a divergence with a substantial gap to the
actual divergence $D_f (P_* \| P_\theta)$ for any fixed $\theta$.
Instead, the implied divergence measure may behave more similarly
as the total variation $D_\TV (P_* \| P_\theta)$ than as $D_f (P_* \| P_\theta)$, due to the boundedness of $h_\gamma(x)$ by a sufficiently small $\kappa$,
so that minimax error rates can still be achieved.

    \section{Simulation studies} \label{sec:Simulation}
    We conducted simulation studies to compare the performance of our logit $f$-GAN and hinge GAN methods with several existing methods in various settings
    depending on $Q$, $\epsilon$, $n$, and $p$. Results about error dependency on $\epsilon$ are provided in the main paper and others are presented in the Supplement.
    Two contamination distributions $Q$ are considered to allow different types of contaminations.

\subsection{Implementation of methods} \label{sec:simulation-method}

Our methods are implemented following Algorithm \ref{alg:1}, where the penalized GAN objective function $K(\theta,\gamma;\lambda)$ is defined
as in (\ref{eq:logit-fgan-transf}) for logit $f$-GAN or with $K_f$ replaced by $K_{\HG}$ for hinge GAN.
As discussed in Section \ref{sec:transform-GAN}, this scheme allows adequate discrimination between
the back-transformed real data, $\Sigma^{-1/2} (x-\mu)$, and the standard Gaussian noises using
spline discriminators with {\it fixed knots}.
In the following, we describe the discriminator optimization and penalty choices. See the Supplement
for further details including the initial values $(\mu_0,\Sigma_0)$ and learning rate $\alpha_t$.

For Algorithm \ref{alg:1} with spline discriminators, the training objective is concave in the discriminator parameter $\gamma$
and hence the discriminator can be fully optimized to provide a proper updating direction for the generator in each iteration.
This implementation via nested optimization can be much more tractable than nonconvex-nonconcave minimax optimization in other GAN methods,
where the discriminator and generator are updated by gradient ascent and descent, with careful tuning of the learning rates.
In numerical work, we formulate the discriminator optimization from our GAN methods in the CVX framework (\cite{CVXR}) and then
apply the convex optimization solver MOSEK (\url{https://docs.mosek.com/latest/rmosek/index.html}).

As dictated by our theory, we employ $L_1$ or $L_2$ penalties on the spline discriminators to control sampling variation,
especially when the sample size $n$ is relatively smaller compared to the dimension of the discriminator parameter $\gamma$.
Numerically, these penalties help stabilize the training process by weakening the discriminator power in the early stage.
We tested our methods under different penalty levels and identified default choices of $\lambda$
for our rKL and JS logit $f$-GANs and hinge GAN. These penalty choices are then fixed in all our subsequent simulations.
See the Supplement for results from our tuning experiments.

For comparison, we also implement ten existing methods for robust estimation.
     \begin{itemize}
         \item \emph{JS-GAN} (\cite{GYZ20}). We use the code from \cite{GYZ20} with minimal modification.
         The batch size is set to $1/10$ of the data size because the default choice $500$ is too large in our experiment settings.
         We use the network structure $p\text{-}2p\text{-}\lfloor p/2 \rfloor\text{-}1$ with LeakyReLU and Sigmoid activations as recommended in \cite{GYZ20}.
         \item \emph{Tyler's M-estimator} (\cite{Tyl87}). This method is included for completeness, being designed for multivariate scatter estimation from elliptical distributions, not Huber's contaminated Gaussian distribution. We use R package \texttt{fastM} for implementation (\url{https://cran.r-project.org/web/packages/fastM/index.html}).
         \item \emph{Kendall's $\tau$ with MAD} (\cite{LT18}). Kendall's $\tau$ (\cite{Ken38}) is used to estimate the correlations after sine transformation
        and the median absolute deviation (MAD) (\cite{Ham74}) is used to estimate the scales. We use the R built-in function \texttt{cor} to compute Kendall's $\tau$ correlations and use the built-in function \texttt{mad} for MAD.
        \item \emph{Spearman's $\rho$ with $Q_n$-estimator} (\cite{OC15}).
         The $Q_n$-estimator (\cite{RC93}) is used for scale estimation
         and Spearman's $\rho$ (\cite{Spe87}) is used with sine transformation for correlation estimation.
         We use the R function \texttt{cor} to compute Spearman's $\rho$ correlations and the \texttt{Qn} function in R package \texttt{robustbase} (\url{https://cran.r-project.org/web/packages/robustbase/index.html}).
         \item \emph{MVE} (\cite{Rou85}). The minimum volume ellipsoid (MVE) estimator is a high-breakdown robust method for multivariate location and scatter estimation. We use the function \texttt{CovMve} in R package \texttt{rrcov} for implementation (\url{https://cran.r-project.org/web/packages/rrcov/index.html}).
             %Detailed algorithms for methods implemented in package \texttt{rrcov} are described in \cite{TF10}.
         {
        \item \emph{MCD} (\cite{Rou85}). The minimum covariance determinant (MCD) estimator is a high-breakdown robust method
        and is shown to be superior to MVE in statistical efficiency (\cite{BDJ93}).
            We use the function \texttt{CovMcd} in R package \texttt{rrcov} for implementation.
         \item \emph{Sest} (\cite{Dav87, Lop89}). This is an S-estimator of multivariate location and scatter,
         based on Tukey's biweight function. It has high-breakdown point and improved statistical efficiency over MVE.
         We use the function \texttt{CovSest} in R package \texttt{rrcov} for implementation.
        \item \emph{Mest} (\cite{Roc96}). This is a constrained M-estimator of multivariate location and scatter, based on a translated biweight function.
        We use the function \texttt{CovMest} in R package \texttt{rrcov} for implementation, with MVE as the initial value.
          \item \emph{MMest} (\cite{TT00}). This is a constrained M-estimator which is a multivariate
          version of MM-functionals in \cite{Yoh87}. We use the function \texttt{CovMMest} in R package \texttt{rrcov} for implementation,
          with an S-estimate as the initial value. }
         \item \emph{$\gamma$-Lasso} (\cite{HFS17}). The method is implemented in the R package \texttt{rsggm}, which has become unavailable on CRAN.
         We use an archived version (\url{https://mran.microsoft.com/snapshot/2017-02-04/web/packages/rsggm/index.html}). We
             set $\gamma=0.05$ and deactivate the vectorized $L_1$ penalty.
     \end{itemize}
    Tyler's M-estimator, Kendall's $\tau$ with MAD, and Spearman's $\rho$ with $Q_n$ deal with scatter estimation only, whereas the other methods
    handle both location and scatter estimation.
    In our experiments, we focus on comparing the performance of
    existing and proposed methods in terms of scatter estimation (i.e., variance matrix estimation).

    \subsection{Simulation settings} \label{sec:simulation-setting}
    The uncontaminated distribution is $\N(0, \Sigma^*)$ where $\Sigma^*$ is %an autoregressive variance matrix, i.e.,
    a Toeplitz matrix with $(i,j)$ component equal to $(1/2)^{|i-j|}$.
    The location parameter is unknown and estimated together with the variance matrix, except for Tyler's $M$-estimator, Kendall's $\tau$, and Spearman's $\rho$.
    Consider two contamination distributions $Q$ of different types. Denote a $p\times p$ identity matrix as $I_p$
    and a $p$-dimensional vector of ones as $1_p$.
    \begin{itemize}
        \item $Q=\N\left(2.25 c, \frac{1}{3}I_p \right)$ where $c = (1, -1, 1, -1, 1, \dots)$ is a $p$-dimensional vector of alternating $\pm 1$. In this setting, the contaminated points may not be seen as outliers marginally in each coordinate.
            %As a result, pairwise methods such as Kendall's $\tau$ or Spearman's $\rho$ based on marginal signs or ranks are expected to have poor performance.
            On the other hand, these contaminated points can be easily separated as outliers from the uncontaminated Gaussian distribution in higher dimensions.

        \item $Q=\N(5 1_p, 5I_p )$.  {Contaminated points may lie in both low-density and high-density regions of the uncontaminated Gaussian distribution.}
        The majority of contaminated points are outliers that are far from the uncontaminated data, and there are also contaminated points that are enclosed by the uncontaminated points. This setting is also used in \cite{GYZ20}.
    \end{itemize}
    The Gaussianity of the above contamination distributions is used for convenience and easy characterization of data patterns,
     given the specific locations and scales. See Figure \ref{fig:Ellipsoid_TypeA} for an illustration of the first contamination,
     and the Supplement for that of the second.

     \subsection{Experiment results}

     Table~\ref{tab:eps-l1} summarizes scatter estimation errors in the maximum norm from $L_1$ penalized hinge GAN and logit $f$-GANs and existing methods, where $p=10$, $n=2000$, and $\epsilon$ increases from $0.025$ to $0.1$. The errors are obtained by averaging $20$ repeated runs and the numbers in brackets are standard deviations.
     From these results, the logit $f$-GANs, JS and rKL, have the best performance, followed closely by the hinge GAN
     and then with more noticeable differences by JS-GAN in \cite{GYZ20} and  {the five high-breakdown methods, MVE, MCD, Sest, Mest, and MMest.}
     The pairwise methods, Kendall's $\tau$ with MAD and Spearman's $\rho$ with $Q_n$-estimator, have relatively poor performance, especially for the first contamination as expected from Figure \ref{fig:Ellipsoid_TypeA}. The $\gamma$-Lasso performs competitively when $\epsilon$ is small (e.g., 2.5\%), but deteriorates quickly as $\epsilon$ increases. A possible explanation is that the default choice $\gamma=0.05$ may not work well in the current settings,
     even though the numerical studies in \cite{HFS17} suggest that $\gamma$ does not require special tuning.
     Tyler's M-estimator performs poorly, primarily because it is not designed for robust estimation under Huber's contamination.

     Estimation errors in the Frobenius norm from our $L_2$ penalized GAN methods and existing methods are shown in Table~\ref{tab:eps-l2}. We observe a similar pattern of comparison as in Table~\ref{tab:eps-l1}.

     From Tables \ref{tab:eps-l1}--\ref{tab:eps-l2}, we see that the estimation errors of our GAN methods, as well as other methods, increase as $\epsilon$ increases.
     However, the dependency on $\epsilon$ is not precisely linear for the hinge GAN, and not in the order $\sqrt{\epsilon}$ for the two logit $f$-GANs.
      {This does not violate our theoretical bounds, which are derived to hold over all possible contamination distributions, i.e., for the worst scenario of contamination.
     For specific contamination settings,
     it is possible for logit $f$-GAN to outperform hinge GAN, and
     for each method to achieve a better error dependency on $\epsilon$ than in the worst scenario.}
        \begin{table}[t]
        \caption{Comparison of existing methods and proposed $L_1$ penalized GAN methods ($p=10, n=2000$).
        Estimation error of the variance matrix is reported in the maximum norm $\|\cdot\|_{\max}$.}
        \label{tab:eps-l1}
        \resizebox{\textwidth}{!}{%
        \begin{tabular}{llllllll}
        \hline
        \multicolumn{1}{l|}{$\epsilon$} & hinge GAN          & JS logit $f$-GAN              & rKL logit $f$-GAN             & GYZ JS-GAN     & Tyler\_M & Kendall\_MAD         & Spearman\_Qn             \\ \hline
        \multicolumn{8}{c}{$Q \sim \N\left(2.25 c,\frac{1}{3}I_p\right)$} \\
        2.5 & 0.0711 (0.0203) & 0.0692 (0.0191) & \textbf{0.0661 (0.0121)} & 0.0960 (0.0384) & 0.3340 (0.0360) & 0.1484 (0.0320) & 0.1434 (0.0245)\\
        5   & 0.0768 (0.0201) & 0.0694 (0.0166) & \textbf{0.0658 (0.0161)} & 0.1049 (0.0423) & 0.3115 (0.0279) & 0.2235 (0.0351) & 0.2351 (0.0284) \\
        7.5 & 0.0828 (0.0203) & 0.0690 (0.0171) & \textbf{0.0661 (0.0092)} & 0.0957 (0.0303) & 0.3048 (0.0234) & 0.3164 (0.0319) & 0.3294 (0.0241) \\
        10  & 0.0981 (0.0221) & 0.0737 (0.0207) & \textbf{0.0720 (0.0151)} & 0.1029 (0.0467) & 0.3526 (0.0226) & 0.4133 (0.0264) & 0.4362 (0.0256) \\
    \multicolumn{8}{c}{$Q \sim \N(5 1_p, 5 I_p)$} \\
    2.5 & 0.0742 (0.0214) & \textbf{0.0732 (0.0207)} & 0.0765 (0.0209) & 0.0994 (0.0334) & 0.3886 (0.0346) & 0.1428 (0.0302) & 0.1578 (0.0266)  \\
    5   & 0.0814 (0.0216) & \textbf{0.0739 (0.0204)} & 0.0824 (0.0196) & 0.1053 (0.0235) & 0.4322 (0.0242) & 0.2211 (0.0286) & 0.2812 (0.0297)  \\
    7.5 & 0.0901 (0.0206) & \textbf{0.0757 (0.0209)} & 0.0893 (0.0155) & 0.1063 (0.0400) & 0.4788 (0.0312) & 0.3149 (0.0278) & 0.4153 (0.0318)\\
    10  & 0.1051 (0.0236) & \textbf{0.0802 (0.0225)} & 0.0935 (0.0219) & 0.1275 (0.0453) & 0.5295 (0.0332) & 0.4074 (0.0278) & 0.5693 (0.0358) \\ \hline

\multicolumn{1}{l|}{$\epsilon$} & $\gamma$-Lasso        &  MVE          & MCD        &  Sest& Mest      & MMest\\ \hline
\multicolumn{8}{c}{$Q \sim \N\left(2.25 c, \frac{1}{3}I_p\right)$}   \\
2.5 & 0.0737 (0.0182) & 0.0935 (0.0266) & 0.0834 (0.0270) & 0.0892 (0.0282) & 0.0802 (0.0236) & 0.0876 (0.0278) \\
5   & 0.2182 (0.0162) & 0.1124 (0.0257) & 0.1078 (0.0267) & 0.1314 (0.0257) & 0.0891 (0.0245) & 0.1299 (0.0259) \\
7.5 & 0.3740 (0.0153) & 0.1373 (0.0276) & 0.1306 (0.0247) & 0.1774 (0.0245) & 0.1026 (0.0239) & 0.1759 (0.0240) \\
10  & 0.5277 (0.0154) & 0.1700 (0.0284) & 0.1619 (0.0258) & 0.2358 (0.0292) & 0.1289 (0.0269) & 0.2349 (0.0292) \\
\multicolumn{8}{c}{$Q \sim \N(5 1_p, 5 I_p)$} \\
2.5 & 0.0769 (0.0243) & 0.0984 (0.0284) & 0.0837 (0.0271) & 0.0892 (0.0282) & 0.0799 (0.0229) & 0.0876 (0.0278) \\
5   & 0.1044 (0.0215) & 0.1171 (0.0258) & 0.1081 (0.0267) & 0.1314 (0.0257) & 0.0890 (0.0243) & 0.1299 (0.0259) \\
7.5 & 0.1668 (0.0312) & 0.1383 (0.0251) & 0.1301 (0.0237) & 0.1774 (0.0245) & 0.1031 (0.0243) & 0.1759 (0.0240) \\
10  & 0.2935 (0.0619) & 0.1697 (0.0240) & 0.1618 (0.0250) & 0.2358 (0.0292) & 0.1289 (0.0262) & 0.2349 (0.0291)\\
\hline
        \end{tabular}%
        }
        \end{table}

    \begin{table}[t]
        \caption{Comparison of existing methods and proposed $L_2$ penalized GAN methods ($p=10, n=2000$).
        Estimation error of the variance matrix is reported in the Frobenius norm $\|\cdot\|_{\fro}$.}
        \label{tab:eps-l2}
        \resizebox{\textwidth}{!}{%
        \begin{tabular}{llllllll}
        \hline
        \multicolumn{1}{l|}{$\epsilon$} & hinge GAN          & JS logit $f$-GAN              & rKL logit $f$-GAN             & GYZ JS-GAN & Tyler\_M        & Kendall\_MAD         & Spearman\_Qn                  \\ \hline
        \multicolumn{8}{c}{$Q \sim \N\left(2.25 c, \frac{1}{3}I_p\right)$}   \\
        2.5 & 0.2673 (0.0457) & 0.2574 (0.0461) & \textbf{0.2502 (0.0345)} & 0.3354 (0.0904) & 1.0973 (0.0890) & 0.7524 (0.0286) & 0.7936 (0.0201) \\
5   & 0.2729 (0.0344) & 0.2649 (0.0385) & \textbf{0.2528 (0.0404)} & 0.3511 (0.0757) & 1.0190 (0.0525) & 1.4888 (0.0282) & 1.5991 (0.0217) \\
7.5 & 0.2985 (0.0573) & 0.2790 (0.0554) & \textbf{0.2568 (0.0339)} & 0.3349 (0.0655) & 1.2802 (0.0439) & 2.3065 (0.0444) & 2.4908 (0.0401) \\
10  & 0.3222 (0.0547) & 0.2898 (0.0565) & \textbf{0.2723 (0.0477)} & 0.3642 (0.0885) & 2.3363 (0.0557) & 3.2143 (0.0489) & 3.4669 (0.0398) ) \\
\multicolumn{8}{c}{$Q \sim \N(5 1_p, 5 I_p)$} \\
2.5 & 0.2705 (0.0475) & 0.2688 (0.0526) & \textbf{0.2454 (0.0376)} & 0.3350 (0.0690) & 1.5503 (0.1226) & 0.7688 (0.1108) & 0.8884 (0.1040) \\
5   & 0.2754 (0.0335) & 0.2738 (0.0365) & \textbf{0.2466 (0.0324)} & 0.3526 (0.0489) & 1.9679 (0.0806) & 1.5229 (0.0793) & 1.8049 (0.0736) \\
7.5 & 0.3071 (0.0621) & 0.3026 (0.0688) & \textbf{0.2685 (0.0435)} & 0.3570 (0.0778) & 2.6099 (0.1545) & 2.3906 (0.1473) & 2.9247 (0.1490) \\
10  & 0.3301 (0.0583) & 0.3049 (0.0635) & \textbf{0.2744 (0.0505)} & 0.3829 (0.0852) & 3.2420 (0.2083) & 3.3202 (0.1462) & 4.1527 (0.1602)  \\
\hline
\multicolumn{1}{l|}{$\epsilon$} & $\gamma$-Lasso        &  MVE          & MCD        &  Sest& Mest      & MMest\\ \hline
\multicolumn{8}{c}{$Q \sim \N\left(2.25 c, \frac{1}{3}I_p\right)$}   \\
2.5 & 0.2679 (0.0339) & 0.3262 (0.0610) & 0.2906 (0.0567) & 0.2982 (0.0619) & 0.2853 (0.0615) & 0.2921 (0.0616) \\
5   & 1.6199 (0.0381) & 0.3583 (0.0452) & 0.3454 (0.0484) & 0.4059 (0.0489) & 0.3016 (0.0425) & 0.4032 (0.0487) \\
7.5 & 3.0849 (0.0461) & 0.4562 (0.0877) & 0.4326 (0.0803) & 0.5805 (0.0863) & 0.3549 (0.0699) & 0.5768 (0.0870) \\
10  & 4.4364 (0.0688) & 0.5409 (0.0799) & 0.5183 (0.0767) & 0.7732 (0.0839) & 0.4079 (0.0681) & 0.7707 (0.0831) \\
\multicolumn{8}{c}{$Q \sim \N(5 1_p, 5 I_p)$} \\
2.5 & 0.3016 (0.0853) & 0.3299 (0.0658) & 0.2896 (0.0568) & 0.2982 (0.0619) & 0.2841 (0.0599) & 0.2921 (0.0616) \\
5   & 0.4872 (0.0670) & 0.3781 (0.0567) & 0.3460 (0.0487) & 0.4059 (0.0489) & 0.3011 (0.0439) & 0.4032 (0.0487) \\
7.5 & 1.0551 (0.2097) & 0.4532 (0.0740) & 0.4342 (0.0792) & 0.5805 (0.0863) & 0.3580 (0.0717) & 0.5768 (0.0870) \\
10  & 2.1356 (0.4419) & 0.5331 (0.0715) & 0.5167 (0.0753) & 0.7732 (0.0839) & 0.4076 (0.0667) & 0.7708 (0.0831)\\
\hline
        \end{tabular}%
        }
    \end{table}
%%%%%%%%%%%%%%%%%%%%%%%%%%%%%%%%%%%%%%%%%%%%
% logit L1

%{lem:spline-L1-upper}
%{lem:logit-upper}
%{pro:logit-L1-upper}

%{lem:logit-concen-L1-lower}
%{lem:logit-L1-lower}
%{pro:logit-L1-lower}

%{lem:local-linear1}
%{lem:local-linear2}
%{pro:logit-L1-combine}

% logit L2

%{lem:spline-L2-upper}
%{pro:logit-L2-upper}

%{lem:logit-concen-L2-lower}
%{lem:logit-L2-lower}
%{pro:logit-L2-lower}

%{pro:logit-L2-combine}

% logit L2 variance matrix

%{pro:logit-L2v-upper}
%{pro:logit-L2v-lower}
%{pro:logit-L2v-combine}

%%% -----------------------------------------------------------------------------------------------------------------------------------------  main proof, spline L1

\section{Main proofs} \label{sec:main-proofs}
We present main proofs of Theorems \ref{thm:pop-robust} and \ref{thm:hinge-L1} in this section.
The main proofs of the other results and details of all main proofs are provided in the Supplementary Material.

 {At the center of our proofs is a unified strategy designed to establish error bounds for GANs. See, for example,
(\ref{eq:logit-pop-main-frame}) and (\ref{eq:main-frame}).
Moreover, we carefully exploit the location transformation and $L_1$ or $L_2$ penalties in our GAN objective functions
and develop suitable concentration properties, in addition to leveraging the concavity in updating the spline discriminators, as discussed in
Remarks \ref{rem:f-condition2-b} and \ref{rem:hinge-concavity}.}

\subsection{Proof of Theorem~\ref{thm:pop-robust}} \label{sec:prf-thm-pop-robust}
We state and prove the following result which implies Theorem~\ref{thm:pop-robust}.
\begin{pro} \label{pro:pop-robust}
Let $\Theta_0 = \{ (\mu,\Sigma): \mu \in \bbR^p, \Sigma $ is a $p\times p$ variance matrix $\}$.

(i) Assume that $f$ satisfies Assumption \ref{ass:f-condition}, and
$\epsilon \in [0, \epsilon_0]$ for a constant $\epsilon_0 \in [0, 1/2)$.
Let $\bar\theta =\argmin_{\theta \in \Theta_0} D_{f}(P_{\epsilon}||P_{\theta})$. If
$\err_{f0}(\epsilon) \le a$ for a constant $a \in [0,1/2)$, then we have
\begin{align*}
\|\bar\mu - \mu^*\|_2 &\le S_{1,a}\|\Sigma^*\|_{\op}^{1/2}\err_{f0}(\epsilon),\\
\|\bar\mu - \mu^*\|_{\infty} &\le S_{1,a}\|\Sigma^*\|_{\max}^{1/2}\err_{f0}(\epsilon),
\end{align*}
where $S_{1,a}=\{ \Phi^\prime(\Phi^{-1}(1/2+a)) \}^{-1}$ and $\err_{f0}(\epsilon) = \sqrt{-2 (f^{\dprime}(1))^{-1}f^{\prime}(1-\epsilon_0)\epsilon}+\epsilon $.
If further $\err_{f0}(\epsilon) \le a/(1+S_{1,a})$, then
\begin{align}
\|\bar\Sigma - \Sigma^*\|_{\op} &\le  2S_{3,a}\|\Sigma^*\|_{\op}\err_{f0}(\epsilon) + S^2_{3,a}\|\Sigma^*\|_{\op}(\err_{f0}(\epsilon))^2, \label{eq:thm-pop-robust-op} \\
\|\bar\Sigma - \Sigma^*\|_{\max}  &\le {4 S_{3,a}}\|\Sigma^*\|_{\max}\err_{f0}(\epsilon) + {2 S_{3,a}^2}\|\Sigma^*\|_{\max}(\err_{f0}(\epsilon))^2, \nonumber
\end{align}
where $S_{3,a}= S_{2,a}(1+S_{1,a})$, $S_{2,a}= \{ \sqrt{z_0/2}\, \erf^\prime(\sqrt{2/z_0}\, \erf^{-1}(1/2+a)) \}^{-1}$, and the constant $z_0$ is defined
such that $\erf(\sqrt{z_0/2}) = 1/2$.
The same inequality as (\ref{eq:thm-pop-robust-op}) also holds with $\|\bar\Sigma - \Sigma^*\|_{\op}$ replaced by
$ p^{-1/2} \|\bar\Sigma - \Sigma^*\|_{\fro}$.

(ii) Let $\bar\theta = \argmin_{\theta \in \Theta_0} D_\TV(P_{\epsilon}||P_{\theta})$. Then the statements in (i) hold
with $\err_{f0}(\epsilon)$ replaced by $\err_{h0}(\epsilon) = 2\epsilon$ throughout.
\end{pro}

\begin{proof}[Proof of Proposition~\ref{pro:pop-robust}]
(i) Our main strategy is to show the following inequalities hold:
\begin{align}
d(\bar\theta, \theta^*) - \Delta_1(\epsilon) \le \sqrt{D_{f}(P_\epsilon || P_{\bar\theta})} \le \Delta_2(\epsilon, f), \label{eq:logit-pop-main-frame}
\end{align}
where $\Delta_1(\epsilon)$ and $\Delta_2(\epsilon, f)$ are bias terms, depending on $\epsilon$ and $(\epsilon, f)$ respectively and $d(\bar\theta, \theta^*)$ is the total variation $D_\TV(P_{\bar\theta} \| P_{\theta^*})$ or simply $\TV(P_{\bar\theta}, P_{\theta^*})$. Under certain conditions, $d(\bar\theta, \theta^*)$ delivers upper bounds, up to scaling constants, on the estimation bias to be controlled, $\|\bar\mu - \mu^*\|_{\infty}$, $\|\bar\mu - \mu^*\|_{2}$, $\|\bar{\Sigma} - \Sigma^*\|_{\max}$, and $\|\bar{\Sigma} - \Sigma^*\|_{\op}$.

(Step 1) The upper bound in (\ref{eq:logit-pop-main-frame}) follows from Lemma \ref{lem:logit-upper} (iv):
for any $f$ satisfying Assumption~\ref{ass:f-condition} and any $\epsilon \in [0, \epsilon_0]$, we have
\begin{align*}
D_{f}(P_\epsilon || P_{\bar\theta}) \le D_{f}(P_\epsilon || P_{\theta^*}) \le -f^\prime(1-\epsilon_0)\epsilon = \Delta_2^2(\epsilon, f), %\label{eq:logit-pop-upper-1}
\end{align*}
where $\Delta_2(\epsilon, f) = \sqrt{-f^\prime(1-\epsilon_0)\epsilon }$.
The constant $-f^\prime(1-\epsilon_0)$ is nonnegative because $f$ is non-increasing by Assumption \ref{ass:f-condition} (ii).

(Step 2) We show the lower bound in (\ref{eq:logit-pop-main-frame}) as follows:
\begin{align}
d(\bar\theta, \theta^*) &\le  \TV(P_{\bar\theta}, P_{\epsilon}) +  \TV(P_{\epsilon}, P_{\theta^*}) \le  \TV(P_{\bar\theta}, P_{\epsilon}) + \Delta_1(\epsilon)
\label{eq:logit-pop-combine-1}\\
&\le  \sqrt{2(f^{\dprime}(1))^{-1} D_f(P_{\epsilon} || P_{\bar\theta})} + \Delta_1(\epsilon) ,\label{eq:logit-pop-combine-2}
\end{align}
where $\Delta_1(\epsilon) = \epsilon$.
Line (\ref{eq:logit-pop-combine-1}) follows by the triangle inequality and the fact that
$\TV(P_{\epsilon}, P_{\theta^*}) \le \epsilon \TV(P_{Q}, P_{\theta^*}) \le \epsilon$.
Line (\ref{eq:logit-pop-combine-2}) follows from Lemma \ref{lem:sqr-tv-lwr-bound}:
for any $f$-divergence satisfying {Assumption \ref{ass:f-condition} (iii)}, we have
\begin{align*}
D_f(P_{\epsilon} || P_{\bar\theta}) &\ge \frac{f^{\dprime}(1)}{2} \TV(P_{\epsilon}, P_{\bar\theta})^2. %\label{eq:logit-pop-lower-1}
\end{align*}
The scaling constant, $\inf_{t\in (0,1]}f^{\dprime}(t)/2$, in Lemma~\ref{lem:sqr-tv-lwr-bound}
reduces to $f^{\dprime}(1) /2$, because $f^\dprime$ is non-increasing by Assumption \ref{ass:f-condition} (iii).

(Step 3) Combining the lower and upper bounds in (\ref{eq:logit-pop-main-frame}), we have
\begin{align*}
d(\bar\theta, \theta^*)
&\le \sqrt{2(f^{\dprime}(1))^{-1} }\Delta_2(\epsilon, f)+ \Delta_1(\epsilon) = \err_{f0}(\epsilon) , %\label{eq:logit-pop-combine-3}
\end{align*}
where $\err_{f0}(\epsilon) = \sqrt{-2 (f^{\dprime}(1))^{-1}f^{\prime}(1-\epsilon_0)\epsilon}+\epsilon $.
The location result then follows from Proposition~\ref{pro:local-linear-pop} provided that $\err_{f0}(\epsilon) \le a$ for a constant $a \in [0, 1/2)$.
The variance matrix result follows if $\err_{f0}(\epsilon) \le a/(1+S_{1,a})$.

(ii) For the TV minimizer $\bar\theta$, Steps 1 and 2 in (i) can be combined to directly obtain an upper bound on $d(\bar\theta, \theta^*)$ as follows:
\begin{align}
d(\bar\theta, \theta^*) &\le \TV( P_{\bar\theta}, P_\epsilon) + \TV(P_\epsilon, P_{\theta^*}) \label{eq:hinge-pop-combine-1} \\
&\le  2\TV(P_\epsilon, P_{\theta^*}) \label{eq:hinge-pop-combine-2}\\
&\le 2\epsilon. \label{eq:hinge-pop-combine-3}
\end{align}
Line (\ref{eq:hinge-pop-combine-1}) is due to the triangle inequality.
Line (\ref{eq:hinge-pop-combine-2}) follows because $\TV( P_{\bar\theta}, P_\epsilon)\le \TV( P_{\theta^*}, P_\epsilon) = \TV(P_\epsilon, P_{\theta^*})$ by the definition of $\bar\theta$ and the symmetry of TV.\
Line (\ref{eq:hinge-pop-combine-3}) follows because $\TV(P_{\epsilon}, P_{\theta^*}) \le \epsilon \TV(P_{Q}, P_{\theta^*}) \le \epsilon$
as in (\ref{eq:logit-pop-combine-1}).

Given the upper bound on $d(\bar\theta, \theta^*)$, the location result then follows from Proposition~\ref{pro:local-linear-pop} provided that $\err_{h0}(\epsilon) \le a$ for a constant $a \in [0, 1/2)$. The variance matrix result follows if $\err_{h0}(\epsilon) \le a/(1+S_{1,a})$.
\end{proof}

\subsection{Proof of Theorem~\ref{thm:hinge-L1}}

We state and prove the following result which implies Theorem \ref{thm:hinge-L1}.
See the Supplement for details about how Proposition~\ref{pro:hinge-L1-detailed} implies Theorem~\ref{thm:hinge-L1}.
For $\delta \in (0,1/7)$, define
\begin{align*}
& \lambda_{11} = \sqrt{\frac{2\log({5 p}) + \log(\delta^{-1})}{n}} + \frac{2\log({5 p}) + \log(\delta^{-1})}{n} , \\
& \lambda_{12} =  2C_{\mathrm{rad4}} \sqrt{\frac{ \log(2p(p+1))}{n}}  + \sqrt{\frac{2 \log(\delta^{-1})}{n}} ,
\end{align*}
where $C_{\mathrm{rad4}}=C_{\mathrm{sg6}} C_{\mathrm{rad3}}$, depending on universal constants $C_{\mathrm{sg6}}$ and $C_{\mathrm{rad3}}$ in
Lemmas~\ref{lem:subg-max} and Corollary \ref{cor:entropy-sg2} in the Supplement. Denote
\begin{align*}
\err_{h1} (n,p, \delta, \epsilon) =  3\epsilon + 2\sqrt{\epsilon/(n\delta)} +  \lambda_{12} + \lambda_1 .
\end{align*}
\begin{pro} \label{pro:hinge-L1-detailed}
Assume that $\|\Sigma_{*}\|_{\max} \le M_1$ and $\epsilon \le 1/5$. Let $\hat \theta=(\hat\mu,\hat\Sigma)$ be a solution to (\ref{eq:hinge-gan-L1}) with $\lambda _1 \ge C_{\mathrm{sp13}} M_{11} \lambda_{11}$ where $M_{11} =  M_1^{1/2} ( M_1^{1/2} +  {2} \sqrt{2\pi})$ and $C_{\mathrm{sp13}} = (5/3) ( C_{\mathrm{sp11}} \vee C_{\mathrm{sp12}})$,
depending on universal constants $ C_{\mathrm{sp11}}$ and $ C_{\mathrm{sp12}}$ in Lemma~\ref{lem:spline-L1-upper} in the Supplement. If $\sqrt{\epsilon(1-\epsilon)/(n \delta) }  \leq 1/5$ and $\err_{h1} (n,p, \delta, \epsilon) \le a $ for a constant $a\in(0,1/2)$,
then the following holds with probability at least $1-7\delta$ uniformly over contamination distribution $Q$,
\begin{align*}
 \| \hat\mu - \mu^* \|_\infty & {\le} S_{4,a}  \err_{h1} (n,p, \delta, \epsilon), \\
 \| \hat\Sigma - \Sigma^*\|_{\max} & {\le} S_{8,a}  \err_{h1} (n,p, \delta, \epsilon) ,
\end{align*}
where $S_{4,a} =  (1 +\sqrt{2 M_1\log\frac{2}{1-2a}})  /a $
and $S_{8,a} =  2 M_1^{1/2} S_{6,a} + S_{7} ( 1+ S_{4,a} + S_{6,a}) $ with
$S_{6,a}= S_{5} ( 1 + S_{4,a}/2)$, $S_{5} = 2\sqrt{2\pi} ( 1- \me^{-2/M_1})^{-1} $, and
$S_{7} =8 \pi M_1  \me^{1/(4M_1)}$
$S_{7} =4 \{ (\frac{1}{\sqrt{2\pi M_1} } \me^{- 1/(8M_1)} ) \vee (1 - 2 \me^{- 1/(8M_1)} ) \}^{-2} $.
\end{pro}

\begin{proof}[Proof of Proposition~\ref{pro:hinge-L1-detailed}]
The main strategy of our proof is to show that the following inequalities hold with high probabilities,
\begin{align}
& d(\hat\theta, \theta^*) - \Delta_{12} \le
\max_{\gamma\in\Gamma}  \; \left\{ K_\HG (P_n, P_{\hat\theta}; h_{\gamma,\hat\mu}) - \lambda_1 \;\pen_1(\gamma) \right\}
\le \Delta_{11} , \label{eq:main-frame}
\end{align}
where $\Delta_{11}$ and $\Delta_{12}$ are error terms,
and $d(\theta^*, \hat\theta)$ is a moment matching term, which under certain conditions delivers upper bounds, up to scaling constants,
on the estimation errors to be controlled, $\| \hat\mu- \mu^* \|_\infty$ and $\|\hat \Sigma - \Sigma^* \|_{\max}$.

(Step 1) For upper bound in (\ref{eq:main-frame}), we show that with probability at least $1-5\delta$,
\begin{align}
& \quad \max_{\gamma\in\Gamma}  \; \left\{ K_\HG (P_n, P_{\hat\theta}; h_{\gamma,\hat\mu}) - \lambda_1 \;\pen_1(\gamma) \right\} \nonumber \\
& \le \max_{\gamma\in\Gamma}  \; \left\{ K_\HG (P_n, P_{\theta^*}; h_{\gamma,\mu^*}) - \lambda_1 \;\pen_1(\gamma) \right\} \label{eq:hinge-L1-upper-1} \\
& \leq \max_{\gamma\in\Gamma}  \; \left\{ \Delta_{11} + \pen_1( \gamma) \tilde\Delta_{11} - \lambda_1 \;\pen_1(\gamma) \right\} .  \label{eq:hinge-L1-upper-2}
\end{align}
Inequality (\ref{eq:hinge-L1-upper-1}) follows from the definition of $\hat\theta$.
Inequality (\ref{eq:hinge-L1-upper-2}) follows from Proposition~\ref{pro:hinge-L1-upper}: it holds with probability at least $1-{7}\delta$ that for any $\gamma \in\Gamma$,
\begin{align*}
K_\HG (P_n, P_{\theta^*}; h_{\gamma,\mu^*}) \le \Delta_{11} + \pen_1( \gamma) \tilde\Delta_{11} ,
\end{align*}
where $\Delta_{11} = 2(\epsilon + \sqrt{\epsilon/(n\delta)} ) $, $\tilde\Delta_{11} = C_{\mathrm{sp13}} M_{11} \lambda_{11}$, and
$$
\lambda_{11} = \sqrt{\frac{2\log(5 p) + \log(\delta^{-1})}{n}} + \frac{2\log(5  p)+ \log(\delta^{-1})}{n} .
$$
From (\ref{eq:hinge-L1-upper-1})--(\ref{eq:hinge-L1-upper-2}), the upper bound in (\ref{eq:main-frame}) holds with probability at least $1-5\delta$,
provided that  the tuning parameter $\lambda_1$ is chosen such that $\lambda_1 \ge \tilde\Delta_{11} $.

(Step 2) For the lower bound in (\ref{eq:main-frame}), we show that with probability at least $1- 2\delta$,
\begin{align}
& \quad \max_{\gamma\in\Gamma}  \; \left\{ K_\HG (P_n, P_{\hat\theta}; h_{\gamma,\hat\mu}) - \lambda_1 \;\pen_1(\gamma) \right\} \nonumber \\
&\geq \max_{\gamma\in\Gamma_0}  \; \left\{ K_\HG (P_n, P_{\hat\theta}; h_{\gamma,\hat\mu}) - \lambda_1 \;\pen_1(\gamma) \right\} \label{eq:hinge-L1-lower-1}\\
&\ge \max_{\gamma\in\Gamma_0}  \; \left\{\E_{P_{\theta^*}} h_{\gamma,\hat\mu}(x) - \E_{P_{\hat\theta}}h_{\gamma,\hat\mu}(x) \right\} - \tilde\Delta_{12} - \lambda_1   . \label{eq:hinge-L1-lower-2}
\end{align}
Inequality (\ref{eq:hinge-L1-lower-1}) holds provided that $\Gamma_0$ is a subset of $\Gamma$.

Take $\Gamma_0 = \{ \gamma \in \Gamma_\rp: \gamma_0=0, \pen_1(\gamma) = 1 \}$,
{where $\Gamma_\rp$ is the subset of $\Gamma$ associated with pairwise ramp functions as in the proof of Theorem \ref{thm:logit-L1}. }
Inequality (\ref{eq:hinge-L1-lower-2}) follows from Proposition \ref{pro:hinge-L1-lower} {because} $h_{\gamma,\hat\mu}(x) \in [-1,1]$ for {$\gamma \in \Gamma_0$, {and hence} the hinge loss reduces to} a moment {matching term}: it holds
with probability at least $1-2 \delta$ that for any $\gamma \in \Gamma_0$,
\begin{align*}
 K_\HG(P_{n}, P_{\hat\theta}; h_{\gamma, \hat\mu}) \ge \left\{\E_{P_{\theta^*}} h_{\gamma, \hat\mu}(x) - \E_{P_{\hat\theta}} h_{\gamma, \hat\mu}(x) \right\} - \tilde{\Delta}_{12}
\end{align*}
where $\tilde{\Delta}_{12} = \epsilon + \lambda_{12}$ and
$$
\lambda_{12} = C_{\mathrm{rad4}} \sqrt{\frac{4 \log(2p(p+1))}{n}} + \sqrt{\frac{2 \log(\delta^{-1})}{n}}.
$$
From (\ref{eq:hinge-L1-lower-1})--(\ref{eq:hinge-L1-lower-2}), the lower bound in (\ref{eq:main-frame}) holds with probability at least $1-2\delta$,
where $\Delta_{12} = \tilde\Delta_{12} + \lambda_1 $ and $d(\hat\theta, \theta^*)
= \max_{\gamma\in\Gamma_0} \{\E_{P_{\theta^*}} h_{\gamma,\hat\mu} (x)- \E_{P_{\hat\theta}}h_{\gamma,\hat\mu}(x)\}$.

(Step 3) We complete the proof by relating the moment matching term $d(\hat\theta,\theta^*)$
to the estimation error between $\hat\theta$ and $\theta^*$.
First, %because $h_{\gamma, \hat\mu}$ is linear in $\gamma$,
combining the lower and upper bounds in (\ref{eq:main-frame}) shows
that with probability at least $1-{9}\delta$,
\begin{align}
\max_{\gamma\in\Gamma_\rp, \pen_1(\gamma)=1 } \left\{\E_{P_{\theta^*}} h_{\gamma,\hat\mu}(x) - \E_{P_{\hat\theta}}h_{\gamma,\hat\mu}(x) \right\} \le  \err_{h1} (n,p, \delta, \epsilon).\label{eq:hinge-L1-combine}
\end{align}
where
\begin{align*}
    \err_{h1} (n,p, \delta, \epsilon)
& =  3\epsilon + 2\sqrt{\epsilon/(n\delta)} +  \lambda_{12} + \lambda_1 .
\end{align*}
The desired result then follows from Proposition~\ref{pro:logit-L1-combine}: provided ${\err_{h1}} (n,p, \delta, \epsilon) \le a $, inequality (\ref{eq:hinge-L1-combine}) implies that
\begin{align*}
& \| \hat\mu - \mu^* \|_\infty \le S_{4,a}  {\err_{h1}} (n,p, \delta, \epsilon), \quad \| \hat\Sigma - \Sigma^* \|_{\max}\le S_{8,a} {\err_{h1}} (n,p, \delta, \epsilon) .
\end{align*}
\end{proof}
\begin{supplement}
    \textbf{Supplement to ``Tractable and Near-Optimal Adversarial Algorithms for Robust Estimation in Contaminated Gaussian Models''}. The Supplement provides
    additional information for simulation studies, and main proofs and technical details. Auxiliary lemmas and technical tools are also provided.
\end{supplement}

\bibliographystyle{imsart-nameyear} % Style BST file (imsart-number.bst or imsart-nameyear.bst)
\bibliography{final}       % Bibliography file (usually '*.bib')

% Supp
\clearpage

\setcounter{page}{1}

\setcounter{section}{0}
\setcounter{equation}{0}

\setcounter{figure}{0}
\setcounter{table}{0}

\renewcommand{\theequation}{S\arabic{equation}}
\renewcommand{\thesection}{\Roman{section}}

\renewcommand\thefigure{S\arabic{figure}}
\renewcommand\thetable{S\arabic{table}}

\setcounter{lem}{0}
\renewcommand{\thelem}{S\arabic{lem}}

\begin{center}
{\Large Supplementary Material for}\\[.05in]
{\Large ``Tractable and Near-Optimal Adversarial Algorithms for}\\
{\Large Robust Estimation in Contaminated Gaussian Models''}

\vspace{.1in} {\large Ziyue Wang and Zhiqiang Tan}
\end{center}

\vspace{.2in}
The Supplementary Material contains Appendices I-V on additional information for simulation studies, main proofs and details, auxiliary lemmas, and technical tools.

\section{Additional information for simulation studies} \label{sec:additional-simulation}
% \usepackage{graphicx}

\subsection{Implementation of proposed methods}

The detailed algorithm to implement our logit $f$-GANs and hinge GAN is shown as Algorithm~\hyperref[alg:2]{S1}. % Have to use customized hyperref for S-named algos in Supplement.
In our experiments, the learning rate $\alpha_t$ is set with $\alpha_0=1$,
and the iteration number is set to $100$. This choice leads to stable convergence while keeping the running time relatively short.

For rKL logit $f$-GAN, we modify the un-penalized objective $K_{\mathrm{rKL}}$ as $\min(K_{\mathrm{rKL}},10)$.
This modification helps stabilize the initial steps of training where the fake data and real data can be almost separable depending on the initial value of $\theta$.

\setcounter{algocf}{0} % customized algorithm number
\renewcommand{\thealgocf}{S\arabic{algocf}} % customized algorithm number, adding "S".

\begin{algorithm} \label{alg:2}
    \caption{Penalized logit $f$-GAN or hinge GAN (in detail)}
    \KwSty{Require} A penalized GAN objective function $K(\theta,\gamma; \lambda)$ as in (\ref{eq:logit-fgan-transf}) for logit $f$-GAN or with $K_f$ replaced by $K_{\HG}$ for hinge GAN, a decaying learning rate $\alpha_t=\alpha_0\exp(-0.05t)$, a base penalty level $\lambda_0$ so that $\lambda_1 = \lambda_0 \sqrt{\log(p)/n}$, $\lambda_2 = \lambda_0 \sqrt{p/n}$, and $\lambda_3 = \lambda_0 \sqrt{p^2/n}$. \\
    %The constants $\alpha_0$ is set to be $1$ and $\lambda_0$ varies for different logit $f$-GANs.
    \KwSty{Initialization}{ Initialize $\mu_0$ by the median of real data $x$. Initialize $\Sigma_0$ by $\hat{S} \hat{K} \hat{S} $ where $\hat{K}$ is Kendall's $\tau$ correlation coefficient matrix and $\hat{S}$ is a diagonal matrix with the MAD scale estimates on the diagonal.
    }
    \Repeat{}{
        \KwSty{Sampling:} Draw $(Z_1,\ldots,Z_n)$ from $\N(0,I)$ and approximate $P_{\theta_{t-1}}$ by the empirical distribution on the fake data $\xi_i = \mu_{t-1} + \Sigma_{t-1}^{1/2} z_i$, $i=1,\ldots,n$. \\
        \KwSty{Updating:} Compute $\gamma_{t} =\argmax_{\gamma} K(\theta_{t-1}, \gamma; \lambda)$ by the optimizer \texttt{MOSEK}, and compute $\theta_t = \theta_{t-1} - \alpha_t \nabla_{\theta} K (\theta, \gamma_t; \lambda) |_{\theta_{t-1}}$ by gradient descent where the gradient $\nabla_{\theta} K (\theta, \gamma_t; \lambda) |_{\theta_{t-1}}$ is clipped at $0.1$ by $L_\infty$ norm.
    }(\KwSty{until} convergence or the maximum iterations reached.)

\end{algorithm}

\vspace{-.1in}

    \subsection{Tuning penalty levels}
     We conducted tuning experiments to identify base penalty levels $\lambda_0$ which are expected to work reasonably well in various settings
     for our logit $f$-GANs and hinge GAN, where the dependency on $(p,n)$ is already absorbed in the penalty parameters $\lambda_1, \lambda_2, \lambda_3$. In the tuning experiments, we tried two contamination proportions $\epsilon$ and two choices of contamination distributions $Q$ as described in Section \ref{sec:simulation-setting}. Results are collected from 20 repeated experiments on a grid of penalty levels for each method.

     As shown in Figure~\ref{fig:Penalty}, although the average estimation error varies as the contamination setting changes, there is a consistent and stable range of
     the penalty level $\lambda_0$ which leads to approximately the best performance for each method, with either $L_1$ or $L_2$ penalty used.
     We manually pick $\lambda_0 = 0.1$ for the hinge GAN, $\lambda_0=0.025$ for the JS logit $f$-GAN, and $\lambda_0=0.3$ for the rKL logit $f$-GAN,
     which are then fixed in all subsequent simulations.
      {It is desirable in future work to design and study automatic tuning of penalty levels for GANs.}
     %These numbers are not guaranteed to be exactly optimal, any number within the appropriate interval will work.

    \begin{figure}[htbp]
        \includegraphics[scale=0.11]{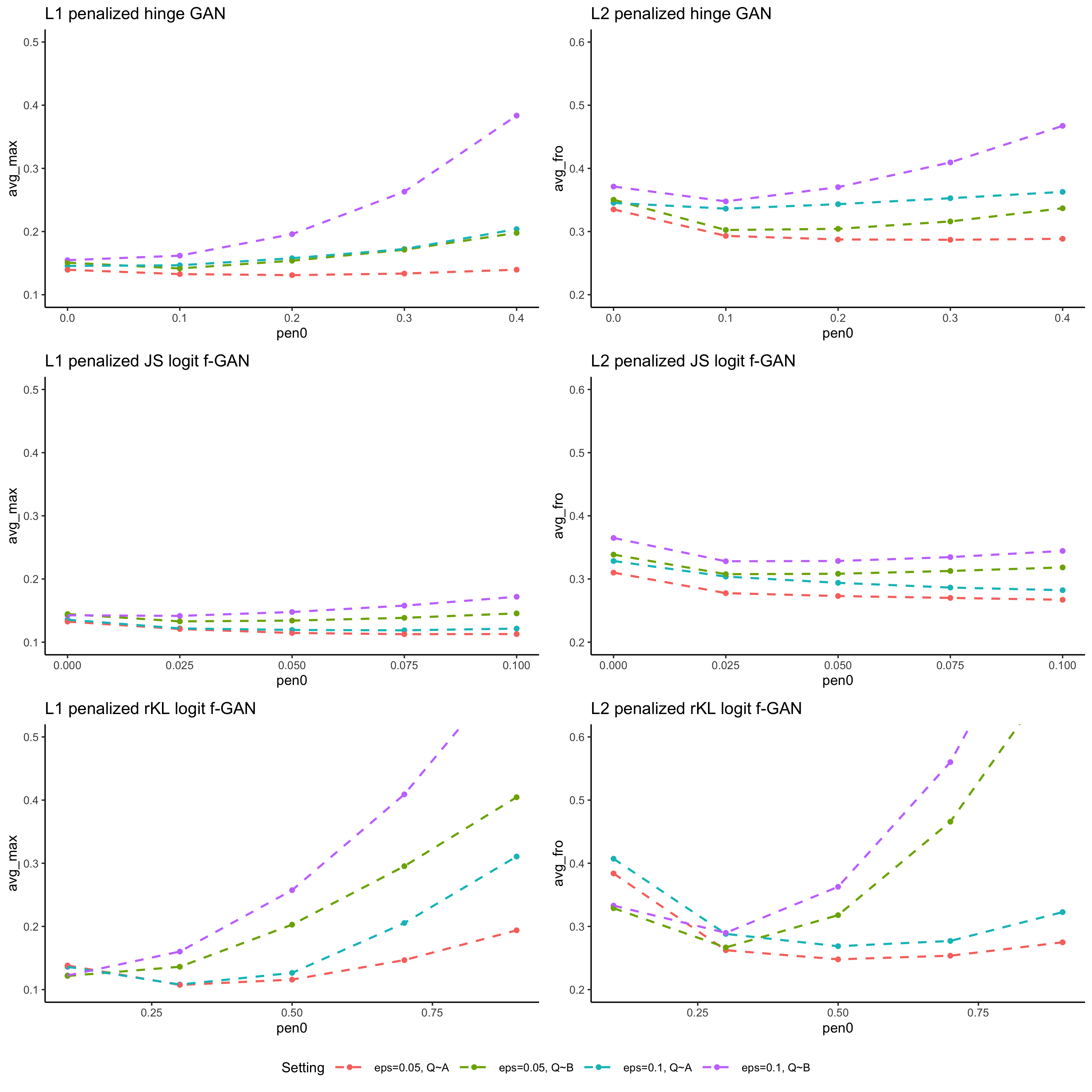}
        \caption{Tuning the penalty level. In this setting $p=5$, $n=500$, $\epsilon \in \{0.05, 0.1\}$, and contamination distribution $Q$ is either (A) $\N(2.25 c, (1/3)I_p)$ or (B) $\N(5 1_p, 5I_p)$. Penalty levels for the hinge GAN are $\{0, 0.1, 0.2, 0.3, 0.4\}$, penalty levels for the JS logit $f$-GAN are $\{0, 0.025, 0.05, 0.075, 0.1\}$, and penalty levels for the rKL logit $f$-GAN are $\{0.1, 0.3, 0.5, 0.7, 0.9\}$.}
        \label{fig:Penalty}
    \end{figure}

    \subsection{Error dependency on $n$ and $p$}
     Tables~\ref{tab:n-l1}--\ref{tab:n-l2} show the performance of various methods depending on sample size $n$ for the two choices of contamination in Section \ref{sec:simulation-setting}. We fix the dimension $p=10$ and $\epsilon =0.1$ and increase $n$ from $500$ to $4000$. Tables~\ref{tab:p-l1}--\ref{tab:p-l2} show how the performance of methods depending on sample size $p$ for the two choices of contamination. We fix $\epsilon=0.1$ and $n=2000$ and increase $p$ from $5$ to $15$. Estimation errors are measured in the maximum norm and the Frobenius norm.

For all methods considered, the estimation errors decrease as $n$ increases. As $p$ increases, however, the estimation errors seem to be affected to a lesser extent when measured in the maximum norm. This is expected because an error rate $\sqrt{\log(p)/n}$ ($\epsilon$ term aside) has been established for our three $L_1$ penalized methods as well as Kendall's $\tau$ and Spearman's $\rho$ (\citeappend{LT18}). When measured in the Frobenius norm, the estimation errors go up as $p$ increases, which is also expected. In Table \ref{tab:p-l1} with $Q\sim \N(5 1_p, 5I_p)$, the $\gamma$-Lasso outperforms our methods in the setting of $p=15$.
     and even has better performance than in the setting of smaller $p$. A possible explanation is that the default choice $\gamma=0.05$ may be near optimal for this particular setting of $p=15$, but much less suitable in other settings.

\begin{table}[htbp]
    \caption{Comparison of existing methods and proposed $L_1$ penalized GAN methods ($p=10, \epsilon=0.1$).
    Estimation error of the variance matrix is reported in the maximum norm $\|\cdot\|_{\max}$.}
    \label{tab:n-l1}
    \resizebox{\textwidth}{!}{%
    \begin{tabular}{llllllll}
    \hline
    \multicolumn{1}{l|}{$n$} & hinge GAN          & JS logit $f$-GAN              & rKL logit $f$-GAN          & GYZ JS-GAN   & Tyler\_M              &   Kendall\_MAD        & Spearman\_Qn         \\ \hline
    \multicolumn{8}{c}{$Q \sim \N\left(2.25 c, \frac{1}{3}I_p\right)$}   \\
500  & 0.1740 (0.0529) & 0.1775 (0.0496) & \textbf{0.1418 (0.0335)} & 0.3357 (0.1771) & 0.4211 (0.0570) & 0.5304 (0.0824) & 0.5074 (0.0555) \\
1000 & 0.1238 (0.0325) & 0.1134 (0.0292) & \textbf{0.1027 (0.0212)} & 0.1982 (0.0917) & 0.3830 (0.0320) & 0.4846 (0.0735) & 0.4715 (0.0431)  \\
2000 & 0.0980 (0.0222) & 0.0737 (0.0207) & \textbf{0.0720 (0.0151)} & 0.1029 (0.0467) & 0.3526 (0.0226) & 0.4133 (0.0264) & 0.4362 (0.0256) \\
4000 & 0.0848 (0.0184) & 0.0528 (0.0124) & \textbf{0.0502 (0.0105)} & 0.0770 (0.0241) & 0.3351 (0.0166) & 0.3948 (0.0259) & 0.4117 (0.0186) \\
\multicolumn{8}{c}{$Q \sim \N(5 1_p, 5 I_p)$} \\

    500  & 0.1855 (0.0536) & 0.1970 (0.0498) & \textbf{0.1548 (0.0353)} & 0.2792 (0.1706) & 0.6061 (0.0579) & 0.5200 (0.0846) & 0.6393 (0.0807)  \\
    1000 & 0.1382 (0.0331) & 0.1231 (0.0324) & \textbf{0.1230 (0.0266)} & 0.1618 (0.0615) & 0.5734 (0.0429) & 0.4724 (0.0732) & 0.6187 (0.0578) \\
    2000 & 0.1050 (0.0237) & \textbf{0.0802 (0.0225)} & 0.0935 (0.0219) & 0.1275 (0.0453) & 0.5295 (0.0332) & 0.4074 (0.0278) & 0.5693 (0.0358)  \\
    4000 & 0.0903 (0.0180) & \textbf{0.0590 (0.0141)} & 0.0739 (0.0167) & 0.0881 (0.0285) & 0.5154 (0.0225) & 0.3863 (0.0251) & 0.5481 (0.0284) \\
    \hline
    \multicolumn{1}{l|}{$n$} & $\gamma$-LASSO & MVE & MCD & Sest & Mest  &MMest   \\ \hline
     \multicolumn{8}{c}{$Q \sim \N\left(2.25 c, \frac{1}{3}I_p\right)$}   \\
   500  & 0.5786 (0.0288) & 0.2270 (0.0518) & 0.2175 (0.0516) & 0.2819 (0.0597) & 0.1900 (0.0461) & 0.2805 (0.0573) \\
1000 & 0.5426 (0.0261) & 0.1877 (0.0362) & 0.1883 (0.0397) & 0.2650 (0.0422) & 0.1551 (0.0394) & 0.2630 (0.0413) \\
2000 & 0.5277 (0.0154) & 0.1637 (0.0269) & 0.1618 (0.0257) & 0.2358 (0.0292) & 0.1292 (0.0260) & 0.2349 (0.0292) \\
4000 & 0.5150 (0.0139) & 0.1457 (0.0211) & 0.1471 (0.0222) & 0.2206 (0.0221) & 0.1127 (0.0218) & 0.2198 (0.0219) \\
 \multicolumn{8}{c}{$Q \sim \N\left(2.25 c, \frac{1}{3}I_p\right)$}   \\
   500  & 0.3472 (0.1021) & 0.2179 (0.0542) & 0.2190 (0.0508) & 0.2819 (0.0597) & 0.1853 (0.0481) & 0.2805 (0.0573) \\
1000 & 0.3118 (0.0843) & 0.1953 (0.0435) & 0.1887 (0.0397) & 0.2650 (0.0422) & 0.1547 (0.0399) & 0.2630 (0.0412) \\
2000 & 0.2935 (0.0619) & 0.1685 (0.0303) & 0.1617 (0.0251) & 0.2358 (0.0292) & 0.1284 (0.0255) & 0.2349 (0.0291) \\
4000 & 0.2627 (0.0422) & 0.1511 (0.0216) & 0.1470 (0.0222) & 0.2206 (0.0221) & 0.1132 (0.0217) & 0.2198 (0.0219)\\
\hline
\end{tabular}%
}
\end{table}

\begin{table}[htbp]
    \caption{Comparison of existing methods and proposed $L_2$ penalized GAN methods ($p=10, \epsilon=0.1$).
    Estimation error of the variance matrix is reported in the Frobenius norm $\|\cdot\|_{\fro}$.}
    \label{tab:n-l2}
    \resizebox{\textwidth}{!}{%
    \begin{tabular}{llllllll}
    \hline
    \multicolumn{1}{l|}{$n$} & hinge GAN          & JS logit $f$-GAN              & rKL logit $f$-GAN          & GYZ JS-GAN   & Tyler\_M  & Kendall\_MAD         &    Spearman\_Qn     \\
    \hline
    \multicolumn{8}{c}{$Q \sim \N\left(2.25 c, \frac{1}{3}I_p\right)$}       \\
    500  & \textbf{0.6855 (0.1138)} & 0.7519 (0.1238) & 0.8294 (0.1381)  & 0.9512 (0.3576) & 2.4034 (0.1264) & 3.3007 (0.1748) & 3.5111 (0.1224)  \\
1000 & 0.4455 (0.0645)   & 0.4304 (0.0683) & \textbf{0.4197 (0.0734)} & 0.6198 (0.2318) & 2.3129 (0.0668) & 3.2204 (0.0862) & 3.4565 (0.0645) \\
2000 & 0.3222 (0.0547)  & 0.2898 (0.0565) & \textbf{0.2723 (0.0477)} & 0.3642 (0.0885) & 2.3363 (0.0557) & 3.2143 (0.0489) & 3.4669 (0.0398)  \\
4000 & 0.2431 (0.0438)  & 0.1948 (0.0380) & \textbf{0.1877 (0.0346)} & 0.2797 (0.0738) & 2.3165 (0.0405) & 3.2030 (0.0417) & 3.4472 (0.0219) \\
\multicolumn{8}{c}{$Q \sim \N(5 1_p, 5 I_p)$} \\
    500  & \textbf{0.7157 (0.1151)} & 0.8054 (0.1400) & 0.7190 (0.0970)  & 0.8889 (0.3253) & 3.3388 (0.2880) & 3.3704 (0.2784) & 4.1824 (0.2852)  \\
    1000 & 0.4589 (0.0695)  & 0.4589 (0.0850) & \textbf{0.3942 (0.0606)} & 0.5367 (0.1282) & 3.2529 (0.2480) & 3.3263 (0.1921) & 4.1522 (0.2150) \\
    2000 & 0.3301 (0.0583) & 0.3049 (0.0635) & \textbf{0.2744 (0.0505)} & 0.3829 (0.0852) & 3.2420 (0.2083) & 3.3202 (0.1462) & 4.1527 (0.1602)  \\
    4000 & 0.2472 (0.0458)& 0.2076 (0.0419) & \textbf{0.1990 (0.0334)} & 0.2986 (0.0536) & 3.1988 (0.1288) & 3.2955 (0.1269) & 4.1117 (0.1182) \\
    \hline
\multicolumn{1}{l|}{$n$} & $\gamma$-LASSO & MVE & MCD & Sest & Mest  &MMest   \\
 \hline
\multicolumn{8}{c}{$Q \sim \N\left(2.25 c, \frac{1}{3}I_p\right)$}   \\
500  & 4.4395 (0.0997) & 0.7574 (0.1373) & 0.7448 (0.1249) & 0.9269 (0.1562) & 0.6615 (0.1164) & 0.9190 (0.1562) \\
1000 & 4.4150 (0.0618) & 0.6005 (0.1132) & 0.5843 (0.0960) & 0.8262 (0.1130) & 0.4894 (0.0948) & 0.8206 (0.1137) \\
2000 & 4.4364 (0.0688) & 0.5416 (0.0816) & 0.5170 (0.0754) & 0.7732 (0.0839) & 0.4078 (0.0670) & 0.7707 (0.0831) \\
4000 & 4.4204 (0.0444) & 0.4663 (0.0613) & 0.4542 (0.0623) & 0.7232 (0.0714) & 0.3453 (0.0577) & 0.7231 (0.0711) \\
\multicolumn{8}{c}{$Q \sim \N(5 1_p, 5 I_p)$} \\
500  & 2.0989 (0.7667) & 0.7577 (0.1291) & 0.7497 (0.1272) & 0.9269 (0.1562) & 0.6581 (0.1176) & 0.9190 (0.1562) \\
1000 & 1.9989 (0.6520) & 0.6021 (0.1016) & 0.5819 (0.0959) & 0.8263 (0.1131) & 0.4889 (0.0988) & 0.8211 (0.1140) \\
2000 & 2.1356 (0.4419) & 0.5314 (0.0835) & 0.5163 (0.0754) & 0.7732 (0.0839) & 0.4083 (0.0675) & 0.7708 (0.0831) \\
4000 & 1.9796 (0.3431) & 0.4740 (0.0628) & 0.4541 (0.0621) & 0.7232 (0.0714) & 0.3460 (0.0574) & 0.7232 (0.0711)\\
\hline

    \end{tabular}%
    }
\end{table}

\begin{table}[htbp]
    \caption{Comparison of existing methods and proposed $L_1$ penalized GAN methods ($n=2000, \epsilon=0.1$).
    Estimation error of the variance matrix is reported in the maximum norm $\|\cdot\|_{\max}$.}
    \label{tab:p-l1}
    \resizebox{\textwidth}{!}{%
    \begin{tabular}{llllllll}
    \hline
    \multicolumn{1}{l|}{$p$} & hinge GAN          & JS logit $f$-GAN              & rKL logit $f$-GAN        & GYZ JS-GAN     & Tyler\_M        & Kendall\_MAD         & Spearman\_Qn  \\
    \hline
    \multicolumn{8}{c}{$Q \sim \N\left(2.25 c, \frac{1}{3}I_p\right)$}  \\
5  & 0.1179 (0.0220) & \textbf{0.0561 (0.0147)} & 0.0576 (0.0138)  & 0.1315 (0.0845) & 0.2153 (0.0231) & 0.3880 (0.0342) & 0.4050 (0.0260)\\
10 & 0.0980 (0.0222) & 0.0737 (0.0207)  & \textbf{0.0720 (0.0151)} & 0.1029 (0.0467) & 0.3526 (0.0226) & 0.4133 (0.0264) & 0.4362 (0.0256) \\
15 & 0.0929 (0.0219) & 0.0905 (0.0222)  & \textbf{0.0761 (0.0130)} & 0.1586 (0.0741) & 0.5856 (0.0194) & 0.4265 (0.0337) & 0.4377 (0.0227)  \\
    \multicolumn{8}{c}{$Q \sim \N(5 1_p, 5 I_p)$}\\
    5  & 0.1344 (0.0230) & \textbf{0.0674 (0.0189)} & 0.0991 (0.0236) & 0.6339 (1.7289) & 0.4503 (0.0286) & 0.3803 (0.0225) & 0.5381 (0.0371)     \\
    10 & 0.1050 (0.0237) & \textbf{0.0802 (0.0225)} & 0.0935 (0.0219) & 0.1275 (0.0453) & 0.5295 (0.0332) & 0.4074 (0.0278) & 0.5693 (0.0358)  \\
    15 & 0.1021 (0.0252) & 0.0957 (0.0232)    & 0.0925 (0.0180) & 0.1294 (0.0432) & 0.5769 (0.0251) & 0.4207 (0.0322) & 0.5782 (0.0301) \\ \hline
    \multicolumn{1}{l|}{$p$} & $\gamma$-LASSO & MVE & MCD & Sest & Mest  &MMest   \\
 \hline
\multicolumn{8}{c}{$Q \sim \N\left(2.25 c, \frac{1}{3}I_p\right)$}   \\
5  & 0.4866 (0.0197) & 0.2107 (0.0270) & 0.1973 (0.0246) & 0.2305 (0.0267) & 0.1398 (0.0296) & 0.2226 (0.0243) \\
10 & 0.5277 (0.0154) & 0.1662 (0.0253) & 0.1608 (0.0241) & 0.2358 (0.0292) & 0.1286 (0.0269) & 0.2349 (0.0292) \\
15 & 0.5556 (0.0204) & 0.1473 (0.0272) & 0.1430 (0.0246) & 0.3918 (0.1740) & 0.1190 (0.0241) & 0.3918 (0.1740) \\
\multicolumn{8}{c}{$Q \sim \N(5 1_p, 5 I_p)$} \\
5  & 1.2973 (0.0860) & 0.2123 (0.0261) & 0.1971 (0.0241) & 0.2306 (0.0266) & 0.1398 (0.0296) & 0.2311 (0.0270) \\
10 & 0.2935 (0.0619) & 0.1698 (0.0258) & 0.1616 (0.0256) & 0.2358 (0.0292) & 0.1292 (0.0264) & 0.2349 (0.0291) \\
15 & \textbf{0.0885 (0.0191)} & 0.1462 (0.0263) & 0.1432 (0.0262) & 0.2362 (0.0287) & 0.1191 (0.0241) & 0.2362 (0.0287)\\
\hline
    \end{tabular}%
    }
\end{table}

\begin{table}[htbp]
    \caption{Comparison of existing methods and proposed $L_2$ penalized GAN methods ($n=2000, \epsilon=0.1$).
    Estimation error of the variance matrix is reported in the Frobenius norm $\|\cdot\|_{\fro}$.}
    \label{tab:p-l2}
    \resizebox{\textwidth}{!}{%
    \begin{tabular}{llllllll}
    \hline
    \multicolumn{1}{l|}{$p$} & hinge GAN          & JS logit $f$-GAN              & rKL logit $f$-GAN         & GYZ JS-GAN    & Tyler\_M        & Kendall\_MAD         & Spearman\_Qn  \\
    \hline
    \multicolumn{8}{c}{$Q \sim \N\left(2.25 c, \frac{1}{3}I_p\right)$}   \\
    5  & 0.2257 (0.0550) & 0.1345 (0.0372) & \textbf{0.1296 (0.0322)} & 0.2774 (0.1523) & 0.6149 (0.0447) & 1.5525 (0.0462) & 1.7124 (0.0383) \\
    10 & 0.3222 (0.0547) & 0.2898 (0.0565) & \textbf{0.2723 (0.0477)} & 0.3642 (0.0885) & 2.3363 (0.0557) & 3.2143 (0.0489) & 3.4669 (0.0398)  \\
    15 & 0.4578 (0.0691) & 0.4582 (0.0754) & \textbf{0.4328 (0.0460)} & 0.6387 (0.2095) & 6.8919 (0.1709) & 4.8584 (0.0640) & 5.1907 (0.0477)\\
    \multicolumn{8}{c}{$Q \sim \N(5 1_p, 5 I_p)$}\\
    5  & 0.2378 (0.0594) & \textbf{0.1514 (0.0477)} & 0.1641 (0.0589) & 2.4720 (7.0730) & 1.5732 (0.1017) & 1.5974 (0.0809) & 2.1323 (0.0963) \\
10 & 0.3301 (0.0583) & 0.3049 (0.0635)  & \textbf{0.2744 (0.0505)} & 0.3829 (0.0852) & 3.2420 (0.2083) & 3.3202 (0.1462) & 4.1527 (0.1602)  \\
15 & 0.4683 (0.0766) & 0.4834 (0.0953)    & \textbf{0.4183 (0.0486)} & 0.5459 (0.0998) & 4.9017 (0.2202) & 4.9803 (0.1710) & 6.0796 (0.1809)\\
\hline
\multicolumn{1}{l|}{$p$} & $\gamma$-LASSO & MVE & MCD & Sest & Mest  &MMest   \\
 \hline
\multicolumn{8}{c}{$Q \sim \N\left(2.25 c, \frac{1}{3}I_p\right)$}   \\
5  & 2.1000 (0.0427) & 0.4682 (0.0742) & 0.4494 (0.0608) & 0.5283 (0.0666) & 0.3014 (0.0649) & 0.5079 (0.0635) \\
10 & 4.4364 (0.0688) & 0.5338 (0.0846) & 0.5175 (0.0752) & 0.7732 (0.0839) & 0.4078 (0.0684) & 0.7707 (0.0831) \\
15 & 6.9967 (0.0857) & 0.6057 (0.0894) & 0.5921 (0.0777) & 3.7412 (3.1416) & 0.5143 (0.0672) & 3.7413 (3.1417) \\
\multicolumn{8}{c}{$Q \sim \N(5 1_p, 5 I_p)$} \\
5  & 5.3996 (0.3406) & 0.4685 (0.0754) & 0.4496 (0.0617) & 0.5284 (0.0669) & 0.3014 (0.0649) & 0.5450 (0.0701) \\
10 & 2.1356 (0.4419) & 0.5419 (0.0812) & 0.5175 (0.0755) & 0.7732 (0.0839) & 0.4075 (0.0689) & 0.7708 (0.0831) \\
15 & 0.4703 (0.1192) & 0.6017 (0.0789) & 0.5938 (0.0791) & 0.9585 (0.1027) & 0.5142 (0.0678) & 0.9585 (0.1027)\\
\hline
    \end{tabular}%
    }
\end{table}

\subsection{Illustration with the second contamination}
\begin{figure}[htbp]
    \includegraphics[scale=0.17]{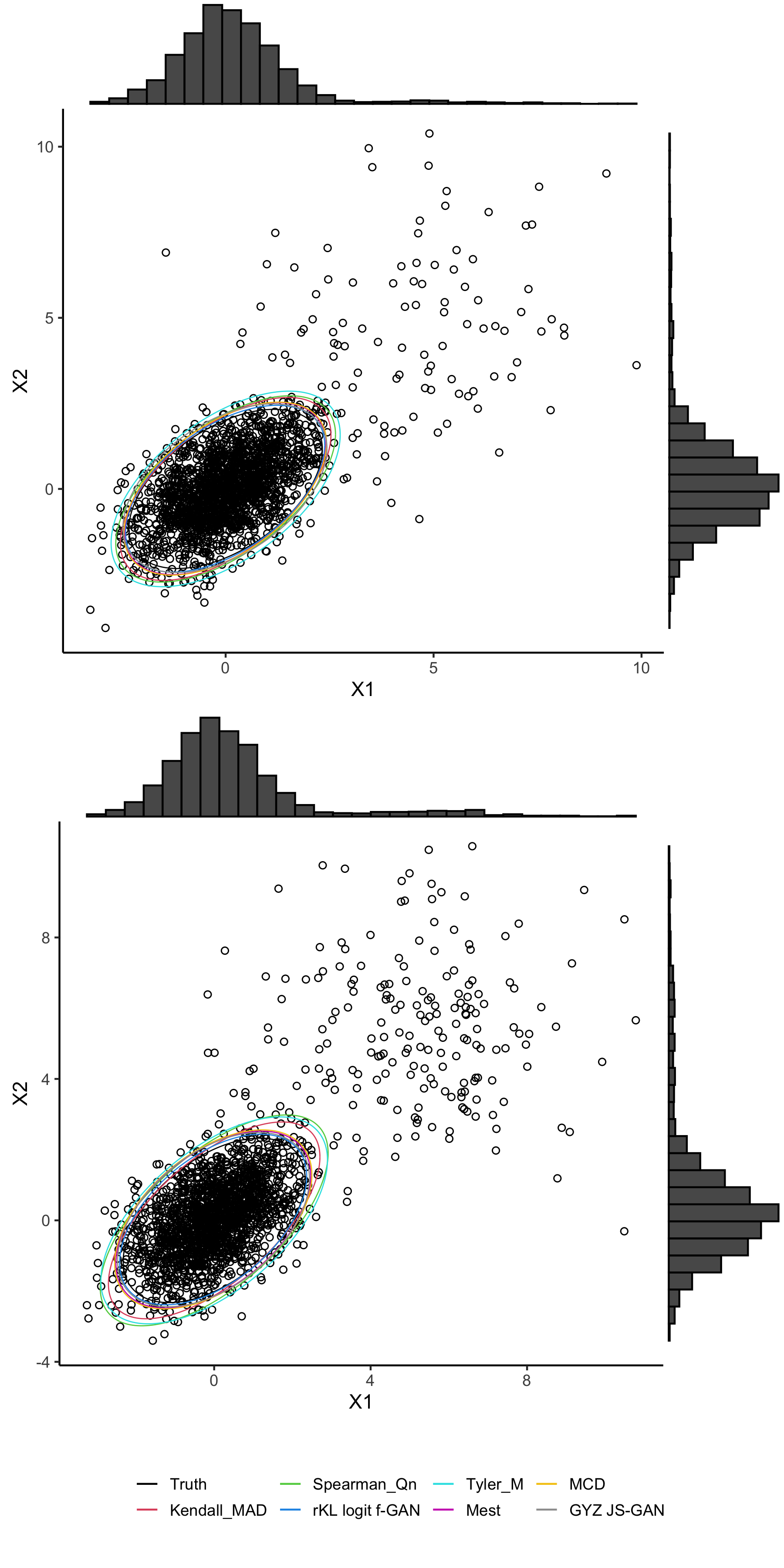}
    \caption{The $95\%$ Gaussian ellipsoids estimated for the first two coordinates and observed marginal histograms,
     from contaminated data based on the second contamination in Section \ref{sec:simulation-setting} with $\epsilon=5\%$ (top) or $10\%$ (bottom).}
    \label{fig:Ellipsoid_TypeB}
\end{figure}

    For completeness, Figure \ref{fig:Ellipsoid_TypeB} shows the 95\% Gaussian ellipsoids estimated for the first two coordinates, similarly as in Figure \ref{fig:Ellipsoid_TypeA} but
    with two samples of size $2000$ from a $10$-dimensional Huber's contaminated Gaussian distributions based on
    the second contamination $Q$ in Section \ref{sec:simulation-setting}.
    Comparison of the methods studied is qualitatively similar to that found in Figure \ref{fig:Ellipsoid_TypeA}.

\section{Main proofs of results}

\subsection{Proof of Theorem~\ref{thm:logit-L1}}\label{sec:logit-L1}
We state and prove the following result which implies Theorem~\ref{thm:logit-L1}. For $b>0$, define two factors $R_{2,b} =\sup_{ |u| \le b} \frac{\dif}{\dif u} f^\prime (\me^u)$
and $R_{3,b} = R_{31,b} + R_{32,b}$ with
 $R_{31,b} = \sup_{|u| \le b} \frac{\dif^2}{ \dif u^2} \{-f^\prime(\me^u) \}$
and $R_{32,b} = \sup_{ |u| \le b} \frac{\dif^2}{ \dif u^2} f^\# (\me^u)$.
For $\delta\in (0,1)$, define
\begin{align*}
& \lambda_{11} = \sqrt{\frac{2\log({5 p}) + \log(\delta^{-1})}{n}} + \frac{2\log({5 p}) + \log(\delta^{-1})}{n} , \\
& \lambda_{12} =  C_{\mathrm{rad4}} \sqrt{\frac{4 \log(2p(p+1))}{n}}  + \sqrt{\frac{2 \log(\delta^{-1})}{n}} ,
\end{align*}
where $C_{\mathrm{rad4}}=C_{\mathrm{sg6}} C_{\mathrm{rad3}}$, depending on universal constants $C_{\mathrm{sg6}}$ and $C_{\mathrm{rad3}}$ in
Lemmas~\ref{lem:subg-max} and Corollary \ref{cor:entropy-sg2} in the Supplement.
Denote
\begin{align*}
 \err_{f1} (n,p, \delta, \epsilon) & =  ( f^{\dprime} (1))^{-1}
 \Big\{-f^\prime(3/5)  (\sqrt{\epsilon} + \sqrt{1/(n\delta)} ) \\
 & \quad - f^{\prime}(\me^{- b_1 }) \sqrt{\epsilon} +
 \frac{1}{2} R_{3, b_1 } (\sqrt{\epsilon} +\sqrt{1/n} ) +  R_{2, b_1 } \lambda_{12} + \lambda_1  \Big\} ,
\end{align*}
where $ b_1  = \sqrt{\epsilon} +\sqrt{1/n} $.
Note that $R_{2, b}$, $R_{3, b}$ are bounded provided that $b$ is bounded, because $f$ is three-times continuously differentiable as required in Assumption~\ref{ass:f-condition2}.

\begin{manualpro}{S1} \label{pro:logit-L1-detailed}
Assume that $\| \Sigma^* \|_{\max}\le M_1$, and $f$ satisfies Assumptions~\ref{ass:f-condition}--\ref{ass:f-condition2}.
Let $\hat \theta=(\hat\mu,\hat\Sigma)$ be a solution to (\ref{eq:logit-fgan-L1}) with
$\lambda _1 \ge C_{\mathrm{sp13}} R_1 M_{11} \lambda_{11}$, where
$M_{11} =  M_1^{1/2} ( M_1^{1/2} +  {2} \sqrt{2\pi})$ and $C_{\mathrm{sp13}} = (5/3) ( C_{\mathrm{sp11}} \vee C_{\mathrm{sp12}})$,
depending on universal constants $ C_{\mathrm{sp11}}$ and $ C_{\mathrm{sp12}}$ in Lemma~\ref{lem:spline-L1-upper} in the Supplement.
If {$\epsilon\le 1/5$}, $\sqrt{\epsilon(1-\epsilon)/(n \delta) }  \leq 1/5$, and $\err_{f1} (n,p, \delta, \epsilon) \le a $ for a constant $a\in(0,1/2)$,
then we have that with probability at least $1-7\delta$,
\begin{align*}
 \| \hat\mu - \mu^* \|_\infty & \le S_{4,a}  \err_{f1} (n,p, \delta, \epsilon), \\
 \| \hat\Sigma - \Sigma^* \|_{\max} & \le S_{8,a}  \err_{f1} (n,p, \delta, \epsilon) ,
\end{align*}
where $S_{4,a} =  (1 +\sqrt{2 M_1\log\frac{2}{1-2a}})  /a $
and $S_{8,a} =  2 M_1^{1/2} S_{6,a} + S_{7} ( 1+ S_{4,a} + S_{6,a}) $ with
$S_{6,a}= S_{5} ( 1 + S_{4,a}/2)$, $S_{5} = 2\sqrt{2\pi} ( 1- \me^{-2/M_1})^{-1} $, and
%$S_{7} =8 \pi M_1  \me^{1/(4M_1)}$
$S_{7} =4 \{ (\frac{1}{\sqrt{2\pi M_1} } \me^{- 1/(8M_1)} ) \vee (1 - 2 \me^{- 1/(8M_1)} ) \}^{-2} $.
\end{manualpro}
\begin{proof}[Proof of Proposition~\ref{pro:logit-L1-detailed}]

The main strategy of our proof is to show that the following inequalities hold with high probabilities,
\begin{align}
& d(\hat\theta, \theta^*) - \Delta_{12} \le
\max_{\gamma\in\Gamma}  \; \left\{ K_f (P_n, P_{\hat\theta}; h_{\gamma,\hat\mu}) - \lambda_1 \;\pen_1(\gamma) \right\}
\le \Delta_{11} , \label{eq:logit-L1-main-frame}
\end{align}
where $\Delta_{11}$ and $\Delta_{12}$ are error terms,
and $d(\theta^*, \hat\theta)$ is a moment matching term, which under certain conditions delivers upper bounds, up to scaling constants,
on the estimation errors to be controlled, $\| \hat\mu- \mu^* \|_\infty$ and $\|\hat\Sigma - \Sigma^*\|_{\max}$.

(Step 1) For the upper bound in (\ref{eq:logit-L1-main-frame}), we show that with probability at least $1-5\delta$,
\begin{align}
& \quad \max_{\gamma\in\Gamma}  \; \left\{ K_f (P_n, P_{\hat\theta}; h_{\gamma,\hat\mu}) - \lambda_1 \;\pen_1(\gamma) \right\} \nonumber \\
& \le \max_{\gamma\in\Gamma}  \; \left\{ K_f (P_n, P_{\theta^*}; h_{\gamma,\mu^*}) - \lambda_1 \;\pen_1(\gamma) \right\} \label{eq:logit-L1-upper-1} \\
& \leq \max_{\gamma\in\Gamma}  \; \left\{ \Delta_{11} + \pen_1( \gamma) \tilde\Delta_{11} - \lambda_1 \;\pen_1(\gamma) \right\} .  \label{eq:logit-L1-upper-2}
\end{align}
Inequality (\ref{eq:logit-L1-upper-1}) follows from the definition of $\hat\theta$.
Inequality (\ref{eq:logit-L1-upper-2}) follows from Proposition~\ref{pro:logit-L1-upper}: it holds with probability at least $1-{5}\delta$ that for any $\gamma \in\Gamma$,
\begin{align*}
K_f (P_n, P_{\theta^*}; h_{\gamma,\mu^*}) \le \Delta_{11} + \pen_1( \gamma) \tilde\Delta_{11} ,
\end{align*}
where $\Delta_{11} = -f^\prime(3/5)  (\epsilon + \sqrt{\epsilon/(n\delta)} ) $, $\tilde\Delta_{11} = C_{\mathrm{sp13}} R_1 M_{11} \lambda_{11}$, and
$$
\lambda_{11} = \sqrt{\frac{2\log({5 p}) + \log(\delta^{-1})}{n}} + \frac{2\log({5 p})+ \log(\delta^{-1})}{n} .
$$
Note that $\lambda_{11}$ is the same as in the proof of Theorem~\ref{thm:hinge-L1}, and the above $\tilde{\Delta}_{11}$ differs from $\tilde{\Delta}_{11}$ in the proof of Theorem~\ref{thm:hinge-L1} only in the factor $R_1$. From (\ref{eq:logit-L1-upper-1})--(\ref{eq:logit-L1-upper-2}), the upper bound in (\ref{eq:logit-L1-main-frame}) holds with probability at least $1-5\delta$,
provided that the tuning parameter $\lambda_1$ is chosen such that $\lambda_1 \ge \tilde\Delta_{11} $.

(Step 2) For the lower bound in (\ref{eq:logit-L1-main-frame}), we show that with probability at least $1- 2\delta$,
\begin{align}
& \quad \max_{\gamma\in\Gamma}  \; \left\{ K_f (P_n, P_{\hat\theta}; h_{\gamma,\hat\mu}) - \lambda_1 \;\pen_1(\gamma) \right\} \nonumber \\
&\geq \max_{\gamma\in\Gamma_0}  \; \left\{ K_f (P_n, P_{\hat\theta}; h_{\gamma,\hat\mu}) - \lambda_1 \;\pen_1(\gamma) \right\} \label{eq:logit-L1-lower-1}\\
&\geq \max_{\gamma\in\Gamma_0}  \; f^{\dprime} (1) \left\{\E_{P_{\theta^*}} h_{\gamma,\hat\mu}(x) - \E_{P_{\hat\theta}}h_{\gamma,\hat\mu}(x) \right\} - \tilde\Delta_{12} - \lambda_1  b_1  . \label{eq:logit-L1-lower-2}
\end{align}
Inequality (\ref{eq:logit-L1-lower-1}) holds provided that $\Gamma_0$ is a subset of $\Gamma$.
As a subset of the pairwise spline class $\mathcal{H}_{\sp}$, define a class of pairwise ramp functions, $\mathcal{H}_{\rp}$, such that each function in $\mathcal{H}_{\rp}$
can be expressed as, for $x= (x_1 , \ldots, x_p)^\T \in \bbR^p$,
\begin{align*}
h_{\rp, \beta,  c }(x) = \beta_0 + \sum_{j=1}^p \beta_{1j} \, \ramp ( x_j -  c _j) + \sum_{1\le i\not= j \le p} \beta_{2,ij} \, \ramp( x_i) \ramp(x_j) ,
\end{align*}
where $\ramp(t) = \frac{1}{2} (t+1)_+ - \frac{1}{2} (t-1)_+$ for $t\in \bbR$, $ c =( c _1,\ldots, c _p)^\T$ with $ c _j \in \{0,1\}$, and $\beta= (\beta_0, \beta_1^\T, \beta_2^\T)^\T$ with
$\beta_1 =(\beta_{1j}: j=1,\ldots,p)^\T$ and $\beta_2 = (\beta_{2,ij}: 1\le i\not= j\le p)^\T$.
For symmetry as in $\gamma_2$, assume that the coefficients in $\beta_2$ are symmetric, $\beta_{2,ij}=\beta_{2,ji}$ for any $i\not=j$.
By the definition of $\ramp(\cdot)$, each function $ h_{\rp,\beta,c} (x)$ can be represented as $h_\gamma(x)$ in the spline class $\mathcal{H}_{\sp}$, where $\beta$ and $\gamma$ satisfy $\beta_0=\gamma_0$,
$ \| \beta_1 \|_1 = \| \gamma_1 \|_1$, and $\| \beta_2 \|_1 = \| \gamma_2 \|_1$.
Incidentally, this relationship also holds when symmetry is not imposed in the coefficients in $\gamma_2$ or in $\beta_2$.
Denote as $\Gamma_\rp$ the subset of $\Gamma$ such that $\mathcal{H}_{\rp} = \{ h_\gamma(x): \gamma \in \Gamma_{\rp} \}$.

Take $\Gamma_0 = \{ \gamma \in \Gamma_\rp: \gamma_0=0, \pen_1(\gamma) = b_1 \}$ for some fixed $ b_1  >0$.
Inequality (\ref{eq:logit-L1-lower-2}) follows from Proposition \ref{pro:logit-L1-lower}: it holds
with probability at least $1-2 \delta$ that for any $\gamma \in \Gamma_0$,
\begin{align*}
& \quad K_f (P_n, P_{\hat\theta}; h_{\gamma,\hat\mu})
\ge f^{\dprime} (1) \left\{\E_{P_{\theta^*}} h_{\gamma,\hat\mu} (x) - \E_{P_{\hat\theta}}h_{\gamma,\hat\mu}(x) \right\} - \tilde\Delta_{12} ,
\end{align*}
where
$\tilde\Delta_{12} = - f^{\prime}(\me^{- b_1 })\epsilon + \frac{1}{2}  b_1 ^2 R_{3, b_1 } +  b_1  R_{2, b_1 } \lambda_{12}$, and
$$
\lambda_{12} =  C_{\mathrm{rad4}} \sqrt{\frac{4 \log(2p(p+1))}{n}}  + \sqrt{\frac{2 \log(\delta^{-1})}{n}} .
$$
Note that $\lambda_{12}$ is the same as in the proof of Theorem~\ref{thm:hinge-L1}. From (\ref{eq:logit-L1-lower-1})--(\ref{eq:logit-L1-lower-2}), the lower bound in (\ref{eq:logit-L1-main-frame}) holds with probability at least $1-2\delta$,
where $\Delta_{12} = \tilde\Delta_{12} + \lambda_1  b_1 $ and $d(\hat\theta, \theta^*)
= f^{\dprime} (1) \max_{\gamma\in\Gamma_0} \{\E_{P_{\theta^*}} h_{\gamma,\hat\mu} (x)- \E_{P_{\hat\theta}}h_{\gamma,\hat\mu}(x)\}$.

(Step 3) We complete the proof by choosing appropriate $ b_1 $ and relating the moment matching term $d(\hat\theta,\theta^*)$
to the estimation error between $\hat\theta$ and $\theta^*$.
First, due to the linearity of $h_{\gamma, \hat\mu}$ in $\gamma$,
combining the lower and upper bounds in (\ref{eq:logit-L1-main-frame}) shows
that with probability at least $1-{7}\delta$,
\begin{align*}
& \quad f^{\dprime} (1)  b_1  \, \max_{\gamma\in\Gamma_\rp, \pen_1(\gamma)=1 } \left\{\E_{P_{\theta^*}} h_{\gamma,\hat\mu}(x) - \E_{P_{\hat\theta}}h_{\gamma,\hat\mu}(x) \right\} \\
& \le -f^\prime(3/5)  (\epsilon + \sqrt{\epsilon/(n\delta)} )  - f^{\prime}(\me^{- b_1 })\epsilon + \frac{1}{2}  b_1 ^2 R_{3, b_1 } +  b_1  R_{2, b_1 } \lambda_{12} + \lambda_1  b_1  .
\end{align*}
Taking $ b_1  = \sqrt{\epsilon} + 1/\sqrt{n}$ in the preceding display and rearranging yields
\begin{align}
& \quad \max_{\gamma\in\Gamma_\rp, \pen_1(\gamma)=1 } \left\{\E_{P_{\theta^*}} h_{\gamma,\hat\mu}(x) - \E_{P_{\hat\theta}}h_{\gamma,\hat\mu}(x) \right\}
\le \err_{f1} (n,p, \delta, \epsilon) , \label{eq:logit-L1-combine}
\end{align}
where
\begin{align*}
\err_{f1} (n,p, \delta, \epsilon)
& =  ( f^{\dprime} (1))^{-1} \Big\{-f^\prime(3/5)  (\sqrt{\epsilon} + \sqrt{1/(n\delta)} ) \\
& \quad - f^{\prime}(\me^{- b_1 }) \sqrt{\epsilon} +
\frac{1}{2} R_{3, b_1 } (\sqrt{\epsilon} +1/\sqrt{n} ) +  R_{2, b_1 } \lambda_{12} + \lambda_1  \Big\}.
\end{align*}
The desired result then follows from Proposition~\ref{pro:logit-L1-combine}: provided $\err_{f1} (n,p, \delta, \epsilon) \le a $, inequality (\ref{eq:logit-L1-combine}) implies that
\begin{align*}
& \| \hat\mu - \mu^* \|_\infty \le S_{4,a}  \err_{f1} (n,p, \delta, \epsilon), \quad  \| \hat\Sigma - \Sigma^* \|_{\max} \le S_{8,a}  \err_{f1} (n,p, \delta, \epsilon) .
\end{align*}
\end{proof}

%%% -----------------------------------------------------------------------------------------------------------------------------------------  main proof, spline L2, location/scale and var matrix

\subsection{Proof of Theorem~\ref{thm:logit-L2}} \label{sec:logit-L2}

We state and prove the following result which implies Theorem~\ref{thm:logit-L2}. For $b>0$, define $R_{4,b} = \inf_{ |u| \le b} \frac{\dif}{ \dif u} f^\# (\me^u)$, in addition to $R_{2,b}$ and $R_{3,b}$ as in Proposition~\ref{pro:logit-L1-detailed}.
For $\delta\in (0,1)$, define
\begin{align*}
& \lambda_{21} = \sqrt{\frac{{5 p} + \log(\delta^{-1})}{n}} , \quad \lambda_{22} = C_{\mathrm{rad5}}  \sqrt{\frac{16 p}{n}} + \sqrt{\frac{2p \log(\delta^{-1})}{n}} ,\\
&  \lambda_{31} = \lambda_{21} + \frac{{5 p} + \log(\delta^{-1})}{n},  \lambda_{32} = C_{\mathrm{rad5}} \sqrt{\frac{6(p-1)}{n}} + \sqrt{\frac{(p-1)\log(\delta^{-1})}{n}} .
\end{align*}
where $C_{\mathrm{rad5}}=C_{\mathrm{sg,12}} C_{\mathrm{rad3}}$, depending on universal constants $C_{\mathrm{sg,12}}$ and $C_{\mathrm{rad3}}$ in
Lemmas~\ref{lem:subg} and Corollary \ref{cor:entropy-sg2} in the Supplement. Denote
\begin{align*}
 \err_{f2} (n,p, \delta, \epsilon)
 & =  ( {\sqrt{2}} R_{4,b_2^\dag} )^{-1} \Big\{-f^\prime(3/5)  (\sqrt{\epsilon} + \sqrt{1/(n\delta)} )- f^{\prime}(\me^{- b_2^\dag }) \sqrt{\epsilon}  \\
 & \quad + 4 C_{\mathrm{sg,12}}^2 M_2  R_{3, b_2^\dag} (\sqrt{\epsilon} + \sqrt{1/(np)} ) +  R_{2, b_2^\dag} \lambda_{22} + \lambda_2  \Big\} ,\\
 \err_{f3} (n,p, \delta, \epsilon) & =  ( {2} R_{4, 2b_3^\dag} )^{-1} \Big\{-f^\prime(3/5)  (\sqrt{\epsilon} + \sqrt{1/(n\delta)} ) - f^{\prime}(\me^{- 2 b_3^\dag }) \sqrt{\epsilon} \\
 & \quad + (80 C_{\mathrm{sg,12}}^2  M_2) R_{3, 2 b_3^\dag} (\sqrt{\epsilon} + \sqrt{1/(np)}) +  R_{2, b_3^\dag} \lambda_{32}
 + \lambda_3/\sqrt{p} \Big\},
\end{align*}
where $ b_2  = \sqrt{\epsilon} + \sqrt{1/(np)} $, $b_2^\dag = b_2 \sqrt{2p}$, $ b_3  = \sqrt{\epsilon/p} + \sqrt{1/(np^2)} $, and  $b_3^\dag = b_2 \sqrt{p(p-1)}$. Note that by the strict convexity and monotonicity of $f$ as required in Assumption~\ref{ass:f-condition}--\ref{ass:f-condition2}, we have that $$R_{4, b}=  \inf_{ |u| \le  b} \frac{\dif}{ \dif u} f^\# (\me^u) = \inf_{ |u| \le  b} \frac{\dif}{ \dif u} \left\{-f^{*} (f^{\prime}(\me^u))\right\}$$ is bounded away from zero provided that $b$ is bounded.
\begin{manualpro}{S2} \label{pro:logit-L2-detailed}
Assume that $\| \Sigma^* \|_\op \le M_2$, and $f$ satisfies Assumptions~\ref{ass:f-condition}--\ref{ass:f-condition2}.
Let $\hat \theta=(\hat\mu,\hat\Sigma)$ be a solution to (\ref{eq:logit-fgan-L2}) with

$$\lambda_2 \ge (5/3) C_{\mathrm{sp21}} M_2^{1/2} R_1 \lambda_{21}
, \quad \lambda_3 /\sqrt{p} \ge {(25 \sqrt{5}/3)} C_{\mathrm{sp22}} M_{21} R_1  \lambda_{31},$$ where
$M_{21} = M_2^{1/2} (M_2^{1/2} +  2\sqrt{2\pi})$,
$C_{\mathrm{sp21}} = \sqrt{2} C_{\mathrm{sg7}}  C_{\mathrm{sg5}}$, and
$C_{\mathrm{sp22}} = \sqrt{2/\pi} C_{\mathrm{sp21}} + C_{\mathrm{sg8}}$,
depending on universal constants $C_{\mathrm{sg5}}$, $C_{\mathrm{sg7}}$, and $C_{\mathrm{sg8}}$ in
Lemmas \ref{lem:subg-concentration}, \ref{lem:subg-vec-norm}, and \ref{lem:subg-vec-var} in the Supplement.
If {$\epsilon\le 1/5$}, $\sqrt{\epsilon(1-\epsilon)/(n \delta) }  \leq 1/5$, and $\err_{f2} (n,p, \delta, \epsilon) \le a $ for a constant $a\in(0,1/2)$,
then we have that with probability at least $1- {8} \delta$,
\begin{align*}
 \| \hat\mu - \mu^* \|_2 & \le S_{4,a}  \err_{f2} (n,p, \delta, \epsilon), \\
 %\| \hat\sigma - \sigma^* \|_2 & \le S_{6,a}  \err_{f2} (n,p, \delta, \epsilon), \\
 p^{-1/2}\| \hat\Sigma - \Sigma^* \|_\fro
& \le S_{9,a} \err_{f2} (n,p, \delta, \epsilon)  +  S_{7} \err_{f3} (n,p, \delta, \epsilon) ,
\end{align*}
where $S_{9,a} = 2 M_2^{1/2} S_{6,a} + \sqrt{2} S_{7} ( S_{4,a} + S_{6,a}) $ and $(S_{4,a}, S_{6,a}, S_{7})$ are defined as in Proposition~\ref{pro:logit-L1-detailed} except with $M_1$ replaced by $M_2$ throughout.
\end{manualpro}
\begin{proof}[Proof of Proposition~\ref{pro:logit-L2-detailed}]
The main strategy of our proof is to show that the following inequalities hold with high probabilities,
\begin{align}
& d(\hat\theta, \theta^*) - \Delta_{22} \le
\max_{\gamma\in\Gamma}  \; \left\{ K_f (P_n, P_{\hat\theta}; h_{\gamma,\hat\mu}) - \lambda_2 \;\pen_2(\gamma_1) - \lambda_3 \;\pen_2(\gamma_2) \right\}
\le \Delta_{21}, \label{eq:logit-L2-main-frame}
\end{align}
where $\Delta_{21}$ and $\Delta_{22}$ are error terms, and
$d(\hat\theta, \theta^*)$ is a moment matching term, similarly as in the proof of Theorem~\ref{thm:logit-L1}.
%which under certain conditions delivers upper bounds, up to scaling constants, on the estimation errors,
%$\| \hat\mu- \mu^* \|_2$, $\| \hat \sigma - \sigma^* \|_2$, and $\| \hat \Sigma - \Sigma ^* \|_\fro$.
However, additional considerations are involved.

We split the proof into several steps. In Step 1, we derive the upper bound in (\ref{eq:logit-L2-main-frame})
by exploiting two tuning parameters $\lambda_2$ and $\lambda_3$ associated with $\gamma_1$ and $\gamma_2$ respectively.
In Steps 2 and 3, we derive the first version of the lower bound in (\ref{eq:logit-L2-main-frame}) and then deduce upper bounds on $\| \hat\mu- \mu^* \|_2$ and $\| \hat \sigma - \sigma^* \|_2$,
where $\hat\sigma$ or $\sigma^*$ is the vector of standard deviations from $\hat\Sigma$ or $\Sigma^*$ respectively.
In Steps 4 and 5, we derive the second version of the lower bound in (\ref{eq:logit-L2-main-frame}) and then deduce an upper bound on $\| \hat \Sigma - \Sigma ^* \|_\fro$.

(Step 1) For the upper bound in (\ref{eq:logit-L2-main-frame}), we show that with probability at least $1-4\delta$,
\begin{align}
& \quad \max_{\gamma\in\Gamma}  \; \left\{ K_f (P_n, P_{\hat\theta}; h_{\gamma,\hat\mu}) - \lambda_2 \;\pen_2(\gamma_1) - \lambda_3 \;\pen_2(\gamma_2) \right\} \nonumber \\
& \le \max_{\gamma\in\Gamma}  \; \left\{ K_f (P_n, P_{\theta^*}; h_{\gamma,\mu^*}) - \lambda_2 \;\pen_2(\gamma_1) - \lambda_3 \;\pen_2(\gamma_2)  \right\} \label{eq:logit-L2-upper-1} \\
& \leq \max_{\gamma\in\Gamma}  \; \left\{ \Delta_{21} + \pen_2( \gamma_1) \tilde\Delta_{21} + \pen_2( \gamma_2) \tilde\Delta_{31} - \lambda_2 \;\pen_2(\gamma_1) - \lambda_3 \;\pen_2(\gamma_2)  \right\} .  \label{eq:logit-L2-upper-2}
\end{align}
Inequality (\ref{eq:logit-L2-upper-1}) follows from the definition of $\hat\theta$.
Inequality (\ref{eq:logit-L2-upper-2}) follows from Proposition~\ref{pro:logit-L2-upper}: it holds with probability at least $1-4\delta$ that for any $\gamma \in\Gamma_1$,
\begin{align*}
K_f (P_n, P_{\theta^*}; h_{\gamma,\mu^*}) \le \Delta_{21} + \pen_2( \gamma_1) \tilde\Delta_{21} + \pen_2( \gamma_2) \sqrt{p} \tilde\Delta_{31} ,
\end{align*}
where
\begin{align*}
    \Delta_{21} &= -f^\prime(3/5)  (\epsilon + \sqrt{\epsilon/(n\delta)} ),\quad
    \tilde\Delta_{21} =  (5/3) C_{\mathrm{sp21}} M_2^{1/2} R_1 \lambda_{21},\\
\tilde\Delta_{31} &= {(25\sqrt{5}/3)} C_{\mathrm{sp22}} M_{21} R_1  \lambda_{31},
\end{align*}
and
$$
\lambda_{21} = \sqrt{\frac{{5 p} + \log(\delta^{-1})}{n}} , \quad \lambda_{31} = \lambda_{21} + \frac{{5 p} + \log(\delta^{-1})}{n} .
$$
From (\ref{eq:logit-L2-upper-1})--(\ref{eq:logit-L2-upper-2}), the upper bound in (\ref{eq:logit-L2-main-frame}) holds with probability at least $1-4\delta$,
provided that the tuning parameters {$\lambda_2$ and $\lambda_3$} are chosen such that $\lambda_2 \ge \tilde\Delta_{21}$ and $\lambda_3 \ge \sqrt{p} \tilde\Delta_{31}$.

(Step 2) For the first version of the lower bound in (\ref{eq:logit-L2-main-frame}),
we show that with probability at least $1- 2\delta$,
\begin{align}
& \quad \max_{\gamma\in\Gamma}  \; \left\{ K_f (P_n, P_{\hat\theta}; h_{\gamma,\hat\mu}) - \lambda_2 \;\pen_2(\gamma_1) - \lambda_3 \;\pen_2(\gamma_2) \right\} \nonumber \\
&\geq \max_{\gamma\in\Gamma_1}  \; \left\{ K_f (P_n, P_{\hat\theta}; h_{\gamma,\hat\mu}) - \lambda_2 \;\pen_2(\gamma_1) \right\} \label{eq:logit-L2-lower-0} \\
&\geq \max_{\gamma\in\Gamma_{10}}  \; \left\{ K_f (P_n, P_{\hat\theta}; h_{\gamma,\hat\mu}) - \lambda_2 \;\pen_2(\gamma_1) \right\} \label{eq:logit-L2-lower-1}\\
&\geq \max_{\gamma\in\Gamma_{10}}  \; R_{4,b_2^\dag}  \left\{\E_{P_{\theta^*}} h_{\gamma,\hat\mu}(x) - \E_{P_{\hat\theta}}h_{\gamma,\hat\mu}(x) \right\} - \tilde\Delta_{22} - \lambda_2 b_2  , \label{eq:logit-L2-lower-2}
\end{align}
{where $\Gamma_1 = \{ (\gamma_0,\gamma_1^\T, \gamma_2^\T)^\T: \gamma_2=0\}$.}
Inequality (\ref{eq:logit-L2-lower-0}) follows because $\Gamma_{1}$ is a subset of $\Gamma$ such that $\gamma_2=0$ and hence $\pen_2(\gamma_2)=0$ for $\gamma\in\Gamma_1$.
Inequality (\ref{eq:logit-L2-lower-1}) holds provided that $\Gamma_{10}$ is a subset of $\Gamma_1$.
As a subset of the main-effect spline class $\mathcal{H}_\spL$, define a main-effect ramp class, $\mathcal{H}_\rpL$, such that each function in $\mathcal{H}_\rpL$
can be expressed as, for $x= (x_1 , \ldots, x_p)^\T \in \bbR^p$,
\begin{align*}
h_{\rpL, \beta,  c }(x) & = \beta_0 + \sum_{j=1}^p \beta_{1j} \ramp ( x_j -  c_j) ,
\end{align*}
where $\ramp(t) = \frac{1}{2} (t+1)_+ - \frac{1}{2} (t-1)_+$ for $t\in \bbR$, $ c =( c _1, \ldots, c _p )^\T$ with $ c _j \in \{0,1\}$,
and $\beta= (\beta_0, \beta_1^\T)^\T$ with
$\beta_1 =( \beta_{11}, \ldots, \beta_{1p} )^\T$.
Only the main-effect ramp functions are included, while the interaction ramp functions are excluded, in $h_{\rpL, \beta, c }(x)$.
By the definition of $\ramp(\cdot)$, each function $ h_{\rpL,\beta,c} (x)$ can be represented as $h_\gamma(x) \in \mathcal{H}_{\spL}$ with $\gamma=(\gamma_0,\gamma_1^\T)^\T \in \Gamma_\rpL$,
such that $\beta$ and $\gamma$ satisfy
$\beta_0=\gamma_0$ and
$ \| \beta_1 \|_2 = \sqrt{2} \| \gamma_1 \|_2$.
For example, for $\ramp(x_1)$, the associated norms are $\|\beta_1\|_2= 1$ and $\| \gamma_1 \|_2 = \sqrt{1/2}$.
Denote as $\Gamma_\rpL$ the subset of $\Gamma_1$ such that $\mathcal{H}_{\rpL} = \{ h_\gamma(x): \gamma \in \Gamma_\rpL \}$.

Take $\Gamma_{10} = \{ \gamma \in \Gamma_\rpL: \pen_2(\gamma) = b_2, \E_{P_{\theta^*}} h_{\gamma, \hat\mu}(x)=0, \E_{P_{\hat \theta}} h_{\gamma, \hat\mu}(x)\le 0 \}$
for some fixed $b_2 >0$.
Inequality (\ref{eq:logit-L2-lower-2}) follows from Proposition \ref{pro:logit-L2-lower}: it holds
with probability at least $1-2 \delta$ that for any $\gamma \in \Gamma_{10}$,
\begin{align*}
& \quad K_f (P_n, P_{\hat\theta}; h_{\gamma,\hat\mu})
\ge R_{4,b_2^\dag} \left\{\E_{P_{\theta^*}} h_{\gamma,\hat\mu} (x) - \E_{P_{\hat\theta}}h_{\gamma,\hat\mu} (x) \right\} - \tilde\Delta_{22} ,
\end{align*}
where $b_2^\dag = b_2 \sqrt{2p}$,
$\tilde\Delta_{22} = -f^{\prime}(\me^{- b_2^\dag})\epsilon
+ 4 C_{\mathrm{sg,12}}^2 M_2  b_2^2 R_{3, b_2^\dag}
+  b_2 R_{2, b_2^\dag} \lambda_{22}$, and
$$
\lambda_{22} = C_{\mathrm{rad5}}  \sqrt{\frac{16 p}{n}} + \sqrt{\frac{2p \log(\delta^{-1})}{n}} .
$$
From (\ref{eq:logit-L2-lower-1})--(\ref{eq:logit-L2-lower-2}), the lower bound in (\ref{eq:logit-L2-main-frame}) holds with probability at least $1-2\delta$,
where $\Delta_{22} = \tilde\Delta_{22} + \lambda_2 b_2 $ and $d(\hat\theta, \theta^*) =
R_{4,b_2^\dag} \{\E_{P_{\theta^*}} h_{\gamma,\hat\mu} (x) - \E_{P_{\hat\theta}}h_{\gamma,\hat\mu} (x) \} $.

(Step 3) We deduce upper bounds on $\| \hat\mu- \mu^* \|_2$ and $\| \hat \sigma - \sigma^* \|_2$, by choosing appropriate $ b_2 $ and relating the moment matching term $d(\hat\theta,\theta^*)$
to the estimation errors.
First, combining the upper bound in (\ref{eq:logit-L2-main-frame}) from Step 1 and the lower bound from Step 2 shows that with probability at least {$1- 6 \delta$,}
\begin{align*}
& \quad R_{4,b_2^\dag}  b_2  \, \max_{\gamma\in\Gamma_\rpL, \pen_2(\gamma)=1 } \left\{\E_{P_{\theta^*}} h_{\gamma,\hat\mu}(x) - \E_{P_{\hat\theta}}h_{\gamma,\hat\mu}(x) \right\} \\
& \le -f^\prime(3/5)  (\epsilon + \sqrt{\epsilon/(n\delta)} ) - f^{\prime}(\me^{- b_2^\dag })\epsilon + 4 C_{\mathrm{sg,12}}^2 M_2  b_2^2 R_{3, b_2^\dag}
+  b_2 R_{2, b_2^\dag} \lambda_{22} + \lambda_2  b_2  .
\end{align*}
Taking $ b_2  = \sqrt{\epsilon} + \sqrt{1/(np)}$ in the preceding display and rearranging yields
\begin{align}
& \quad \max_{\gamma\in\Gamma_\rpL, \pen_2(\gamma)= {\sqrt{1/2}} } \left\{\E_{P_{\theta^*}} h_{\gamma,\hat\mu}(x) - \E_{P_{\hat\theta}}h_{\gamma,\hat\mu}(x) \right\}
\le \err_{f2} (n,p, \delta, \epsilon) , \label{eq:logit-L2-combine}
\end{align}
where \begin{align*}
\err_{f2} (n,p, \delta, \epsilon)
& =  ( {\sqrt{2}} R_{4,b_2^\dag} )^{-1} \Big\{-f^\prime(3/5)  (\sqrt{\epsilon} + \sqrt{1/(n\delta)} ) \\
& \quad - f^{\prime}(\me^{- b_2^\dag }) \sqrt{\epsilon} +
4 C_{\mathrm{sg,12}}^2 M_2  R_{3, b_2^\dag} (\sqrt{\epsilon} + \sqrt{1/(np)} ) +  R_{2, b_2^\dag} \lambda_{22} + \lambda_2  \Big\}.
\end{align*}
The error bounds for $(\hat\mu,\hat\sigma)$ then follows from Proposition~\ref{pro:logit-L2-combine}: provided $\err_{f2} (n,p, \delta, \epsilon) \le a $, inequality (\ref{eq:logit-L2-combine}) implies that
\begin{align}
& \| \hat\mu - \mu^* \|_2\le S_{4,a}  \err_{f2} (n,p, \delta, \epsilon),  \label{eq:logit-L2-mu-bd}\\
&  \| \hat\sigma - \sigma^* \|_2 \le S_{6,a}  \err_{f2} (n,p, \delta, \epsilon) . \label{eq:logit-L2-sigma-bd}
\end{align}

(Step 4) For the second version of the lower bound in (\ref{eq:logit-L2-main-frame}), we show that with probability at least $1- 2\delta$,
\begin{align}
& \quad \max_{\gamma\in\Gamma}  \; \left\{ K_f (P_n, P_{\hat\theta}; h_{\gamma,\hat\mu}) - \lambda_2 \;\pen_2(\gamma_1) - \lambda_3 \;\pen_2(\gamma_2) \right\} \nonumber \\
&\geq \max_{\gamma\in\Gamma_2}  \; \left\{ K_f (P_n, P_{\hat\theta}; h_{\gamma,\hat\mu}) - \lambda_3 \;\pen_2(\gamma_2) \right\} \label{eq:logit-L2v-lower-0} \\
&\geq \max_{\gamma\in\Gamma_{20}}  \; \left\{ K_f (P_n, P_{\hat\theta}; h_{\gamma,\hat\mu}) - \lambda_3 \;\pen_2(\gamma_2) \right\} \label{eq:logit-L2v-lower-1}\\
&\geq \max_{\gamma\in\Gamma_{20}}  \; f^{\dprime} (1) \left\{\E_{P_{\theta^*}} h_{\gamma,\hat\mu}(x) - \E_{P_{\hat\theta}}h_{\gamma,\hat\mu}(x) \right\} - \tilde\Delta_{32} - \lambda_3 b_3  , \label{eq:logit-L2v-lower-2}
\end{align}
{where $\Gamma_2 = \{ (\gamma_0,\gamma_1^\T, \gamma_2^\T)^\T: \gamma_1=0\}$.}
Inequality (\ref{eq:logit-L2v-lower-0}) follows because $\Gamma_{2}$ is a subset of $\Gamma$ such that $\gamma_1=0$ and hence $\pen_2(\gamma_1)=0$ for $\gamma\in\Gamma_2$.
Inequality (\ref{eq:logit-L2v-lower-1}) holds provided that $\Gamma_{20}$ is a subset of $\Gamma_2$.
As a subset of the interaction spline class $\mathcal{H}_\spQ$, define an interaction ramp class, $\mathcal{H}_\rpQ$, such that each function in $\mathcal{H}_\rpQ$
can be expressed as, for $x= (x_1 , \ldots, x_p)^\T \in \bbR^p$,
\begin{align*}
h_{\rpQ, \beta }(x) & = \beta_0 + \sum_{1\le i\not= j \le p} \beta_{2,ij} \ramp( x_i) \ramp(x_j) ,
\end{align*}
where $\ramp(t) = \frac{1}{2} (t+1)_+ - \frac{1}{2} (t-1)_+$ for $t\in \bbR$,
and $\beta= (\beta_0, \beta_2^\T)^\T$ with
$\beta_2 = (\beta_{2,ij}: 1\le i\not= j\le p)^\T$.
In contrast with the function $h_{\rpL, \beta, c}(x)$ in $\mathcal{H}_\spL$, only the interaction ramp functions are included, while the main-effect ramp functions are excluded, in $h_{\rpQ, \beta}(x)$.
For symmetry as in $\gamma_2$, assume that the coefficients in $\beta_2$ are symmetric, $\beta_{2,ij}=\beta_{2,ji}$ for any $i\not=j$.
By the definition of $\ramp(\cdot)$, each function $ h_{\rpQ,\beta} (x)$ can be represented as $h_\gamma(x) \in \mathcal{H}_{\spQ}$ with $\gamma=(\gamma_0,\gamma_2^\T)^\T \in \Gamma_\rpQ$,
such that $\beta$ and $\gamma$ satisfy
$\beta_0=\gamma_0$ and
$ \| \beta_2 \|_2 = 2\| \gamma_2 \|_2$.
For example, for $\ramp(x_1)\ramp(x_2)$, the associated norms are $\|\beta_2\|_2= 1$ and $\| \gamma_2 \|_2 =1/2$.
Denote as $\Gamma_\rpQ$ the subset of {$\Gamma_2$} such that $\mathcal{H}_{\rpQ} = \{ h_\gamma(x): \gamma \in \Gamma_\rpQ \}$.

Take $\Gamma_{20} = \{ \gamma \in \Gamma_\rpQ: \pen_2(\gamma)=  b_3 , \E_{P_{\theta^*}} h_{\gamma, \hat\mu}(x)=0, \E_{P_{\hat \theta}} h_{\gamma, \hat\mu}(x)\le 0 \}$ for some fixed $ b_3  >0$.
Inequality (\ref{eq:logit-L2v-lower-2}) follows from Proposition \ref{pro:logit-L2v-lower}: it holds
with probability at least $1-2 \delta$ that for any $\gamma \in \Gamma_{20}$,
\begin{align*}
& \quad K_f (P_n, P_{\hat\theta}; h_{\gamma,\hat\mu})
\ge R_{4, 2b_3^\dag} \left\{ \E_{P_{\theta^*}} h(x) - \E_{P_{\hat\theta}} h(x) \right\} - \tilde\Delta_{32} ,
\end{align*}
where $b_3^\dag =  b_3 \sqrt{p(p-1)}$,
$\tilde \Delta_{32} = - f^{\prime}(\me^{- 2 b_3^\dag })\epsilon
+ (80 C_{\mathrm{sg,12}}^2  M_2) p b_3^2 R_{3, 2 b_3^\dag}
+ \sqrt{p} b_3  R_{2, b_3^\dag} \lambda_{32}$, and
$$
\lambda_{32} = C_{\mathrm{rad4}} \sqrt{\frac{6(p-1)}{n}} + \sqrt{\frac{(p-1)\log(\delta^{-1})}{n}} .
$$
From (\ref{eq:logit-L2v-lower-1})--(\ref{eq:logit-L2v-lower-2}), the lower bound in (\ref{eq:logit-L2-main-frame}) holds with probability at least $1-2\delta$,
where $\Delta_{22} = \tilde \Delta_{32} + \lambda_3 b_3 $ and $d(\hat\theta, \theta^*) =
R_{4, 2b_3^\dag} \max_{\gamma\in\Gamma_{20}} \{\E_{P_{\theta^*}} h_{\gamma,\hat\mu} (x)- \E_{P_{\hat\theta}}h_{\gamma,\hat\mu}(x)\}$.

(Step 5) We deduce an upper bound on $\| \hat \Sigma - \Sigma^* \|_\fro$, by choosing appropriate $ b_3 $ and relating the moment matching term $d(\hat\theta,\theta^*)$
to the estimation error.
First, combining the upper bound in (\ref{eq:logit-L2-main-frame}) from Step 1 and the lower bound from Step 4
shows that with probability {$1-6\delta$,}
\begin{align*}
& \quad R_{4, 2b_3^\dag} b_3 \, \max_{\gamma\in\Gamma_\rpQ, \pen_2(\gamma)=1 } \left\{\E_{P_{\theta^*}} h_{\gamma,\hat\mu}(x) - \E_{P_{\hat\theta}}h_{\gamma,\hat\mu}(x) \right\} \\
& \le  -f^\prime(3/5)  (\epsilon + \sqrt{\epsilon/(n\delta)} ) - f^{\prime}(\me^{- 2 b_3^\dag })\epsilon
+ (80 C_{\mathrm{sg,12}}^2  M_2) p b_3^2 R_{3, 2 b_3^\dag}
+ \sqrt{p} b_3  R_{2, b_3^\dag} \lambda_{32}
+ \lambda_3  b_3  .
\end{align*}
Taking $ b_3  = \sqrt{\epsilon /p} + \sqrt{1/(np^2)} $ in the preceding display and rearranging yields
\begin{align}
& \quad \max_{\gamma\in\Gamma_\rpQ, \pen_2(\gamma)= {1/2} } \left\{\E_{P_{\theta^*}} h_{\gamma,\hat\mu}(x) - \E_{P_{\hat\theta}}h_{\gamma,\hat\mu}(x) \right\}
\le \sqrt{p} \; \err_{f3} (n,p, \delta, \epsilon) , \label{eq:logit-L2v-combine}
\end{align}
where
\begin{align*}
& \err_{f3} (n,p, \delta, \epsilon) =  ( {2} R_{4, 2b_3^\dag} )^{-1} \Big\{-f^\prime(3/5)  (\sqrt{\epsilon} + \sqrt{1/(n\delta)} ) \\
& \quad - f^{\prime}(\me^{- 2 b_3^\dag }) \sqrt{\epsilon}
+ (80 C_{\mathrm{sg,12}}^2  M_2) R_{3, 2 b_3^\dag} (\sqrt{\epsilon} + \sqrt{1/(np)})
+  R_{2, b_3^\dag} \lambda_{32}
+ \lambda_3/\sqrt{p} \Big\}.
\end{align*}
The error bound for $\hat\Sigma$ then follows from Proposition~\ref{pro:logit-L2v-combine}: inequality (\ref{eq:logit-L2v-combine})
{together with the error bounds (\ref{eq:logit-L2-mu-bd})--(\ref{eq:logit-L2-sigma-bd}) implies that}
{
\begin{align*}
\frac{1}{\sqrt{p}} \| \hat \Sigma - \Sigma^* \|_\fro
& \le  2 M_2^{1/2}  \| \hat \sigma - \sigma^*\|_2 + S_{7} \left\{ \sqrt{2} \Delta_{\hat\mu,\hat\sigma} + \err_{f3} (n,p, \delta, \epsilon) \right\} \\
& \le  S_{9,a} \err_{f2} (n,p, \delta, \epsilon) + S_{7} \err_{f3} (n,p, \delta, \epsilon) ,
\end{align*}
where $\Delta_{\hat\mu,\hat\sigma} = ( \| \hat\mu-\mu^*\|_2^2 + \| \hat\sigma - \sigma^*\|_2^2 )^{1/2}$
and $S_{9,a} = 2 M_2^{1/2} S_{6,a} + \sqrt{2} S_{7} (S_{4,a} + S_{6,a})$. }
\end{proof}

\subsection{Proof of Theorem~\ref{thm:hinge-L2}} \label{sec:hinge-L2}
We state and prove the following result which implies Theorem~\ref{thm:hinge-L2}. For $\delta \in (0,1)$, define $(\lambda_{21}, \lambda_{31}, \lambda_{22}, \lambda_{32})$ the same as in Sections~\ref{sec:logit-L1} and \ref{sec:logit-L2}. Denote
\begin{align*}
& \err_{h2} (n,p, \delta, \epsilon)
=  3\epsilon(2p)^{1/2} + 2\sqrt{2p\epsilon/(n\delta)} + \lambda_2 + \lambda_{22},\\
&  \err_{{h3}} (n,p, \delta, \epsilon) =  3\epsilon\sqrt{p-1}+2\sqrt{\epsilon(p-1)/(n\delta)}+\lambda_{32}/2+(25 \sqrt{5}/6) C_{\mathrm{sp22}} M_{21}  \lambda_{31}.
\end{align*}

\begin{manualpro}{S3} \label{pro:hinge-L2-detailed}
Assume that $\| \Sigma^* \|_{\op} \le M_2$.
Let $\hat \theta=(\hat\mu,\hat\Sigma)$ be a solution to (\ref{eq:hinge-gan-L2}) with $\lambda_3 /\sqrt{p} \ge (25 \sqrt{5}/3) C_{\mathrm{sp22}} M_{21} \lambda_{31}$ and $\lambda_2 \ge (5/3) C_{\mathrm{sp21}} M_2^{1/2} \lambda_{21}$, where
$M_{21}$, $C_{\mathrm{sp21}}$, and $C_{\mathrm{sp22}}$ are defined as in Proposition~\ref{pro:logit-L2-detailed}.
If {$\epsilon\le 1/5$}, $\sqrt{\epsilon(1-\epsilon)/(n \delta) }  \leq 1/5$ and $\err_{h2} (n,p, \delta, \epsilon) \le a $ for a constant $a\in(0,1/2)$,
then we have that with probability at least $1-8\delta$,
\begin{align*}
\| \hat\mu - \mu^* \|_2 &\le S_{4,a}  \err_{h2} (n,p, \delta, \epsilon), \\
p^{-1/2} \| \hat\Sigma - \Sigma^*\|_\fro  &\le S_{9,a}\err_{h2}(n,p,\delta,\epsilon) + S_{7} \err_{h3}(n,p,\delta,\epsilon),
\end{align*}
where $(S_{4,a}, S_{6,a}, S_{7}, S_{9,a})$ are defined as in Proposition~\ref{pro:logit-L2-detailed}.
\end{manualpro}
\begin{proof}[Proof of Proposition~\ref{pro:hinge-L2-detailed}]

The main strategy of our proof is to show that the following inequalities hold with high probabilities,
\begin{align}
& d(\hat\theta, \theta^*) - \Delta_{22} \le
\max_{\gamma\in\Gamma}  \; \left\{ K_\HG (P_n, P_{\hat\theta}; h_{\gamma,\hat\mu}) - \lambda_2 \;\pen_2(\gamma_1) - \lambda_3 \;\pen_2(\gamma_2) \right\}
\le \Delta_{21}, \label{eq:hinge-L2-main-frame}
\end{align}
where $\Delta_{21}$ and $\Delta_{22}$ are error terms, and
$d(\hat\theta, \theta^*)$ is a moment matching term, similarly as in the proof of {Theorem~\ref{thm:hinge-L1}.}
%which under certain conditions delivers upper bounds, up to scaling constants, on the estimation errors,
%$\| \hat\mu- \mu^* \|_2$, $\| \hat \sigma - \sigma^* \|_2$, and $\| \hat \Sigma - \Sigma ^* \|_\fro$.
However, additional considerations are involved.

(Step 1) For the upper bound in (\ref{eq:hinge-L2-main-frame}), we show that with probability at least $1-4\delta$,
\begin{align}
& \quad \max_{\gamma\in\Gamma}  \; \left\{ K_\HG (P_n, P_{\hat\theta}; h_{\gamma,\hat\mu}) - \lambda_2 \;\pen_2(\gamma_1) - \lambda_3 \;\pen_2(\gamma_2) \right\} \nonumber \\
& \le \max_{\gamma\in\Gamma}  \; \left\{ K_\HG (P_n, P_{\theta^*}; h_{\gamma,\mu^*}) - \lambda_2 \;\pen_2(\gamma_1) - \lambda_3 \;\pen_2(\gamma_2)  \right\} \label{eq:hinge-L2-upper-1} \\
& \leq \max_{\gamma\in\Gamma}  \; \left\{ \Delta_{21} + \pen_2( \gamma_1) \tilde\Delta_{21} + \pen_2( \gamma_2) \tilde\Delta_{31} - \lambda_2 \;\pen_2(\gamma_1) - \lambda_3 \;\pen_2(\gamma_2)  \right\} .  \label{eq:hinge-L2-upper-2}
\end{align}
Inequality (\ref{eq:hinge-L2-upper-1}) follows from the definition of $\hat\theta$.
Inequality (\ref{eq:hinge-L2-upper-2}) follows from Proposition~\ref{pro:hinge-L2-upper}: it holds with probability at least $1-4\delta$ that for any $\gamma \in\Gamma_1$,
\begin{align*}
K_\HG (P_n, P_{\theta^*}; h_{\gamma,\mu^*}) \le \Delta_{21} + \pen_2( \gamma_1) \tilde\Delta_{21} + \pen_2( \gamma_2) \sqrt{p} \tilde\Delta_{31} ,
\end{align*}
where
\begin{align*}
    \Delta_{21} = 2(\epsilon + \sqrt{\epsilon/(n\delta)} ),\quad \tilde\Delta_{21} =  (5/3) C_{\mathrm{sp21}} M_2^{1/2} \lambda_{21},\quad
\tilde\Delta_{31} =(25\sqrt{5}/3) C_{\mathrm{sp22}} M_{21}  \lambda_{31},
\end{align*}
and $\lambda_{21}$ and $\lambda_{31}$ are the same as in the proof of Theorem~\ref{thm:logit-L2}. Note that $\tilde{\Delta}_{21}$ and $\tilde{\Delta}_{31}$ differ from those in the proof of Theorem~\ref{thm:logit-L2} only in that $R_1$ is removed. From (\ref{eq:hinge-L2-upper-1})--(\ref{eq:hinge-L2-upper-2}), the upper bound in (\ref{eq:hinge-L2-main-frame}) holds with probability at least $1-4\delta$,
provided that the tuning parameters $\lambda_2$ and $\lambda_3$ are chosen such that $\lambda_2 \ge \tilde\Delta_{21}$ and $\lambda_3 \ge \sqrt{p} \tilde\Delta_{31}$.

(Step 2) For the first version of the lower bound in (\ref{eq:hinge-L2-main-frame}),
we show that with probability at least $1- 2\delta$,
\begin{align}
& \quad \max_{\gamma\in\Gamma}  \; \left\{ K_\HG (P_n, P_{\hat\theta}; h_{\gamma,\hat\mu}) - \lambda_2 \;\pen_2(\gamma_1) - \lambda_3 \;\pen_2(\gamma_2) \right\} \nonumber \\
&\ge \max_{\gamma\in\Gamma_1}  \; \left\{ K_\HG (P_n, P_{\hat\theta}; h_{\gamma,\hat\mu}) - \lambda_2 \;\pen_2(\gamma_1) \right\}    \label{eq:hinge-L2-lower-1} \\
&\ge \max_{\gamma\in\Gamma_{10}}  \; \left\{ K_\HG (P_n, P_{\hat\theta}; h_{\gamma,\hat\mu})\right\} - \lambda_2 \;  (2p)^{-1/2} \label{eq:hinge-L2-lower-2} \\
&\ge \max_{\gamma\in\Gamma_{10}}  \; \left\{\E_{P_{\theta^*}} h_{\gamma,\hat\mu}(x) - \E_{P_{\hat\theta}}h_{\gamma,\hat\mu}(x) \right\} - \tilde\Delta_{22} - \lambda_2 (2p)^{-1/2}  , \label{eq:hinge-L2-lower-3}
\end{align}
where $\Gamma_1=\{(\gamma_0, \gamma_1^\T, \gamma_2^\T)^\T:\gamma_2 = 0\}.$ Inequality (\ref{eq:hinge-L2-lower-1}) follows because $\Gamma_{1}$ is defined as a subset of $\Gamma$ such that {$\gamma_2=0$ and} hence $\pen_2(\gamma_2)=0$ for $\gamma\in\Gamma_1$.

Take $\Gamma_{10} = \{ \gamma \in \Gamma_\rpL: \gamma_0 = 0, \pen_2(\gamma) = (2p)^{-1/2}\}$,
{where $\Gamma_\rpL$ is the subset of $\Gamma_1$ associated with main-effect ramp functions as in the proof of Theorem \ref{thm:logit-L2}.}
Inequality (\ref{eq:hinge-L2-lower-2}) holds {because $\Gamma_{10} \subset \Gamma_1$ by definition.}
\ Inequality (\ref{eq:hinge-L2-lower-3}) follows from Proposition \ref{pro:hinge-L2-lower}: it holds
with probability at least $1-2 \delta$ that for any $\gamma \in \Gamma_{10}$,
\begin{align*}
& \quad K_\HG (P_n, P_{\hat\theta}; h_{\gamma,\hat\mu})
\ge  \E_{P_{\theta^*}} h_{\gamma,\hat\mu} (x) - \E_{P_{\hat\theta}}h_{\gamma,\hat\mu} (x) - \tilde\Delta_{22} ,
\end{align*}
where $\tilde\Delta_{22} = \epsilon + \lambda_{22}(2p)^{-1/2}$,  and
$$
\lambda_{22} = C_{\mathrm{rad5}}  \sqrt{\frac{16 p}{n}} + \sqrt{\frac{2p \log(\delta^{-1})}{n}} .
$$
Note that $\lambda_{22}$ is the same as in the proof of Theorem~\ref{thm:logit-L2}. From (\ref{eq:hinge-L2-lower-1})--(\ref{eq:hinge-L2-lower-3}), the lower bound in (\ref{eq:hinge-L2-main-frame}) holds with probability at least $1-2\delta$,
where $\Delta_{22} = \tilde\Delta_{22} + \lambda_2 (2p)^{-1/2} $ and $d(\hat\theta, \theta^*) =
\E_{P_{\theta^*}} h_{\gamma,\hat\mu} (x) - \E_{P_{\hat\theta}}h_{\gamma,\hat\mu} (x)  $.

(Step 3) We deduce upper bounds on $\| \hat\mu- \mu^* \|_2$ and $\| \hat \sigma - \sigma^* \|_2$, by choosing appropriate $ b_2 $ and relating the moment matching term $d(\hat\theta,\theta^*)$
to the estimation errors.
First, combining the upper bound in (\ref{eq:hinge-L2-main-frame}) from Step 1 and the lower bound from Step 2 shows that with probability at least $1- {6}\delta$,
\begin{align*}
& \quad (2p)^{-1/2} \, \max_{\gamma\in\Gamma_\rpL, \pen_2(\gamma)=1 } \left\{\E_{P_{\theta^*}} h_{\gamma,\hat\mu}(x) - \E_{P_{\hat\theta}}h_{\gamma,\hat\mu}(x) \right\} \\
& \le 3\epsilon + 2\sqrt{\epsilon/(n\delta)} +  (\lambda_{2} + \lambda_{22})(2p)^{-1/2},
\end{align*}
which gives
\begin{align}
& \quad \max_{\gamma\in\Gamma_\rpL, \pen_2(\gamma)=\sqrt{1/2} } \left\{\E_{P_{\theta^*}} h_{\gamma,\hat\mu}(x) - \E_{P_{\hat\theta}}h_{\gamma,\hat\mu}(x) \right\}
\le \err_{h2} (n,p, \delta, \epsilon) , \label{eq:hinge-L2-combine}
\end{align}
where \begin{align*}
    \err_{h2} (n,p, \delta, \epsilon) =  3\epsilon\sqrt{p} + 2\sqrt{p\epsilon/(n\delta)} + (\lambda_2 + \lambda_{22})/\sqrt{2}.
\end{align*}
The desired result then follows from Proposition~\ref{pro:logit-L2-combine}: provided $\err_{h2} (n,p, \delta, \epsilon) \le a $, inequality (\ref{eq:hinge-L2-combine}) implies that
\begin{align*}
& \| \hat\mu - \mu^* \|_2\le S_{4,a}  \err_{h2} (n,p, \delta, \epsilon), \\
&  \| \hat\sigma - \sigma^* \|_2 \le S_{6,a}  \err_{h2} (n,p, \delta, \epsilon) .
\end{align*}

(Step 4) For the second version of the lower bound in (\ref{eq:logit-L2-main-frame}), we show that with probability at least $1- 2\delta$,
\begin{align}
& \quad \max_{\gamma\in\Gamma}  \; \left\{ K_\HG (P_n, P_{\hat\theta}; h_{\gamma,\hat\mu}) - \lambda_2 \;\pen_2(\gamma_1) - \lambda_3 \;\pen_2(\gamma_2) \right\} \nonumber \\
&\ge \max_{\gamma\in\Gamma_2}  \; \left\{ K_\HG (P_n, P_{\hat\theta}; h_{\gamma,\hat\mu}) - \lambda_3 \;\pen_2(\gamma_2) \right\}    \label{eq:hinge-L2v-lower-1} \\
&\ge \max_{\gamma\in\Gamma_{20}}  \; \left\{ K_\HG (P_n, P_{\hat\theta}; h_{\gamma,\hat\mu})\right\} - \lambda_3 \;  (4q)^{-1/2}\label{eq:hinge-L2v-lower-2} \\
&\ge \max_{\gamma\in\Gamma_{20}}  \; \left\{\E_{P_{\theta^*}} h_{\gamma,\hat\mu}(x) - \E_{P_{\hat\theta}}h_{\gamma,\hat\mu}(x) \right\} - \tilde\Delta_{32} - \lambda_3 (4q)^{-1/2} , \label{eq:hinge-L2v-lower-3}
\end{align}
where $\Gamma_2 = \{(\gamma_0, \gamma_1^\T, \gamma_2^\T)^\T: \gamma_1 = 0\}$. Inequality (\ref{eq:hinge-L2v-lower-1}) follows because $\Gamma_{2}$ is a subset of $\Gamma$ such that $\gamma_1=0$ {and hence} $\pen_2(\gamma_1)=0$ for $\gamma\in\Gamma_2$.

Take $\Gamma_{20} = \{ \gamma \in \Gamma_\rpQ: \Gamma_0=0, \pen_2(\gamma)=  (4q)^{-1/2} \}$ for $q=p(1-p)$,
{where  $\Gamma_\rpQ$ is the subset of $\Gamma_2$ associated with interaction ramp functions as in the proof of Theorem \ref{thm:logit-L2}.}
Inequality (\ref{eq:hinge-L2v-lower-2}) {holds because $\Gamma_{20} \subset \Gamma_2$ by definition.} Inequality (\ref{eq:hinge-L2v-lower-3}) follows from Proposition \ref{pro:hinge-L2v-lower}: it holds
with probability at least $1-2 \delta$ that for any $\gamma \in \Gamma_{20}$,
\begin{align*}
& \quad K_\HG (P_n, P_{\hat\theta}; h_{\gamma,\hat\mu})
\ge  \E_{P_{\theta^*}} h(x) - \E_{P_{\hat\theta}} h(x)  - \tilde\Delta_{32} ,
\end{align*}
where $\tilde \Delta_{32} = \epsilon
+ \sqrt{p}\lambda_{32}(4q)^{-1/2}$ and
$$
\lambda_{32} = C_{\mathrm{rad4}} \sqrt{\frac{6(p-1)}{n}} + \sqrt{\frac{(p-1)\log(\delta^{-1})}{n}} .
$$
Note that $\lambda_{32}$ is the same as in the proof of Theorem~\ref{thm:logit-L2}. From (\ref{eq:hinge-L2v-lower-1})--(\ref{eq:hinge-L2v-lower-3}), the lower bound in (\ref{eq:hinge-L2-main-frame}) holds with probability at least $1-2\delta$,
where $\Delta_{22} = \tilde \Delta_{32} + \lambda_3 (4q)^{-1/2} $ and $d(\hat\theta, \theta^*) =
\max_{\gamma\in\Gamma_{20}} \{\E_{P_{\theta^*}} h_{\gamma,\hat\mu} (x)- \E_{P_{\hat\theta}}h_{\gamma,\hat\mu}(x)\}$.

(Step 5) We deduce an upper bound on $\| \hat \Sigma - \Sigma^* \|_\fro$, by relating the moment matching term $d(\hat\theta,\theta^*)$
to the estimation error.
First, combining the upper bound in (\ref{eq:hinge-L2-main-frame}) from Step 1 and the lower bound from Step 4
shows that with probability $1- {6}\delta$,
\begin{align*}
& \quad (4q)^{-1/2}\max_{\gamma\in\Gamma_\rpQ, \pen_2(\gamma)=1 } \left\{\E_{P_{\theta^*}} h_{\gamma,\hat\mu}(x) - \E_{P_{\hat\theta}}h_{\gamma,\hat\mu}(x) \right\} \\
& \le  3\epsilon + 2\sqrt{\epsilon/(n\delta)} - \sqrt{p}\lambda_{32}(4q)^{-1/2} -(25/3)\sqrt{5p} C_{\mathrm{sp22}} M_{21}  \lambda_{31}(4q)^{-1/2},
\end{align*}
which gives
\begin{align}
& \quad \max_{\gamma\in\Gamma_\rpQ, \pen_2(\gamma)=1/2 } \left\{\E_{P_{\theta^*}} h_{\gamma,\hat\mu}(x) - \E_{P_{\hat\theta}}h_{\gamma,\hat\mu}(x) \right\}
\le \sqrt{p} \; \err_{h3} (n,p, \delta, \epsilon) , \label{eq:hinge-L2v-combine}
\end{align}
where
\begin{align*}
& \err_{h3} (n,p, \delta, \epsilon) =  3\epsilon\sqrt{p-1}+2\sqrt{\epsilon(p-1)/(n\delta)}+\lambda_{32}/2+(25\sqrt{5}/6) C_{\mathrm{sp22}} M_{21}  \lambda_{31}
\end{align*}
The desired result then follows from Proposition~\ref{pro:logit-L2v-combine}: inequality (\ref{eq:hinge-L2v-combine}) implies that
\begin{align*}
\frac{1}{\sqrt{p}} \| \hat \Sigma - \Sigma^* \|_\fro &\le 2 M_2^{1/2}  \| \hat \sigma - \sigma^*\|_2  + S_{7} (\sqrt{2}\Delta_{\hat\mu,\hat\sigma} + \err_{h3} (n,p, \delta, \epsilon)  )\\
&\le S_{9,a}\err_{h2}(n,p,\delta,\epsilon) + S_{7} \err_{h3}(n,p,\delta,\epsilon),
\end{align*}
where $\Delta_{\hat\mu, \hat\sigma}=(\|\hat\mu - \mu^*\|_2^2 + \|\hat\sigma - \sigma^*\|_2^2)^{1/2}$ and $S_{9,a}=2M_2^{1/2}S_{6,a}+\sqrt{2}S_{7}(S_{4,a}+S_{6,a})$.

\end{proof}
\color{black}
%%% ----------------------------------------------------------------------------------------------------------------------------------------- two obj
\subsection{Proof of Corollary~\ref{cor:two-obj}}
    (i)
    In the {proofs of Theorems \ref{thm:logit-L1} and \ref{thm:logit-L2}, we used the main frame,}
    \begin{align}
    d(\hat\theta, \theta^*) - \Delta_1 \le \max_{\gamma \in \Gamma} \{K_f(P_n, P_{\hat\theta};h_{\gamma, \hat\mu}) -{\pen(\gamma;\lambda)} \} \le \Delta_2, \label{eq:two-obj-main-frame}
    \end{align}
    {where $\pen(\gamma; \lambda)$ is $\lambda_1 ( \| \gamma_1 \|_1 + \|\gamma_2\|_1) $ or
     $\lambda_2 \| \gamma_1\|_2 + \lambda_3 \| \gamma_2\|_2$.}
    For Theorems \ref{thm:logit-L1} and \ref{thm:logit-L2}, we showed the upper bound in (\ref{eq:two-obj-main-frame})
    using the fact that $\hat\theta$ is the minimizer of
    $$\max_{\gamma \in \Gamma} \{K_f(P_n, P_{\theta};h_{\gamma, \mu}) - \lambda \pen(\gamma)\},$$
    which is a function of $\theta$ {by the definition of (\ref{eq:logit-fgan-L1}) and (\ref{eq:logit-fgan-L2}) as nested optimization (see Remark \ref{rem:nash}).}
    {Now $\hat\theta$ is not defined as a minimizer of {the} above function, but a solution to an alternating optimization problem (\ref{eq:two-obj-gan-centered})
    with two objectives. We need to develop new arguments.
    On the other hand, we showed the lower bound in (\ref{eq:two-obj-main-frame}) for Theorems \ref{thm:logit-L1} and \ref{thm:logit-L2},
    through choosing different subsets of $\Gamma$. The previous arguments are still applicable here.}

    (Step 1) For the upper bound in (\ref{eq:two-obj-main-frame}), we show that the following holds with probability at least $1-\delta$,
    \begin{align}
    &\quad \max_{\gamma \in \Gamma} \{K_f(P_n, P_{\hat\theta};h_{\gamma, \hat\mu}) - {\pen(\gamma;\lambda)}\} \nonumber\\
    &\le  \max_{\gamma \in \Gamma} \{-f^{\prime}(1 - \hat\epsilon)\hat{\epsilon} + R_1\left|\E_{P_{\theta^{*},n}} h_{\gamma, \hat\mu}(x) - \E_{P_{\theta^{*}}} h_{\gamma, \hat\mu}(x)\right|- {\pen(\gamma;\lambda)}\} \label{eq:two-obj-upper-prf-1}\\
    &\le \Delta_1 + \max_{\gamma \in \Gamma} \{ R_1\left|\E_{P_{\theta^{*},n}} h_{\gamma, \hat\mu}(x) - \E_{P_{\theta^{*}}} h_{\gamma, \hat\mu}(x)\right|-{\pen(\gamma;\lambda)} \} \label{eq:two-obj-upper-prf-2},
    \end{align}
    where $\hat\epsilon$ is the (unobserved) fraction of contamination in $(X_1,\ldots,X_n)$.
    % and $R_1$ is the Lipschitz constant in Assumption {\ref{ass:f-condition2}(iii)}.
    Inequality (\ref{eq:two-obj-upper-prf-1}) follows from Supplement Lemma \ref{lem:two-obj-centered-upr-bound}, and
    is the most important step for connecting two-objective GAN with logit $f$-GAN.
    Inequality (\ref{eq:two-obj-upper-prf-2}) follows from an upper bound on $\hat\epsilon$ as proved in Supplement Proposition \ref{pro:logit-L1-upper},
    where $\Delta_1= -f^{\prime}(3/5)(\epsilon + \sqrt{\epsilon/(n\delta)})$,
    the same as {$\Delta_{11}$ and $\Delta_{21}$} in the proofs of Theorems \ref{thm:logit-L1} and \ref{thm:logit-L2}.

    Similarly as in Supplement Proposition \ref{pro:logit-L1-upper} or \ref{pro:logit-L2-upper},
    the term $|\E_{P_{\theta^{*},n}} h_{\gamma, \hat\mu}(x) - \E_{P_{\theta^{*}}} h_{\gamma, \hat\mu}(x)|$ can be controlled in terms of
    the $L_1$ or $L_2$ norms of $(\gamma_1, \gamma_2)$, using Supplement Lemma \ref{lem:spline-L1-upper} or \ref{lem:spline-L2-upper}.
    Then for $\pen(\gamma; \lambda)$ defined as an $L_1$ or $L_2$ penalty, it can be shown that
    the following holds with probability at least $1-4\delta$ or $1-6 \delta$,
    \begin{align}
    R_1\left|\E_{P_{\theta^{*},n}} h_{\gamma, \hat\mu}(x) - \E_{P_{\theta^{*}}} h_{\gamma, \hat\mu}(x)\right|- {\pen(\gamma;\lambda)} \le 0. \label{eq:two-obj-upper-prf-3}
    \end{align}
    provided that the tuning parameters $\lambda_1$ or $(\lambda_2,\lambda_3)$ are chosen as in Theorem \ref{thm:logit-L1} or \ref{thm:logit-L2} respectively.
    From (\ref{eq:two-obj-upper-prf-1})--(\ref{eq:two-obj-upper-prf-3}),
    the upper bound holds in (\ref{eq:two-obj-main-frame}) with probability $1-5 \delta$ or $1-7 \delta$.

    (Steps 2,3) The lower bound step and the estimation error step for $L_1$ or $L_2$ penalized two-objective GAN are
    the same as in the proofs of Theorems \ref{thm:logit-L1} and \ref{thm:logit-L2} respectively.

    (ii) For the two-objective hinge GAN hinge (\ref{eq:two-obj-hinge-centered}), the result follows similarly using Lemma~\ref{lem:hinge-two-obj-upr-bound} with $\Delta_1 = 2(\epsilon+\sqrt{\epsilon/(n\delta)})$ and $R_1=1$.

\section{Technical details}

\subsection{Details in main proof of Theorem~\ref{thm:pop-robust}}
\begin{lem} \label{lem:sqr-tv-lwr-bound}
Suppose that $f: (0,\infty) \to \bbR$ is convex with $f(1)=0$ and satisfies Assumption \ref{ass:f-condition}(i).
Denote $C_f = \inf_{t \in (0,1]}{f^{\dprime}(t)}$. Then
\begin{equation*}
D_f(P||Q) \ge \frac{C_{f}}{2} \TV(P,Q)^2.
\end{equation*}
If further Assumption \ref{ass:f-condition}(iii) holds, then $D_f(P||Q) \ge \frac{f^\dprime(1)}{2} \TV(P,Q)^2$.
\end{lem}

\begin{proof}
Because $x^2$ is convex in $x$, by Jensen's inequality, we have
\begin{equation*}
\TV(P,Q)^2 \le \int f^2_{\TV}(p/q) \dif Q,
\end{equation*}
where $f_{\TV}(t) = (1-t)_+$ and $p/q$ is the density ratio $\dif P / \dif Q$.
Note that $D_f ( P\|Q )$ can be equivalently obtained as $D_{\tilde f} (P \| Q)$,
where $\tilde f(t) = f(t) - f^{\prime}(1)(t-1)$.
Therefore, it suffices to show that $\tilde f(t) \ge \frac{C_{f}}{2}f^2_{TV}(t)$ for $t \in (0, \infty)$.

By a Taylor expansion of $f$, we have
\begin{align}
f(t) &= f(1) + f^{\prime}(1)(t-1) + \frac{f^{\prime\prime}(\tilde t)}{2}(t-1)^2  \nonumber \\
& \ge f^{\prime}(1)(t-1) + \frac{C_{f}}{2}(t-1)^2,  \label{eq:sqr-tv-lwr-bound-prf}
\end{align}
where $\tilde t$ lies between $t$ and 1. If $t \in (0,1]$, then (\ref{eq:sqr-tv-lwr-bound-prf}) gives
\begin{equation*}
\tilde f(t) \ge \frac{C_{f}}{2}(t-1)^2 = \frac{C_{f}}{2}(1-t)_+^2 = \frac{C_{f}}{2} f^2_{\TV}(t).
\end{equation*}
If $t \in (1,\infty)$, then because $C_f \ge 0$ by convexity of $f$, (\ref{eq:sqr-tv-lwr-bound-prf}) gives
\begin{equation*}
\tilde f(t) \ge  \frac{C_f}{2}(t-1)^2 \ge 0 = \frac{C_{f}}{2}(1-t)^{2}_+ =\frac{C_{f}}{2}f^2_{TV}(t).
\end{equation*}
Combining the two cases completes the proof.
\end{proof}

Denote as $\Phi(\cdot)$ the cumulative distribution function of $\N(0,1)$, and {$\erf(x)$ the probability of $[-\sqrt{2}x, \sqrt{2}x]$ under $\N(0,1)$ for $x \ge 0$}.

\begin{lem} \label{lem:invert-pop}
{Let $a \in [0, 1/2)$ be arbitrarily fixed.}

(i) If $\Phi(x) \le 1/2+ a$ for $x \ge 0$ , then
$$ x \le S_{1, a} \left\{ \Phi(x) - 1/2 \right\} , $$
where $S_{1,a}=\{ \Phi^\prime(\Phi^{-1}(1/2+a)) \}^{-1}$.

(ii) If $|\erf ( x\sqrt{z_0/2}) - 1/2| \le  a$ for {$x \ge 0$}, then
$$ |x - 1| \le  S_{2,a} \left| \erf (x\sqrt{z_0/2}) - 1/2 \right|, $$
where
$S_{2,a}= \{ \sqrt{z_0/2}\erf^\prime(\sqrt{2/z_0}\erf^{-1}(1/2+a)) \}^{-1}$ and $z_0$ is an universal constant such that $\erf(\sqrt{z_0/2}) = 1/2$.
\end{lem}

\begin{proof}
(i) By the mean value theorem, we have $\Phi(x) \ge \frac{1}{2} + S^{-1}_{1,a}x$,
because $\Phi^\prime (\cdot)$ is decreasing on $[0, +\infty)$.

(ii) By the mean value theorem, we have
$|\text{erf}(x\sqrt{z_0/2}) - 1/2| \ge  S^{-1}_{2,a} {|x-1|}$,
because $\erf^\prime (\cdot)$ is decreasing on $[0, +\infty)$.
\end{proof}

\begin{manualpro}{S4} \label{pro:local-linear-pop}
For two multivariate Gaussian distributions, $P_{\bar\theta}$ and $P_{\theta^*}$, with $\bar\theta=(\bar\mu, \bar\Sigma)$ and $\theta^*=(\mu^*, \Sigma^*)$,
denote $d(\bar\theta, \theta^*) = \TV(P_{\bar\theta}, P_{\theta^*})$.
\begin{itemize}
\item[(i)]
If $ d(\bar\theta, \theta^*) \le a$ for a constant $a \in [0,1/2)$, then
\begin{align*}
\|\bar\mu - \mu^*\|_2 &\le S_{1,a}\|\Sigma^*\|_{\op}^{1/2} d(\bar\theta, \theta^*),\\
\|\bar\mu - \mu^*\|_{\infty} &\le S_{1,a}\|\Sigma^*\|_{\max}^{1/2} d(\bar\theta, \theta^*) ,
\end{align*}
where $S_{1,a}=\{ \Phi^\prime(\Phi^{-1}(1/2+a)) \}^{-1}$ as in Lemma~\ref{lem:invert-pop}.

\item[(ii)]
If further $ d(\bar\theta, \theta^*) \le a/(1+S_{1,a})$, then
\begin{align*}
\|\bar\Sigma - \Sigma^*\|_{\op} & \le 2 S_{3,a}\|\Sigma^*\|_{\op} d(\bar\theta, \theta^*) + S^2_{3,a} \|\Sigma^*\|_{\op}( d(\bar\theta, \theta^*))^2, \\
\|\bar\Sigma - \Sigma^*\|_{\max} & \le 4S_{3,a} \|\Sigma^*\|_{\max} d(\bar\theta, \theta^*) + 2S^2_{3,a}\|\Sigma^*\|_{\max}( d(\bar\theta, \theta^*))^2,
\end{align*}
\end{itemize}
where $S_{3,a}= S_{2,a}(1+S_{1,a})$, $S_{2,a}= \{ \sqrt{z_0/2}\, \erf^\prime(\sqrt{2/z_0}\, \erf^{-1}(1/2+a)) \}^{-1}$, and the constant $z_0$ is defined
such that $\erf(\sqrt{z_0/2}) = 1/2$, as in Lemma~\ref{lem:invert-pop}.
\end{manualpro}

\begin{proof}
The TV distance, $D_\TV(P_1 \| P_2)$, can be equivalently defined as
\begin{align*}
\TV (P_1, P_2) = \sup_{A \in \mathcal{A}}| P_1(A)-P_2(A)|,
\end{align*}
for $P_1$ and $P_2$ defined in a probability space $(\mathcal{X}, \mathcal{A})$.
This definition is applicable to multivariate Gaussian distributions with singular variance matrices.
To derive the desired results, we choose specific events $A$ and show that the differences in the means
and variance matrices can be upper bounded by $ |P_{\bar\theta} (A) - P_{\theta^*} (A)|$.

We first show results (i) and (ii), when $\Sigma^*$ and $\bar \Sigma$ are nonsingular. Then
we show that the results remain valid when $\Sigma^*$ or $\bar \Sigma$ is singular.

(i) Assume that both $\Sigma^*$ and $\bar \Sigma$ are nonsingular.
For any $u \in \bbR^p $, we have by the definition of TV,
$$
 P_{\mu^*, \Sigma^*}(u^\T X \le u^\T \bar\mu) - P_{\bar{\mu}, \bar{\Sigma}}(u^\T X \le u^\T \bar\mu) \le d(\bar\theta, \theta^*) .
$$
For nonzero $u \in \bbR^p$, because $u^\T \bar \Sigma u \not= 0$ and $u^\T \Sigma^* u \not= 0$, we have
\begin{align*}
P_{\bar{\mu}, \bar{\Sigma}}(u^\T X \le u^\T \bar\mu) & = \frac{1}{2},\\
P_{\mu^*, \Sigma^*}(u^\T X \le u^\T \bar\mu) & = \Phi\left(\frac{u^\T(\bar \mu- \mu^*)}{\sqrt{u^\T \Sigma^* u}}\right).
\end{align*}
Combining the preceding three displays shows that for nonzero $u \in \bbR^p$,
\begin{align*}
\Phi\left(\frac{u^\T(\bar \mu- \mu^*) }{\sqrt{u^\T \Sigma^* u}}\right) \le  \frac{1}{2} + d (\bar\theta, \theta^*).
\end{align*}
By Lemma \ref{lem:invert-pop} (i), if $d(\bar\theta,\theta^*) \le a$ {for a constant $a \in [0, 1/2)$}, then for any $u\in \bbR^p$ satisfying
{$u^\T(\bar\mu - \mu^*) \ge 0$},
\begin{align}
{ 0 \le } u^\T(\bar \mu- \mu^*) & \le \sqrt{u^\T \Sigma^* u} S_{1,a} d (\bar\theta, \theta^*) . \label{eq:local-linear-pop-0}
\end{align}
Let $\mathcal{U}_2 = \{u \in \bbR^p: \|u\|_2=1\}$.
By (\ref{eq:local-linear-pop-0}) with $u$ restricted such that $u \in \mathcal{U}_2 $
and {$u^\T(\bar\mu - \mu^*) \ge 0$}, we have
\begin{align*}
\|\bar{\mu} - \mu^* \|_2 &=\sup_{u\in\mathcal{U}_2} u^\T(\bar \mu- \mu^*)\\
& \le S_{1,a}\|\Sigma^*\|_{\op}^{1/2} d(\bar\theta, \theta^*).
\end{align*}
Similarly, let $\mathcal{U}_{\infty} = \{\pm \me_j: j =1,\dots, p\}$,
where $e_j$ is a vector with $j$th coordinate being one and others being zero.
By (\ref{eq:local-linear-pop-0}) with $u$ restricted such that $u \in \mathcal{U}_\infty $
and {$u^\T(\bar\mu - \mu^*) \ge 0$}, we have
\begin{align*}
\|\bar{\mu} - \mu^* \|_\infty &=\sup_{u\in\mathcal{U}_{\infty}} u^\T(\bar{\mu} - \mu^*) \\
& \le S_{1,a}\|\Sigma^*\|_{\max}^{1/2} d(\bar\theta, \theta^*).
\end{align*}
The last line uses the fact that $\sup_{u\in\mathcal{U}_{\infty}} u^\T \Sigma^* u = \| \diag(\Sigma^*) \|_{\infty} = \| \Sigma^* \|_{\max}$ by the nature of variance matrices.

(ii) Assume that $\Sigma^*$ and $\bar \Sigma$ are nonsingular.
We first separate the bias caused by the location difference between $P_{\bar\theta}$ and $P_{\theta^*}$. By the triangle inequality, we have
\begin{align}
\TV(P_{\bar\mu, \Sigma^*}, P_{\bar\mu, \bar\Sigma}) \le \TV(P_{\bar\mu, \Sigma^*}, P_{\mu^*, \Sigma^*}) + \TV( P_{\mu^*, \Sigma^*},  P_{\bar\mu, \bar\Sigma}).  \label{eq:local-linear-pop-triangle}
\end{align}
By Lemma~\ref{lem:sqr-tv-lwr-bound}, we know that $TV(P,Q) \le \sqrt{2 D_{\mathrm{KL}}(P||Q)}$. Then we have
\begin{align}
\TV(P_{\bar\mu, \Sigma^*}, P_{\mu^*, \Sigma^*}) &\le \sqrt{2 D_{\mathrm{KL}}(P_{\bar\mu, \Sigma^*} || P_{\mu^*, \Sigma^*})}\nonumber\\
&\le S_{1,a} d (\bar\theta, \theta^*) .\label{eq:KL-Gaussian}
\end{align}
provided that $d (\bar\theta, \theta^*) \le a $. Inequality (\ref{eq:KL-Gaussian}) follows because by standard calculation
\begin{align*}
D_{\text{KL}}(N(\bar \mu, \Sigma^*)|| N(\mu^*, \Sigma^*)) & = \frac{1}{2}(\bar\mu - \mu^*)^\T \Sigma^{*-1}(\bar\mu - \mu^*),
\end{align*}
and taking $u = \Sigma^{*-1}(\bar\mu - \mu^*)$ in (\ref{eq:local-linear-pop-0}) gives
\begin{align*}
\sqrt{ (\bar\mu - \mu^*)^\T \Sigma^{*-1}(\bar\mu - \mu^*) }
& \le S_{1,a} d (\bar\theta, \theta^*) .
\end{align*}
Combining (\ref{eq:local-linear-pop-triangle}) and (\ref{eq:KL-Gaussian}) yields
\begin{align}
\TV(P_{\bar\mu, \Sigma^*}, P_{\bar\mu, \bar\Sigma})
&\le d(\bar\theta, \theta^*) +  {S_{1,a}}d(\bar\theta, \theta^*). \label{eq:local-linear-pop-main}
\end{align}

For any $u \in \bbR^p$ {such that $ u^\T (\bar\Sigma- \Sigma^*) u \ge 0$}, (\ref{eq:local-linear-pop-main}) implies
\begin{align*}
0 &\le P_{\bar\mu, \Sigma^*}\left\{ (u^\T X - u^\T \bar\mu)^2 \le  z_0 u^\T \bar\Sigma u \right\}- P_{\bar{\mu}, \bar{\Sigma}}\left\{ (u^\T X - u^\T \bar\mu)^2 \le z_0 u^\T \bar\Sigma u \right\}\\
&= P_{0, \Sigma^*}\left\{ (u^\T X)^2 \le   z_0 u^\T \bar\Sigma u \right\}- P_{0, \bar{\Sigma}}\left\{ (u^\T X)^2 \le  z_0 u^\T \bar\Sigma u \right\}\\
&\le d(\bar\theta, \theta^*) +  S_{1,a}d(\bar\theta, \theta^*),
\end{align*}
where $z_0$ is an universal constant such that $\erf(\sqrt{z_0/2}) = 1/2$.  Similarly, for any $u \in \bbR^p$ such that $ u^\T (\bar\Sigma- \Sigma^*) u \le 0$, (\ref{eq:local-linear-pop-main}) implies
\begin{align*}
0 &\le P_{\bar\mu, \Sigma^*}\left\{ (u^\T X - u^\T \bar\mu)^2 \ge  z_0 u^\T \bar\Sigma u \right\}- P_{\bar{\mu}, \bar{\Sigma}}\left\{ (u^\T X - u^\T \bar\mu)^2 \ge z_0 u^\T \bar\Sigma u \right\}\\
&= P_{0, \Sigma^*}\left\{ (u^\T X)^2 \ge   z_0 u^\T \bar\Sigma u \right\}- P_{0, \bar{\Sigma}}\left\{ (u^\T X)^2 \ge  z_0 u^\T \bar\Sigma u \right\}\\
&\le d(\bar\theta, \theta^*) +  S_{1,a}d(\bar\theta, \theta^*).
\end{align*}
Notice that the choice of $z_0$ ensures that for $Z \in \N(0,1)$,
\begin{align*}
P_{0, \bar{\Sigma}}\left\{ (u^\T X)^2 \le z_0 u^\T \bar\Sigma u \right\} & = \pr(Z^2 \le z_0) = \frac{1}{2} \\
& =\pr(Z^2 \ge z_0)= P_{0, \bar{\Sigma}}\left\{ (u^\T X)^2 \ge z_0 u^\T \bar\Sigma u \right\} .
\end{align*}
{Moreover}, for any {nonzero} $u \in \bbR^p$, we have by the definition of $\erf$,
\begin{align*}
& P_{0, \Sigma^*}\left\{ (u^\T X)^2 \le  z_0 u^\T \bar\Sigma u \right\}  = \erf \left(\sqrt{ \frac{z_0 u^\T \bar\Sigma u}{2u^\T \Sigma^* u}}\right).
\end{align*}
Combining the preceding {four} displays shows that for any {nonzero} $u \in \bbR^p$,
\begin{align*}
\left|\erf\left(\sqrt{ \frac{z_0 u^\T \bar\Sigma u}{2u^\T \Sigma^* u}}\right)-\frac{1}{2}\right| \le   (1 +  S_{1,a})d(\bar\theta, \theta^*). %\label{eq:local-linear-pop-cov-l2}
\end{align*}
By Lemma~\ref{lem:invert-pop} (ii), {if $d(\bar\theta,\theta^*) \le \min(a, a/(1+S_{1,a})) = a/(1+S_{1,a})$,} then for any nonzero $u\in \bbR^p$,
\begin{align*}
\left|\sqrt{ \frac{u^\T \bar\Sigma u}{u^\T \Sigma^* u}} -1\right|  &\le  S_{2,a}(1+S_{1,a})d(\bar\theta, \theta^*), %\label{eq:local-linear-pop-1}
\end{align*}
or equivalently {for any $u\in \bbR^p$,}
\begin{align}
\left|\sqrt{u^\T \bar\Sigma u} -\sqrt{u^\T \Sigma^* u}\right|  &\le S_{2,a}(1+S_{1,a})\sqrt{u^\T \Sigma^* u} d(\bar\theta, \theta^*). \label{eq:local-linear-pop-2}
\end{align}
Notice that for any $a, b, c {\ge} 0$, if $|\sqrt{a}-\sqrt{b} |\le c$ then $|a-b|\le 2\sqrt{b}c+c^2$. Thus, inequality (\ref{eq:local-linear-pop-2}) implies
\begin{align}
&\quad \|\bar\Sigma - \Sigma^*\|_{\op} = \sup_{u \in \mathcal{U}_2} \left|u^\T( \bar\Sigma - \Sigma^*) u\right| \nonumber\\
&\le 2S_{3,a}\|\Sigma^*\|_{\op}d(\bar\theta, \theta^*) + S^2_{3,a}\|\Sigma^*\|_{\op}(d(\bar\theta, \theta^*))^2, \label{eq:local-linear-pop-3}
\end{align}
where $S_{3,a} = S_{2,a}(1+S_{1,a})$.

To handle $\|\bar\Sigma - \Sigma^*\|_{\max}$, let $\mathcal{U}_{2, \infty} = \{\pm e_{ij}: i,j=1,\dots, p, i \neq j\}$, where $e_{ij}$ is a vector in $\mathcal{U}_2$ with only $i$th and $j$th coordinates possibly being nonzero.
For $ u \in \mathcal{U}_{2, \infty}$, we have
$u^\T \Sigma^* u = u_{ij}^\T \Sigma^*_{ij} u_{ij}$ and $u^\T \bar\Sigma u = u_{ij}^\T \bar\Sigma_{ij} u_{ij}$,
where $u_{ij} \in \bbR^2$ is formed by $i$th and $j$th coordinates of $u$,
and $\Sigma^*_{ij}$ and $\bar\Sigma_{ij}$ are $2\times 2$ matrices, formed by selecting $i$th and $j$th rows and columns from $\Sigma^*$ and $\bar\Sigma$ respectively.
Similarly as in the deviation of (\ref{eq:local-linear-pop-3}), applying inequality (\ref{eq:local-linear-pop-2}) with $u \in \mathcal{U}_{2, \infty}$, we have
\begin{align*}
\|\bar\Sigma_{ij} - \Sigma_{ij}^*\|_{\op}  &= \sup_{u \in U_{2, \infty}} \left| u^\T(\bar\Sigma - \Sigma^*)u  \right|\\
&\le  2S_{3,a}\|\Sigma_{ij}^*\|_{\op}d(\bar\theta, \theta^*) + S^2_{3,a}\|\Sigma_{ij}^*\|_{\op}(d(\bar\theta, \theta^*))^2.
\end{align*}
Because for a matrix $A \in \bbR^{m_1 \times m_2 }$, $\|A\|_{\max}\le \|A\|_{\op} \le \sqrt{m_1 m_2}\|A\|_{\max}$, the above inequality implies
that for any $i\not= j \in \{1,\ldots,p\}$,
\begin{align}
&\quad \|\bar\Sigma_{ij} - \Sigma_{ij}^*\|_{\max} \nonumber \\
&\le  4S_{3,a}\|\Sigma_{ij}^*\|_{\max}d(\bar\theta, \theta^*) +2S^2_{3,a}\|\Sigma_{ij}^*\|_{\max}(d(\bar\theta, \theta^*))^2. \label{eq:local-linear-pop-4}
\end{align}
Taking the maximum on both sides of (\ref{eq:local-linear-pop-4}) over $i\not=j$ gives the desired result:
\begin{align*}
&\quad \|\bar\Sigma - \Sigma^*\|_{\max} {=}  \max_{i\neq j \in \{1,\dots, p\}}\|\bar\Sigma_{ij} - \Sigma_{ij}^*\|_{\max} \nonumber \\
&\le    4S_{3,a}\|\Sigma^*\|_{\max}d(\bar\theta, \theta^*) + 2S^2_{3,a}\|\Sigma^*\|_{\max}(d(\bar\theta, \theta^*))^2 . %\label{eq:local-linear-pop-5}
\end{align*}

(iii) Consider the case where $\Sigma^*$ or $\bar\Sigma$ is singular. As the following argument is symmetric in $\Sigma^*$ and $\bar\Sigma$,
we assume without loss of generality that $\Sigma^*$ is singular.
Fix any nonzero $u$ such that $u^\T \Sigma^* u = 0$.

First, we show that for $\bar\theta = (\bar\mu, \bar\Sigma)$ such that $\TV(P_{\bar\theta}, P_{\theta^*}) < 1$, we also have $u^\T \bar\Sigma u = 0$.
In fact, {$\TV(P_{\bar\theta}, P_{\theta^*}) < 1$ implies}
\begin{align}
    \left| P_{\mu^*, \Sigma^*}(u^\T X = u^\T \mu^*) - P_{\bar{\mu}, \bar{\Sigma}}(u^\T X = u^\T \mu^*) \right| \le d(\bar\theta, \theta^*) < 1. \label{eq:pd-contradict-1}
\end{align}
Note that
$P_{\mu^*, \Sigma^*}(u^\T X = u^\T \mu^*) =1$ because $u^\T \Sigma^* u =0$.
If $u^\T \bar\Sigma u > 0$, then $P_{\bar{\mu}, \bar{\Sigma}}(u^\T X = u^\T \mu^*) =0$, and hence (\ref{eq:pd-contradict-1}) gives
$$
|1 - 0| \le d(\bar\theta, \theta^*) < 1,
$$
which is a contradiction. Thus $u^\T \bar\Sigma u = 0$.

Next we show that for $\bar\theta = (\bar\mu, \bar\Sigma)$ such that $\TV(\bar\theta, \theta^*) < 1$, we also have $u^\T(\mu^* - \bar\mu) = 0$. In fact, with $u^\T \bar\Sigma u =0$ as shown above, we have that
$P_{\bar\mu, \bar\Sigma}(u^\T X = u^\T \mu^*)=1$ if $u^\T\mu^* = u^\T\bar\mu$ and $P_{\bar\mu, \bar\Sigma}(u^\T X = u^\T \mu^*)=0$ otherwise. If $u^\T\mu^* \neq u^\T\bar\mu$, then inequality (\ref{eq:pd-contradict-1}) gives
$$
|1 - 0| \le d(\bar\theta, \theta^*) < 1,
$$
which is a contradiction. Thus $u^\T\mu^* = u^\T\bar\mu$.

From the two preceding results, we see that
the upper bounds (\ref{eq:local-linear-pop-0}) and (\ref{eq:local-linear-pop-2}) derived in (i) and (ii) remain valid
for any $u\in\bbR^p$ satisfying $u^\T \Sigma^* u = 0$.
Hence the desired results hold by the remaining proofs in (i) and (ii).
\end{proof}

\subsection{Details in main proof of Theorem \ref{thm:logit-L1}}

\begin{lem} \label{lem:spline-L1-upper}
Suppose that $X_1, \ldots, X_n$ are independent and identically distributed as
$X \sim N_p(0, \Sigma)$ with $ \|\Sigma\|_{\max} \leq M_1$.
For $k$ fixed knots $\xi_1,\ldots,\xi_k$ in $\bbR$,
denote $\varphi(x) = (\varphi^\T_1(x), \ldots, \varphi^\T_k(x))^\T$,
where $\varphi_l(x) \in \bbR^p$ is obtained by applying $t \mapsto (t-\xi_l)_+$ componentwise to $x \in \bbR^p$ for {$l=1,\ldots,k$.}
Then the following results hold.

(i) Each component of the random vector $\varphi(X) - \E \varphi(X)$ is a sub-gaussian random variable with tail parameter $M_1^{1/2}$.

(ii) For any $\delta >0$, we have that with probability at least $1-2\delta$,
\begin{align*}
&\quad \sup_{\| w \|_{1}=1}\left| w^\T \left\{\frac{1}{n}\sum_{i=1}^{n}\varphi(X_i)  - \E \varphi(X) \right\}\right| \\
&\le C_{\mathrm{sp11}} M_1^{1/2} \sqrt{\frac{2\log (kp) + \log(\delta^{-1})}{n}} ,
\end{align*}
where $C_{\mathrm{sp11}} = \sqrt{2} C_{\mathrm{sg5}}$, depending on the universal constant $C_{\mathrm{sg5}}$ in Lemma~\ref{lem:subg-concentration}.

(iii) Let $\mathcal{A}_1=\{A \in \bbR^{kp \times kp}: \| A \|_{1,1} =1\}$. For any $\delta >0$, we have that with probability at least $1-4\delta$,
both the inequality in (ii) and
\begin{align*}
& \quad  \sup_{A \in \mathcal{A}_1} \left| \frac{1}{n} \sum_{i=1}^{n} \varphi^\T(X_{i}) A \varphi(X_{i}) - \E \varphi^\T(X) A \varphi(X) \right| \\
& \le  C_{\mathrm{sp12}} M_{11} \left\{ \sqrt{\frac{2\log(kp) +\log(\delta^{-1})}{n}} + \frac{2\log(kp)+ \log(\delta^{-1})}{n}\right\},
\end{align*}
where $M_{11} =  M_1^{1/2} ( M_1^{1/2} +  \sqrt{2\pi}  \|\xi \|_\infty)$,
$\|\xi \|_\infty = \max_{{l=1,\ldots,k}} |\xi_l|$, and $C_{\mathrm{sp12}} = \sqrt{2/\pi}C_{\mathrm{sp11}} + C_{\mathrm{sx7}} C_{\mathrm{sx6}} C_{\mathrm{sx5}}$. Constants $(C_{\mathrm{sx5}}, C_{\mathrm{sx6}}, C_{\mathrm{sx7}})$ are the universal constants in Lemmas~\ref{lem:subexp-prod}, \ref{lem:subexp-centering}, and \ref{lem:subexp-concentration}.
\end{lem}

\begin{proof}
(i) This can be obtained as the univariate case of Lemma~\ref{lem:spline-L2-upper} (i). It is only required that the marginal variance of each component of $X$ is upper bounded by $M_1$.

(ii) Notice that
\begin{align*}
\sup_{\| w \|_{1} = 1}\left| w^\T \left\{\frac{1}{n}\sum_{i=1}^{n}\varphi(X_i)  - \E \varphi(X) \right\}\right|
= \left\| \frac{1}{n}\sum_{i=1}^{n}\varphi(X_i)  - \E \varphi(X) \right\|_\infty.
\end{align*}
By (i) and sub-gaussian concentration (Lemma~\ref{lem:subg-concentration}),
each component of $n^{-1} \sum_{i=1}^n \varphi(X_i) - \E \varphi(X)$ is sub-gaussian with tail parameter $C_{\mathrm{sg5}} (M_1/n)^{1/2}$.
Then for any $t>0$, by the union bound, we have that with probability at $1-2 k^2p^2 \me^{-t}$,
\begin{align*}
\left\| \frac{1}{n}\sum_{i=1}^{n}\varphi(X_i)  - \E \varphi(X) \right\|_\infty \le \sqrt{2} C_{\mathrm{sg5}} (M_1/n)^{1/2} t^{1/2}.
\end{align*}
Taking $t =2 \log (kp) + \log (\delta^{-1})$ gives the desired result.

(iii)
The difference of interest can be expressed in terms of the centered variables as
\begin{align}
& \quad \frac{1}{n} \sum_{i=1}^{n} \varphi^\T_i A \varphi_i - \E \varphi^\T A \varphi \nonumber \\
& = \frac{1}{n} \sum_{i=1}^{n} (\varphi_i - \E \varphi)^\T A (\varphi_i -\E\varphi) - \E \{ (\varphi - \E \varphi)^\T A (\varphi-\E \varphi)\}  \label{eq:lem-spline-L1-upper-prf1} \\
& \quad + \frac{1}{n} \sum_{i=1}^n 2 (\E \varphi)^\T A (\varphi_i - \E \varphi) .  \label{eq:lem-spline-L1-upper-prf2}
\end{align}
We handle the concentration of the two terms separately.
Denote $\varphi_i = \varphi(X_i)$, $\varphi=\varphi(X)$, $\tilde\varphi_i = \varphi_i - \E \varphi$, and $\tilde\varphi = \varphi - \E \varphi$.

First, for $A \in \mathcal{A}_1$, the term in (\ref{eq:lem-spline-L1-upper-prf2}) can be bounded as follows:
\begin{align*}
& \quad \left| 2 (\E \varphi)^\T A \frac{1}{n} \sum_{i=1}^n \tilde\varphi_i \right|
\le 2 \| \E \varphi \|_\infty \|A\|_{1,1} \left\| \frac{1}{n} \sum_{i=1}^n \tilde\varphi_i \right\|_\infty
= 2 \| \E \varphi \|_\infty \left\| \frac{1}{n} \sum_{i=1}^n \tilde\varphi_i \right\|_\infty \\
& \le 2 \left( \frac{M_1^{1/2}}{\sqrt{2\pi}} + \| \xi \|_\infty \right) \left\| \frac{1}{n} \sum_{i=1}^n \tilde\varphi_i \right\|_\infty,
\end{align*}
where $ \| \xi \|_\infty = \max_{l=1,\ldots,k} |\xi_l|$.
The second step holds because $\|A \|_{1,1} =1$ for $A \in \mathcal{A}_1$.
The third step holds because $\| \E \varphi_l \|_\infty \le M_1^{1/2} / \sqrt{2\pi} + | \xi_l|$ for $l=1,\dots,5$ by Lemma~\ref{lem:truncated-mean}.
By (ii), for any $\delta>0$, we have that with probability at least $1-2\delta$,
\begin{align*}
&\quad \left\| \frac{1}{n} \sum_{i=1}^n \tilde\varphi_i \right\|_\infty
\le  C_{\mathrm{sp11}} M_1^{1/2} \sqrt{\frac{2\log(kp) + \log(\delta^{-1})}{n} } .
\end{align*}
From the preceding two displays, we obtain that with probability at least $1-2\delta$,
\begin{align}
& \quad \sup_{A \in \mathcal{A}_1 }  \left| 2 (\E \varphi)^\T A \frac{1}{n} \sum_{i=1}^n \tilde\varphi_i \right| \nonumber \\
& \le \sqrt{\frac{2}{\pi}} C_{\mathrm{sp11}} M_1^{1/2} \left( M_1^{1/2} +  \sqrt{2\pi}  \|\xi \|_\infty \right) \sqrt{\frac{2\log(kp) + \log(\delta^{-1})}{n} }. \label{eq:lem-spline-L1-upper-prf3}
\end{align}

Next, notice that
\begin{align*}
& \quad \sup_{A \in \mathcal{A}_1} \left| \frac{1}{n} \sum_{i=1}^{n} \tilde\varphi^\T_i A \tilde\varphi_i - \E \tilde\varphi^\T A \tilde\varphi \right|
= \left\| \frac{1}{n}\sum_{i=1}^{n} \tilde\varphi_i \otimes \tilde\varphi_i - \E \tilde\varphi \otimes\tilde\varphi \right\|_{\max}.
\end{align*}
From (i), each component of $\tilde\varphi_i$ is sub-gaussian with tail parameter $M_1^{1/2}$.
By Lemma~\ref{lem:subexp-prod},
each element of $\tilde\varphi_i \otimes \tilde\varphi_i$ is sub-exponential with tail parameter $C_{\mathrm{sx5}} M_1$.
By Lemma \ref{lem:subexp-centering}, each element of the centered version, $\tilde\varphi_i \otimes \tilde\varphi_i - \E \tilde\varphi \otimes\tilde\varphi$,
is sub-exponential with tail parameter $C_{\mathrm{sx6}} C_{\mathrm{sx5}} M_1$. Then for any $t>0$, by Lemma~\ref{lem:subexp-concentration} and the union bound, we have that
with probability at least $ 1 - 2 k^2 p^2 \me^{-t}$,
\begin{align*}
& \quad \left\| \frac{1}{n}\sum_{i=1}^{n} \tilde\varphi_i \otimes \tilde\varphi_i - \E\tilde\varphi \otimes\tilde\varphi \right\|_{\max}
\le C_{\mathrm{sx7}} C_{\mathrm{sx6}} C_{\mathrm{sx5}} M_1 \left( \sqrt{\frac{t}{n}} \vee \frac{t}{n} \right) .
\end{align*}
Taking $t =2 \log (kp) + \log (\delta^{-1})$, we obtain that with probability at least $1-2\delta$,
\begin{align}
& \quad \left\| \frac{1}{n}\sum_{i=1}^{n} \tilde\varphi_i \otimes \tilde\varphi_i - \E \tilde\varphi \otimes\tilde\varphi \right\|_{\max} \nonumber \\
& \le C_{\mathrm{sx7}} C_{\mathrm{sx6}} C_{\mathrm{sx5}} M_1 \left\{ \sqrt{\frac{2 \log (kp)+ \log (\delta^{-1})}{n}} \vee \frac{2 \log (kp) + \log (\delta^{-1})}{n} \right\} .\label{eq:lem-spline-L1-upper-prf4}
\end{align}
Combining the two bounds (\ref{eq:lem-spline-L1-upper-prf3}) and (\ref{eq:lem-spline-L1-upper-prf4}) gives the desired result.
\end{proof}

\begin{lem} \label{lem:logit-upper}
Suppose that $f: (0,\infty) \to \bbR$ is convex, non-increasing, and differentiable, and $f(1)=0$.
Denote $f^\#(t) = t f^\prime (t) - f(t)$.

(i) For any $t>0$ and $\epsilon \in [0,1)$, we have
\begin{align}
& (1-\epsilon) f^\prime (t) - f^\# (t) \le - f^{\prime}(1-\epsilon) \epsilon . \label{eq:lem-logit-upper-eps}
\end{align}

(ii) Let $\epsilon_0\in (0,1)$ be fixed. For any $\epsilon \in [0, \epsilon_0]$ and any function $h : \bbR^p \to \bbR$, we have
\begin{align*}
K_f(P_{\epsilon}, P_{\theta^*}; h) \leq -f^\prime(1-\epsilon_0) \epsilon .
\end{align*}

(iii) Suppose, in addition, that {$f^\prime(\me^u)$ is concave in $u$.} Let $\epsilon_1\in (0,1)$ be fixed. If $\hat{\epsilon} = n^{-1} \sum_{i=1}^{n} U_i \in [0, \epsilon_1]$, then
for any function $h: \bbR^p \to \bbR$,
\begin{align}
K_f(P_n, P_{\theta^* } ; h) \le -f^\prime(1-\epsilon_1) \hat\epsilon + R_1 | \E_{P_{\theta^* ,n}} h(x) - \E_{P_{\theta^* }} h(x) | ,\label{eq:lem-logit-upper-emp}
\end{align}
where $P_{\theta^*,n}$ denotes the empirical distribution of $\{X_i: U_i=0,i=1,\ldots,n\}$ in the latent representation of Huber's contamination model.
\end{lem}

\begin{proof}
(i) Notice that by definition,
\begin{align*}
& \quad (1-\epsilon) f^\prime (t) - f^\# (t) = (1-\epsilon) f^\prime (t) - t f^\prime(t) + f(t) \\
& = f(t) + f^{\prime}(t)((1-\epsilon)-t) .
\end{align*}
By the convexity of $f$, we have that for any $t>0$ and $\epsilon \in [0,1)$,
$$
f(t) + f^{\prime}(t)((1-\epsilon)-t) \le f(1-\epsilon).
$$
Moreover, by the convexity of $f$ and $f(1)=0$, we have
$$
f(1-\epsilon) \le f(1) + f^{\prime}(1-\epsilon)((1-\epsilon)-1) =  - f^{\prime}(1-\epsilon) \epsilon .
$$
Combining the preceding three displays yields the desired result.

(ii) For any function $h$, we have
\begin{align}
& \quad  K_f(P_{\epsilon}, P_{\theta^* }; h)
= \epsilon  \, \E_Q f^{\prime}(\me^{h(x)}) +  \E_{P_{\theta^* }} \left\{
(1-\epsilon)f^{\prime}(\me^{h(x)})- f^\# (\me^{h(x)}) \right\} \nonumber \\
& \le \E_{P_{\theta^* }} \left\{
(1-\epsilon)f^{\prime}(\me^{h(x)})- f^\# (\me^{h(x)}) \right\}, \label{eq:lem-logit-upper-prf1}
\end{align}
using the fact that $f$ is non-increasing and hence $f^{\prime}(t) \le 0$ for $t>0$.
Setting $t = \me^{h(x)}$ in (\ref{eq:lem-logit-upper-eps}) shows that for $\epsilon \le \epsilon_0$,
\begin{align}
(1-\epsilon)f^{\prime}(\me^{h(x)})-f^\# (\me^{h(x)}) \le  - f^{\prime}(1-\epsilon) \epsilon \le  - f^{\prime}(1-\epsilon_0) \epsilon .\label{eq:lem-logit-upper-prf2}
\end{align}
where $f^\prime(1-\epsilon) \ge f^\prime(1-\epsilon_0)$ for $\epsilon \le \epsilon_0$ by the convexity of $f$.
Combining (\ref{eq:lem-logit-upper-prf1}) and (\ref{eq:lem-logit-upper-prf2}) leads to the desired result.

(iii) For any function $h$, $K_f(P_{n}, P_{\theta^* ,n}; h)$ can be bounded as follows:
\begin{align}
& \quad K_f(P_n, P_{\theta^* } ; h) \nonumber \\
& = \frac{1}{n}\sum_{i=1}^{n}U_i f^{\prime}(\me^{h(X_i)}) + \frac{1}{n}\sum_{i=1}^{n}(1-U_i) f^{\prime}(\me^{h(X_i)}) - \E_{P_{\theta^* }}  f^\# (\me^{h(x)}) \nonumber\\
&\le (1-\hat{\epsilon})\E_{P_{\theta^* , n}}f^{\prime}(\me^{h(x)}) - \E_{P_{\theta^* }}  f^\# (\me^{h(x)}) \label{eq:lem-logit-upper-prf3} \\
&\le (1-\hat{\epsilon})f^{\prime}(\me^{\E_{P_{\theta^* , n}}h(x)}) - f^\# (\me^{\E_{P_{\theta^* }} h(x)} )  \label{eq:lem-logit-upper-prf4} \\
&\le -f^{\prime}(1-\epsilon_1) \hat{\epsilon} + |f^\# (\me^{\E_{P_{\theta^* ,n}} h(x) }) - f^\# (\me^{ \E_{P_{\theta^* }} h(x)} ) | \label{eq:lem-logit-upper-prf5}\\
&\le -f^{\prime}(1-\epsilon_1) \hat{\epsilon} + R_1  | \E_{P_{\theta^* ,n}} h(x) - \E_{P_{\theta^* }} h(x) |  .\label{eq:lem-logit-upper-prf6}
\end{align}
Line (\ref{eq:lem-logit-upper-prf3}) follows because $f^{\prime}(t) \le 0$ for $t>0$.
Line (\ref{eq:lem-logit-upper-prf4}) follows from Jensen's inequality by the concavity of $f^\prime(\me^u)$ and $-f^\# (\me^u)$ in $u$.
Line (\ref{eq:lem-logit-upper-prf5}) follows because
\begin{align*}
(1-\hat{\epsilon})f^{\prime}(\me^{\E_{P_{\theta^* , n}}h(x)}) - f^\# (\me^{\E_{P_{\theta^* ,n}} h(x)} ) \le -f^{\prime}(1-\epsilon_1) \hat{\epsilon} ,
\end{align*}
obtained by taking $\epsilon=\hat\epsilon$ and $t=\me^{ \E_{P_{\theta^* , n}}h(x) }$ in (\ref{eq:lem-logit-upper-eps}) and using
$f^\prime(1-\hat \epsilon) \ge f^\prime(1-\epsilon_1)$ for $\hat\epsilon \le \epsilon_1$.
Finally, line (\ref{eq:lem-logit-upper-prf6}) follows because $f^\# (\me^u)$ is $R_1$-Lipschitz in $u$.
\end{proof}

\begin{manualrem}{S5} \label{rem:logit-upper}
Compared with (\ref{eq:lem-logit-upper-emp}), $K_f(P_{n}, P_{\theta^*}; h)$ can also be bounded as
\begin{align}
K_f(P_{n}, P_{\theta^*}; h) \leq -f^\prime(1-\epsilon_1) \hat{\epsilon} + | \E_{P_{\theta^* ,n}} f^\# (\me^{h(x)}) - \E_{P_{\theta^* }} f^\# (\me^{h(x)}) |. \label{eq:rem-logit-upper-bound-emp}
\end{align}
In fact, this follows directly from (\ref{eq:lem-logit-upper-prf3}), because for $\hat\epsilon \le \epsilon_1$,
\begin{align*}
& \quad (1-\hat \epsilon)f^{\prime}(\me^{h(x)}) - f^\# (\me^{h(x)} )  \le  - f^{\prime}(1-\epsilon_1) \hat \epsilon ,
\end{align*}
which can be obtained by taking $\epsilon=\hat\epsilon$ and $t = \me^{h(x)}$ in (\ref{eq:lem-logit-upper-eps}), similarly as (\ref{eq:lem-logit-upper-prf2}).
However, the bound (\ref{eq:rem-logit-upper-bound-emp}) involves the moment difference of $f^\# (\me^{h(x)})$ between $P_{\theta^* }$ and $P_{\theta^* ,n}$,
which is difficult to control for $h$ in our spline class, even with $f^\# (\me^u)$ Lipschitz in $u$.
In contrast, by exploiting the concavity of $f^\prime(\me^u)$ and $-f^\#(\me^u)$ in $u$, the bound (\ref{eq:lem-logit-upper-emp}) is derived such that
it involves the moment difference of $h(x)$, which can be controlled by Lemma~\ref{lem:spline-L1-upper} in Proposition~\ref{pro:logit-L1-upper}
or by Lemma~\ref{lem:spline-L2-upper} in Proposition~\ref{pro:logit-L2-upper}.
\end{manualrem}

\begin{manualpro}{S5} \label{pro:logit-L1-upper}
In the setting of Proposition~\ref{pro:logit-L1-detailed}, it holds with probability at least $1-{5} \delta$ that for any $\gamma\in\Gamma$,
\begin{align*}
& K_f(P_{n}, P_{\theta^* }; h_{\gamma, \mu^*})
\le -f^\prime(3/5)  (\epsilon + \sqrt{\epsilon/(n\delta)} ) + \pen_1(\gamma) C_{\mathrm{sp13}}  R_1  M_{11}  \lambda_{11} ,
\end{align*}
where $C_{\mathrm{sp13}} = (5/3) ( C_{\mathrm{sp11}} \vee C_{\mathrm{sp12}})$ {with $C_{\mathrm{sp11}}$ and $C_{\mathrm{sp12}}$ as in Lemma~\ref{lem:spline-L1-upper}, }
$M_{11} =  M_1^{1/2} ( M_1^{1/2} + {2} \sqrt{2\pi})$, and
$$
\lambda_{11} = \sqrt{\frac{2\log({5 p} ) + \log(\delta^{-1})}{n}} + \frac{2\log({5 p})+ \log(\delta^{-1})}{n} .
$$
\end{manualpro}

\begin{proof}
Consider the event $\Omega_1 = \{ | \hat\epsilon-\epsilon | \le \sqrt{\epsilon(1-\epsilon)/(n\delta)}\}$.
By Chebyshev's inequality, we have $\pr (\Omega_1) \ge 1-\delta$.
In the event $\Omega_1$, we have $ |\hat\epsilon -\epsilon | \le 1/5$ by the assumption $\sqrt{\epsilon(1-\epsilon)/(n\delta)} \le 1/5$
and hence $\hat\epsilon \le 2/5$ by the assumption $\epsilon \le 1/5$.
By Lemma~\ref{lem:logit-upper} with $\epsilon_1=2/5$, it holds in the event $\Omega_1$ that for any $\gamma\in\Gamma$,
\begin{align}
& \quad K_f(P_n, P_{\theta^* } ; h_{\gamma, \mu^*}) \nonumber \\
& \le -f^\prime(3/5) \hat\epsilon + R_1 \left| \E_{P_{\theta^* ,n}} h_{\gamma, \mu^*}(x) - \E_{P_{\theta^* }} h_{\gamma, \mu^*}(x) \right| \nonumber \\
& \le -f^\prime(3/5) (\epsilon + \sqrt{\epsilon/(n\delta)} ) + R_1 \left| \E_{P_{\theta^*,n}} h_{\gamma}(x-\mu^*) - \E_{P_{(0,\Sigma^*)}} h_{\gamma}(x) \right| . \label{eq:pro-logit-L1-upper-prf1}
\end{align}
The last step uses the fact that
$ \E_{P_{\theta^* }} h_{\gamma, \mu^*}(x) = \E_{P_{(0,\Sigma^*)}} h_{\gamma}(x)$ and $\E_{P_{\theta^* ,n}} h_{\gamma, \mu^*}(x) =  \E_{P_{\theta^*,n}} h_{\gamma}(x-\mu^*)$,
by the definition $h_{\gamma, \mu^*}(x) = h_{\gamma} ( x- \mu^*)$.

Next, conditionally on the contamination indicators $(U_1,\ldots,U_n)$ such that the event $\Omega_1$ holds,
we have that $\{X_i: U_i=1, i=1,\ldots,n\}$ are $n_1$ independent and identically distributed observations from $P_{\theta^*}$,
where $n_1 = \sum_{i=1}^n (1-U_i) = n(1-\hat\epsilon) \ge (3/5)n$.
Denote as $\Omega_2$ the event that for any $\gamma_1$ and $\gamma_2$,
\begin{align*}
& \quad \left| \E_{P_{\theta^*,n}} \gamma_1^\T \varphi (x-\mu^*) - \E_{P_{(0,\Sigma^*)}} \gamma_1^\T \varphi (x) \right|
\le \|\gamma_1\|_1 C_{\mathrm{sp11}} M_1^{1/2} \sqrt{\frac{2\log(5 p) +\log(\delta^{-1})}{(3/5)n}},
\end{align*}
and
\begin{align*}
&  \quad \left| \E_{P_{\theta^*,n}} \gamma_2^\T (\varphi(x-\mu^*)\otimes \varphi(x-\mu^*)) - \E_{P_{(0,\Sigma^*)}} \gamma_2^\T (\varphi(x)\otimes \varphi(x)) \right| \\
& \le \|\gamma_2\|_1  C_{\mathrm{sp12}} M_{11} \left\{ \sqrt{\frac{2\log(5 p) +\log(\delta^{-1})}{(3/5) n}} + \frac{2\log(5 p)+ \log(\delta^{-1})}{(3/5) n}\right\},
\end{align*}
where $C_{\mathrm{sp11}}$, $C_{\mathrm{sp12}}$, and $M_{11}$ are defined as in Lemma~\ref{lem:spline-L1-upper} with {$\|\xi\|_\infty=2$}.
In the event $\Omega_2$, the preceding inequalities imply that for any $\gamma = (\gamma_0, \gamma_1^\T, \gamma_2^\T)^\T \in \Gamma$,
\begin{align}
& \quad \left| \E_{P_{\theta^*,n}} h_{\gamma}(x-\mu^*) - \E_{P_{(0,\Sigma^*)}} h_{\gamma}(x) \right| \nonumber \\
& \le \pen_1(\gamma) (5/3) ( C_{\mathrm{sp11}} \vee C_{\mathrm{sp12}}) M_{11} \lambda_{11} , \label{eq:pro-logit-L1-upper-prf2}
\end{align}
where $h_{\gamma}(x) = \gamma_0 +  \gamma_1^\T \varphi (x)+ \gamma_2^\T (\varphi(x)\otimes \varphi(x))$ and
$\pen_1 (\gamma) = \|\gamma_1\|_1 + \|\gamma_2\|_1$.
By applying Lemma \ref{lem:spline-L1-upper} with {$k=5$} to $\{X_i -\mu^*: U_i=1, i=1,\ldots,n\}$, we have
$\pr( \Omega_2 | U_1, \ldots, U_n) \ge 1-{4}\delta$ for any $(U_1,\ldots, U_n)$ such that $\Omega_1$ holds.
Taking the expectation over $(U_1,\ldots,U_n)$ given $\Omega_1$
shows that $\pr ( \Omega_2 | \Omega_1 ) \ge 1- {4}\delta$
and hence $\pr (\Omega_1 \cap \Omega_2 ) \ge (1-\delta)(1-4\delta)\ge 1-{5}\delta$.

Combining (\ref{eq:pro-logit-L1-upper-prf1}) and (\ref{eq:pro-logit-L1-upper-prf2}) in the event $\Omega_1 \cap \Omega_2$ indicates that, with probability at least $1-5\delta$,
the desired inequality holds for any $\gamma\in\Gamma$.
\end{proof}

\begin{lem} \label{lem:logit-concen-L1-lower}
Suppose that $X_1, \ldots, X_n$ are independent and identically distributed as
$X \sim P_\epsilon$.
Let $b>0$ be fixed and $g : \bbR^p \to [0,1]^{q }$ be a vector of fixed functions.
For a convex and twice differentiable function $f: (0,\infty) \to \bbR$, define
\begin{align*}
F (X_1, \dots, X_n) & = \sup_{\| w \|_1= 1, \mu \in \bbR^p} \left\{K_f (P_n, P_{\hat\theta}; b w^\T g_{\mu})  - K_f (P_{\epsilon}, P_{\hat\theta}; b w^\T g_{\mu})  \right\} \\
& = \sup_{\| w \|_1 = 1, \mu \in \bbR^p} \left\{
\frac{1}{n}\sum_{i=1}^{n} f^{\prime}(\me^{b w^\T g_{\mu} (X_i) }) - \E f^{\prime}(\me^{b w^\T g_{\mu}(X) })
\right\} ,
\end{align*}
where $g_{\mu} (x)= g (x-\mu)$.
Suppose that conditionally on $(X_1,\ldots,X_n)$, the random variable $Z_{n,j} = \sup_{\mu \in \bbR^p} |n^{-1} \sum_{i=1}^{n}\epsilon_i g_{\mu,j}(X_i) |$
is sub-gaussian with tail parameter $\sqrt{V_g/n}$ for $j=1,\ldots, q$,
where $(\epsilon_1,\ldots,\epsilon_n)$ are Rademacher variables, independent of $(X_1,\ldots,X_n)$,
and $g_{\mu,j}: \bbR^p \to [0,1]$ denotes the $j$th component of $g_{\mu}$.
Then for any $\delta > 0$, we have that
with probability at least $1-2 \delta$,
\begin{align*}
F(X_1, \dots, X_n) \le b R_{2,b} \left\{ C_{\mathrm{sg6}} \sqrt{\frac{V_g\log(2q )}{n}} + \sqrt{\frac{2 \log(\delta^{-1})}{n}} \right\},
\end{align*}
where $R_{2,b} =\sup_{ |u| \le b} \frac{\dif}{\dif u} f^\prime (\me^u)$ and $C_{\mathrm{sg6}}$ is the universal constant in Lemma~\ref{lem:subg-max}.
\end{lem}

\begin{proof} %[Proof of Lemma \ref{lem:logit-concen-L1-lower}]
First, $F$ satisfies the bounded difference condition, because $|b w^\T g_{\mu} | \le b $ with $\| w \|_1= 1$ and $f^{\prime}$ is non-decreasing by the convexity of $f$:
\begin{align*}
& \quad \sup_{X_1, \dots, X_n, X_{i}^{\prime}} \left|
F(X_1,\dots,X_n) - F(X_1,\dots, X_{i}^{\prime}, \dots, X_n) \right| \\
& \le \frac{f^{\prime}(\me^{b }) - f^{\prime}(\me^{- b })}{n} \le \frac{2 b  R_{2,b} }{n},
\end{align*}
where $R_{2,b} =\sup_{ |u| \le b } \frac{\dif}{\dif u} f^\prime (\me^u)$.
By McDiarmid's inequality (\citeappend{MC89}), for any $t > 0$, we have that with probability at least $1-2\me^{-2nt^2}$,
\begin{align*}
|F(X_1, \dots, X_n) - \E F(X_1, \dots, X_n)| \le  2 b R_{2,b} t.
\end{align*}
For any $\delta >0$, taking $t= \sqrt{ \log(\delta^{-1}) / (2n) }$ shows that with probability at least $1-2\delta$,
\begin{align*}
|F(X_1, \dots, X_n) - \E F(X_1, \dots, X_n)| \le b  R_{2,b} \sqrt{\frac{2 \log(\delta^{-1})}{n}} .
\end{align*}

Next, the expectation of $F(X_1, \dots, X_n)$ can be bounded as follows:
\begin{align}
& \quad \E \sup_{\| w \|_1 =1, \mu \in \bbR^p} \left\{
\frac{1}{n}\sum_{i=1}^{n} f^{\prime}(\me^{ b w^\T g_{\mu}( X_i)}) - \E f^{\prime}(\me^{ b w^\T g_{\mu}(X)})
\right\} \nonumber \\
& \le 2\E \sup_{\| w \|_1 =1, \mu \in \bbR^p} \left\{
\frac{1}{n}\sum_{i=1}^{n} \epsilon_i f^{\prime}(\me^{ b w^\T g_{\mu}(  X_i)})
\right\} \label{eq:lem-logit-concen-L1-lower-2} \\
& \le 2 R_{2,b} \,\E  \sup_{\| w \|_1 =1, \mu \in \bbR^p}\left\{
\frac{1}{n}\sum_{i=1}^{n} \epsilon_i  b w^\T g_{\mu}( X_i) \right\}  \label{eq:lem-logit-concen-L1-lower-3} \\
& \le 2 b R_{2,b} \,\E \sup_{ \mu \in \bbR^p}  \left\|\frac{1}{n}\sum_{i=1}^{n} \epsilon_i g_{\mu}(X_i)\right\|_\infty \nonumber \\
&\le 2 b R_{2,b} C_{\mathrm{sg6}} \sqrt{\frac{V_g\log(2q ) }{n}} . \label{eq:lem-logit-concen-L1-lower-4}
\end{align}
Line (\ref{eq:lem-logit-concen-L1-lower-2}) follows from the symmetrization Lemma~\ref{lem:symm}, where $(\epsilon_1,\ldots,\epsilon_n)$ are Rademacher variables, independent of $(X_1,\ldots,X_n)$.
Line (\ref{eq:lem-logit-concen-L1-lower-3}) follows by Lemma \ref{lem:contraction}, because $f^{\prime}(\me^u)$ is $R_{2,b}$-Lipschitz in $u \in [-b ,b]$.
Line (\ref{eq:lem-logit-concen-L1-lower-4}) follows because
\begin{align*}
&\quad \E  \sup_{\mu \in \bbR^p} \left\|\frac{1}{n}\sum_{i=1}^{n}\epsilon_i g_{\mu}(X_i)\right\|_\infty
= \E  \sup_{\mu \in \bbR^p}  \max_{j=1,\dots, q }\left| \frac{1}{n}\sum_{i=1}^{n}\epsilon_i g_{\mu,j}(X_i)  \right|\\
&= \E  \max_{j=1,\dots, q }\sup_{\mu \in \bbR^p} \left| \frac{1}{n}\sum_{i=1}^{n}\epsilon_i g_{\mu,j}(X_i)  \right| \\
& \le C_{\mathrm{sg6}} \sqrt{\frac{V_g \log(2 q )}{n}}.
\end{align*}
For the last step,
we use the assumption that conditionally on $(X_1,\ldots,X_n)$, the random variable $Z_{n,j} = \sup_{\mu \in \bbR^p} |n^{-1} \sum_{i=1}^{n}\epsilon_i g_{\mu,j}(X_i) |$ is sub-gaussian with tail parameter
$\sqrt{ V_g /n}$ for each $j=1,\ldots,q $, and apply Lemma~\ref{lem:subg-max} to obtain
$\E ( \max_{j=1,\ldots,q } Z_{n,j} | X_1,\ldots,X_n) \le C_{\mathrm{sg6}}  \sqrt{V_g \log(2q )/n}$,
and then $\E ( \max_{j=1,\ldots,q } Z_{n,j})\le C_{\mathrm{sg6}}  \sqrt{V_g \log(2q )/n}$.

Combining the tail probability and expectation bounds yields the desired result.
\end{proof}
\begin{manualrem}{S2} \label{rem:hinge-concen-lower}
    Results of Lemma~\ref{lem:logit-concen-L1-lower} and Lemma~\ref{lem:logit-concen-L2-lower} still hold with $R_{2,b}$ and $R_{2,b\sqrt{q}}$ replaced by $1$ if $K_f$ is replaced by $K_{\HG}$. This is true because $f^{\prime}(\me^{u})$ will be replaced by identity function which is just $1$-Lipschitz.
\end{manualrem}

\begin{lem} \label{lem:logit-L1-lower}
Suppose that $f: (0,\infty) \to \bbR$ is convex and three-times differentiable. Let $b >0$ be fixed. For any function $h: \bbR^p \to [-b,b]$, we have
\begin{align*}
K_f(P_{\epsilon}, P_{\hat\theta} ; h) \ge f^{\prime}(\me^{-b})\epsilon+  f^{\dprime}(1)\left\{\E_{P_{\theta^{*}}} h (x)  -  \E_{P_{\hat\theta}} h (x)\right\} - \frac{1}{2} b^2 R_{3,b} ,
\end{align*}
where $R_{3,b} =  R_{31,b} + R_{32,b} $, $R_{31,b} = \sup_{|u| \le b} \frac{\dif^2}{ \dif u^2} \{-f^\prime(\me^u) \}$,
and $R_{32,b} = \sup_{ |u| \le b} \frac{\dif^2}{ \dif u^2} f^\# (\me^u)$.
\end{lem}

\begin{proof}
First, $K_f(P_{\epsilon}, P_{\hat\theta} ; h) $ can be bounded as
\begin{align*}
& \quad  K_f(P_{\epsilon}, P_{\hat\theta} ; h) \\
& = \epsilon  \E_Q f^{\prime}(\me^{h(x)})+ (1-\epsilon)  \E_{P_{\theta^{*}}} f^{\prime}(\me^{h(x)}) -  \E_{P_{\hat\theta}}  f^\# (\me^{h(x)}) \\
& \ge f^{\prime}(\me^{-b})\epsilon +   \E_{P_{\theta^{*}}} f^{\prime}(\me^{h(x)}) - \E_{P_{\hat\theta}} f^\# (\me^{h(x)}) \\
& = f^{\prime}(\me^{-b}) \epsilon + K_f(P_{\theta^{*}},P_{\hat\theta} ;h) ,
\end{align*}
where the inequality follows because $ f^{\prime} (\me^{h(x)})  \ge f^{\prime}(\me^{-b})$ for $h(x) \in [-b,b]$ by the convexity of $f$.
Next, consider the function $\kappa (t) = K_f(P_{\theta^{*}},P_{\hat\theta} ; t h)$. A Taylor expansion of $\kappa(1) = K_f(P_{\theta^{*}},P_{\hat\theta} ; h)$ about $t=0$ yields
\begin{align*}
K_f(P_{\theta^{*}}, P_{\hat\theta};h) & = f^{\dprime}(1) \left\{ \E_{P_{\theta}} h(x)  -  \E_{P_{\hat\theta}} h(x) \right\} - \frac{1}{2} \kappa^{\dprime}(t),
\end{align*}
where for some $t \in [0,1]$,
\begin{align*}
\kappa^{\dprime} (t) & = - \E_{P_{\theta^{*}}}\left\{ h^2(x) \frac{\dif^2}{ \dif u^2} f^\prime(\me^u)|_{u=th(x)} \right\} +
\E_{P_{\hat\theta}}\left\{ h^2(x) \frac{\dif^2}{ \dif u^2} f^\# (\me^u)  |_{u=th(x)} \right\} .
\end{align*}
The desired result then follows because $h(x) \in [-b,b]$ and $t h(x) \in [-b,b]$ for $t \in [0,1]$, and hence $\kappa^{\dprime} (t) \le R_{3,b}$ by the definition of $R_{3,b}$.
\end{proof}

\begin{manualpro}{S6} \label{pro:logit-L1-lower}
Let $ b_1 >0$ be fixed. In the setting of Proposition~\ref{pro:logit-L1-detailed}, it holds with probability at least $1-2 \delta$ that for any $\gamma\in \Gamma_{\rp}$ with $\gamma_0=0$ and $\pen_1(\gamma)=  b_1 $,
\begin{align*}
& \quad  K_f(P_{n}, P_{\hat\theta}; h_{\gamma, \hat\mu}) \\
& \ge f^{\dprime}(1) \left\{ \E_{P_{\theta^*}} h_{\gamma, \hat\mu}(x) - \E_{P_{\hat\theta}} h_{\gamma, \hat\mu}(x) \right\}  +  f^{\prime}(\me^{- b_1 })\epsilon - \frac{1}{2}  b_1 ^2 R_{3, b_1 }
-  b_1  R_{2, b_1 } \lambda_{12}
\end{align*}
where, with $C_{\mathrm{rad4}}=C_{\mathrm{sg6}} C_{\mathrm{rad3}}$,
$$
\lambda_{12} = C_{\mathrm{rad4}} \sqrt{\frac{4 \log(2p(p+1))}{n}} + \sqrt{\frac{2 \log(\delta^{-1})}{n}} ,
$$
depending on the universal constants $C_{\mathrm{sg6}}$ and $ C_{\mathrm{rad3}}$ in Lemma~\ref{lem:subg-max} and Corollary~\ref{cor:entropy-sg2}.
\end{manualpro}

\begin{proof}
By definition, for any $\gamma\in \Gamma_{\rp}$, $h_\gamma(x)$ can be represented as
$h_{\rp,\beta, c }(x)$ such that $\beta_0=\gamma_0$ and $\pen_1 (\beta) = \pen_1(\gamma)$:
\begin{align*}
h_{\rp, \beta,  c }(x) = \beta_0 + \sum_{j=1}^p \beta_{1j} \, \ramp ( x_j -  c _j) + \sum_{1\le i\not= j \le p} \beta_{2,ij} \, \ramp( x_i) \ramp(x_j) ,
\end{align*}
where $ c =( c _1,\ldots, c _p)^\T$ with $ c _j \in \{0,1\}$, and $\beta= (\beta_0, \beta_1^\T, \beta_2^\T)^\T$ with
$\beta_1 =(\beta_{1j}: j=1,\ldots,p)^\T$ and $\beta_2 = (\beta_{2,ij}: 1\le i\not= j\le p)$.
Then for any $\gamma\in \Gamma_{\rp}$ with $\gamma_0=0$ and $\pen_1(\gamma)=  b_1 $, we have $\beta_0=0$ and $\pen_1(\beta) =  b_1$ correspondingly,
and hence $h _\gamma (x) = h_{\rp, \beta,  c }(x) \in [- b_1 ,  b_1 ]$ by the boundedness of the ramp function in $[0, 1]$.
Moreover, $h_{\rp,\beta, c }(x)$ with $\beta_0=0$ and $\pen_1 (\beta) =  b_1 $ can be expressed in the form
$ b_1  w^\T g(x)$, where for $q =2 p + p(p-1)$, $w \in \bbR^{q }$ is an $L_1$ unit vector, and
$g: \bbR^p \to [0,1]^{q }$ is a vector of functions including $\ramp(x_j)$ and $\ramp(x_j-1)$ for $j=1,\ldots,p$, and $\ramp(x_i)\ramp(x_j)$ for $1 \le i\not= j\le p$.
For symmetry, $\ramp(x_i)\ramp(x_j)$ and $\ramp(x_j)\ramp(x_i)$ are included as two distinct components in $g$,
and the corresponding coefficients are identical to each other in $w$.
Parenthetically, at most one of the coefficients in $w$ associated with $\ramp(x_j)$ and $\ramp(x_j-1)$ is nonzero for each $j$,
but this property is not used in the subsequent discussion.

Next, $K_f(P_{n}, P_{\hat\theta}; h_{\gamma, \hat\mu})$ can be bounded as
\begin{align*}
& \quad  K_f(P_{n}, P_{\hat\theta}; h_{\gamma, \hat\mu}) \\
& \ge  K_f(P_\epsilon, P_{\hat\theta}; h_{\gamma, \hat\mu})  -  \{K_f(P_{n}, P_{\hat\theta}; h_{\gamma, \hat\mu}) - K_f(P_\epsilon, P_{\hat\theta}; h_{\gamma, \hat\mu})\} .
\end{align*}
For any $\gamma\in \Gamma_{\rp}$ with $\gamma_0=0$ and $\pen_1(\gamma)=  b_1 $, applying Lemma~\ref{lem:logit-L1-lower} with $h= h_{\gamma, \hat\mu}$ and $b= b_1 $ yields
\begin{align*}
& \quad   K_f(P_\epsilon, P_{\hat\theta}; h_{\gamma, \hat\mu}) \\
& \ge f^{\dprime}(1) \left\{ \E_{P_{\theta^*}} h_{\gamma, \hat\mu}(x) - \E_{P_{\hat\theta}} h_{\gamma, \hat\mu}(x) \right\}  +  f^{\prime}(\me^{- b_1 })\epsilon - \frac{1}{2}  b_1 ^2 R_{3, b_1 }.
\end{align*}
By Lemma~\ref{lem:logit-concen-L1-lower} with $b= b_1 $ and $g(x)\in [0,1]^{q }$ defined above, it holds with probability at least $1-2\delta$ that
for any $\gamma\in \Gamma_{\rp}$ with $\gamma_0=0$ and $\pen_1(\gamma)=  b_1 $,
\begin{align*}
& \quad \{ K_f(P_{n}, P_{\hat\theta}; h_{\gamma, \hat\mu}) - K_f(P_\epsilon, P_{\hat\theta}; h_{\gamma, \hat\mu})\} \\
& \le  b_1  R_{2, b_1 } \left\{  C_{\mathrm{sg6}} \sqrt{\frac{V_g \log(2q )}{n}} + \sqrt{\frac{2 \log(\delta^{-1})}{n}} \right\} \\
& =  b_1  R_{2, b_1 } \left\{  C_{\mathrm{sg6}} C_{\mathrm{rad3}}\sqrt{\frac{4 \log(2p(p+1))}{n}} + \sqrt{\frac{2 \log(\delta^{-1})}{n}} \right\} ,
\end{align*}
where $V_g  = 4 C_{\mathrm{rad3}}^2$ is determined in Lemma~\ref{lem:logit-concen-L1-lower} as follows.
For $j=1,\ldots, q$,
consider the function class $\mathcal{G}_j = \{g_{\mu,j}: \mu \in \bbR^p\}$,
where $\mu = (\mu_1,\ldots,\mu_p)^\T$ and, as defined in Lemma~\ref{lem:logit-concen-L1-lower}, $g_{\mu,j}(x)$ is either a moving-knot ramp function, $\ramp(x_{j_1} - \mu_{j_1})$ or $\ramp(x_{j_1}- \mu_{j_1}-1)$
or a product of moving-knot ramp functions, $\ramp(x_{j_1} - \mu_{j_1}) \ramp(x_{j_2} - \mu_{j_2})$ for $1 \le j_1 \not= j_2 \le p$.
By Lemma~\ref{lem:ramp-VC}, the VC index of moving-knot ramp functions is 2.
By applying Corollary~\ref{cor:entropy-sg2} (i) and (ii) with vanishing $\mathcal H$,
we obtain that conditionally on $(X_1,\ldots,X_n)$, the random variable
$Z_{n,j} = \sup_{\mu \in \bbR^p}  |n^{-1} \sum_{i=1}^{n}\epsilon_i g_{\mu,j}(X_i) | = \sup_{f_j \in\mathcal{G}_j}  |n^{-1} \sum_{i=1}^{n}\epsilon_i f_j (X_i) |$
is sub-gaussian with tail parameter $C_{\mathrm{rad3}} \sqrt{4/n}$ for $j=1,\ldots,q$.

Combining the preceding three displays leads to the desired result.
\end{proof}

\begin{lem}[Local linearity 1] \label{lem:local-linear1}
For $\delta \in \bbR $ and $0 \le \sigma_1, \sigma_2\le M^{1/2}$, denote
$D_h = \E h(\sigma_1 Z + \delta) - \E h(\sigma_2 Z) $, where $h$ is a function on $\bbR$ and $Z$ is a standard Gaussian random variable.
For $h_1(x) = \pm \ramp(x)$, if $|D_{h_1}| \le a$ for $a \in (0,1/2)$, then we have
\begin{align}
& |\delta| \le S_{4,a} | D_{h_1} | , \label{eq:lem-local-linear-1}
\end{align}
where $S_{4,a} =  (1 +\sqrt{2 M\log\frac{2}{1-2a}})  /a $.
For $h_2(x) = \pm \ramp(x-1)$, we have
\begin{align}
& |\sigma_1 - \sigma_2| \le S_{5} (| D_{h_2} | + |\delta|/2 ), \label{eq:lem-local-linear-2}
\end{align}
where $S_{5} = 2\sqrt{2\pi} ( 1- \me^{-2/M})^{-1} $.
\end{lem}

\begin{manualrem}{S6} \label{rem:local-linear1}
Define a ramp function class
$$
\mathcal{R}_1 = \left\{ \pm \ramp ( x- c ),  x\in\bbR:   c=0,1\right\}.
$$
%where $\ramp(x) = \frac{1}{2}(x+1)_{+}-\frac{1}{2}(x-1)_{+}$.
In the setting of Lemma~\ref{lem:local-linear1}, suppose that for fixed $a \in (0,1/2)$,
\begin{align*}
D \stackrel{\text{def}}{=} \sup_{h \in \mathcal{R}_1} \left\{\E h(\sigma_1 Z + \delta) - \E h(\sigma_2 Z) \right\} \le a .
\end{align*}
Then we have
\begin{align*}
|\delta| \le S_{4,a} D , \quad  |\sigma_1 - \sigma_2| \le S_{6,a} D , %\label{eq:lem-local-linear-3}
\end{align*}
where $S_{6,a} = S_{5} ( 1 + S_{4,a}/2)$.
This shows that the moment matching discrepancy $D$ over $\mathcal{R}_1$ delivers upper bounds, up to scaling constants, on the mean and standard deviation differences,
provided that $D$ is sufficiently small, for example, $D \le 1/3$.
\end{manualrem}

\begin{proof}

[Proof of (\ref{eq:lem-local-linear-1})]
First, assume that $\delta$ is nonnegative.
The other direction will be discussed later.
Take $h(x) = \ramp(x)$. Then $ h(x) + h(-x) = 1$ for all $x \in \mathbb{R}$ and
\begin{align*}
\E h(\sigma_2 Z) = \E h(\sigma_1 Z) = \frac{1}{2}. %\label{eq:ramp-expectation}
\end{align*}
Define $g(t) = \E h(\sigma_1 Z + t) -\E h(\sigma_1 Z) = \E h(\sigma_1 Z +t) -\frac{1}{2}$ for $t\in \bbR$.
Then $g(0)=0$ and $g(\delta) = D_h \le a$.
We notice the following properties for the function $g$.
\begin{itemize}
\item[(i)] $g(t)$ is non-decreasing and concave for $t \ge 0$.
\item[(ii)] $g(t)/t$ is non-increasing for $t >0$.
\item[(iii)] $g(t) \ge \frac{1}{2} - \exp\{ -(t-1)^2 / (2 \sigma_1^2)\}$ for $t\ge 1$.
\end{itemize}
For property (i), the derivative of $g(t)$ is $g^\prime(t) =  \frac{1}{2}\E \{\mathbbm{1}_{\sigma_1 Z +t \in [-1,1]} \} \ge 0$.
Moreover, $g^\prime(t)$ is non-increasing: for $0\le t_1 < t_2$,
\begin{align*}
& \quad g^\prime(t_1) = \frac{1}{2} \pr (  -1-t_1 \le \sigma_1 Z \le 1-t_1  ) \\
& = \frac{1}{2} \pr (  -1-t_1 \le \sigma_1 Z \le 1-t_2 ) + \frac{1}{2} \pr (  1-t_2 \le \sigma_1 Z \le 1-t_1 ) ,\\
& \quad g^\prime(t_2) = \frac{1}{2} \pr (  -1-t_2 \le \sigma_1 Z \le 1-t_2 ) \\
& = \frac{1}{2} \pr (  -1-t_1 \le \sigma_1 Z \le 1-t_2 ) + \frac{1}{2} \pr (  -1-t_2 \le \sigma_1 Z \le -1-t_1 ) ,
\end{align*}
and
\begin{align*}
& \quad \pr (  1-t_2 \le \sigma_1 Z \le 1-t_1 )
= \pr (  t_1-1 \le \sigma_1 Z \le t_2-1 ) \\
& \le \pr (  t_1 + 1 \le \sigma_1 Z \le t_2 +1 )
= \pr (  -1-t_2 \le \sigma_1 Z \le -1-t_1 ) .
\end{align*}
The last inequality holds because $(t -1)^2 \le (t +1)^2$  for any $t\ge 0$ and hence $\int_{t_1}^{t_2} \exp\{ -(t-1)^2/(2\sigma_1^2)\}\,\dif t \ge  \int_{t_1}^{t_2} \exp\{ -(t+1)^2/(2\sigma_1^2)\}\,\dif t$.
To show (ii), we write $g(t)=g(t)-g(0) =t \int_0^1 g^\prime (t z)\,\dif z$.
Then $g(t)/t = \int_0^1 g^\prime (t z)\,\dif z$ is non-increasing in $t$ because $g^\prime$ is non-increasing.
To show (iii), we notice that $h(x) \ge \mathbbm{1}_{x > 1}$ and hence for $t \ge 1$,
\begin{align*}
& \quad g(t) + \frac{1}{2} = \E h(\sigma_1 Z+t)  \\
& \ge 1 - P ( \sigma_1 Z + t \le 1 ) = 1 - \pr ( \sigma_1 Z \ge t-1 )\\
& \ge 1- \exp\{ -(t-1)^2 / (2 \sigma_1^2)\}.
\end{align*}
The last inequality follows by the Gaussian tail bound: {$\pr (Z \ge z) \le \me^{- z^2/2}$} for $z >0$.

By the preceding properties, we show that $\delta \le S_{4,a} D_h$ and hence (\ref{eq:lem-local-linear-1}) holds.
Without loss of generality, assume that $\delta\not= 0$.
For $a\in(0, 1/2)$, let $t_a >0$ be determined such that $g(t_a) =a$. Then $\delta \le t_a$ and, by property (ii), $g(\delta)/ \delta \ge a / t_a$.
If $t_a \ge 1$, then, by property (iii), $ a= g(t_a) \ge \frac{1}{2} - \exp\{ -(t_a-1)^2 / (2 \sigma_1^2)\}$,
and hence
$$ t_a \le 1 + \sqrt{2} \sigma_1 \sqrt{\log\frac{2}{1-2a}} .$$
This inequality remains valid if $t_a <1$. Therefore, $\delta \le g(\delta) t_a /a \le g(\delta) S_{4,a} = D_h S_{4,a}$,
by the assumption $\sigma_1 \le M^{1/2}$
and the definition of $S_{4,a}$.

When $\delta$ is negative, a similar argument taking $h(x)= -\ramp(x)$, which is the same as $\ramp(-x)-1$, shows that $-\delta \le S_{4,a} D_h$
and hence (\ref{eq:lem-local-linear-1}) holds.

[Proof of (\ref{eq:lem-local-linear-2})]
First, assume that $\sigma_1 - \sigma_2$ is nonnegative. Take $h(x) = \ramp(x-1)$.
Notice that $h(x)$ is $(1/2)$-Lipschitz and hence $|h(x+\delta) - h(x)| \leq |\delta|/2$. Then by the triangle inequality, we have
\begin{align*}
\E h(\sigma_1 Z ) - \E h(\sigma_2 Z) \leq D_h + |\delta| /2.
\end{align*}
Define $g(t) = \E h(t Z) - \E h(\sigma_2 Z)$ for $t\ge 0$. Then $g(\sigma_2)=0$ and $g(\sigma_1) \le D_h + |\delta| /2$.
The derivative $g^\prime (t) =  \frac{1}{2} \E \{Z\mathbbm{1}_{t Z \in [0,2]}\} $ can be
calculated as
\begin{align*}
g^\prime (t) = \frac{1}{2} \int_0^{2/t} \frac{1}{\sqrt{2\pi}} z \me^{-z^2/2} \,\dif z = \frac{1}{2\sqrt{2\pi}} \left\{ 1- \me^{-\frac{(2/t)^2}{2}} \right\}.
\end{align*}
By the mean value theorem, $g(\sigma_1) = g(\sigma_1)-g(\sigma_2) = g^\prime (t) (\sigma_1 - \sigma_2) \ge g^\prime(M^{1/2}) (\sigma_1-\sigma_2)$,
where $t \in [\sigma_2,\sigma_1]$ and hence $t\le M^{1/2}$ because $\sigma_1\le M^{1/2}$. Therefore, we have
\begin{align*}
\sigma_1 - \sigma_2 \le g^\prime(M^{1/2}) ^{-1} g(\sigma_1) \le S_{5} (D_h+|\delta|/2),
\end{align*}
by the definition $S_{5} = g^\prime(M^{1/2}) ^{-1} =2\sqrt{2\pi} ( 1- \me^{-2/M})^{-1} $.

When $\sigma_1-\sigma_2$ is negative, a similar argument taking $h(x) = -\ramp(x-1)$ shows that $\sigma_2-\sigma_1\le S_{5} (D_h + |\delta|/2)$
and hence (\ref{eq:lem-local-linear-2}) holds.
\end{proof}

\begin{lem}[Local linearity 2] \label{lem:local-linear2}
For $\delta_1,\delta_2\in\bbR$, $0\le \sigma_1, \sigma_2, \tilde\sigma_1, \tilde\sigma_2 \le M^{1/2}$, and $\rho,\tilde\rho \in [-1,1]$,
denote $D_h =  \E h(\tilde X) - \E h( X) $,
where $h$ is a function on $\bbR^2$, $X=(X_1, X_2)^\T$ is a Gaussian random vector in $\bbR^2$ with mean 0 and variance matrix
$ \begin{pmatrix}
\sigma_1^2 & \sigma_1\sigma_2 \rho \\
\sigma_1\sigma_2\rho & \sigma_2^2
\end{pmatrix}$,
and $\tilde X = (\tilde X_1, \tilde X_2)^\T$ is a Gaussian random vector in $\bbR^2$ with mean $\delta=(\delta_1,\delta_2)^\T$ and variance matrix
$ \begin{pmatrix}
\tilde\sigma_1^2 & \tilde\sigma_1\tilde\sigma_2 \tilde\rho \\
\tilde\sigma_1\tilde\sigma_2\tilde\rho & \tilde\sigma_2^2
\end{pmatrix}$.
For $h(x) = \pm \ramp(x_1) \ramp(x_2) $, we have
\begin{align*}
|\tilde\rho \tilde\sigma_1 \tilde\sigma_2 -\rho \sigma_1\sigma_2 | \le  M^{1/2} \| \tilde\sigma-\sigma\|_1 +  S_{7} ( |D_h| + \Delta /2)  ,
\end{align*}

where $S_{7} =4 \{ (\frac{1}{\sqrt{2\pi M}} \me^{- 1/(8M)} ) \vee (1 - 2 \me^{- 1/(8M)} ) \}^{-2} $,
which behaves like to $4 (1 - 2 \me^{- 1/(8M)} )^{-2}$ as $M \to 0$
or $8\pi M \me^{1/(4M)}$ as $M \to \infty$,
and $\Delta = \|\delta\|_1 + \|\tilde\sigma - \sigma \|_1$,
with $\sigma=(\sigma_1,\sigma_2)^\T$ and $\tilde\sigma=(\tilde\sigma_1,\tilde\sigma_2)^\T$.
\end{lem}

\begin{manualrem}{S4} \label{rem:local-linear2}
Define a ramp main-effect and interaction class
\begin{align*}
\mathcal{R}_2 & = \left\{ \pm \ramp ( x_j- c ),  (x_1,x_2)\in\bbR^2 : j=1,2,  \, c =0,1\right\} \\
& \quad \bigcup \left\{ \pm \ramp(x_1)\ramp(x_2),  (x_1,x_2)\in\bbR^2 \right\}.
\end{align*}
In the setting of Lemma~\ref{lem:local-linear2}, suppose that for fixed $a \in (0,1/2)$,
\begin{align*}
D \stackrel{\text{def}}{=} \sup_{h \in \mathcal{R}_2} \left\{ \E h(\tilde X) - \E h( X)  \right\} \le a .
\end{align*}
Then combining Lemma~\ref{lem:local-linear1}--\ref{lem:local-linear2} yields
\begin{align*}
& \max( |\delta_1|, |\delta_2| ) \le S_{4,a} D , \quad
\max( |\tilde\sigma_1 - \sigma_1|,|\tilde\sigma_2 - \sigma_2|)  \le S_{6,a} D ,  \\
& \max( |\tilde\sigma_1^2 - \sigma_1^2|,|\tilde\sigma_2^2 - \sigma_2^2|, |\tilde\rho \tilde\sigma_1 \tilde\sigma_2 -\rho \sigma_1\sigma_2 | )
\le S_{8,a} D
\end{align*}
where $S_{8,a} =S_{7} ( 1+ S_{4,a} + S_{6,a}) + 2 M^{1/2} S_{6,a} $.
This shows that the moment matching discrepancy $D$ over $\mathcal{R}_2$ delivers upper bounds, up to scaling constants, on the mean, variance, and covariance differences,
provided that $D$ is sufficiently small.
\end{manualrem}

\begin{proof}
First, we handle the effect of different means and standard deviations between $\tilde X$ and $X$ in $D_h$.
Denote $D^\dag_h =  \E h(\tilde Y) - \E h( X)$,
where $\tilde Y = (\tilde Y_1,\tilde Y_2)^\T$ is a Gaussian random vector
with mean 0 and variance matrix
$  \begin{pmatrix}
\sigma_1^2 & \sigma_1\sigma_2 \tilde \rho \\
\sigma_1\sigma_2\tilde\rho & \sigma_2^2
\end{pmatrix}$.
Then we have
\begin{align}
|D^\dag_h| \le |D_h| + \Delta /2 . \label{eq:matching-rho-Ddag}
\end{align}
In fact, assume that $\tilde Y = ( \sigma_1 Z_1, \sigma_2 Z_2)^\T$ and
$\tilde X =\delta+ ( \tilde\sigma_1 Z_1, \tilde\sigma_2 Z_2)^\T$, where $(Z_1,Z_2)^\T$ is a Gaussian random vector
with mean 0 and variance matrix
$\begin{pmatrix}
1 & \tilde\rho \\
\tilde \rho & 1
\end{pmatrix}$.
For $h(x) = \pm \ramp(x_1) \ramp(x_2) $, we have
\begin{align*}
&\quad | \E h(\tilde Y) - \E h( X)| \\
& \le | \E h (\tilde X) - \E h(X)| + \sum_{j=1,2} \E \, |\ramp(\tilde Y_j) -\ramp (\tilde X_j)| \\
& \le | \E h (\tilde X) - \E h(X)| + \sum_{j=1,2} \{ \delta_j^2 + (\tilde\sigma_j-\sigma_j)^2 \}^{1/2} /2 \\
& \le | \E h (\tilde X) - \E h(X)| + \Delta /2.
\end{align*}
The first inequality follows by the triangle inequality and the fact that $\ramp(\cdot)$ is bounded in $[0,1]$.
{The second step uses $\E\, |\ramp(\tilde Y_j) -\ramp (\tilde X_j)|
\le (1/2) \E | \tilde Y_j - \tilde X_j|
\le  (1/2) \E^{1/2} [ \{ \delta_j + (\tilde\sigma_j-\sigma_j) Z_j \}^2 ]
= \{ \delta_j^2 + (\tilde\sigma_j-\sigma_j)^2 \}^{1/2} /2$,}
by the fact that $\ramp(\cdot)$ is $(1/2)$-Lipschitz, $\E Z_j =0$, and $\E (Z_j^2)=1$.
The third inequality follows because $\sqrt{u_1+u_2} \le \sqrt{u_1} + \sqrt{u_2}$.

Next, we show that
\begin{align}
|\tilde\rho-\rho| \sigma_1\sigma_2  \le 8\pi M  \me^{1/(4M)} |D^\dag_h| . \label{eq:matching-rho-bd}
\end{align}
Assume that $\tilde \rho - \rho$ is nonnegative. The other direction will be discussed later. Now take $h(x) = \ramp(x_1) \ramp(x_2) $, and define
$g(t) = \E h(Y) - \E h( X) $ for $t \in[-1,1]$, where $Y = (Y_1,Y_2)^\T$ is a Gaussian random vector with mean 0 and variance matrix
$\begin{pmatrix}
\sigma_1^2 & \sigma_1\sigma_2 t \\
\sigma_1\sigma_2 t & \sigma_2^2
\end{pmatrix}$.
Then $g(\rho)=0$ and $g( \tilde \rho) = D^\dag_h$. The derivative $g^\prime(t)$ can be calculated as
\begin{align}
g^\prime(t) = \frac{1}{4} \sigma_1\sigma_2  \E \mathbbm{1}_{ Y_1 \in [-1,1]} \mathbbm{1}_{Y_2 \in [-1,1]} . \label{eq:ramp-prod-derivative}
\end{align}
In fact, $Y_1$ can be represented as $Y_1= t \frac{\sigma_1}{\sigma_2} Y_2 + \sqrt{1-t^2} \sigma_1 Z_1$, where $Z_1$ is a standard Gaussian variable independent of $Y_2$.
Then direct calculation yields
\begin{align*}
g^\prime(t)& = \frac{1}{2} \E \mathbbm{1}_{ Y_1 \in [-1,1]} \left( \frac{\sigma_1}{\sigma_2} Y_2 -\frac{t}{ \sqrt{1-t^2}} \sigma_1 Z_1 \right) \ramp(Y_2)  \\
& = \frac{1}{2} \E \mathbbm{1}_{ Y_1 \in [-1,1]} \left\{ \frac{\sigma_1}{\sigma_2} Y_2 -\frac{t}{1-t^2} ( Y_1- t \frac{\sigma_1}{\sigma_2} Y_2 ) \right\} \ramp(Y_2)  \\
& = \frac{1}{2} \E \mathbbm{1}_{ Y_1 \in [-1,1]} \left( \frac{1}{1-t^2} \frac{\sigma_1}{\sigma_2} Y_2 -\frac{t}{1-t^2} Y_1  \right) \ramp(Y_2) \\
& = \frac{1}{2} \frac{1}{1-t^2} \frac{\sigma_1}{\sigma_2}   \E \mathbbm{1}_{ Y_1 \in [-1,1]} \left( Y_2 - t \frac{\sigma_2}{\sigma_1} Y_1  \right) \ramp(Y_2).
\end{align*}
By Stein's lemma using the fact that $Y_2$ given $Y_1$ is Gaussian with mean $t \frac{\sigma_2}{\sigma_1} Y_1$ and variance $(1-t^2) \sigma_2^2$,  we have
\begin{align*}
\E \left\{ \left( Y_2 - t \frac{\sigma_2}{\sigma_1} Y_1  \right) \ramp(Y_2) \Big| Y_1 \right\}
= \E \left\{ (1-t^2) \sigma_2^2 \frac{1}{2} \mathbbm{1}_{Y_2 \in [-1,1]} \Big| Y_1 \right\} .
\end{align*}
Substituting this into the expression for $g^\prime(t)$ gives the formula (\ref{eq:ramp-prod-derivative}).

To show (\ref{eq:matching-rho-bd}), we derive a lower bound on $g^\prime(t)$.
Assume without loss of generality that $\sigma_1 \le \sigma_2$, because formula (\ref{eq:ramp-prod-derivative}) is symmetric in $Y_1$ and $Y_2$.
By the previous representation of $Y_1$, we have $|Y_1| \le \sqrt{1-t^2} \sigma_1 |Z_1|  + \frac{\sigma_1}{\sigma_2} |Y_2| \le  \sigma_1 |Z_1| + |Y_2|$ and hence
\begin{align*}
g ^\prime(t) % & = \frac{1}{4} \sigma_1\sigma_2  \E \mathbbm{1}_{ Y_1 \in [-1,1]} \mathbbm{1}_{Y_2 \in [-1,1]} \\
& \ge \frac{1}{4} \sigma_1\sigma_2  \E \mathbbm{1}_{ Y_1 \in [-1,1]} \mathbbm{1}_{Y_2 \in [-1/2,1/2]} \\
& \ge \frac{1}{4} \sigma_1\sigma_2  \E \mathbbm{1}_{ \sigma_1 |Z_1| \in [-1/2,1/2]} \mathbbm{1}_{Y_2 \in [-1/2,1/2]} \\
& \ge S_{7}^{-1} \sigma_1\sigma_2 .
\end{align*}

where $S_{7} = 4 \{ (\frac{1}{\sqrt{2\pi} M^{1/2}} \me^{- 1/(8M)} ) \vee (1 - 2 \me^{- 1/(8M)} ) \}^{-2} $.
The last inequality follows by the independence of $Z_1$ and $Y_2$, $\sigma_1\le M^{1/2}$, $\sigma_2 \le M^{1/2}$, and
the two probability bounds: $\pr (M^{1/2} |Z_1| \le 1/2) \ge \frac{1}{\sqrt{2\pi M} } \me^{- (1/2)^2/(2M)}$
and $\pr (M^{1/2} |Z_1| \le 1/2) \ge 1 - 2 \me^{- (1/2)^2/(2M) }$.
Hence by the mean value theorem, we have
\begin{align*}
g(\tilde\rho) = g(\tilde\rho)-g(\rho) \ge (\tilde\rho-\rho) \sigma_1\sigma_2 S_{7}^{-1} ,
\end{align*}
which gives the desired bound (\ref{eq:matching-rho-bd}) in the case of $\tilde\rho \ge \rho$:
\begin{align*}
(\tilde\rho-\rho) \sigma_1\sigma_2   \le S_{7} g(\tilde \rho) = S_{7} D^\dag_h.
\end{align*}

When $\tilde\rho-\rho$ is negative, a similar argument taking $h(x) = -\ramp(x_1) \ramp(x_2) $
and interchanging the roles of $\tilde Y$ and $X$
leads to (\ref{eq:matching-rho-bd}).

Finally, by the triangle inequality, we find
\begin{align*}
| \tilde \rho \tilde\sigma_1 \tilde\sigma_2 - \rho \sigma_1 \sigma_2 |
& \le |\tilde \sigma_1 \tilde \sigma_2 - \sigma_1 \sigma_2 | + |\tilde \rho-\rho| \sigma_1 \sigma_2 \\
& \le M^{1/2} ( |\tilde\sigma_1 -\sigma_1| + |\tilde \sigma_2 - \sigma_2 | ) +  |\tilde \rho-\rho| \sigma_1 \sigma_2 .
\end{align*}
Combining this with (\ref{eq:matching-rho-Ddag}) and (\ref{eq:matching-rho-bd}) leads to the desired result.
\end{proof}

\begin{manualpro}{S7} \label{pro:logit-L1-combine}
In the setting of Proposition~\ref{pro:logit-L1-detailed} or Proposition~\ref{pro:hinge-L1-detailed}, suppose that for $a \in (0,1/2)$,
\begin{align}
D \stackrel{\text{def}}{=} \sup_{\gamma \in \Gamma_{\rp}, \pen_1(\gamma)=1} \left\{  \E_{P_{\theta^*}} h_{\gamma, \hat\mu}(x) - \E_{P_{\hat\theta}} h_{\gamma, \hat\mu}(x) \right\} \le a. \label{eq:pro-logit-L1-combine-cond}
\end{align}
Then we have
\begin{align*}
& \| \hat \mu- \mu^* \|_\infty \le S_{4,a} D,\\
& \| \hat \Sigma - \Sigma^* \|_{\max} \le S_{8,a} D ,
\end{align*}
where $S_{4,a}$ and $S_{8,a}$ are defined as in Lemma~\ref{lem:local-linear1} and Remark \ref{rem:local-linear2} {with $M = M_1$}.
\end{manualpro}

\begin{proof}
For any $\gamma \in \Gamma_{\rp}$ with $\pen_1(\gamma) =1$, the function $h_{\gamma,\hat\mu}(x) = h_\gamma ( x- \hat\mu) $ can be expressed as $h_{\rp, \beta,  c }(x - \hat\mu)$ with $\pen_1 (\beta) = 1$.
For $j=1,\ldots,p$, by restricting $h_{\gamma,\hat\mu}(x)$ in (\ref{eq:pro-logit-L1-combine-cond}) to those with $h_\gamma(x)$ defined as
ramp functions of $x_j$ and using Remark~\ref{rem:local-linear1}, we obtain
\begin{align*}
| \hat \mu_j - \mu^*_j | \le S_{4,a} D, \quad  | \hat \sigma_j - \sigma^*_j | \le S_{6,a} D.
\end{align*}
For $1 \le i \not= j \le p$, by restricting $h_{\gamma,\hat\mu}(x)$ in (\ref{eq:pro-logit-L1-combine-cond}) to those with $h_\gamma(x)$ defined as
ramp interaction functions of $(x_i,x_j)$ and using Remark~\ref{rem:local-linear2}, we obtain
\begin{align*}
\max \left( |\hat\Sigma_{ii} - \Sigma_{ii}^* |, |\hat \Sigma_{jj} - \Sigma_{jj}^*|, | \hat \Sigma_{ij} - \Sigma_{ij}^* | \right) \le S_{8,a} D ,
\end{align*}
where $\hat \Sigma_{ij}$ and $\Sigma_{ij}^*$ are the $(i,j)$th elements of $\hat\Sigma$ and $\Sigma^*$ respectively. Combining the preceding two displays leads to the desired result.
\end{proof}

%%% -----------------------------------------------------------------------------------------------------------------------------------------  details, spline, L2, location/scale and var matrix

\subsection{Details in main proof of Theorem \ref{thm:logit-L2}}

\begin{lem} \label{lem:spline-L2-upper}
Suppose that $X_1, \ldots, X_n$ are independent and identically distributed as
$X \sim N_p(0, \Sigma)$ with $\|\Sigma\|_{\op} \leq M_2$.
For $k$ fixed knots $\xi_1,\ldots,\xi_k$ in $\bbR$,
denote $\varphi(x) = (\varphi^\T_1(x), \ldots, \varphi^\T_k(x))^\T$,
where $\varphi_l(x) \in \bbR^p$ is obtained by applying $t \mapsto (t-\xi_l)_+$ componentwise to $x \in \bbR^p$ for {$l=1,\ldots,k$.}
Then the following results hold.

(i) $\varphi(X)-\E\varphi(X)$ is a sub-gaussian random vector with tail parameter $(k M_2)^{1/2}$.

(ii) For any $\delta >0$, we have that with probability at least $1-\delta$,
\begin{align*}
& \quad \sup_{\| w \|_2 = 1} \left| \frac{1}{n} \sum_{i=1}^{n} w^\T \varphi(X_{i}) - \E w^\T \varphi(X) \right| \\
& \le  C_{\mathrm{sp21}} (k M_2)^{1/2} \sqrt{\frac{kp +\log(\delta^{-1})}{n}} ,
\end{align*}
where $C_{\mathrm{sp21}} = \sqrt{2} C_{\mathrm{sg7}}  C_{\mathrm{sg5}}$ and $(C_{\mathrm{sg5}}, C_{\mathrm{sg7}})$ are the universal constants in Lemmas~\ref{lem:subg-concentration} and \ref{lem:subg-vec-norm}.

(iii) Let $\mathcal{A}_2=\{A \in \bbR^{kp \times kp}: \|A\|_{\fro}= 1, A^\T=A\}$. For any $\delta >0$, we have that with probability at least $1-3\delta$,
both the inequality in (ii) and
\begin{align*}
& \quad \frac{1}{\sqrt{kp}}\sup_{A \in \mathcal{A}_2} \left| \frac{1}{n} \sum_{i=1}^{n} \varphi^\T(X_{i}) A \varphi(X_{i}) - \E \varphi^\T(X) A \varphi(X) \right| \\
& \le  C_{\mathrm{sp22}} k M_{21} \left\{ \sqrt{\frac{kp + \log(\delta^{-1})}{n}} + \frac{kp+ \log(\delta^{-1})}{n}\right\},
\end{align*}
where $M_{21} = M_2^{1/2} (M_2^{1/2} +  \sqrt{2\pi}  \|\xi \|_\infty)$, $ \| \xi \|_\infty = \max_{{l=1,\ldots,k}} |\xi_l|$,
$C_{\mathrm{sp22}} = \sqrt{2/\pi} C_{\mathrm{sp21}} + C_{\mathrm{sg8}}$, and $C_{\mathrm{sg8}}$ is the universal constant in Lemma~\ref{lem:subg-vec-var}.
\end{lem}

\begin{proof}
(i) First, we show that $w^\T \varphi(x)$ is a $k^{1/2}$-Lipschitz function for any $L_2$ unit vector $w \in \bbR^{kp}$.
For any $x_1, x_2 \in \bbR^p$, we have
$ | w^\T ( \varphi(x_1) - \varphi(x_2)) | \le \sum_{l=1}^k \|w_l\|_2 \|\varphi_l(x_1) - \varphi_l(x_2)\|_2
\le \sum_{l=1}^k \|w_l\|_2 \|x_1 - x_2\| \le k^{1/2} \|x_1-x_2\| $,
where $w$ is partitioned as $w = (w_1^\T,\ldots, w_k^\T)^\T $,
and {$\|\varphi_l(x_1) - \varphi_l(x_2)\|_2 \le \|x_1-x_2\|_2 $} because each component of $\varphi_l(x)$ is 1-Lipschitz, as a function of only the corresponding component of $x$.
Next, $X$ can be represented as $\Sigma^{1/2}Z$, where $Z$ is a standard Gaussian random vector.
For any $z_1, z_2 \in \bbR^p$ and $L_2$ unit vector $w \in \bbR^{kp}$, we have
\begin{align*}
&\quad | w^\T\varphi (\Sigma^{1/2}z_1) - w^\T\varphi (\Sigma^{1/2}z_2) |\\
&\le k^{1/2} \|\Sigma^{1/2} (z_1 -z_2) \|_2
\le k^{1/2} \|\Sigma^{1/2}\|_{\op} \|z_1 - z_2\|_2
\le (k M_2)^{1/2}\|z_1 -z_2\|_2 .
\end{align*}
Hence $w^\T\varphi(X)$ is a $(k M_2)^{1/2}$-Lipschitz function of the standard Gaussian vector $Z$.
By Theorem 5.6 in \citeappend{BLM13}, the centered version satisfies that for any $t>0$,
\begin{align*}
\pr \left( | w^\T (\varphi(X)-\E\varphi(X)) | > t \right) \le 2 \me^{- t^2/(2k M_2)}.
\end{align*}
That is, $w^\T (\varphi(X)-\E\varphi(X))$ is sub-gaussian with tail parameter $(k M_2)^{1/2}$.
The desired result follows by the definition of sub-gaussian random vectors.

(ii) As shown above, $w^\T (\varphi(X)-\E\varphi(X))$ is sub-gaussian with tail parameter $(k M_2)^{1/2}$ for any $L_2$ unit vector $w$. Then
$w^\T \{ n^{-1} \sum_{i=1}^n \varphi(X_i) -\E \varphi(X) \}$ is sub-gaussian with tail parameter $C_{\mathrm{sg5}} (k M_2/n)^{1/2}$ by sub-gaussian concentration (Lemma~\ref{lem:subg-concentration}).
Hence by definition, we have that $n^{-1} \sum_{i=1}^n \varphi(X_i) -\E \varphi(X) $ is a sub-gaussian random vector with tail parameter $C_{\mathrm{sg5}} (k M_2/n)^{1/2}$.
Notice that
\begin{align*}
\sup_{\| w \|_2 = 1} \left| \frac{1}{n} \sum_{i=1}^{n} w^\T \varphi(X_{i}) - \E w^\T \varphi(X) \right|  =
\left\| \frac{1}{n} \sum_{i=1}^n \varphi(X_i) -\E \varphi(X)  \right\|_2.
\end{align*}
The desired result follows from Lemma~\ref{lem:subg-vec-norm}: with probability at least $1-\delta$, we have
\begin{align*}
\quad \left\| \frac{1}{n} \sum_{i=1}^n \varphi(X_i) - \E \varphi(X) \right\|_2
& \le C_{\mathrm{sg7}}  C_{\mathrm{sg5}} (k M_2)^{1/2} \left\{\sqrt{\frac{kp}{n}} + \sqrt{\frac{ \log(\delta^{-1})}{n}} \right\} \\
& \le \sqrt{2} C_{\mathrm{sg7}}  C_{\mathrm{sg5}} (k M_2)^{1/2} \sqrt{\frac{kp + \log(\delta^{-1})}{n} } .
\end{align*}

(iii)
The difference of interest can be expressed in terms of the centered variables as
\begin{align}
& \quad \frac{1}{n} \sum_{i=1}^{n} \varphi^\T_i A \varphi_i - \E \varphi^\T A \varphi \nonumber \\
& = \frac{1}{n} \sum_{i=1}^{n} (\varphi_i - \E \varphi)^\T A (\varphi_i -\E\varphi) - \E \{ (\varphi - \E \varphi)^\T A (\varphi-\E \varphi)\}  \label{eq:lem-spline-L2-upper-prf1} \\
& \quad + \frac{1}{n} \sum_{i=1}^n 2 (\E \varphi)^\T A (\varphi_i - \E \varphi) .  \label{eq:lem-spline-L2-upper-prf2}
\end{align}
We handle the concentration of the two terms separately.
Denote $\varphi_i = \varphi(X_i)$, $\varphi=\varphi(X)$, $\tilde\varphi_i = \varphi_i - \E \varphi$, and $\tilde\varphi = \varphi - \E \varphi$.

First, for $A \in \mathcal{A}_2$ the term in (\ref{eq:lem-spline-L2-upper-prf2}) can be bounded as follows:
\begin{align*}
& \quad \left| 2 (\E \varphi)^\T A \frac{1}{n} \sum_{i=1}^n \tilde\varphi_i \right|
%\le 2 \| \E \varphi \|_2 \left\| A  \frac{1}{n} \sum_{i=1}^n \tilde\varphi_i \right\|_2
\le 2 \| \E \varphi \|_2 \|A\|_{\op} \left\| \frac{1}{n} \sum_{i=1}^n \tilde\varphi_i \right\|_2
\le 2 \| \E \varphi \|_2 \left\| \frac{1}{n} \sum_{i=1}^n \tilde\varphi_i \right\|_2 \\
& \le 2 \sqrt{kp} \left( \frac{M_2^{1/2}}{\sqrt{2\pi}} + \| \xi \|_\infty \right) \left\| \frac{1}{n} \sum_{i=1}^n \tilde\varphi_i \right\|_2,
\end{align*}
where $ \| \xi \|_\infty = \max_{l=1,\ldots,k} |\xi_l|$.
The second inequality holds because $\|A \|_{\op} \le \|A \|_{\fro} =1$.
The third inequality holds because $ \| \E \varphi_l \|_2 \le \sqrt{p} ( M_2^{1/2} / \sqrt{2\pi} + | \xi_l|)$ by Lemma~\ref{lem:truncated-mean}.
By (ii), for any $\delta>0$, we have that with probability at least $1-\delta$,
\begin{align*}
&\quad \left\| \frac{1}{n} \sum_{i=1}^n \tilde\varphi_i \right\|_2
\le C_{\mathrm{sp21}} (k M_2)^{1/2} \sqrt{\frac{kp + \log(\delta^{-1})}{n} } .
\end{align*}
From the preceding two displays, we obtain that with probability at least $1-\delta$,
\begin{align}
& \quad \frac{1}{\sqrt{kp}} \sup_{A \in \mathcal{A}_2 }  \left| 2 (\E \varphi)^\T A \frac{1}{n} \sum_{i=1}^n \tilde\varphi_i \right| \nonumber \\
& \le \sqrt{\frac{2}{\pi}} C_{\mathrm{sp21}} (k M_2)^{1/2} \left( M_2^{1/2} +  \sqrt{2\pi}  \|\xi \|_\infty \right) \sqrt{\frac{kp + \log(\delta^{-1})}{n}} . \label{eq:lem-spline-L2-upper-prf3}
\end{align}

Next, consider an eigen-decomposition $A = \sum_{l=1}^{kp} \lambda_l w_l w_l^{\T} $,
where $\lambda_l$'s are eigenvalues and $w_l$'s are the eigenvectors with $\|w_l\|_2 =1$.
The concentration of the term in (\ref{eq:lem-spline-L2-upper-prf1}) can be controlled as follows:
\begin{align*}
&\quad \sup_{A \in \mathcal{A}_2 }  \left| \frac{1}{n}  \sum_{i=1}^{n} \tilde\varphi_{i}^\T A \tilde\varphi_{i} - \E \tilde\varphi^\T A \tilde\varphi \right|
=  \sup_{A \in \mathcal{A}_2 }  \left| \sum_{l=1}^{kp} \lambda_l w_l^\T \left( \frac{1}{n} \sum_{i=1}^{n} \tilde\varphi_{i}\tilde\varphi_{i}^\T  - \E \tilde\varphi \tilde\varphi^\T \right) w_l \right| \\
&\le \sup_{A \in \mathcal{A}_2 }  \left( \sum_{l=1}^{kp} |\lambda_l | \right) \left\| \frac{1}{n} \sum_{i=1}^{n} \tilde\varphi_{i}\tilde\varphi_{i}^\T  - \E \tilde\varphi \tilde\varphi^\T \right\|_{\op}
\le  \sqrt{kp} \left\| \frac{1}{n} \sum_{i=1}^{n} \tilde\varphi_{i}\tilde\varphi_{i}^\T  - \E \tilde\varphi \tilde\varphi^\T \right\|_{\op} .
\end{align*}
The last inequality uses the fact that $\|A\|_\fro = (\sum_{l=1}^{kp} \lambda_l^2)^{1/2} =1$ and hence $ \sum_{l=1}^{kp} |\lambda_l |  \le \sqrt{kp}$ for $A \in \mathcal{A}_2$.
From (i), $\tilde\varphi_i$ is a sub-gaussian random vector with tail parameter $(k M_2)^{1/2}$.
By Lemma \ref{lem:subg-vec-var}, for any $\delta >0$, we have that with probability at least $1- 2 \delta$,
\begin{align*}
\left\| \frac{1}{n} \sum_{i=1}^{n} \tilde\varphi_{i}\tilde\varphi_{i}^\T  - \E \tilde\varphi \tilde\varphi^\T \right\|_{\op}
\le C_{\mathrm{sg8}} k M_2 \left\{\sqrt{\frac{kp + \log(\delta^{-1})}{n}} + \frac{kp+ \log(\delta^{-1})}{n}\right\} ,
\end{align*}
From the preceding two displays, we obtain that with probability at least $1-2\delta$,
\begin{align}
& \quad \frac{1}{\sqrt{kp}} \sup_{A \in \mathcal{A}_2 }  \left| \frac{1}{n}  \sum_{i=1}^{n} \tilde\varphi_{i}^\T A \tilde\varphi_{i} - \E \tilde\varphi^\T A \tilde\varphi \right| \nonumber \\
& \le C_{\mathrm{sg8}} k M_2 \left\{\sqrt{\frac{kp + \log(\delta^{-1})}{n}} + \frac{kp+ \log(\delta^{-1})}{n}\right\}. \label{eq:lem-spline-L2-upper-prf4}
\end{align}
Combining the two bounds (\ref{eq:lem-spline-L2-upper-prf3}) and (\ref{eq:lem-spline-L2-upper-prf4}) gives the desired result.
\end{proof}

\begin{manualpro}{S8} \label{pro:logit-L2-upper}
In the setting of Proposition~\ref{pro:logit-L2-detailed}, it holds with probability at least $1-{4}\delta$ that for any $\gamma = (\gamma_0,\gamma_1, \gamma_2)^\T \in {\Gamma}$,
\begin{align*}
K_f(P_{n}, P_{\theta^* }; h_{\gamma, \mu^*})
& \le -f^\prime(3/5)  (\epsilon + \sqrt{\epsilon/(n\delta)} )  \\
& \quad + \pen_2(\gamma_1) (5/3) C_{\mathrm{sp21}} M_2^{1/2} R_1  \lambda_{21}\\
& \quad +  \pen_2 (\gamma_2) {(25\sqrt{5}/3)} C_{\mathrm{sp22}} M_{21} R_1 \sqrt{p} \lambda_{31} ,
\end{align*}
where {$C_{\mathrm{sp21}}$ and $C_{\mathrm{sp22}}$ are defined as in Lemma \ref{lem:spline-L2-upper}, $M_{21} = M_2^{1/2} (M_2^{1/2} + 2 \sqrt{2\pi}) $, and}
\begin{align*}
\lambda_{21} = \sqrt{\frac{{5 p} + \log(\delta^{-1})}{n}}, \quad \lambda_{31} = \lambda_{21} + \frac{{5 p}+ \log(\delta^{-1})}{n} .
\end{align*}
\end{manualpro}

\begin{proof}
Consider the event $\Omega_1 = \{ | \hat\epsilon-\epsilon | \le \sqrt{\epsilon(1-\epsilon)/(n\delta)}\}$.
By Chebyshev's inequality, we have $\pr (\Omega_1) \ge 1-\delta$.
In the event $\Omega_1$, we have $ |\hat\epsilon -\epsilon | \le 1/5$ by the assumption $\sqrt{\epsilon(1-\epsilon)/(n\delta)} \le 1/5$
and hence $\hat\epsilon \le 2/5$ by the assumption $\epsilon \le 1/5$.
By Lemma~\ref{lem:logit-upper} with $\epsilon_1=2/5$, it holds in the event $\Omega_1$ that for any $\gamma\in\Gamma$,
\begin{align}
& \quad K_f(P_n, P_{\theta^* } ; h_{\gamma, \mu^*}) \nonumber \\
& \le -f^\prime(3/5) \hat\epsilon + R_1 \left| \E_{P_{\theta^* ,n}} h_{\gamma, \mu^*}(x) - \E_{P_{\theta^* }} h_{\gamma, \mu^*}(x) \right| \nonumber \\
& \le -f^\prime(3/5) (\epsilon + \sqrt{\epsilon/(n\delta)} ) + R_1 \left| \E_{P_{\theta^*,n}} h_{\gamma}(x-\mu^*) - \E_{P_{(0,\Sigma^*)}} h_{\gamma}(x) \right| . \label{eq:pro-logit-L2-upper-prf1}
\end{align}
The last step also uses the fact that
$ \E_{P_{\theta^* }} h_{\gamma, \mu^*}(x) = \E_{P_{(0,\Sigma^*)}} h_{\gamma}(x)$ and also $\E_{P_{\theta^* ,n}} h_{\gamma, \mu^*}(x) =  \E_{P_{\theta^*,n}} h_{\gamma}(x-\mu^*)$,
by the definition $h_{\gamma, \mu^*}(x) = h_{\gamma} ( x- \mu^*)$.

Next, conditionally on the contamination indicators $(U_1,\ldots,U_n)$ such that the event $\Omega_1$ holds,
we have that $\{X_i: U_i=1, i=1,\ldots,n\}$ are $n_1$ independent and identically distributed observations from $P_{\theta^*}$,
where $n_1 = \sum_{i=1}^n (1-U_i) = n(1-\hat\epsilon) \ge (3/5)n$.
Denote as $\Omega_2$ the event that for any $\gamma_1$ and $\gamma_2$,
\begin{align*}
& \quad \left| \E_{P_{\theta^*,n}} \gamma_1^\T \varphi (x-\mu^*) - \E_{P_{(0,\Sigma^*)}} \gamma_1^\T \varphi (x) \right|
\le \|\gamma_1\|_2 C_{\mathrm{sp21}} M_2^{1/2} \sqrt{\frac{{5 p} +\log(\delta^{-1})}{(3/5)n}} ,
\end{align*}
and
\begin{align*}
&  \quad \left| \E_{P_{\theta^*,n}} \gamma_2^\T (\varphi(x-\mu^*)\otimes \varphi(x-\mu^*)) - \E_{P_{(0,\Sigma^*)}} \gamma_2^\T (\varphi(x)\otimes \varphi(x)) \right| \nonumber \\
& \le \|\gamma_2\|_2  C_{\mathrm{sp22}} {5} M_{21} \sqrt{{5 p}} \left\{ \sqrt{\frac{{5 p} +\log(\delta^{-1})}{(3/5) n}} + \frac{{5 p}+ \log(\delta^{-1})}{(3/5) n}\right\},
\end{align*}
where $C_{\mathrm{sp21}}$, $C_{\mathrm{sp22}}$, and $M_{21}$ are defined as in Lemma~\ref{lem:spline-L2-upper} with $ \|\xi\|_\infty=1$.
In the event $\Omega_2$, the preceding inequalities imply that for any $\gamma = (\gamma_0, \gamma_1^\T, \gamma_2^\T)^\T \in \Gamma$,
\begin{align}
& \quad \left| \E_{P_{\theta^*,n}} h_{\gamma}(x-\mu^*) - \E_{P_{(0,\Sigma^*)}} h_{\gamma}(x) \right| \nonumber \\
& \le \pen_2 (\gamma_1) (5/3) C_{\mathrm{sp21}}  M_2^{1/2} \lambda_{21} +
\pen_2 (\gamma_2) (5/3) C_{\mathrm{sp22}}  {5} M_{21} \sqrt{{5 p}} \lambda_{31}, \label{eq:pro-logit-L2-upper-prf2}
\end{align}
where $h_{\gamma}(x) = \gamma_0 +  \gamma_1^\T \varphi (x) + \gamma_2^\T (\varphi(x) \otimes \varphi(x))$,
$\pen_2 (\gamma_1) = \|\gamma_1\|_2$, and $\pen_2 (\gamma_2) = \|\gamma_2\|_2$.
By applying Lemma \ref{lem:spline-L2-upper} with {$k=5$} to $\{X_i -\mu^*: U_i=1, i=1,\ldots,n\}$, we have
$\pr( \Omega_2 | U_1, \ldots, U_n) \ge 1-3\delta$ for any $(U_1,\ldots, U_n)$ such that $\Omega_1$ holds.
Taking the expectation over $(U_1,\ldots,U_n)$ given $\Omega_1$
shows that $\pr ( \Omega_2 | \Omega_1 ) \ge 1- {3}\delta$
and hence $\pr (\Omega_1 \cap \Omega_2 ) \ge (1-\delta)(1-{3}\delta)\ge 1-{4}\delta$.

Combining (\ref{eq:pro-logit-L2-upper-prf1}) and (\ref{eq:pro-logit-L2-upper-prf2}) in the event $\Omega_1 \cap \Omega_2$ indicates that, with probability at least $1-4\delta$,
the desired inequality holds for any $\gamma\in\Gamma$.
\end{proof}

\begin{lem} \label{lem:logit-concen-L2-lower}
Suppose that $X_1, \ldots, X_n$ are independent and identically distributed as
$X \sim P_\epsilon$.
Let $b>0$ be fixed and $g : \bbR^p \to [0,1]^{q }$ be a vector of fixed functions.
For a convex and twice differentiable function $f: (0,\infty) \to \bbR$, define
\begin{align*}
& \quad F (X_1, \dots, X_n)  \\
& = \sup_{\| w \|_2 = 1, \mu \in \bbR^p, \eta_0\in[0,1]^{q }} \left\{K_f (P_n, P_{\hat\theta}; b w^\T g_{\mu,\eta_0} ) - K_f (P_{\epsilon}, P_{\hat\theta}; b w^\T g_{\mu,\eta_0} ) \right\} \\
& = \sup_{\| w \|_2 = 1, \mu \in \bbR^p, \eta_0\in[0,1]^{q }} \left\{
\frac{1}{n}\sum_{i=1}^{n} f^{\prime}(\me^{b w^\T g_{\mu,\eta_0} (X_i) }) - \E f^{\prime}(\me^{b w^\T g_{\mu,\eta_0}(X) })
\right\} ,
\end{align*}
where $g_{\mu,\eta_0} (x)= g (x-\mu) - \eta_0$.
Suppose that conditionally on $(X_1,\ldots,X_n)$, the random variable $Z_{n,j} = \sup_{\mu \in \bbR^p, \eta_0\in [0,1]^q }  |n^{-1} \sum_{i=1}^{n}\epsilon_i g_{\mu,\eta_0,j}(X_i) |$
is sub-gaussian with tail parameter $\sqrt{V_g/n}$ for $j=1,\ldots, q$,
where $(\epsilon_1,\ldots,\epsilon_n)$ are Rademacher variables, independent of $(X_1,\ldots,X_n)$,
and $g_{\mu, \eta_0, j}: \bbR^p \to [-1,1]$ denotes the $j$th component of $g_{\mu,\eta_0}$.
Then for any $\delta > 0$, we have that
with probability at least $1-2 \delta$,
\begin{align*}
F(X_1, \dots, X_n) \le b R_{2,b \sqrt{q}} \left\{ C_{\mathrm{sg,12}} \sqrt{\frac{2q  V_g}{n}} + \sqrt{\frac{2 q  \log(\delta^{-1})}{n}} \right\},
\end{align*}
where $R_{2,b\sqrt{q}} =\sup_{ |u| \le b\sqrt{q }} \frac{\dif}{\dif u} f^{\prime}(\me^u)$
and $C_{\mathrm{sg,12}}$ is the universal constant in Lemma~\ref{lem:subg}.
\end{lem}

\begin{proof} %[Proof of Lemma \ref{lem:logit-concen-L2-lower}]
First, $F$ satisfies the bounded difference condition, because $|b w^\T g_{\mu,\eta_0}| \le b \sqrt{q }$ with $\| w \|_2\le 1$ and $f^{\prime}$ is non-decreasing by the convexity of $f$:
\begin{align*}
& \quad \sup_{X_1, \dots, X_n, X_{i}^{\prime}} \left|
F(X_1,\dots,X_n) - F(X_1,\dots, X_{i}^{\prime}, \dots, X_n) \right| \\
& \le \frac{f^{\prime}(\me^{b \sqrt{q }}) - f^{\prime}(\me^{- b\sqrt{q }})}{n} \le \frac{2 b\sqrt{q } R_{2,b} }{n},
\end{align*}
where $R_{2,b\sqrt{q}} =\sup_{ |u| \le b\sqrt{q }} \frac{\dif}{\dif u} f^{\prime}(\me^u)$.
By McDiarmid's inequality (\citeappend{MC89}), for any $t > 0$, we have that with probability at least $1-2\me^{-2nt^2}$,
\begin{align*}
|F(X_1, \dots, X_n) - \E F(X_1, \dots, X_n)| \le  2 b \sqrt{q } R_{2,b} t.
\end{align*}
For any $\delta >0$, taking $t= \sqrt{ \log(\delta^{-1}) / (2n) }$ shows that with probability at least $1-2\delta$,
\begin{align*}
|F(X_1, \dots, X_n) - \E F(X_1, \dots, X_n)| \le b  R_{2,b} \sqrt{\frac{2 q  \log(\delta^{-1})}{n}} .
\end{align*}

Next, the expectation of $F(X_1, \dots, X_n)$ can be bounded as follows:
\begin{align}
& \quad \E \sup_{\| w \|_2 = 1, \mu \in \bbR^p, \eta_0\in[0,1]^{q }} \left\{
\frac{1}{n}\sum_{i=1}^{n} f^{\prime}(\me^{ b w^\T g_{\mu,\eta_0}( X_i)}) - \E f^{\prime}(\me^{ b w^\T g_{\mu,\eta_0}(X)})
\right\} \nonumber \\
& \le 2\E \sup_{\| w \|_2 = 1, \mu \in \bbR^p, \eta_0\in[0,1]^{q }} \left\{
\frac{1}{n}\sum_{i=1}^{n} \epsilon_i f^{\prime}(\me^{ b w^\T g_{\mu,\eta_0}(  X_i)})
\right\} \label{eq:lem-logit-concen-L2-lower-2} \\
& \le 2 R_{2,b\sqrt{q}} \,\E  \sup_{\| w \|_2 = 1, \mu \in \bbR^p, \eta_0\in[0,1]^{q }}\left\{
\frac{1}{n}\sum_{i=1}^{n} \epsilon_i  b w^\T g_{\mu,\eta_0}( X_i) \right\}  \label{eq:lem-logit-concen-L2-lower-3} \\
& \le 2 b R_{2,b\sqrt{q}} \,\E \sup_{ \mu \in \bbR^p, \eta_0\in[0,1]^{q }}  \left\|\frac{1}{n}\sum_{i=1}^{n} \epsilon_i g_{\mu,\eta_0}(X_i)\right\|_2 \nonumber \\
&\le 2 b R_{2,b \sqrt{q}} C_{\mathrm{sg,12}} \sqrt{\frac{2 q  V_g}{n}} . \label{eq:lem-logit-concen-L2-lower-4}
\end{align}
Line (\ref{eq:lem-logit-concen-L2-lower-2}) follows from the symmetrization Lemma~\ref{lem:symm}, where $(\epsilon_1,\ldots,\epsilon_n)$ are Rademacher variables, independent of $(X_1,\ldots,X_n)$.
Line (\ref{eq:lem-logit-concen-L2-lower-3}) follows by Lemma \ref{lem:contraction}, because $f^{\prime}(\me^{t})$ is $R_{2,b\sqrt{q}}$-Lipschitz in $u \in [-b\sqrt{q },b\sqrt{q }]$.
Line (\ref{eq:lem-logit-concen-L2-lower-4}) follows because
\begin{align*}
&\quad \E  \sup_{\mu \in \bbR^p, \eta_0\in[0,1]^{q }} \left\|\frac{1}{n}\sum_{i=1}^{n}\epsilon_i g_{\mu,\eta_0}(X_i)\right\|_2
\le \left\{\E \sup_{\mu \in \bbR^p, \eta_0\in[0,1]^{q }} \left\|\frac{1}{n}\sum_{i=1}^{n}\epsilon_ig_{\mu,\eta_0}(X_i)\right\|_2^2 \right\}^{1/2} \\
&\le \left\{ \sum_{j=1}^{q } \E \sup_{\mu \in \bbR^p, \eta_0\in[0,1]^{q }} \left| \frac{1}{n}\sum_{i=1}^{n}\epsilon_i g_{\mu,\eta_0, j}(X_i)  \right|^2 \right\}^{1/2}\\
& \le C_{\mathrm{sg,12}} \sqrt{\frac{2 q V_g}{n}}.
\end{align*}
For the last step,
we use the assumption that conditionally on $(X_1,\ldots,X_n)$, the random variable $Z_{n,j} = \sup_{\mu \in \bbR^p, \eta_0\in [0,1]^q}  |n^{-1} \sum_{i=1}^{n}\epsilon_i g_{\mu,\eta_0,j}(X_i) |$
is sub-gaussian with tail parameter
$\sqrt{ V_g /n}$ for each $j=1,\ldots,q $, and apply Lemma~\ref{lem:subg} to obtain
$\E ( Z^2_{n,j} | X_1,\ldots,X_n) \le C_{\mathrm{sg,12}}^2 (2 V_g /n)$,
and then $\E ( Z^2_{n,j})\le C_{\mathrm{sg,12}}^2 (2 V_g /n)$.

Combining the tail probability and expectation bounds yields the desired result.
\end{proof}

\begin{lem} \label{lem:logit-L2-lower}
Suppose that $f: (0,\infty) \to \bbR$ is convex and three-times differentiable. Let $b >0$ be fixed.

(i) For any function $h: \bbR^p \to [-b,b]$, we have
\begin{align*}
K_f(P_{\epsilon}, P_{\hat\theta} ; h) \ge f^{\prime}(\me^{-b})\epsilon+  f^{\prime}(\me^{\E_{P_{\theta^*}} h(x) }) -  f^\# (\me^{\E_{P_{\hat\theta}} h(x)} ) - \frac{1}{2} R_{33,b} ,
\end{align*}
where $R_{33,b} = R_{31,b} \var_{P_{\theta^* }} h (x) + R_{32,b} \var_{P_{\hat\theta}}  h (x)$, $R_{31,b} = \sup_{|u|\le b} \frac{\dif^2}{ \dif u^2} \{-f^\prime(\me^u) \}$,
and $R_{32,b} = \sup_{ |u| \le b} \frac{\dif^2}{ \dif u^2} f^\# (\me^u)$.

(ii) If, in addition, $\E_{P_{\theta^*}} h(x) =0$ and $\E_{P_{\hat\theta}} h(x) \le 0$, then
\begin{align*}
K_f(P_{\epsilon}, P_{\hat\theta} ; h) \ge f^{\prime}(\me^{-b})\epsilon + R_{4,b} \left\{ \E_{P_{\theta^*}} h(x) - \E_{P_{\hat\theta}} h(x) \right\} - \frac{1}{2} R_{33,b} ,
\end{align*}
where $R_{4,b} = \inf_{ |u| \le b} \frac{\dif}{ \dif u} f^\# (\me^u)$.
\end{lem}

\begin{proof}
(i) First, $K_f(P_{\epsilon}, P_{\hat\theta} ; h) $ can be bounded as follows:
\begin{align*}
& \quad  K_f(P_{\epsilon}, P_{\hat\theta} ; h) \\
& =     \epsilon  \E_Q f^{\prime}(\me^{h(x)})+ (1-\epsilon)  \E_{P_{\theta^* }} f^{\prime}(\me^{h(x)}) -  \E_{P_{\hat\theta}}  f^\# ( \me^{h(x)} ) \\
& \ge f^{\prime}( \me^{-b} )\epsilon + \E_{P_{\theta^* }} f^{\prime}(\me^{h(x)}) - \E_{P_{\hat\theta}} f^\# (\me^{h(x)} ) , \\
& = f^{\prime}( \me^{-b} )\epsilon + K_f(P_{\theta^* },P_{\hat\theta} ;h) ,
\end{align*}
where the inequality follows because $ f^{\prime} (\me^{h(x)})  \ge f^{\prime}( \me^{-b} )$ for $h(x)\in [-b,b]$ by the convexity of $f$.
Next, consider the function
\begin{align*}
\kappa (t) = \E_{P_{\theta^* }} f^{\prime}(\me^{E_1 + t \tilde h_1 (x) }) - \E_{P_{\hat\theta}} f^\# ( \me^{E_2 + t \tilde h_2 (x)} ),
\end{align*}
where $E_1 = \E_{P_{\theta^*}} h(x)$, $E_2 = \E_{P_{\hat\theta}} h(x)$, $\tilde h_1(x) = h(x) - E_1$, and $\tilde h_2(x) = h(x) - E_2$.
A Taylor expansion of $\kappa(1) = K_f(P_{\theta^* },P_{\hat\theta} ; h)$ about $t=0$ yields
\begin{align*}
K_f(P_{\theta^* }, P_{\hat\theta} ;h)
& = f^{\prime}(\me^{E_1 }) - f^\# (\me^{ E_2} ) - \frac{1}{2} \kappa^{\dprime}(t),
\end{align*}
where for some $t \in [0,1]$,
\begin{align*}
\kappa^{\dprime} (t) & = - \E_{P_{\theta^* }}\left\{ \tilde h_1^2(x) \frac{\dif^2}{ \dif u^2} f^\prime(\me^{E_1+u} )|_{u=t \tilde h_1(x)} \right\} +
\E_{P_{\hat\theta}}\left\{ \tilde h_2^2(x) \frac{\dif^2}{ \dif u^2} f^\# (\me^{E_2 + u}) |_{u=t \tilde h_2(x)} \right\} .
\end{align*}
The desired result then follows because $E_1 + t \tilde h_1(x) \in [-b,b]$ and $E_2 + t \tilde h_2(x) \in [-b,b]$ for $t\in [0,1]$ and hence $\kappa^{\dprime} (t) \le R_{33,b}$ by the definition of $R_{33,b}$.

(ii) The inequality from (i) can be rewritten as
\begin{align*}
& \quad K_f(P_{\epsilon}, P_{\hat\theta} ; h) \\
& \ge f^{\prime}(\me^{-b})\epsilon + \left\{ f^\prime (\me^{\E_{P_{\theta^*}} h(x) })  -  f^\# (\me^{\E_{P_{\theta^*}} h(x) }) \right\}
+ f^\#(\me^{\E_{P_{\theta^*}} h(x) }) -  f^\# (\me^{\E_{P_{\hat\theta}} h(x)} ) - \frac{1}{2} R_{33,b} .
\end{align*}
If $\E_{P_{\theta^*}} h(x) =0$, then
\begin{align*}
& \quad K_f(P_{\epsilon}, P_{\hat\theta} ; h)
\ge f^{\prime}(\me^{-b})\epsilon + f^\#(\me^{\E_{P_{\theta^*}} h(x) }) -  f^\# (\me^{\E_{P_{\hat\theta}} h(x)} ) - \frac{1}{2} R_{33,b} .
\end{align*}
Moreover, if $\E_{P_{\theta^*}} h(x) - \E_{P_{\hat\theta}} h(x) \ge 0$, then
\begin{align*}
f^\#(\me^{\E_{P_{\theta^*}} h(x) }) -  f^\# (\me^{\E_{P_{\hat\theta}} h(x)} ) \ge R_{4,b} \left\{ \E_{P_{\theta^*}} h(x) - \E_{P_{\hat\theta}} h(x) \right\} .
\end{align*}
by the mean value theorem and the definition of $R_{4,b}$. Combining the preceding two displays gives the desired result.
\end{proof}

\begin{manualpro}{S9} \label{pro:logit-L2-lower}
Let $ b_2 >0$ be fixed and $b_2^\dag = b_2 \sqrt{2p}$. In the setting of Proposition~\ref{pro:logit-L2-detailed}, it holds with probability at least $1-2 \delta$ that for any $\gamma\in \Gamma_{\rpL}$ with
$\pen_2(\gamma)=  b_2 $, $\E_{P_{\theta^*}} h_{\gamma, \hat\mu}(x)=0$, and $\E_{P_{\hat \theta}} h_{\gamma, \hat\mu}(x)\le 0$,
\begin{align*}
& \quad  K_f(P_{n}, P_{\hat\theta}; h_{\gamma, \hat\mu}) \\
& \ge R_{4,b_2^\dag} \left\{ \E_{P_{\theta^*}} h_{\gamma, \hat\mu}(x) - \E_{P_{\hat\theta}} h_{\gamma, \hat\mu}(x) \right\}  +  f^{\prime}(\me^{- b_2^\dag})\epsilon
- 4 C_{\mathrm{sg,12}}^2 M_2  b_2^2 R_{3, b_2^\dag}
-  b_2 R_{2, b_2^\dag} \lambda_{22}
\end{align*}
where $R_{3,b} = R_{31,b} + R_{32,b}$ as in {Lemma~\ref{lem:logit-L1-lower}} and, with $C_{\mathrm{rad5}}  = C_{\mathrm{sg,12}} C_{\mathrm{rad3}} $,
$$
\lambda_{22} = C_{\mathrm{rad5}}  \sqrt{\frac{16 p}{n}} + \sqrt{\frac{2p \log(\delta^{-1})}{n}} ,
$$
depending on the universal constants $C_{\mathrm{sg, 12}}$ and $ C_{\mathrm{rad3}}$ in Lemma~\ref{lem:subg} and Corollary~\ref{cor:entropy-sg2}.
\end{manualpro}

\begin{proof}
By definition, for any $\gamma\in \Gamma_{\rpL}$, $h_\gamma(x)$ can be represented as $h_{\rpL,\beta, c }(x)$ such that $\beta_0=\gamma_0$ and $\pen_2 (\beta) =\sqrt{2} \pen_2(\gamma)$:
\begin{align*}
h_{\rpL, \beta,  c }(x) & = \beta_0 + \sum_{j=1}^p \beta_{1j} \ramp ( x_j -  c_j) \\
&  = \beta_0 + \beta_1^\T \varphi_{\rp,c}(x)  ,
\end{align*}
where $ c =( c _1, \ldots, c _p )^\T$ with $ c _j \in \{0,1\}$,
$\beta= (\beta_0, \beta_1^\T)^\T$ with
$\beta_1 =( \beta_{11}, \ldots, \beta_{1p} )^\T$, and
$\varphi_{\rp,c} (x): \bbR^p \to [0,1]^p$ denotes the vector of functions with the $j$th component $\ramp(x_j -c_j)$ for $j=1,\ldots,p$.
Then for any $\gamma\in \Gamma_{\rpL}$ with
$\pen_2(\gamma)=  b_2 $, we have $\beta_0=\gamma_0$ and $\pen_2(\beta)=  \sqrt{2} b_2$ correspondingly, and hence
$h_\gamma(x) -\gamma_0 = h_{\rpL, \beta,  c }(x) - \beta_0\in [- b_2  \sqrt{2p},  b_2  \sqrt{2p}]$
by the Cauchy--Schwartz inequality and the boundedness of the ramp function in $[0,1]$.
Moreover, $h_{\rpL,\beta, c }(x)$ can be expressed in the form
$\beta_0 + \pen_2(\beta) w^\T g(x)$, where for $q =2p $, $w \in \bbR^{q }$ is an $L_2$ unit vector,
$g: \bbR^p \to [0,1]^{q }$ is a vector of functions, including $\ramp(x_j)$ and $\ramp(x_j-1)$ for $j=1,\ldots,p$.
Parenthetically, at most one of the coefficients in $w$ associated with $\ramp(x_j)$ and $\ramp(x_j-1)$ is nonzero for each $j$,
although this property is not used in the subsequent discussion.

For any $\gamma\in \Gamma_{\rpL}$ with $\pen_2(\gamma)=  b_2 $ and $\E_{P_{\theta^*}} h_{\gamma, \hat\mu}(x)=0$, the function
$ h_{\gamma, \hat\mu}(x) $ can be expressed as
\begin{align*}
h_{\gamma, \hat\mu}(x)= \beta_1^\T \left\{ \varphi_{\rp,c}(x-\hat\mu) - \beta_{01} \right\} ,
\end{align*}
where $\beta_{01} = \E_{P_{\theta^*}} \varphi_{\rp,c}(x-\hat\mu)$.
The mean-centered ramp functions in $\varphi_{\rp,c}(x-\hat\mu) - \beta_{01} $ are bounded between $[-1,1]$, and hence
$ h_{\gamma, \hat\mu}(x) \in [- b_2  \sqrt{2p},  b_2  \sqrt{2p}]$ similarly as above.
Moreover, such $ h_{\gamma, \hat\mu}(x) $ can be expressed in the form
$ \pen_2(\beta) w^\T \{g(x -\hat\mu) - \eta_0\} $, where $w \in \bbR^{q }$ is an $L_2$ unit vector,
$g(x): \bbR^p \to [0,1]^{q }$ is defined as above, and
$\eta_0 = \E_{P_{\theta^*}} g(x - \hat\mu) \in [0,1]^{q }$ by the boundedness of the ramp function in $[0,1]$.

Next, $K_f(P_{n}, P_{\hat\theta}; h_{\gamma, \hat\mu})$ can be bounded as
\begin{align}
& \quad  K_f(P_{n}, P_{\hat\theta}; h_{\gamma, \hat\mu}) \nonumber \\
& \ge  K_f(P_\epsilon, P_{\hat\theta}; h_{\gamma, \hat\mu})  - \{ K_f(P_{n}, P_{\hat\theta}; h_{\gamma, \hat\mu}) - K_f(P_\epsilon, P_{\hat\theta}; h_{\gamma, \hat\mu})\} . \label{eq:pro-logit-L2-lower-prf1}
\end{align}
For any $\gamma\in \Gamma_{\rpL}$ with $\pen_2(\gamma)=  b_2 $, $E_{P_{\theta^*}} h_{\gamma, \hat\mu}(x)=0$, and $E_{P_{\hat \theta}} h_{\gamma, \hat\mu}(x)\le 0$,
applying {Lemma~\ref{lem:logit-L2-lower}(ii)} with $h= h_{\gamma, \hat\mu}$ and $b= b_2^\dag =b_2\sqrt{2p}$ yields
\begin{align*}
& K_f(P_\epsilon, P_{\hat\theta}; h_{\gamma, \hat\mu})
\ge f^{\prime}(\me^{- b_2^\dag })\epsilon \\
& \quad + R_{4, b_2^\dag} \left\{ \E_{P_{\theta^*}} h(x) - \E_{P_{\hat\theta}} h(x) \right\} -
\frac{1}{2}  \left\{ R_{31, b_2^\dag}\var_{P_{\theta^*}} h_{\gamma, \hat\mu}(x) + R_{32, b_2^\dag}\var_{P_{\hat\theta}} h_{\gamma, \hat\mu}(x) \right\}.
\end{align*}
By {Lemma~\ref{lem:lip-gaus}(i)}, $\var_{P_{\theta^*}} h_{\gamma, \hat\mu}(x)$ can bounded as follows:
\begin{align*}
& \quad \var_{P_{\theta^*}} h_{\gamma, \hat\mu}(x)
= \var_{P_{\theta^*}} \beta_1^\T \varphi_{\rp,c}(x-\hat\mu)  \\
& \le \|\beta_1\|_2^2 \cdot 2 C_{\mathrm{sg,12}}^2 (\sqrt{2})^2 \| \Sigma^*\|_\op
= 4 \pen_2^2(\beta) C_{\mathrm{sg,12}}^2 M_2 .
\end{align*}
Similarly, $\var_{P_{\theta^*}} h_{\gamma, \hat\mu}(x)$ can also be bounded by $4 \pen_2^2(\beta) C_{\mathrm{sg,12}}^2 M_2 $, because $\| \hat \Sigma\|_\op \le M_2$.
Hence $K_f(P_\epsilon, P_{\hat\theta}; h_{\gamma, \hat\mu}) $ can be bounded as
\begin{align}
& \quad K_f(P_\epsilon, P_{\hat\theta}; h_{\gamma, \hat\mu}) \nonumber \\
& \ge f^{\prime}(\me^{- b_2^\dag})\epsilon + R_{4, b_2^\dag} \left\{ \E_{P_{\theta^*}} h(x) - \E_{P_{\hat\theta}} h(x) \right\} - 4 C_{\mathrm{sg,12}}^2 M_2  b_2^2 R_{3, b_2^\dag},\label{eq:pro-logit-L2-lower-prf2}
\end{align}
where $R_{3,b} = R_{31,b} + R_{32,b}$ {as in {Lemma~\ref{lem:logit-L1-lower}}.}
Moreover, by Lemma~\ref{lem:logit-concen-L1-lower} with $b=\sqrt{2}  b_2 $ and $g(x) \in [0,1]^{q }$ defined above, it holds with probability at least $1-2\delta$ that
for any $\gamma\in \Gamma_{\rpL}$ with $\gamma_0=0$ and $\pen_2(\gamma)=  b_2 $,
\begin{align}
& \quad \{ K_f(P_{n}, P_{\hat\theta}; h_{\gamma, \hat\mu}) - K_f(P_\epsilon, P_{\hat\theta}; h_{\gamma, \hat\mu})\} \nonumber \\
& \le  b_2 R_{2, b_2\sqrt{q}} \left\{  C_{\mathrm{sg,12}} \sqrt{\frac{2 q V_g }{n}} + \sqrt{\frac{q  \log(\delta^{-1})}{n}} \right\} \nonumber \\
& \le  b_2 R_{2, b_2\sqrt{2p}} \left\{  C_{\mathrm{sg,12}} C_{\mathrm{rad3}} \sqrt{\frac{16 p}{n}} + \sqrt{\frac{2p \log(\delta^{-1})}{n}} \right\} , \label{eq:pro-logit-L2-lower-prf3}
\end{align}
where $V_g  = 4 C_{\mathrm{rad3}}^2$ is determined in Lemma~\ref{lem:logit-concen-L1-lower} as follows.
For $j=1,\ldots, q$,
consider the function class $\mathcal{G}_j = \{g_{\mu, \eta_0, j}: \mu \in \bbR^p, \eta_0 \in [0,1]^q \}$,
where $\mu = (\mu_1,\ldots,\mu_p)^\T$, $\eta_0=(\eta_{01},\ldots,\eta_{0q})^\T$, and, as defined in Lemma~\ref{lem:logit-concen-L1-lower}, $g_{\mu,\eta_0,j}(x)$ is
of the form $\ramp(x_{j_1} - \mu_{j_1}) - \eta_{0j}$ or $\ramp(x_{j_1}- \mu_{j_1}-1) - \eta_{0j}$.
By Lemma~\ref{lem:ramp-VC}, the VC index of moving-knot ramp functions is 2.
By Lemma~\ref{lem:intercept-VC}, the VC index of constant functions is also 2.
By applying Corollary~\ref{cor:entropy-sg2} (ii) with vanishing $\mathcal G$,
we obtain that conditionally on $(X_1,\ldots,X_n)$, the random variable
$Z_{n,j} = \sup_{\mu \in \bbR^p, \eta_0\in [0,1]^q }  |n^{-1} \sum_{i=1}^{n}\epsilon_i g_{\mu,\eta_0,j}(X_i) | = \sup_{f_j \in\mathcal{G}_j}  |n^{-1} \sum_{i=1}^{n}\epsilon_i f_j (X_i) |$
is sub-gaussian with tail parameter $C_{\mathrm{rad3}} \sqrt{4/n}$ for $j=1,\ldots,q$.

Combining the inequalities (\ref{eq:pro-logit-L2-lower-prf1})--(\ref{eq:pro-logit-L2-lower-prf3}) leads to the desired result.
\end{proof}

\begin{manualpro}{S10} \label{pro:logit-L2-combine}
In the setting of Proposition~\ref{pro:logit-L2-detailed} or Proposition~\ref{pro:hinge-L2-detailed}, suppose that for $a \in (0,1/2)$,
\begin{align}
D  \stackrel{\text{def}}{=} \sup_{\gamma \in \Gamma_{\rpL}, \pen_2(\gamma)=\sqrt{1/2}}
\left\{  \E_{P_{\theta^*}} h_{\gamma, \hat\mu}(x) - \E_{P_{\hat\theta}} h_{\gamma, \hat\mu}(x) \right\}  \le a. \label{eq:pro-spline-L1-combine-cond}
\end{align}
Then we have
\begin{align}
& \| \hat \mu- \mu^* \|_2\le S_{4,a} D ,  \label{eq:pro-logit-L2-combine-1} \\
& \| \hat \sigma - \sigma^* \|_2 \le S_{5} \left( D   + \|\hat\mu-\mu^*\|_2/2 \right) \le S_{6,a} D ,  \label{eq:pro-logit-L2-combine-2}
\end{align}
where {$S_{4,a}$, $S_{5}$, and $S_{6,a}$ are defined as in Lemma~\ref{lem:local-linear1} and Remark \ref{rem:local-linear1} with $M= M_2$.}
\end{manualpro}

\begin{proof}
For any $\gamma \in \Gamma_{\rpL}$ with $\pen_2(\gamma) =\sqrt{1/2}$, the function $ h_\gamma ( x) $ can be obtained as $h_{\rpL, \beta,  c }(x )$
with $ \pen_2 (\beta) = 1$.
For $j=1,\ldots,p$, we restrict $h_{\gamma,\hat\mu}(x)$ in (\ref{eq:pro-spline-L1-combine-cond}) such that $h_{\gamma}(x)$ is a ramp function of $x_j$,
in the form $\pm \ramp(x_j -c)$ for $c \in \{0,1\}$. Applying Lemma~\ref{lem:local-linear1} shows
that there exists $h_j^{(1)} (x_j)$ in the form $\pm \ramp(x_j)$ and $h_j^{(2)} (x_j)$ in the form $\pm \ramp (x_j-1)$ such that
\begin{align}
& | \hat \mu_j - \mu^*_j | \le S_{4,a} \left\{  \E_{P_{\theta^*}} h_j^{(1)} (x_j - \hat\mu_j) - \E_{P_{\hat\theta}} h_j^{(1)} (x_j - \hat\mu_j) \right\}, \label{eq:pro-logit-L2-combine-prf1}  \\
& | \hat \sigma_j - \sigma^*_j | \le S_{5} \left\{  \E_{P_{\theta^*}} h_j^{(2)} (x_j - \hat\mu_j) - \E_{P_{\hat\theta}} h_j^{(2)} (x_j - \hat\mu_j) \right\} + S_{5} | \hat\mu_j - \mu^*_j|/2 .
\label{eq:pro-logit-L2-combine-prf2}
\end{align}

From (\ref{eq:pro-logit-L2-combine-prf1}), we have that for any $L_2$ unit vector $w = (w_1, \ldots, w_p)^\T$,
\begin{align*}
\sum_{j=1}^p  | w_j ( \hat \mu_j- \mu_j) |
& \le  S_{4,a} \sum_{j=1}^p  |w_j| \left\{  \E_{P_{\theta^*}} h_j^{(1)} (x_j - \hat\mu_j) - \E_{P_{\hat\theta}} h_j^{(1)} (x_j - \hat\mu_j) \right\} \\
&  =  S_{4,a} \left\{ \E_{P_{\theta^*}} h^{(1)} (x - \hat\mu) - \E_{P_{\hat\theta}} h^{(1)} (x - \hat\mu) \right\} ,
\end{align*}
where $h^{(1)} (x) = \sum_{j=1}^p  |w_j| h_j^{(1)} (x_j )$. In fact, $h^{(1)} (x)$ can be expressed as $h_{\rpL, \beta,  c }(x )$ such that $c=(0,\ldots,0)^\T$ and
each component in $\beta_1$ is either $|w_j|$ or $-|w_j|$ for $j=1,\ldots,p$, which implies that $\pen_2( \beta ) = \| w\|_2 =1$.
Hence by the definition of $D $, we obtain (\ref{eq:pro-logit-L2-combine-1}).

Similarly, from (\ref{eq:pro-logit-L2-combine-prf2}), we have that for any $L_2$ unit vector $w = (w_1, \ldots, w_p)^\T$,
\begin{align*}
& \quad \sum_{j=1}^p  | w_j ( \hat \sigma_j- \sigma_j) | \\
& \le  S_{5} \sum_{j=1}^p  |w_j| \left\{  \E_{P_{\theta^*}} h_j^{(2)} (x_j - \hat\mu_j) - \E_{P_{\hat\theta}} h_j^{(2)} (x_j - \hat\mu_j) \right\} + S_{5} \sum_{j=1}^p |w_j (\hat\mu_j - \mu_j^*)|/2 \\
&  =  S_{5} \left\{ \E_{P_{\theta^*}} h^{(2)} (x - \hat\mu) - \E_{P_{\hat\theta}} h^{(2)} (x - \hat\mu) \right\} + S_{5} | w^\T (\hat \mu-\mu^*)|/2,
\end{align*}
where $h^{(2)} (x) = \sum_{j=1}^p  |w_j| h_j^{(2)} (x_j )$, which can be expressed in the form $h_{\rpL, \beta,  c }(x )$ with $c=(1,\ldots,1)^\T$ and $\pen_2( \beta ) = \| w\|_2 =1$.
Hence by the definition of $D $, we obtain (\ref{eq:pro-logit-L2-combine-2}).
\end{proof}

%%% -----------------------------------------------------------------------------------------------------------------------------------------   details, spline, L2, var matrix

\begin{manualpro}{S11} \label{pro:logit-L2v-lower}
Let $ b_3 >0$ be fixed and $b_3^\dag = b_3 \sqrt{p(p-1)}$. In the setting of {Proposition~\ref{pro:logit-L2-detailed}}, it holds with probability at least $1-2 \delta$ that for any $\gamma\in \Gamma_{\rpQ}$ with
$\pen_2(\gamma)=  b_3 $, $\E_{P_{\theta^*}} h_{\gamma, \hat\mu}(x)=0$, and $\E_{P_{\hat \theta}} h_{\gamma, \hat\mu}(x)\le 0$,
\begin{align*}
& \quad  K_f(P_{n}, P_{\hat\theta}; h_{\gamma, \hat\mu}) \\
& \ge R_{4, 2b_3^\dag} \left\{ \E_{P_{\theta^*}} h(x) - \E_{P_{\hat\theta}} h(x) \right\} + f^{\prime}(\me^{- 2 b_3^\dag })\epsilon - (80 C_{\mathrm{sg,12}}^2  M_2) p b_3^2 R_{3, 2 b_3^\dag}
- \sqrt{p} b_3  R_{2, b_3^\dag} \lambda_{32},
\end{align*}
{where $R_{3,b} = R_{31,b} + R_{32,b}$ as in Lemma~\ref{lem:logit-L1-lower} and, with $C_{\mathrm{rad5}} = C_{\mathrm{sg,12}} C_{\mathrm{rad3}}$,}
$$
\lambda_{32} = {C_{\mathrm{rad5}}} \sqrt{\frac{12 (p-1)}{n}} + \sqrt{\frac{(p-1)\log(\delta^{-1})}{n}} ,
$$
{depending on the universal constants $C_{\mathrm{sg, 12}}$ and $ C_{\mathrm{rad3}}$ in Lemma~\ref{lem:subg} and Corollary~\ref{cor:entropy-sg2}.}
\end{manualpro}

\begin{proof}
By definition, for any $\gamma\in \Gamma_{\rpQ}$, $h_\gamma(x)$ can be represented as $h_{\rpQ,\beta }(x)$ such that $\beta_0=\gamma_0$ and $\pen_2 (\beta) =2 \pen_2(\gamma)$, where
\begin{align*}
h_{\rpQ, \beta }(x) & = \beta_0 + \sum_{1\le i\not= j \le p} \beta_{2,ij} \ramp( x_i) \ramp(x_j) \\
& = \beta_0 + \beta_2^\T \vec(\varphi_{\rp} (x) \otimes \varphi_{\rp}(x)) ,
\end{align*}
where $\beta= (\beta_0, \beta_2^\T)^\T$ with
$\beta_2 = (\beta_{2,ij}: 1\le i\not= j\le p)^\T$, and $\varphi_{\rp} (x): \bbR^p \to [0,1]^p$ denotes the vector of functions with the $j$th component $\ramp(x_j)$ for $j=1,\ldots,p$.
Then for any $\gamma\in \Gamma_{\rpQ}$ with $\pen_2(\gamma)=  b_3 $, we have
$\beta_0 = \gamma_0$ and $\pen_2(\beta)= 2 b_3$ correspondingly, and hence
$h_\gamma(x) -\gamma_0 = h_{\rpQ, \beta }(x) - \beta_0\in [- 2 b_3 \sqrt{p(p-1)},  2 b_3 \sqrt{p(p-1)}]$,
by the boundedness of the ramp function in $[0,1]$ and the Cauchy--Schwartz inequality, $\| \beta_2 \|_1 \le \sqrt{p(p-1)} \|\beta_2\|_2$.
Moreover, $h_{\rpQ,\beta }(x)$ can be expressed in the form
$\beta_0 + \pen_2(\beta) w^\T g(x)$, where for $q =p (p-1)$, $w \in \bbR^{q }$ is an $L_2$ unit vector,
$g: \bbR^p \to [0,1]^{q }$ is a vector of functions, including $\ramp(x_i)\ramp(x_j)$ for $1 \le i\not= j\le p$.
For symmetry, $ \ramp(x_i)\ramp(x_j)$ and $\ramp(x_j)\ramp(x_i)$ are included as two distinct components in $g$,
and the corresponding coefficients are assumed to be identical to each other in $w$.

For any $\gamma\in \Gamma_{\rpQ}$ with $\pen_2(\gamma)=  b_3 $ and $\E_{P_{\theta^*}} h_{\gamma, \hat\mu}(x)=0$, the function
$ h_{\gamma, \hat\mu}(x) $ can be expressed as
\begin{align*}
h_{\gamma, \hat\mu}(x)= \beta_2^\T \left\{ \varphi_\rp(x-\hat\mu) - \beta_0 \right\} ,
\end{align*}
where $\beta_0 = \E_{P_{\theta^*}} \varphi_\rp(x-\hat\mu)$.
The mean-centered ramp functions in $\varphi_\rp(x-\hat\mu) - \beta_0 $  are bounded between $[-1,1]$, and hence
$ h_{\gamma, \hat\mu}(x) \in [- b_3  2p,  b_3  2p]$ similarly as above.
Moreover, such $ h_{\gamma, \hat\mu}(x) $ can be expressed in the form
$ \pen_2(\beta) w^\T \{g(x -\hat\mu) - \eta_0\} $, where $w \in \bbR^{q }$ is an $L_2$ unit vector,
$g(x): \bbR^p \to [0,1]^{q }$ is defined as above, and
$\eta_0 = \E_{P_{\theta^*}} g(x - \hat\mu) \in [0,1]^{q }$ by the boundedness of the ramp function in $[0,1]$.

Next, $K_f(P_{n}, P_{\hat\theta}; h_{\gamma, \hat\mu})$ can be bounded as
\begin{align}
& \quad  K_f(P_{n}, P_{\hat\theta}; h_{\gamma, \hat\mu}) \nonumber \\
& \ge  K_f(P_\epsilon, P_{\hat\theta}; h_{\gamma, \hat\mu})  - | K_f(P_{n}, P_{\hat\theta}; h_{\gamma, \hat\mu}) - K_f(P_\epsilon, P_{\hat\theta}; h_{\gamma, \hat\mu})| . \label{eq:pro-logit-L2v-lower-prf1}
\end{align}
For any $\gamma\in \Gamma_{\rpQ}$ with $\pen_2(\gamma)=  b_3 $, $E_{P_{\theta^*}} h_{\gamma, \hat\mu}(x)=0$, and $E_{P_{\hat \theta}} h_{\gamma, \hat\mu}(x)\le 0$,
applying {Lemma~\ref{lem:logit-L2-lower}(ii)} with $h= h_{\gamma, \hat\mu}$ and $b= 2 b_3^\dag = 2 b_3 \sqrt{p(p-1)}$ yields
\begin{align*}
& K_f(P_\epsilon, P_{\hat\theta}; h_{\gamma, \hat\mu})
\ge f^{\prime}(\me^{- 2 b_3^\dag })\epsilon \\
& \quad + R_{4, 2 b_3^\dag} \left\{ \E_{P_{\theta^*}} h(x) - \E_{P_{\hat\theta}} h(x) \right\} -
\frac{1}{2}  \left\{ R_{31, 2 b_3^\dag}\var_{P_{\theta^*}} h_{\gamma, \hat\mu}(x) + R_{32, 2 b_3^\dag}\var_{P_{\hat\theta}} h_{\gamma, \hat\mu}(x) \right\}.
\end{align*}
By {Lemma~\ref{lem:lip-gaus2}(ii)}, $\var_{P_{\theta^*}} h_{\gamma, \hat\mu}(x)$ can bounded as follows:
\begin{align*}
& \quad \var_{P_{\theta^*}} h_{\gamma, \hat\mu}(x) \\
& = \var_{P_{\theta^*}} \beta_2^\T \vec(\varphi_\rp (x-\hat\mu) \otimes \varphi_\rp(x-\hat\mu)) \\
& \le \| \beta_2\|_2^2 \cdot 20 C_{\mathrm{sg,12}}^2 (\sqrt{2})^2 p \| \Sigma^*\|_\op \\
& \le 40 \pen_2^2(\beta) C_{\mathrm{sg,12}}^2 p M_2 .
\end{align*}
Similarly, $\var_{P_{\theta^*}} h_{\gamma, \hat\mu}(x)$ can also be bounded by $40 \pen_2^2(\beta) C_{\mathrm{sg,12}}^2 p M_2 $, because $\| \hat \Sigma\|_\op \le M_2$.
Hence, with $\pen_2(\beta) = 2b_3$, $K_f(P_\epsilon, P_{\hat\theta}; h_{\gamma, \hat\mu}) $ can be bounded as
\begin{align}
& \quad K_f(P_\epsilon, P_{\hat\theta}; h_{\gamma, \hat\mu}) \nonumber \\
& \ge  f^{\prime}(\me^{- 2 b_3^\dag })\epsilon + R_{4, 2 b_3^\dag} \left\{ \E_{P_{\theta^*}} h(x) - \E_{P_{\hat\theta}} h(x) \right\} -80 C_{\mathrm{sg,12}}^2  p M_2 b_3^2 R_{3, 2 b_3^\dag},
\label{eq:pro-logit-L2v-lower-prf2}
\end{align}
where $R_{3,b} = R_{31,b} + R_{32,b}$ {as in Lemma~\ref{lem:logit-L1-lower}.}
Moreover, by Lemma~\ref{lem:logit-concen-L1-lower} with $b=2 b_3 $ and $g(x) \in [0,1]^{q }$ defined above, it holds with probability at least $1-2\delta$ that
for any $\gamma\in \Gamma_{\rpQ}$ with $\gamma_0=0$ and $\pen_2(\gamma)=  b_3 $,
\begin{align}
& \quad | K_f(P_{n}, P_{\hat\theta}; h_{\gamma, \hat\mu}) - K_f(P_\epsilon, P_{\hat\theta}; h_{\gamma, \hat\mu})| \nonumber \\
& \le  b_3  R_{2, b_3 \sqrt{q} } \left\{  C_{\mathrm{sg,12}} \sqrt{\frac{2 q V_g }{n}} + \sqrt{\frac{q  \log(\delta^{-1})}{n}} \right\} \nonumber \\
& =  b_3  R_{2, b_3^\dag} \left\{  C_{\mathrm{sg,12}} C_{\mathrm{rad3}} \sqrt{\frac{12p(p-1)}{n}} + \sqrt{\frac{p(p-1) \log(\delta^{-1})}{n}} \right\} , \label{eq:pro-logit-L2v-lower-prf3}
\end{align}
where $V_g  = 6 C_{\mathrm{rad3}}^2$ is determined in Lemma~\ref{lem:logit-concen-L1-lower} as follows.
For $j=1,\ldots, q$,
consider the function class $\mathcal{G}_j = \{g_{\mu, \eta_0, j}: \mu \in \bbR^p, \eta_0 \in [0,1]^q \}$,
where $\mu = (\mu_1,\ldots,\mu_p)^\T$, $\eta_0=(\eta_{01},\ldots,\eta_{0q})^\T$, and, as defined in Lemma~\ref{lem:logit-concen-L1-lower}, $g_{\mu,\eta_0,j}(x)$ is
of the form $\ramp(x_{j_1} - \mu_{j_1}) \ramp(x_{j_2} - \mu_{j_2}) - \eta_{0j}$ for $1 \le j_1 \not= j_2 \le p$.
By Lemma~\ref{lem:ramp-VC}, the VC index of moving-knot ramp functions is 2.
By Lemma~\ref{lem:intercept-VC}, the VC index of constant functions is also 2.
By applying Corollary~\ref{cor:entropy-sg2} (ii),
we obtain that conditionally on $(X_1,\ldots,X_n)$, the random variable
$Z_{n,j} = \sup_{\mu \in \bbR^p, \eta_0\in [0,1]^q }  |n^{-1} \sum_{i=1}^{n}\epsilon_i g_{\mu,\eta_0,j}(X_i) | = \sup_{f_j \in\mathcal{G}_j}  |n^{-1} \sum_{i=1}^{n}\epsilon_i f_j (X_i) |$
is sub-gaussian with tail parameter $C_{\mathrm{rad3}} \sqrt{6/n}$ for $j=1,\ldots,q$.

Combining the inequalities (\ref{eq:pro-logit-L2v-lower-prf1})--(\ref{eq:pro-logit-L2v-lower-prf3}) leads to the desired result.
\end{proof}

\begin{manualpro}{S12} \label{pro:logit-L2v-combine}
In the setting of {Proposition~\ref{pro:logit-L2-detailed} or Proposition~\ref{pro:hinge-L2-detailed}}, denote
\begin{align*}
D  = \sup_{\gamma \in \Gamma_{\rpQ}, \pen_2(\gamma)=1/2}
\left\{  \E_{P_{\theta^*}} h_{\gamma, \hat\mu}(x) - \E_{P_{\hat\theta}} h_{\gamma, \hat\mu}(x) \right\} .
\end{align*}
Then we have
{
\begin{align*}
& \| \hat \Sigma - \Sigma^* \|_\fro \le  2 M_2^{1/2} \sqrt{p}  \| \hat \sigma - \sigma^*\|_2  + S_{7} (\sqrt{2p} \Delta_{\hat\mu,\hat\sigma} + D ),
\end{align*}
where $\Delta_{\hat\mu,\hat\sigma} = ( \| \hat\mu-\mu^*\|_2^2 + \| \hat\sigma - \sigma^*\|_2^2 )^{1/2}$ and $S_{7}$ is defined as in Lemma~\ref{lem:local-linear2} with $M=M_2$. }
\end{manualpro}

\begin{proof}
For any $\gamma \in \Gamma_{\rpQ}$ with $\pen_2(\gamma) =1/2$, the function $ h_\gamma ( x) \in \mathcal{H}_\rpQ $ can be obtained as $h_{\rpQ, \beta }(x )$
with $ \pen_2 (\beta) = 1$.
First, we handle the effect of different means and standard deviations between $P_{\theta^*}$ and $P_{\hat\theta}$ in $D $. Denote
\begin{align*}
D^\dag = \sup_{\gamma \in \Gamma_{\rpQ}, \pen_2(\gamma)=1/2}
\left\{  \E_{P_{\theta^\dag}} h_{\gamma, \hat\mu}(x) - \E_{P_{\hat\theta}} h_{\gamma, \hat\mu}(x) \right\} ,
\end{align*}
{where %$P_{\theta^\dag }$ denotes $\N ( \hat\mu, \hat\sigma, \Sigma_0^*)$ and $P_{\theta^*}$ denotes $\N( \mu^*, \sigma^*, \Sigma_0^*)$.
$\theta^\dag = (\hat\mu, \diag(\hat\sigma) \Sigma_0^* \diag (\hat\sigma))$
and $\Sigma_0^*$ is defined as the correlation matrix such that $\Sigma^* = \diag(\sigma^*) \Sigma_0^* \diag (\sigma^*)$. }
Then $D^\dag$ can be related to $D $ as follows:
\begin{align}
D^\dag \le D  + \sqrt{2p} \Delta_{\hat\mu,\hat\sigma} , \label{eq:pro-spline-L2v-combine-prf1}
\end{align}
where $\Delta_{\hat\mu,\hat\sigma} = ( \| \hat\mu-\mu^*\|_2^2 + \| \hat\sigma - \sigma^*\|_2^2 )^{1/2}$.
In fact, by Lemma~\ref{lem:lip-matching-bd} with $g(x)$ set to $\varphi_\rp(x)$, which is {($1/2$)}-Lipschitz and componentwise bounded in $[0,1]$,
we have that for any $\gamma \in \Gamma_{\rpQ}$ with $\pen_2(\gamma) =1/2$,
\begin{align*}
\left| \E_{P_{\theta^\dag }}  h_{\gamma,\hat\mu}(x) -\E_{P_{\theta^* }} h_{\gamma,\hat\mu}(x) \right|
\le \sqrt{2p} \Delta_{\hat\mu,\hat\sigma} .
\end{align*}

For each pair $1 \le i \not=j \le p$, we restrict $h_{\gamma,\hat\mu}(x)$ such that $h_{\gamma}(x)$ is
$\ramp(x_i)\ramp(x_j)$ or $-\ramp(x_i)\ramp(x_j)$.
Applying Lemma~\ref{lem:local-linear2} shows
that there exists $r_{ij}\in \{-1,1\}$ such that
\begin{align*}
| \hat \rho_{ij} - \rho^*_{ij} | \hat\sigma_i \hat\sigma_j
\le S_{7} r_{ij} \left\{  \E_{P_{\theta^\dag}} h_{ij} (x-\hat\mu) - \E_{P_{\hat\theta}} h_{ij} (x-\hat\mu) \right\} ,
\end{align*}
where $h_{ij}(x) = \ramp(x_i) \ramp(x_j)$.
By the triangle inequality, we have
\begin{align*}
& \quad | \hat \rho_{ij}  \hat\sigma_i \hat\sigma_j  - \rho^*_{ij} \sigma_i^* \sigma_j^* |\\
& \le \hat\sigma_i | \hat\sigma_j -\sigma_j^*| + \sigma_j^* |\hat \sigma_i^* - \sigma_i^* | + | \hat \rho_{ij} - \rho^*_{ij} | \hat\sigma_i \hat\sigma_j \\
& \le M_2^{1/2} |\hat\sigma_i-\sigma_i^*| + M_2^{1/2} |\hat\sigma_j -\sigma_j^*|
+ S_{7} r_{ij} \left\{  \E_{P_{\theta^\dag}} h_{ij} (x-\hat\mu) - \E_{P_{\hat\theta}} h_{ij} (x-\hat\mu) \right\} . \nonumber
\end{align*}
In addition, we have $|\hat\sigma_i^2 - {\sigma_i^*}^2 | =  | (\hat \sigma_i + \sigma_i^*)( \hat \sigma_i - \sigma_i^*)| \le 2 M_2^{1/2}  | \hat \sigma_i - \sigma_i^*|$.
Then for any $L_2$ unit vector $w = (w_{ij}: 1\le i,j \le p)^\T \in \bbR^{p\times p}$,
\begin{align*}
& \quad \sum_{i=1}^p \left| w_{ii} (\hat\sigma_i^2 - {\sigma_i^*}^2) \right|  + \sum_{1\le i\not=j\le p}  \left| w_{ij} (\hat \rho_{ij}  \hat\sigma_i \hat\sigma_j  - \rho^*_{ij} \sigma_1^* \sigma_2^*) \right| \\
& \le 2 M_2^{1/2} \sum_{i=1}^p |w_{ii}| | \hat \sigma_i - \sigma_i^*| + M_2^{1/2}  \sum_{1\le i\not=j\le p} |w_{ij}| \left( |\hat\sigma_i-\sigma_i^*| + |\hat\sigma_j -\sigma_j^*|  \right)  \\
& \quad + S_{7} \sum_{1\le i\not=j\le p} |w_{ij}| r_{ij} \left\{  \E_{P_{\theta^\dag}} h_{ij} (x-\hat\mu) - \E_{P_{\hat\theta}} h_{ij} (x-\hat\mu) \right\} \\
& =  M_2^{1/2}  \sum_{1\le i,j\le p} |w_{ij}| \left( |\hat\sigma_i-\sigma_i^*| + |\hat\sigma_j -\sigma_j^*|  \right)
+ S_{7} \left\{ \E_{P_{\theta^\dag}} h (x - \hat\mu) - \E_{P_{\hat\theta}} h (x - \hat\mu) \right\},
\end{align*}
where $h(x) = \sum_{1\le i \not=j \le p} |w_{ij}| r_{ij} h_{ij} (x )$.
The function $h (x)$ can be expressed as $h_{\rpQ, \beta }(x )$ such that $\beta_{2,ii}=0$ for $i=1,\ldots,p$ and $\beta_{2,ij} = |w_{ij} | r_{ij}$ for $1 \le i \not=j \le p$,
and hence $\pen_2 ( \beta) \le \| w \|_2 =1$. By the definition of $D^\dag$, we have
$ \E_{P_{\theta^\dag}} h (x - \hat\mu) - \E_{P_{\hat\theta}} h (x - \hat\mu) \le D^\dag$.
Moreover, by the Cauchy--Schwartz inequality,
$ \sum_{1\le i,j\le p} |w_{ij}| |\hat\sigma_i-\sigma_i^*| \le \sqrt{p} \| \hat \sigma - \sigma^*\|_2 $.
Substituting these inequalities into the preceding display shows that
\begin{align}
& \quad \sum_{i=1}^p \left| w_{ii} (\hat\sigma_i^2 - {\sigma_i^*}^2) \right|
+ \sum_{1\le i\not=j\le p}  \left| w_{ij} (\hat \rho_{ij}  \hat\sigma_i \hat\sigma_j  - \rho^*_{ij} \sigma_1^* \sigma_2^*) \right| \nonumber \\
& \le 2 M_2^{1/2} \sqrt{p}  \| \hat \sigma - \sigma^*\|_2  + S_{7} D^\dag .  \label{eq:pro-spline-L2v-combine-prf2}
\end{align}
Combining (\ref{eq:pro-spline-L2v-combine-prf1}) and (\ref{eq:pro-spline-L2v-combine-prf2}) yields the desired result.
\end{proof}

%%%-----------------------------------------------------------------------------------------------------------------------------------------  Details for Hinge L1

\subsection{Details in main proof of Theorem \ref{thm:hinge-L1}}

For completeness, we restate Proposition \ref{pro:hinge-L1-detailed} in the main proof of Theorem \ref{thm:hinge-L1} below.
For $\delta \in (0,1/7)$, define
\begin{align*}
& \lambda_{11} = \sqrt{\frac{2\log({5 p}) + \log(\delta^{-1})}{n}} + \frac{2\log({5 p}) + \log(\delta^{-1})}{n} , \\
& \lambda_{12} =  2C_{\mathrm{rad4}} \sqrt{\frac{ \log(2p(p+1))}{n}}  + \sqrt{\frac{2 \log(\delta^{-1})}{n}} ,
\end{align*}
where $C_{\mathrm{rad4}}=C_{\mathrm{sg6}} C_{\mathrm{rad3}}$, depending on universal constants $C_{\mathrm{sg6}}$ and $C_{\mathrm{rad3}}$ in
Lemmas~\ref{lem:subg-max} and Corollary \ref{cor:entropy-sg2} in the Supplement. Denote
\begin{align*}
\err_{h1} (n,p, \delta, \epsilon) =  3\epsilon + 2\sqrt{\epsilon/(n\delta)} +  \lambda_{12} + \lambda_1 .
\end{align*}
\begin{manualpro}{3}
Assume that $\|\Sigma_{*}\|_{\max} \le M_1$ and $\epsilon \le 1/5$. Let $\hat \theta=(\hat\mu,\hat\Sigma)$ be a solution to (\ref{eq:hinge-gan-L1}) with $\lambda _1 \ge C_{\mathrm{sp13}} M_{11} \lambda_{11}$ where $M_{11} =  M_1^{1/2} ( M_1^{1/2} +  {2} \sqrt{2\pi})$ and $C_{\mathrm{sp13}} = (5/3) ( C_{\mathrm{sp11}} \vee C_{\mathrm{sp12}})$,
depending on universal constants $ C_{\mathrm{sp11}}$ and $ C_{\mathrm{sp12}}$ in Lemma~\ref{lem:spline-L1-upper} in the Supplement. If $\sqrt{\epsilon(1-\epsilon)/(n \delta) }  \leq 1/5$ and $\err_{h1} (n,p, \delta, \epsilon) \le a $ for a constant $a\in(0,1/2)$,
then the following holds with probability at least $1-7\delta$ uniformly over contamination distribution $Q$,
\begin{align*}
 \| \hat\mu - \mu^* \|_\infty & {\le} S_{4,a}  \err_{h1} (n,p, \delta, \epsilon), \\
 \| \hat\Sigma - \Sigma^*\|_{\max} & {\le} S_{8,a}  \err_{h1} (n,p, \delta, \epsilon) ,
\end{align*}
where $S_{4,a} =  (1 +\sqrt{2 M_1\log\frac{2}{1-2a}})  /a $
and $S_{8,a} =  2 M_1^{1/2} S_{6,a} + S_{7} ( 1+ S_{4,a} + S_{6,a}) $ with
$S_{6,a}= S_{5} ( 1 + S_{4,a}/2)$, $S_{5} = 2\sqrt{2\pi} ( 1- \me^{-2/M_1})^{-1} $, and
$S_{7} =8 \pi M_1  \me^{1/(4M_1)}$
$S_{7} =4 \{ (\frac{1}{\sqrt{2\pi M_1} } \me^{- 1/(8M_1)} ) \vee (1 - 2 \me^{- 1/(8M_1)} ) \}^{-2} $.
\end{manualpro}

 To see why Proposition~\ref{pro:hinge-L1-detailed} leads to Theorem~\ref{thm:hinge-L1}, we show that conditions in Proposition~\ref{pro:hinge-L1-detailed} are satisfied under the setting of Theorem~\ref{thm:hinge-L1}. For a constant $a \in (0, 1/2)$, let
 \begin{align*}
     C_1 &= 2\sqrt{3}C_{\mathrm{sp13}} M_{11},\\
     C_2 &= \frac{1}{5} \wedge \frac{\sqrt{3}C_1}{6} \wedge \left( 3\vee\left(\frac{\sqrt{2}C_1}{4C_{\mathrm{rad4}}+2}+1\right)\right),\\
     C &= S_{8, a} \left( 3\vee\left(\frac{\sqrt{2}C_1}{4C_{\mathrm{rad4}}+2}+1\right)\right) .
 \end{align*}
  Then the following conditions
 \begin{itemize}
     \item[(i)] $\lambda_1 \ge C_1 \left( \sqrt{\log{p}/{n}} + \sqrt{\log{(1/\delta)}/{n}}\right)$,
     \item[(ii)] $\epsilon + \sqrt{\epsilon/(n\delta)} + \lambda_{1} \le C_2$,
\end{itemize}
imply the conditions
\begin{itemize}
     \item[(iii)] $\lambda _1 \ge C_{\mathrm{sp13}} M_{11} \lambda_{11}$,
     \item[(iv)] $\err_{h1} (n,p, \delta, \epsilon) \le a $,
     \item[(v)] $\epsilon \le 1/5$ and $\sqrt{\epsilon(1-\epsilon)/(n\delta)} \le 1/5$.
 \end{itemize}
 In fact, condition (v) follows directly from condition (ii) because $C_2 \le 1/5$. For conditions (iii) and (iv), we first show that $\lambda_{11}$ and $\lambda_{12}$ can be upper bounded as follows:
\begin{align}
    \lambda_{11} &\le \sqrt{\frac{2\log{p} + 3\log(\delta^{-1})}{n}}  + {\frac{2\log{p} + 3\log(\delta^{-1})}{n}}  \label{eq:hinge-L1-aux-1}  \\
    & \le  2\sqrt{\frac{2\log{p} + 3\log(\delta^{-1})}{n}}  \le 2\sqrt{\frac{2\log{p}}{n}}  + 2\sqrt{\frac{3\log(\delta^{-1})}{n}},   \label{eq:hinge-L1-aux-2}
\end{align}
and
\begin{align}
    \lambda_{12} &\le 2C_{\mathrm{rad4}} \sqrt{\frac{2\log{2} + 2\log{p}}{n}}  + \sqrt{\frac{2 \log(\delta^{-1})}{n}} \label{eq:hinge-L1-aux-3}  \\
     &\le 2C_{\mathrm{rad4}} \sqrt{\frac{ 2\log{p}}{n}}  + \left(2C_{\mathrm{rad4}}+1\right)\sqrt{\frac{2 \log(\delta^{-1})}{n}}  \label{eq:hinge-L1-aux-4}.
\end{align}
Lines (\ref{eq:hinge-L1-aux-1}) and (\ref{eq:hinge-L1-aux-4}) hold because $\log(1/\delta) \ge \log(5)$ for $\delta \in (0, 1/7)$. Line (\ref{eq:hinge-L1-aux-2}) holds because $\sqrt{\frac{2\log{p} + 3\log(\delta^{-1})}{n}} \le 1$ and hence the linear term in $\lambda_{11}$ is upper bounded by the square root term. To see this, by conditions (i) and (ii) we have
\begin{align*}
    \sqrt{\frac{2\log{p} + 3\log(\delta^{-1})}{n}} &\le \sqrt{\frac{3\log{p}}{n}} + \sqrt{\frac{3\log(\delta^{-1})}{n}}
     \le \frac{\sqrt{3}\lambda_1}{C_1}  \le \frac{\sqrt{3}C_2}{C_1} \le 1.
\end{align*}
Line (\ref{eq:hinge-L1-aux-3}) holds because $\log{(2p(p+1))} \le 2\log{2} + 2\log{p}$ for $p \ge 1$. With the above upper bounds for $\lambda_{11}$ and $\lambda_{12}$, we show that condition (i) implies condition (iii) as follows:
\begin{align*}
    \lambda _1 &\ge C_{1} \left( \sqrt{\frac{\log{p}}{n}} + \sqrt{\frac{\log{(1/\delta)}}{n}}\right) \\
    & = 2\sqrt{3}C_{\mathrm{sp13}} M_{11} \left( \sqrt{\frac{\log{p}}{n}} + \sqrt{\frac{\log{(1/\delta)}}{n}}\right)
    \ge C_{\mathrm{sp13}} M_{11} \lambda_{11} \nonumber,
\end{align*}
and condition (ii) implies condition (iv) as follows:
\begin{align*}
    \err_{h1} (n,p, \delta, \epsilon) &=  3\epsilon + 2\sqrt{\epsilon/(n\delta)} +  \lambda_{12} + \lambda_1  \nonumber\\
    &\le  3\left(\epsilon +\sqrt{\frac{\epsilon}{n\delta}}\right) + \left(\frac{\sqrt{2}C_1}{4C_{\mathrm{rad4}}+2}+1\right)\lambda_1 \\
    &\le  \left( 3\vee\left(\frac{\sqrt{2}C_1}{4C_{\mathrm{rad4}}+2}+1\right)\right) \left(\epsilon +\sqrt{\frac{\epsilon}{n\delta}} + \lambda_1\right) \nonumber \\
    & \le  \left( 3\vee\left(\frac{\sqrt{2}C_1}{4C_{\mathrm{rad4}}+2}+1\right)\right) C_2
    \le a.
\end{align*}
Therefore, Proposition~\ref{pro:hinge-L1-detailed} implies Theorem~\ref{thm:hinge-L1} with constant $C = S_{8, a} \left( 3\vee\left(\frac{\sqrt{2}C_1}{4C_{\mathrm{rad4}}+2}+1\right)\right)$.

\begin{lem} \label{lem:hinge-upr-bound}
    Consider the hinge GAN (\ref{eq:hinge-gan}).

    (i) For any $\epsilon \in [0, 1]$ and any function $h: \bbR^p \to \bbR$, we have
    \begin{align*}
    K_{\HG}(P_{\epsilon}, P_{\theta^*}; h) \le 2\epsilon .
    \end{align*}

    (ii) For any function $h: \bbR^p \to \bbR$, we have
    \begin{align}
    K_{\HG}(P_n, P_{\theta^{*}} ; h) \le 2 \hat\epsilon + | \E_{P_{\theta^{*},n}} h(x) - \E_{P_{\theta^{*}}} h(x) | ,\label{eq:lem-logit-upper-bound-emp}
    \end{align}
    where {$\hat{\epsilon} = n^{-1} \sum_{i=1}^{n} {U_i}$ and} $P_{\theta^*,n}$ denotes the empirical distribution of $\{X_i: {U_i}=0,i=1,\ldots,n\}$ in the latent representation of Huber's contamination model.
    \end{lem}

    \begin{proof}
    (i) For any $h: \bbR^p \to \bbR$ we have
    \begin{align}
    &\quad K_{\HG}(P_{\epsilon}, P_{\theta^*}; h) \nonumber\\
    &= \epsilon \E_{Q} \min(1, h(x)) + (1-\epsilon) \E_{P_{\theta^*}} \min(1, h(x)) + \E_{P_{\theta^*}} \min(1, -h(x)) \nonumber\\
    &\le \epsilon + (1-\epsilon)\E_{P_{\theta^*}} \min(1, h(x)) + \E_{P_{\theta^*}} \min(1, -h(x)) \label{eq: lem-hinge-upr-bound-prf1}\\
    &\le 2\epsilon + (1-\epsilon)\left\{\E_{P_{\theta^*}} \min(1, h(x)) + \E_{P_{\theta^*}} \min(1, -h(x))\right\} \label{eq: lem-hinge-upr-bound-prf2}\\
    &\le 2\epsilon. \label{eq: lem-hinge-upr-bound-prf3}
    \end{align}
    Inequalities (\ref{eq: lem-hinge-upr-bound-prf1}) and (\ref{eq: lem-hinge-upr-bound-prf2}) hold because $\min(1, u) \vee \min(1, -u) \le 1$ for all $u \in \bbR$. Inequality (\ref{eq: lem-hinge-upr-bound-prf3}) holds because $\min(1 ,u) + \min(1, -u) \le 0$ for all $u \in \bbR$.

    (ii) Because both $\min(1, u)$ and $\min(1, -u)$ are concave in $u \in \bbR$ and upper bounded by $1$, we have
    \begin{align}
    & \quad K_{\HG}(P_n, P_{\theta^{*}} ; h) \nonumber \\
    & = \frac{1}{n}\sum_{i=1}^{n}R_i \min(1, h(X_i)) + \frac{1}{n}\sum_{i=1}^{n}(1-R_i)\min(1, h(X_i))   + \E_{P_{\theta^{*}}}  \min(1, -h(X_i))  \nonumber\\
    &\le \hat{\epsilon} + (1-\hat{\epsilon})\E_{P_{\theta^{*}, n}} \min(1, h(X_i))  + \E_{P_{\theta^{*}}}  \min(1, -h(X_i))  \label{eq:lem-hinge-upr-bound-prf4} \\
    &\le \hat{\epsilon} + (1-\hat{\epsilon})\min(1, \E_{P_{\theta^{*}, n}} h(X_i))  +  \min(1, -\E_{P_{\theta^{*}}} h(X_i))  \label{eq:lem-hinge-upr-bound-prf5} \\
    &\le 2\hat{\epsilon} + (1-\hat{\epsilon})\left\{\min(1, \E_{P_{\theta^{*}, n}} h(X_i)) + \min(1, -\E_{P_{\theta^{*}}} h(X_i))\right\}  \label{eq:lem-hinge-upr-bound-prf6} \\
    &\le 2\hat{\epsilon} + 0 + |\min(1, -\E_{P_{\theta^{*},n}} h(x)) - \min(1, -\E_{P_{\theta^{*}}} h(x)) | \label{eq:lem-hinge-upr-bound-prf7}\\
    &\le 2\hat{\epsilon} + 0 + |\E_{P_{\theta^{*},n}} h(x) - \E_{P_{\theta^{*}}} h(x) | \label{eq:lem-hinge-upr-bound-prf8}
    \end{align}
    Lines (\ref{eq:lem-hinge-upr-bound-prf4}) and (\ref{eq:lem-hinge-upr-bound-prf6}) hold because $\min(1, u) \vee \min(1, -u) \le 1$ for all $u \in \bbR$.
    Line (\ref{eq:lem-hinge-upr-bound-prf5}) follows from Jensen's inequality by the concavity of $\min(1, u)$ and $\min(1, -u)$. Line (\ref{eq:lem-hinge-upr-bound-prf7}) follows because $\min(1 ,u) + \min(1, -u) \le 0$ for all $u \in \bbR$, and the last line (\ref{eq:lem-hinge-upr-bound-prf8}) holds because $\min(1, -u)$ is $1$-Lipschitz {in $u$}.
    \end{proof}

    \begin{manualpro}{S13} \label{pro:hinge-L1-upper}
    In the setting of Proposition~\ref{pro:hinge-L1-detailed}, it holds with probability at least $1-{5} \delta$ that for any $\gamma\in\Gamma$,
    \begin{align*}
    & K_\HG(P_{n}, P_{\theta^* }; h_{\gamma, \mu^*})
    \le 2(\epsilon + \sqrt{\epsilon/(n\delta)} ) + \pen_1(\gamma) C_{\mathrm{sp13}}  M_{11}  \lambda_{11} ,
    \end{align*}
    where $C_{\mathrm{sp13}} = (5/3) ( C_{\mathrm{sp11}} \vee C_{\mathrm{sp12}})$ with $C_{\mathrm{sp11}}$ and  $C_{\mathrm{sp12}}$ as in Lemma~\ref{lem:spline-L1-upper}, $M_{11} =  M_1^{1/2} ( M_1^{1/2} +  2\sqrt{2\pi})$, and
    $$
    \lambda_{11} = \sqrt{\frac{2\log(5p) + \log(\delta^{-1})}{n}} + \frac{2\log(5p)+ \log(\delta^{-1})}{n} .
    $$
    \end{manualpro}

    \begin{proof}
    The proof is similar to that of Proposition~\ref{pro:logit-L1-upper} and we {use the same definition of $\Omega_1$ and $\Omega_2$}.
    In {the} event $\Omega_1$ we have $ |\hat\epsilon -\epsilon | \le 1/5$ by the assumption $\sqrt{\epsilon(1-\epsilon)/(n\delta)} \le 1/5$
    and hence $\hat\epsilon \le 2/5$ by the assumption $\epsilon \le 1/5$.
    Thus, by Lemma~\ref{lem:hinge-upr-bound} with $\epsilon_1=2/5$, it holds in {the} event $\Omega_1$ that for any $\gamma\in\Gamma$,
    \begin{align}
    & \quad K_\HG(P_n, P_{\theta^* } ; h_{\gamma, \mu^*}) \nonumber \\
    & \le 2 \hat\epsilon + \left| \E_{P_{\theta^* ,n}} h_{\gamma, \mu^*}(x) - \E_{P_{\theta^* }} h_{\gamma, \mu^*}(x) \right| \nonumber \\
    & \le 2 (\epsilon + \sqrt{\epsilon/(n\delta)} ) + \left| \E_{P_{\theta^*,n}} h_{\gamma}(x-\mu^*) - \E_{P_{(0,\Sigma^*)}} h_{\gamma}(x) \right|. \label{pro:hinge-spline-L1-upper-prf1}
    \end{align}
    The last step uses the fact that
    $ \E_{P_{\theta^* }} h_{\gamma, \mu^*}(x) = \E_{P_{(0,\Sigma^*)}} h_{\gamma}(x)$ and $\E_{P_{\theta^* ,n}} h_{\gamma, \mu^*}(x) =  \E_{P_{\theta^*,n}} h_{\gamma}(x-\mu^*)$,
    by the definition $h_{\gamma, \mu^*}(x) = h_{\gamma} ( x- \mu^*)$.

    {Next, as shown in Proposition~\ref{pro:logit-L1-upper}, it holds in the event $\Omega_2$ while conditionally on $\Omega_1$ that for any} $\gamma = (\gamma_0, \gamma_1^\T, \gamma_2^\T)^\T \in \Gamma$,
    \begin{align}
    & \quad \left| \E_{P_{\theta^*,n}} h_{\gamma}(x-\mu^*) - \E_{P_{(0,\Sigma^*)}} h_{\gamma}(x) \right| \nonumber \\
    & \le \pen_1(\gamma) (5/3) ( C_{\mathrm{sp11}} \vee C_{\mathrm{sp12}}) M_{11} \lambda_{11} , \label{pro:hinge-spline-L1-upper-prf2}
    \end{align}
    where $h_{\gamma}(x) = \gamma_0 +  \gamma_1^\T \varphi (x)+ \gamma_2^\T (\varphi(x)\otimes \varphi(x))$ and
    $\pen_1 (\gamma) = \|\gamma_1\|_1 + \|\gamma_2\|_1$.
    Combining (\ref{pro:hinge-spline-L1-upper-prf1}) and (\ref{pro:hinge-spline-L1-upper-prf2}) indicates that in {the event $\Omega_1 \cap \Omega_2$ with  probability at least $1-{5}\delta$},
    the desired inequality holds for any $\gamma\in\Gamma$.
    \end{proof}

    \begin{manualpro}{S14} \label{pro:hinge-L1-lower} In the setting of Proposition~\ref{pro:hinge-L1-detailed}, it holds with probability at least $1-2 \delta$ that for any $\gamma\in \Gamma_{\rp}$ with $\gamma_0=0$ and $\pen_1(\gamma)=  1 $,
    \begin{align*}
    & \quad  K_\HG(P_{n}, P_{\hat\theta}; h_{\gamma, \hat\mu}) \\
    & \ge \left\{\E_{P_{\theta^*}} h_{\gamma, \hat\mu}(x) - \E_{P_{\hat\theta}} h_{\gamma, \hat\mu}(x) \right\} - \epsilon -  \lambda_{12}
    \end{align*}
    where, with $C_{\mathrm{rad4}}=C_{\mathrm{sg6}} C_{\mathrm{rad3}}$ is the same constant in Proposition~\ref{pro:logit-L1-lower},
    $$
    \lambda_{12} = C_{\mathrm{rad4}} \sqrt{\frac{4 \log(2p(p+1))}{n}} + \sqrt{\frac{2 \log(\delta^{-1})}{n}} ,
    $$
    depending on the universal constants $C_{\mathrm{sg6}}$ and $ C_{\mathrm{rad3}}$ in Lemma~\ref{lem:subg-max} and Corollary~\ref{cor:entropy-sg2}.
    \end{manualpro}

    \begin{proof}
    By definition, for any $\gamma\in \Gamma_{\rp}$, $h_\gamma(x)$ can be represented as
    $h_{\rp,\beta, c }(x)$ such that $\beta_0=\gamma_0$ and $\pen_1 (\beta) = \pen_1(\gamma)$ in the same way as in Proposition~\ref{pro:logit-L1-lower}.
    Then for any $\gamma\in \Gamma_{\rp}$ with $\gamma_0=0$ and $\pen_1(\gamma)=  1 $, we have $\beta_0=0$ and $\pen_1(\beta) =  1$ correspondingly,
    and hence $h _\gamma (x) = h_{\rp, \beta,  c }(x) \in [- 1 ,  1 ]$ by the boundedness of the ramp function in $[0, 1]$.
    Moreover, $h_{\rp,\beta, c }(x)$ with $\beta_0=0$ and $\pen_1 (\beta) =  1 $ can be expressed in the form
    $ w^\T g(x)$, where for $q =2 p + p(p-1)$, $w \in \bbR^{q }$ is an $L_1$ unit vector, and
    $g: \bbR^p \to [0,1]^{q }$ is a vector of functions including $\ramp(x_j)$ and $\ramp(x_j-1)$ for $j=1,\ldots,p$, and $\ramp(x_i)\ramp(x_j)$ for $1 \le i\not= j\le p$.
    For symmetry, $\ramp(x_i)\ramp(x_j)$ and $\ramp(x_j)\ramp(x_i)$ are included as two distinct components in $g$,
    and the corresponding coefficients are identical to each other in $w$.

    Next, $K_\HG(P_{n}, P_{\hat\theta}; h_{\gamma, \hat\mu})$ can be bounded as
    \begin{align}
    & \quad  K_\HG(P_{n}, P_{\hat\theta}; h_{\gamma, \hat\mu}) \nonumber\\
    & \ge  K_\HG(P_\epsilon, P_{\hat\theta}; h_{\gamma, \hat\mu})  - \{ K_\HG(P_{n}, P_{\hat\theta}; h_{\gamma, \hat\mu}) - K_\HG(P_\epsilon, P_{\hat\theta}; h_{\gamma, \hat\mu})\} . \label{eq:hinge-L1-lower-pro-1}
    \end{align}
    For any $\gamma\in \Gamma_{\rp}$ with $\gamma_0=0$ and $\pen_1(\gamma)=  1 $, because $h_{\gamma,\hat{\mu}}(x) \in [-1, 1]$, we have $\min(h_{\gamma,\hat{\mu}}(x),1)=h_{\gamma,\hat{\mu}}(x)$ and $\min(-h_{\gamma,\hat{\mu}}(x),1)=-h_{\gamma,\hat{\mu}}(x)$. Hence the {hinge} term $ K_\HG(P_\epsilon, P_{\hat\theta}; h_{\gamma, \hat\mu}) $ in (\ref{eq:hinge-L1-lower-pro-1}) reduces to a moment matching term and can be lower bounded as follows:
    \begin{align*}
    &\quad  K_\HG(P_\epsilon, P_{\hat\theta}; h_{\gamma, \hat\mu}) \\
    &= \E_{P_\epsilon} \min(h_{\gamma,\hat{\mu}}(x), 1) + \E_{P_{\hat\theta}} \min(-h_{\gamma,\hat{\mu}}(x), 1)\\
    &= \E_{P_\epsilon} h_{\gamma,\hat{\mu}}(x) - \E_{P_{\hat\theta}} h_{\gamma,\hat{\mu}}(x)\\
    &=\epsilon\E_{Q} h_{\gamma,\hat{\mu}}(x) + (1-\epsilon)\E_{P_\theta^*} h_{\gamma,\hat{\mu}}(x)- \E_{P_{\hat\theta}} h_{\gamma,\hat{\mu}}(x)\\
    &\ge -\epsilon + \left\{\E_{P_\theta^*} h_{\gamma,\hat{\mu}}(x)- \E_{P_{\hat\theta}} h_{\gamma,\hat{\mu}}(x)\right\}.
    \end{align*}
    Similarly, the two other {hinge} terms in (\ref{eq:hinge-L1-lower-pro-1}) also reduce to moment matching terms:
    \begin{align*}
    &\quad \{ K_\HG(P_{n}, P_{\hat\theta}; h_{\gamma, \hat\mu}) - K_\HG(P_\epsilon, P_{\hat\theta}; h_{\gamma, \hat\mu})\}\\
    &= \left\{\E_{P_n} h_{\gamma,\hat{\mu}}(x)- \E_{P_{\hat\theta}} h_{\gamma,\hat{\mu}}(x)\right\} - \left\{\E_{P_\epsilon} h_{\gamma,\hat{\mu}}(x)- \E_{P_{\hat\theta}} h_{\gamma,\hat{\mu}}(x)\right\}\\
    &=\E_{P_n} h_{\gamma,\hat{\mu}}(x)- \E_{P_\epsilon} h_{\gamma,\hat{\mu}}(x).
    \end{align*}
     We apply Lemma~\ref{lem:logit-concen-L1-lower} {with $b= 1 $, $g(x)\in [0,1]^{q }$ defined above, and $f^{\prime}(\me^u)$ and $f^{\#}(\me^u)$ replaced by the identity function in $u$.} It holds with probability at least $1-2\delta$ that
    for any $\gamma\in \Gamma_{\rp}$ with $\gamma_0=0$ and $\pen_1(\gamma)=  b_1 $,
    \begin{align*}
    & \quad \E_{P_n} h_{\gamma,\hat{\mu}}(x)- \E_{P_\epsilon} h_{\gamma,\hat{\mu}}(x) \\
    & \le    C_{\mathrm{sg6}} \sqrt{\frac{V_g \log(2q )}{n}} + \sqrt{\frac{2 \log(\delta^{-1})}{n}}  \\
    & =   C_{\mathrm{sg6}} C_{\mathrm{rad3}}\sqrt{\frac{4 \log(2p(p+1))}{n}} + \sqrt{\frac{2 \log(\delta^{-1})}{n}} ,
    \end{align*}
    as shown in Proposition~\ref{pro:logit-L1-lower}. Combining the preceding three displays leads to the desired result.
    \end{proof}

%%%-----------------------------------------------------------------------------------------------------------------------------------------  Details for Hinge L2
\subsection{Details in main proof of Theorem~\ref{thm:hinge-L2}}

    \begin{manualpro}{S15} \label{pro:hinge-L2-upper}
    In the setting of Proposition~\ref{pro:hinge-L2-detailed}, it holds with probability at least $1-{4} \delta$ that for any $\gamma = (\gamma_0,\gamma_1, \gamma_2)^\T \in /Gamma$,
    \begin{align*}
    K_\HG(P_{n}, P_{\theta^* }; h_{\gamma, \mu^*})
    & \le 2(\epsilon + \sqrt{\epsilon/(n\delta)} ) + \pen_2(\gamma_1) (5/3) C_{\mathrm{sp21}} M_2^{1/2} R_1  \lambda_{21} \\
    & \quad   +  \pen_2 (\gamma_2) {(25\sqrt{5}/3)} C_{\mathrm{sp22}} M_{21} R_1 \sqrt{p} \lambda_{31} ,
    \end{align*}
    where $C_{\mathrm{sp21}}$ and $C_{\mathrm{sp22}}$ are defined as in Lemma \ref{lem:spline-L2-upper}, $M_{21} = M_2^{1/2} (M_2^{1/2} + 2 \sqrt{2\pi}) $, and
    \begin{align*}
    \lambda_{21} = \sqrt{\frac{5 p + \log(\delta^{-1})}{n}}, \quad \lambda_{31} = \lambda_{21} + \frac{5 p+ \log(\delta^{-1})}{n} .
    \end{align*}
    \end{manualpro}

    \begin{proof}
    The proof is similar to that of Proposition~\ref{pro:logit-L2-upper} and we use the same definition of $\Omega_1$ and $\Omega_2$. In the event $\Omega_1$ we have $ |\hat\epsilon -\epsilon | \le 1/5$ by the assumption $\sqrt{\epsilon(1-\epsilon)/(n\delta)} \le 1/5$
    and hence $\hat\epsilon \le 2/5$ by the assumption $\epsilon \le 1/5$.
    By Lemma~\ref{lem:hinge-upr-bound} with $\epsilon_1=2/5$, it holds that {in the event $\Omega_1$} for any $\gamma\in\Gamma$,
    \begin{align}
    & \quad K_\HG(P_n, P_{\theta^* } ; h_{\gamma, \mu^*}) \nonumber \\
    & \le 2 \hat\epsilon + \left| \E_{P_{\theta^* ,n}} h_{\gamma, \mu^*}(x) - \E_{P_{\theta^* }} h_{\gamma, \mu^*}(x) \right| \nonumber \\
    & \le 2 (\epsilon + \sqrt{\epsilon/(n\delta)} ) + \left| \E_{P_{\theta^*,n}} h_{\gamma}(x-\mu^*) - \E_{P_{(0,\Sigma^*)}} h_{\gamma}(x) \right| . \label{eq:pro-hinge-L2-upper-prf1}
    \end{align}
    The last step uses the fact that
    $ \E_{P_{\theta^* }} h_{\gamma, \mu^*}(x) = \E_{P_{(0,\Sigma^*)}} h_{\gamma}(x)$ and $\E_{P_{\theta^* ,n}} h_{\gamma, \mu^*}(x) =  \E_{P_{\theta^*,n}} h_{\gamma}(x-\mu^*)$,
    by the definition $h_{\gamma, \mu^*}(x) = h_{\gamma} ( x- \mu^*)$.

    Next, as shown in Proposition~\ref{pro:logit-L2-upper}, it holds {in event $\Omega_2$ while conditionally on $\Omega_1$} that for any $\gamma = (\gamma_0, \gamma_1^\T, \gamma_2^\T)^\T \in \Gamma$,
    \begin{align}
    & \quad \left| \E_{P_{\theta^*,n}} h_{\gamma}(x-\mu^*) - \E_{P_{(0,\Sigma^*)}} h_{\gamma}(x) \right| \nonumber \\
    & \le \pen_2 (\gamma_1) (5/3) C_{\mathrm{sp21}}  M_2^{1/2} \lambda_{21} +
    \pen_2 (\gamma_2) (5/3) C_{\mathrm{sp22}}  3 M_{21} \sqrt{3 p} \lambda_{31}, \label{eq:pro-hinge-L2-upper-prf2}
    \end{align}
    where $h_{\gamma}(x) = \gamma_0 +  \gamma_1^\T \varphi (x) + \gamma_2^\T (\varphi(x) \otimes \varphi(x))$,
    $\pen_2 (\gamma_1) = \|\gamma_1\|_2$, and $\pen_2 (\gamma_2) = \|\gamma_2\|_2$.
    Combining (\ref{eq:pro-hinge-L2-upper-prf1}) and (\ref{eq:pro-hinge-L2-upper-prf2}) indicates that, {in the event $\Omega_1 \cap \Omega_2$  with probability at least $1-4\delta$}, the desired inequality holds for any $\gamma\in\Gamma$.
    \end{proof}

\begin{manualpro}{S16} \label{pro:hinge-L2-lower} In the setting of Proposition~\ref{pro:hinge-L2-detailed}, it holds with probability at least $1-2 \delta$ that for any $\gamma\in \Gamma_{10}$,
    \begin{align*}
    & \quad  K_\HG(P_{n}, P_{\hat\theta}; h_{\gamma, \hat\mu}) \\
    & \ge \left\{\E_{P_{\theta^*}} h_{\gamma, \hat\mu}(x) - \E_{P_{\hat\theta}} h_{\gamma, \hat\mu}(x) \right\} - \epsilon -  \lambda_{22}(2p)^{-1/2}
    \end{align*}
    where, with $C_{\mathrm{rad5}}  = C_{\mathrm{sg,12}} C_{\mathrm{rad3}} $, and
    $$
    \lambda_{22} = C_{\mathrm{rad5}}  \sqrt{\frac{16 p}{n}} + \sqrt{\frac{2p \log(\delta^{-1})}{n}} ,
    $$
    depending on the universal constants $C_{\mathrm{sg, 12}}$ and $ C_{\mathrm{rad3}}$ in Lemma~\ref{lem:subg} and Corollary~\ref{cor:entropy-sg2}.
    \end{manualpro}

    \begin{proof}
    For any $\gamma\in \Gamma_{10} \subset \Gamma_{\rpL}$, because
    $\pen_2(\gamma)=  (2p)^{-1/2}$ and $\Gamma_0 = 0$, we have $\beta_0=0$ and $\pen_2(\beta)=  p^{-1/2}$. Hence, $h_\gamma(x) = h_{\rpL, \beta,  c }(x)\in [- 1,  1]$
    by the Cauchy--Schwartz inequality and the boundedness of the ramp function in $[0,1]$.
    {Then the mean-centered version $h_{\gamma, \hat\mu}(x)$, with $\E_{P_{\theta^*}} h_{\gamma, \hat\mu}(x)=0$, is also bounded in $[-1,1]$.
    Moreover, such $ h_{\gamma, \hat\mu}(x) $ can be expressed in the form
$ \pen_2(\beta) w^\T \{g(x -\hat\mu) - \eta_0\} $, where, for $q=2p$, $w \in \bbR^{q }$ is an $L_2$ unit vector,
$g: \bbR^p \to [0,1]^{q }$ is a vector of functions, including $\ramp(x_j)$ and $\ramp(x_j-1)$ for $j=1,\ldots,p$, and
$\eta_0 = \E_{P_{\theta^*}} g(x - \hat\mu) \in [0,1]^{q }$.}

    Next, $K_\HG(P_{n}, P_{\hat\theta}; h_{\gamma, \hat\mu})$ can be bounded as
    \begin{align}
    & \quad  K_\HG(P_{n}, P_{\hat\theta}; h_{\gamma, \hat\mu}) \nonumber\\
    & \ge  K_\HG(P_\epsilon, P_{\hat\theta}; h_{\gamma, \hat\mu})  - \{ K_\HG(P_{n}, P_{\hat\theta}; h_{\gamma, \hat\mu}) - K_\HG(P_\epsilon, P_{\hat\theta}; h_{\gamma, \hat\mu})\} \label{eq:pro-hinge-L2-lower-prf0}.
    \end{align}
    For any $\gamma\in \Gamma_{10}$, {because $h_{\gamma,\hat{\mu}}(x) \in [-1, 1]$, we have $\min(h_{\gamma,\hat{\mu}}(x),1)=h_{\gamma,\hat{\mu}}(x)$ and $\min(-h_{\gamma,\hat{\mu}}(x),1)=-h_{\gamma,\hat{\mu}}(x)$. Then the hinge term $K_\HG(P_{n}, P_{\hat\theta}; h_{\gamma, \hat\mu}) $ reduces to a moment matching term and can be lower bounded as follows:}
    \begin{align}
    &\quad  K_\HG(P_\epsilon, P_{\hat\theta}; h_{\gamma, \hat\mu})\nonumber \\
    &= \E_{P_\epsilon} \min(h_{\gamma,\hat{\mu}}(x), 1) + \E_{P_{\hat\theta}} \min(-h_{\gamma,\hat{\mu}}(x), 1)\nonumber\\
    &= \E_{P_\epsilon} h_{\gamma,\hat{\mu}}(x) - \E_{P_{\hat\theta}} h_{\gamma,\hat{\mu}}(x)\nonumber\\
    &=\epsilon\E_{Q} h_{\gamma,\hat{\mu}}(x) + (1-\epsilon)\E_{P_\theta^*} h_{\gamma,\hat{\mu}}(x)- \E_{P_{\hat\theta}} h_{\gamma,\hat{\mu}}(x) \nonumber\\
    &\ge -\epsilon + \left\{\E_{P_\theta^*} h_{\gamma,\hat{\mu}}(x)- \E_{P_{\hat\theta}} h_{\gamma,\hat{\mu}}(x)\right\} \label{eq:pro-hinge-L2-lower-prf1}.
    \end{align}
    Similarly, the {absolute difference term in (\ref{eq:pro-hinge-L2-lower-prf0}) can be simplified} as follows:
    \begin{align*}
    &\quad \{ K_\HG(P_{n}, P_{\hat\theta}; h_{\gamma, \hat\mu}) - K_\HG(P_\epsilon, P_{\hat\theta}; h_{\gamma, \hat\mu})\}\\
    &=  \left\{\E_{P_n} h_{\gamma,\hat{\mu}}(x)- \E_{P_{\hat\theta}} h_{\gamma,\hat{\mu}}(x)\right\} - \left\{\E_{P_\epsilon} h_{\gamma,\hat{\mu}}(x)- \E_{P_{\hat\theta}} h_{\gamma,\hat{\mu}}(x)\right\}\\
    &=\E_{P_n} h_{\gamma,\hat{\mu}}(x)- \E_{P_\epsilon} h_{\gamma,\hat{\mu}}(x).
    \end{align*}
    {We apply Lemma~\ref{lem:logit-concen-L2-lower} with $b= 1$, $g(x)\in [0,1]^{q }$ and $\eta_0$ defined above, and
    $f^{\prime}(\me^u)$ and $f^{\#}(\me^u)$ replaced by the identity function in $u$.}
    It holds with probability at least $1-2\delta$ that
    for any $\gamma\in \Gamma_{01}$,
    \begin{align}
    & \quad \E_{P_n} h_{\gamma,\hat{\mu}}(x)- \E_{P_\epsilon} h_{\gamma,\hat{\mu}}(x)\nonumber \\
    & \le  (2p)^{-1/2} \left\{  C_{\mathrm{sg,12}} \sqrt{\frac{2 q V_g }{n}} + \sqrt{\frac{q  \log(\delta^{-1})}{n}} \right\} \nonumber \\
    & \le  (2p)^{-1/2} \left\{  C_{\mathrm{sg,12}} C_{\mathrm{rad3}} \sqrt{\frac{16 p}{n}} + \sqrt{\frac{2p \log(\delta^{-1})}{n}} \right\} , \label{eq:pro-hinge-L2-lower-prf2}
    \end{align}
    as shown in Proposition~\ref{pro:logit-L2-lower}. Combining the inequalities (\ref{eq:pro-hinge-L2-lower-prf0})--(\ref{eq:pro-hinge-L2-lower-prf2}) leads to the desired result.
    \end{proof}

    \begin{manualpro}{S17} \label{pro:hinge-L2v-lower} In the setting of Proposition~\ref{pro:hinge-L2-detailed}, it holds with probability at least $1-2 \delta$ that for any $\gamma\in \Gamma_{20}$,
    \begin{align*}
    & \quad  K_\HG(P_{n}, P_{\hat\theta}; h_{\gamma, \hat\mu}) \\
    & \ge \left\{\E_{P_{\theta^*}} h_{\gamma, \hat\mu}(x) - \E_{P_{\hat\theta}} h_{\gamma, \hat\mu}(x) \right\} - \epsilon -  \sqrt{p}\lambda_{32}(4q)^{-1/2}
    \end{align*}
    where $q=p(1-p)$, and, with $C_{\mathrm{rad6}} = C_{\mathrm{sg,12}} C_{\mathrm{rad3}}$,
    $$
    \lambda_{32} = C_{\mathrm{rad6}} \sqrt{\frac{12 (p-1)}{n}} + \sqrt{\frac{(p-1)\log(\delta^{-1})}{n}} .
    $$
    \end{manualpro}

    \begin{proof}
    For any $\gamma\in \Gamma_{20}\subset \Gamma_{\rpQ}$, {because} $\pen_2(\gamma)=  (4q)^{-1/2} $ and $\gamma_0=0$, we have $\pen_2(\beta)= 2 \pen_2(\gamma) = q^{-1/2}$. Hence
    $h_\gamma(x) = h_{\rpQ, \beta }(x) \in [- 1,1]$ by the boundedness of the ramp function in $[0,1]$ and the Cauchy--Schwartz inequality, $\| \beta_2 \|_1 \le q^{1/2} \|\beta_2\|_2$.
    {Then the mean-centered version $h_{\gamma, \hat\mu}(x)$, with $\E_{P_{\theta^*}} h_{\gamma, \hat\mu}(x)=0$, is also bounded in $[-1,1]$.
    Moreover, such $ h_{\gamma, \hat\mu}(x) $ can be expressed in the form
$ \pen_2(\beta) w^\T \{g(x -\hat\mu) - \eta_0\} $, where, for $q=p(p-1)$, $w \in \bbR^{q }$ is an $L_2$ unit vector,
$g: \bbR^p \to [0,1]^{q }$ is a vector of functions, including $\ramp(x_i)\ramp(x_j)$ for $1 \le i\not= j\le p$, and
$\eta_0 = \E_{P_{\theta^*}} g(x - \hat\mu) \in [0,1]^{q }$.}

    Next, $K_\HG(P_{n}, P_{\hat\theta}; h_{\gamma, \hat\mu})$ can be bounded as
    \begin{align}
    & \quad  K_\HG(P_{n}, P_{\hat\theta}; h_{\gamma, \hat\mu}) \nonumber\\
    & \ge  K_\HG(P_\epsilon, P_{\hat\theta}; h_{\gamma, \hat\mu})  - | K_\HG(P_{n}, P_{\hat\theta}; h_{\gamma, \hat\mu}) - K_\HG(P_\epsilon, P_{\hat\theta}; h_{\gamma, \hat\mu})| \label{eq:pro-hinge-L2v-lower-prf1}.
    \end{align}
    For any $\gamma\in \Gamma_{20}$, {because} $h_{\gamma,\hat{\mu}}(x) \in [-1, 1]$, we have $\min(h_{\gamma,\hat{\mu}}(x),1)=h_{\gamma,\hat{\mu}}(x)$ and $\min(-h_{\gamma,\hat{\mu}}(x),1)=-h_{\gamma,\hat{\mu}}(x)$. Then the hinge term $K_\HG(P_{n}, P_{\hat\theta}; h_{\gamma, \hat\mu}) $ reduces to a moment matching term and can be lower bounded as follows:
    \begin{align}
    &\quad  K_\HG(P_\epsilon, P_{\hat\theta}; h_{\gamma, \hat\mu})\nonumber \\
    &= \E_{P_\epsilon} \min(h_{\gamma,\hat{\mu}}(x), 1) + \E_{P_{\hat\theta}} \min(-h_{\gamma,\hat{\mu}}(x), 1)\nonumber\\
    &= \E_{P_\epsilon} h_{\gamma,\hat{\mu}}(x) - \E_{P_{\hat\theta}} h_{\gamma,\hat{\mu}}(x)\nonumber\\
    &=\epsilon\E_{Q} h_{\gamma,\hat{\mu}}(x) + (1-\epsilon)\E_{P_\theta^*} h_{\gamma,\hat{\mu}}(x)- \E_{P_{\hat\theta}} h_{\gamma,\hat{\mu}}(x) \nonumber\\
    &\ge -\epsilon + \left\{\E_{P_\theta^*} h_{\gamma,\hat{\mu}}(x)- \E_{P_{\hat\theta}} h_{\gamma,\hat{\mu}}(x)\right\} \label{eq:pro-hinge-L2v-lower-prf2}.
    \end{align}
    Similarly, the {absolute difference term in (\ref{eq:pro-hinge-L2v-lower-prf1}) can be simplified as follows:}
    \begin{align*}
    &\quad | K_\HG(P_{n}, P_{\hat\theta}; h_{\gamma, \hat\mu}) - K_\HG(P_\epsilon, P_{\hat\theta}; h_{\gamma, \hat\mu})|\\
    &= \left| \left\{\E_{P_n} h_{\gamma,\hat{\mu}}(x)- \E_{P_{\hat\theta}} h_{\gamma,\hat{\mu}}(x)\right\} - \left\{\E_{P_\epsilon} h_{\gamma,\hat{\mu}}(x)- \E_{P_{\hat\theta}} h_{\gamma,\hat{\mu}}(x)\right\}\right|\\
    &=|\E_{P_n} h_{\gamma,\hat{\mu}}(x)- \E_{P_\epsilon} h_{\gamma,\hat{\mu}}(x)|.
    \end{align*}
    {We apply Lemma~\ref{lem:logit-concen-L2-lower} with $b= 1$, $g(x)\in [0,1]^{q }$ and $\eta_0$ defined above, and
    $f^{\prime}(\me^u)$ and $f^{\#}(\me^u)$ replaced by the identity function in $u$.}
    It holds with probability at least $1-2\delta$ that
   for any $\gamma\in \Gamma_{20}$,
    \begin{align}
    & \quad | \E_{P_n} h_{\gamma,\hat{\mu}}(x)- \E_{P_\epsilon} h_{\gamma,\hat{\mu}}(x)| \nonumber \\
    & \le  (4q)^{-1/2} \left\{  C_{\mathrm{sg,12}} \sqrt{\frac{2 q V_g }{n}} + \sqrt{\frac{q  \log(\delta^{-1})}{n}} \right\} \nonumber \\
    & =   \sqrt{p}(4q)^{-1/2} \left\{  C_{\mathrm{sg,12}} C_{\mathrm{rad3}} \sqrt{\frac{12(p-1)}{n}} + \sqrt{\frac{(p-1) \log(\delta^{-1})}{n}} \right\} , \label{eq:pro-hinge-L2v-lower-prf3}
    \end{align}
    as shown in Proposition~\ref{pro:logit-L2v-lower}. Combining the inequalities (\ref{eq:pro-hinge-L2v-lower-prf1})--(\ref{eq:pro-hinge-L2v-lower-prf3}) leads to the desired result.
    \end{proof}

\subsection{Details in proof of Corollary~\ref{cor:two-obj}} \label{sec:two-obj}

\begin{lem} \label{lem:two-obj-centered-upr-bound}
    Assume that $f$ satisfies Assumptions \ref{ass:f-condition} and \ref{ass:f-condition2},
    and $G$ satisfies Assumption \ref{ass:generator-condition}.
    Let $(\hat{\gamma}, \hat{\theta})$ be a solution to the alternating optimization problem (\ref{eq:two-obj-gan-centered}).

    (i) Let $\epsilon_0\in (0,1)$ be fixed. For any $\epsilon \in [0, \epsilon_0]$ and any function $h : \bbR^p \to \bbR$, we have
    \begin{align*}
    K_{f}(P_{\epsilon}, P_{\hat{\theta}}; h) \le -f^\prime(1-\epsilon_0) \epsilon .
    \end{align*}

    (ii) Let $\epsilon_1\in (0,1)$ be fixed. If $\hat{\epsilon} = n^{-1} \sum_{i=1}^{n} U_i \in [0, \epsilon_1]$, then
    for any function $h : \bbR^p \to \bbR$, we have
    \begin{align}
    K_f(P_n, P_{\hat{\theta}} ; h) \le -f^\prime(1-\epsilon_1) \hat\epsilon + R_1| \E_{P_{\theta^{*},n}} h(x) - \E_{P_{\theta^{*}}} h(x) | ,
    \end{align}
    where $P_{\theta^*,n}$ denotes the empirical distribution of $\{X_i: U_i=0,i=1,\ldots,n\}$ in the latent representation of Huber's contamination model.
    \end{lem}

    \begin{proof}
    (i) As mentioned in Remark \ref{rem:f-gan}, the logit $f$-GAN objective in (\ref{eq:fgan}) can be equivalently written as
    \begin{align*}
    K_f(P_{\epsilon}, P_{\hat{\theta}}; h)
    &= \E_{P_{\epsilon}}f^{\prime}(\me^{h(x)}) - \E_{P_{\hat{\theta}}} f^\#(\me^{h(x)}) \\
    &= \E_{P_{\epsilon}}T(h(x)) - \E_{P_{\hat{\theta}}}f^{*}\{T(h(x))\},
    \end{align*}
    where $T(u)=f^{\prime}(\me^{u})$ and $f^{*}$ is the convex conjugate of $f$.
    Because $f$ is convex and {non-decreasing} by Assumptions $\ref{ass:f-condition}$ and $\ref{ass:f-condition2}$, we have that $f^{*}$ is convex and {non-decreasing.}

    Denote as $L_G(\theta, \gamma)$ the generator objective function, $\E_{P_{\epsilon}}f^{\prime} (\me^{h_{\gamma, \mu}(x)})-\E_{P_{\theta}}G (h_{\gamma, \mu}(x))$.
    By the definition of a solution to alternating optimization  (see Remark \ref{rem:nash}), we have
    \begin{align*}
    & L_G\{(\hat\mu, \hat\Sigma), \hat\gamma\} \le L_G\{(\mu^*, \hat\Sigma), \hat\gamma\}, \\
    & L_G\{(\hat\mu, \hat\Sigma), \hat\gamma\} \le L_G\{(\hat\mu, \Sigma^*), \hat\gamma\},
    \end{align*}
    that is,  the generator loss at $\hat\theta$ is less than that at any $\theta$, with the discriminator parameter fixed at $\hat{\gamma}$.
    The preceding inequalities can be written out as
    \begin{align*}
    \E_{P_{\epsilon}}T(\me^{h_{\hat\gamma, \hat\mu}(x)})-\E_{P_{\hat\theta}}G (h_{\hat\gamma, \hat\mu}(x)) \le \E_{P_{\epsilon}}T (\me^{h_{\hat\gamma, \mu^*}(x)})-\E_{P_{\mu^*, \hat\Sigma}}G (h_{\hat\gamma, \mu^*}(x)) ,
    \end{align*}
    and
    \begin{align*}
    \E_{P_{\epsilon}}T (\me^{h_{\hat\gamma, \hat\mu}(x)})-\E_{P_{\hat\theta}}G (h_{\hat\gamma, \hat\mu}(x)) \le \E_{P_{\epsilon}}T (\me^{h_{\hat\gamma, \hat\mu}(x)})-\E_{P_{\hat\mu, \Sigma^*}}G (h_{\hat\gamma, \hat\mu}(x)).
    \end{align*}
    Note that either the location or the variance matrix, but not both, is changed on the two sides in each inequality.
    In the first inequality, we have $\E_{P_{\hat\theta}}G (h_{\hat\gamma, \hat\mu}(x))=\E_{P_{\mu^*, \hat\Sigma}}G (h_{\hat\gamma, \mu^*}(x))$ because both are equal to $\E_{P_{0, \hat\Sigma}}G (h_{\hat\gamma, 0}(x))$. In the second inequality, the term $\E_{P_{\epsilon}}G (h_{\hat\gamma, \hat\mu}(x))$ is on both sides.
    Then the two inequalities yield
    \begin{align*}
    \E_{P_{\epsilon}}T (\me^{h_{\hat\gamma, \hat\mu}(x)}) &\le \E_{P_{\epsilon}}T (\me^{h_{\hat\gamma, \mu^*}(x)}), \\
    -\E_{P_{\hat\theta}}G (h_{\hat\gamma, \hat\mu}(x)) &\le -\E_{P_{\hat\mu, \Sigma^*}}G (h_{\hat\gamma, \hat\mu}(x)).
    \end{align*}

    Now we are ready to derive an upper bound for $K_{f}(P_{\epsilon}, P_{\hat{\theta}}; h_{\hat{\gamma}, \hat{\mu}}) $:
    \begin{align}
    & \quad  K_{f}(P_{\epsilon}, P_{\hat{\theta}}; h_{\hat{\gamma}, \hat{\mu}}) \nonumber \\
    &  \E_{P_{\epsilon}}T(h_{\hat{\gamma}, \hat{\mu}}) -  \E_{P_{\hat{\theta}}} f^{*}\{T(G^{-1}(G(h_{\hat{\gamma}, \hat{\mu}})))\} \nonumber \\
    & \le \E_{P_{\epsilon}}T(h_{\hat{\gamma}, \hat{\mu}}) -   f^{*}\{T(G^{-1}(\E_{P_{\hat{\theta}}} G(h_{\hat{\gamma}, \hat{\mu}})))\} \label{eq:lem-two-obj-centered-upr-bound-prf1}\\
    & \le \E_{P_{\epsilon}}T(h_{\hat{\gamma}, \mu^*}) -   f^{*}\{T(G^{-1}(\E_{P_{\hat\mu, \Sigma^*}} G(h_{\hat{\gamma}, \hat{\mu}})))\} \label{eq:lem-two-obj-centered-upr-bound-prf2}\\
    & \le (1-\epsilon)\E_{P_{\theta^*}}T(h_{\hat{\gamma}, \mu^*}) -   f^{*}\{T(G^{-1}(\E_{P_{\mu^*, \Sigma^*}} G(h_{\hat{\gamma}, \mu^*})))\} \label{eq:lem-two-obj-centered-upr-bound-prf3}\\
    & \le (1-\epsilon)T(\E_{P_{\theta^*}}h_{\hat{\gamma}, \mu^*}) -   f^{*}\{T(\E_{P_{\mu^*,\Sigma^*}} h_{\hat{\gamma}, \mu^*})\} \label{eq:lem-two-obj-centered-upr-bound-prf4}\\
    & \le f(1-\epsilon) \le -f^{\prime}(1-\epsilon_0)\epsilon \label{eq:lem-two-obj-centered-upr-bound-prf5}.
    \end{align}
    Line (\ref{eq:lem-two-obj-centered-upr-bound-prf1}) follows from Jensen's inequality by the convexity of $f^*$.
    Line (\ref{eq:lem-two-obj-centered-upr-bound-prf2}) follows from the two inequalities derived above, together with
    {the fact that $-f^{*}(T(G^{-1}))$ is non-increasing, by non-decreasingness of $f^*$, $T$, and $G^{-1}$.}
    In (\ref{eq:lem-two-obj-centered-upr-bound-prf3}) we use the fact that $\E_{P_{\hat\mu, \Sigma^*}} G(h_{\hat{\gamma}, \hat{\mu}})=\E_{P_{\mu^*, \Sigma^*}}G (h_{\hat\gamma, \mu^*}(x))$ and drop the $\E_Q$ term because $T \le 0$.
    Line (\ref{eq:lem-two-obj-centered-upr-bound-prf4}) follows from Jensen's inequality by the convexity of $G$ and the concavity of $T$, together with the fact that
    $-f^{*}(T(G^{-1}))$ is non-increasing.
    For the last line (\ref{eq:lem-two-obj-centered-upr-bound-prf5}), by the definition of Fenchel conjugate we have
    \begin{align*}
    (1-\epsilon)s - f^{*}(s)\le f(1-\epsilon) \le -f^{\prime}(1-\epsilon_0)\epsilon,
    \end{align*}
    with $s$ set to $\E_{P_{\theta^*}}T(h_{\hat{\gamma}, \mu^*})$.

    (ii) To derive an upper bound for $K_f(P_n, P_{\hat{\theta}} ; h_{\hat{\gamma}})$, we first argue similarly as in part (i):
    \begin{align}
    & \quad  K_{f}(P_{n}, P_{\hat{\theta}}; h_{\hat{\gamma}, \hat{\mu}}) \nonumber \\
    & = \hat{\epsilon}  \, \E_Q T(h_{\hat{\gamma}, \hat{\mu}}) + (1-\hat{\epsilon})\E_{P_{\theta^*, n}}T(h_{\hat{\gamma}, \hat{\mu}}) -  \E_{P_{\hat{\theta}}} f^{*}\{T(h_{\hat{\gamma}, \hat{\mu}})\} \nonumber \\
    & \le (1-\hat\epsilon)T(\E_{P_{\theta^*, n}}h_{\hat{\gamma}, \mu^*}) -   f^{*}\{T(\E_{P_{\mu^*,\Sigma^*}} h_{\hat{\gamma}, \mu^*})\} \nonumber.
    \end{align}
    Then we {use} the $R_1$-Lipschitz property of $f^*(T)$ and obtain
    \begin{align*}
    & \quad (1-\hat\epsilon)T(\E_{P_{\theta^*, n}}h_{\hat{\gamma}, \mu^*}) -   f^{*}\{T(\E_{P_{\mu^*,\Sigma^*}} h_{\hat{\gamma}, \mu^*})\} \nonumber\\
    &\le (1-\hat{\epsilon})T(\E_{P_{\theta^*, n}}h_{\hat{\gamma}, \mu^*}) -   f^{*}\{T(\E_{P_{\theta^*, n}} h_{\hat{\gamma}, \mu^*})\} + R_1|\E_{P_{\theta^*, n}} h_{\hat{\gamma}, \mu^*} - \E_{P_{\theta^*}} h_{\hat{\gamma}, \mu^*}|\nonumber\\
    &\le f(1-\hat{\epsilon})+R_1|\E_{P_{\theta^*, n}} h_{\hat{\gamma}, \mu^*} - \E_{P_{\theta^*}} h_{\hat{\gamma}, \mu^*}|\nonumber\\
    &\le -f^{\prime}(1-\epsilon_1)\hat{\epsilon}+R_1|\E_{P_{\theta^*, n}} h_{\hat{\gamma}, \mu^*} - \E_{P_{\theta^*}} h_{\hat{\gamma}, \mu^*}|\nonumber.
    \end{align*}
    Combining the preceding displays completes the proof.
    \end{proof}
    \begin{lem} \label{lem:hinge-two-obj-upr-bound}
       Let $(\hat{\gamma}, \hat{\theta})$ be a solution to the alternating optimization problem (\ref{eq:two-obj-hinge-centered}).

        (i) For any $\epsilon \in [0, 1]$ and any function $h : \bbR^p \to \bbR$, we have
        \begin{align*}
        K_{\HG}(P_{\epsilon}, P_{\hat{\theta}}; h) \le 2\epsilon .
        \end{align*}

        (ii) If $\hat{\epsilon} = n^{-1} \sum_{i=1}^{n} U_i \in [0, 1]$, then
        for any function $h : \bbR^p \to \bbR$, we have
        \begin{align}
        K_\HG(P_n, P_{\hat{\theta}} ; h) \le 2\hat\epsilon + |\E_{P_{\theta^{*},n}} h(x) - \E_{P_{\theta^{*}}} h(x) | ,
        \end{align}
        where $P_{\theta^*,n}$ denotes the empirical distribution of $\{X_i: U_i=0,i=1,\ldots,n\}$ in the latent representation of Huber's contamination model.
        \end{lem}

        \begin{proof}
            (i) By the same argument used in the proof of Lemma~\ref{lem:two-obj-centered-upr-bound} with $T(u)$ replaced by $\min(u, 1)$ and $-f^*(T(u))$ replaced by $\min(-u, 1)$ we have:
        \begin{align}
        & \quad  K_{\HG}(P_{\epsilon}, P_{\hat{\theta}}; h_{\hat{\gamma}, \hat{\mu}}) \nonumber \\
        & \le \epsilon + (1-\epsilon)\min(\E_{P_{\theta^*}}h_{\hat{\gamma}, \mu^*}, 1) + \min(-\E_{P_{\theta^*}} h_{\hat{\gamma},\mu^*}, 1) \label{eq:lem-two-obj-hinge-upr-bound-prf1}\\
        & \le 2\epsilon \label{eq:lem-two-obj-hinge-upr-bound-prf2}.
        \end{align}
        Inequality (\ref{eq:lem-two-obj-hinge-upr-bound-prf1}) is derived by the same argument that leads to (\ref{eq:lem-two-obj-centered-upr-bound-prf1})-(\ref{eq:lem-two-obj-centered-upr-bound-prf4}) with the fact that $\min(u, 1)$ is concave and non-decreasing and that $\min(-u, 1)$ is concave and non-increasing just like $T$ and $-f^*(T)$ in Lemma~\ref{lem:two-obj-centered-upr-bound} respectively. The $\epsilon$ term is a result of the fact that $\min(u, 1) \le 1$. Inequality (\ref{eq:lem-two-obj-hinge-upr-bound-prf2}) is by the same argument used in Lemma~\ref{lem:hinge-upr-bound}:
        \begin{align*}
            &\quad (1-\epsilon)\min(u, 1) + \min(-u, 1) \\
            & \le \epsilon + (1-\epsilon)\{\min(u, 1) + \min(-u, 1)\}\\
            &\le \epsilon,
        \end{align*}
        with $u$ set to be $h_{\hat\gamma, \mu^*}$.

        (ii) To derive an upper bound for $K_\HG(P_n, P_{\hat{\theta}} ; h_{\hat{\gamma}})$, we first argue similarly as in part (i):
        \begin{align}
        & \quad  K_\HG(P_{n}, P_{\hat{\theta}}; h_{\hat{\gamma}, \hat{\mu}}) \nonumber \\
        & = \hat{\epsilon}  \, \E_Q \min(h_{\hat{\gamma}, \hat{\mu}}, 1) + (1-\hat{\epsilon})\E_{P_{\theta^*, n}}\min(h_{\hat{\gamma}, \hat{\mu}}, 1) + \E_{P_{\hat{\theta}}}\min(-h_{\hat{\gamma}, \hat{\mu}},1) \nonumber \\
        & \le \hat\epsilon + (1-\hat\epsilon)\min(\E_{P_{\theta^*, n}}h_{\hat{\gamma}, \mu^*}, 1) +  \min(-\E_{P_{\mu^*,\Sigma^*}} h_{\hat{\gamma}, \mu^*}, 1) \nonumber.
        \end{align}
        Then we use the $1$-Lipschitz property of $\min(-u ,1)$ and obtain
        \begin{align*}
        & \quad (1-\hat\epsilon)\min(\E_{P_{\theta^*, n}}h_{\hat{\gamma}, \mu^*}, 1) + \min(-\E_{P_{\mu^*,\Sigma^*}} h_{\hat{\gamma}, \mu^*}, 1) \nonumber\\
        &\le (1-\hat\epsilon)\min(\E_{P_{\theta^*, n}}h_{\hat{\gamma}, \mu^*}, 1) + \min(-\E_{P_{\theta^*, n}} h_{\hat{\gamma}, \mu^*}, 1) + |\E_{P_{\theta^*, n}} h_{\hat{\gamma}, \mu^*} - \E_{P_{\theta^*}} h_{\hat{\gamma}, \mu^*}|\nonumber\\
        &\le \hat{\epsilon} + |\E_{P_{\theta^*, n}} h_{\hat{\gamma}, \mu^*} - \E_{P_{\theta^*}} h_{\hat{\gamma}, \mu^*}|\nonumber.
        \end{align*}
        Combining the preceding displays completes the proof.
        \end{proof}

    \subsection{Proofs in Section~\ref{sec:discussion}}

\begin{proof}[Proof of Proposition~\ref{pro:GYZ}]
    We first verify that $f(t) = \frac{1+t}{2} g_0(\frac{2t}{1+t})$ is convex on $[0, +\infty)$ and $f(1)=0$ so that $D_f$ is a valid $f$-divergence. Because $g_0$ is convex by the convexity of $g$ and $f^{\dprime}(t) = \frac{2}{(1+t)^3}g_0^{\dprime}(\frac{2t}{1+t})$, it follows that $f$ is convex on $[0, +\infty)$. Direct calculation gives $f(1) = g_0(1) = 0$. Thus, $f$ defines a valid $f$-divergence.

    Next we show that $L_g(P_*, P_\theta; q_\gamma) = K_f(P_*, P_\theta; h_\gamma)$. Denote $\me^{h_\gamma(x)}$ by $t$ and $q_\gamma(x)$ by $q$. Then we have
    \begin{align}
        & \quad K_f(P_*, P_\theta; h_\gamma) = \E_{P_*}f^{\prime}(t) - \E_{P_\theta}\left\{t f^{\prime}(t) - f(t) \right\} \nonumber\\
        &= \E_{P_*}\left\{\frac{1}{1+t}g_0^{\prime}\left(\frac{2t}{1+t} \right) + \frac{1}{2}g_0\left(\frac{2t}{1+t}\right)\right\}
         - \E_{P_\theta}\left\{\frac{t}{1+t}g_0^{\prime}\left(\frac{2t}{1+t}\right) - \frac{1}{2}g_0\left(\frac{2t}{1+t}\right) \right\} \label{eq:GYZ-1}\\
        &=\E_{P_*}\left\{(1-q)g_0^{\prime}\left(2q\right) + \frac{1}{2}g_0\left(2q\right)\right\} - \E_{P_\theta}\left\{q g_0^{\prime}\left(2q\right) - \frac{1}{2}g_0\left(2q\right) \right\} \nonumber\\
        &=\E_{P_*}\left\{\frac{1-q}{2}g^{\prime}\left(q\right) + \frac{1}{2}g\left(q\right) - \frac{1}{2}g\left(\frac{1}{2}\right)\right\}
         - \E_{P_\theta}\left\{\frac{q}{2} g^{\prime}\left(q\right) - \frac{1}{2}g\left(q\right) + \frac{1}{2}g\left(\frac{1}{2}\right)\right\}\label{eq:GYZ-2} \\
        &=\frac{1}{2} \left\{ \E_{P_*} S_g(q, 1) - \E_{P_\theta}S_g(q, 0)\right\} - g\left(\frac{1}{2}\right). \label{eq:GYZ-3}
    \end{align}
    Line (\ref{eq:GYZ-1}) is by direct calculation. Lines (\ref{eq:GYZ-2})--(\ref{eq:GYZ-3}) are by the definition of $g$ and $S_g$.

     Finally, by the definition of $f$ from $g_0$, direct calculation gives $$\int q f\left(\frac{p}{q}\right) = \int \frac{p+q}{2} g_0\left(\frac{2p}{p+q}\right),$$ which implies that $D_{g_0}(P_* || (P_* + P_\theta)/2) = D_f(P_* || P_\theta)$.
\end{proof}

%%% -----------------------------------------------------------------------------------------------------------------------------------------  auxilary lemmas

\section{Auxiliary lemmas}

\subsection{Truncated linear basis}
The following result gives upper bounds on the moments of the truncated linear basis, which are used in the proofs of
Lemma~\ref{lem:spline-L1-upper} and~\ref{lem:spline-L2-upper}.

\begin{lem} \label{lem:truncated-mean}
For $X \sim \N(0,\sigma^2)$ and $\xi \in \bbR$, we have
\begin{align*}
& \E (X-\xi)_+ \le \frac{\sigma}{\sqrt{2\pi}} + |\xi|, \\
&  \E [ \{ (X-\xi)_+\}^2 ] \le \sigma^2 + \xi^2.
\end{align*}
\end{lem}

\begin{proof}
The second result is immediate: $\E [ \{ (X-\xi)_+\}^2 ] \le \E \{ (X-\xi)^2 \} = \sigma^2 + \xi^2$.
The first result can be shown as follows:
\begin{align*}
& \quad \E (X-\xi)_+ = \int_\xi^\infty (x-\xi) \frac{1}{\sqrt{2\pi} \sigma} \me^{-\frac{x^2}{2\sigma^2}} \,\dif x \\
& = \frac{\sigma}{\sqrt{2\pi}} \me^{- \frac{\xi^2}{2\sigma^2}} - \xi \int_\xi^\infty \frac{1}{\sqrt{2\pi}\sigma} \me^{-\frac{x^2}{2\sigma^2}} \,\dif x  \\
& \le \frac{\sigma}{\sqrt{2\pi}} + |\xi|.
\end{align*}
The last inequality holds because $0 \le \int_\xi^\infty \frac{1}{\sqrt{2\pi}\sigma} \me^{-\frac{x^2}{2\sigma^2}} \,\dif x \le 1$.
\end{proof}

\subsection{VC index of ramp functions}

For a collection $\mathcal{C}$ of subsets of $\mathcal{X}$, and points $x_1, \dots, x_n \in \mathcal{X}$, define
\begin{equation*}
\Delta_{n}^{\mathcal{C}}(x_1, \dots, x_n) = \# \{C \cap \{x_1, \dots, x_n\}: C \in \mathcal{C}  \},
\end{equation*}
that is, $\Delta_{n}^{\mathcal{C}}(x_1, \dots, x_n)$ is the number of subsets of $\{x_1, \dots, x_n\}$ picked out by the collection $\mathcal{C}$.
We say that a subset $\{x_i, \dots, x_j\} \subset \{x_1, \dots, x_n\}$ is picked up by $\mathcal{C}$ if $\{x_i, \dots, x_j\} \in\{C \cap \{x_1, \dots, x_n\}: C \in \mathcal{C}  \}$.
For convenience, we also say that $\{x_i, \dots, x_j\}$ is picked up by $C$ if $\{x_i, \dots, x_j\} = C \cap \{x_1, \dots, x_n\}$.
Moreover, define
\begin{align*}
& m^{\mathcal{C}}(n) = \max_{x_1, \dots, x_n} \Delta_{n}^{\mathcal{C}}(x_1, \dots, x_n),
\end{align*}
and the Vapnik--Chervonenkis (VC) index of $\mathcal{C}$ as (\citeappend{VW})
\begin{align*}
& V(\mathcal{C}) = \inf\{n \ge 1 :  m^{\mathcal{C}}(n) < 2^n \} .
\end{align*}
where the infimum over the empty set is taken to be infinity.

The subgraph of a function $f: \mathcal{X} \to \bbR$ is defined as $G_f= \{(x, t) \in \mathcal{X} \times \bbR: t < f(x)\}$.
For a collection of functions $\mathcal{F}$, denote the collection of corresponding subgraphs as
$G_{\mathcal{F}} = \{G_f : f \in \mathcal{F}\}$, and define the VC index of $\mathcal{F}$ as $V(\mathcal{F})=V(G_{\mathcal{F}})$.

\begin{lem} \label{lem:ramp-VC}
For $\mathcal{F} =\{f_b(x)=1-(x+b)_+ + (x+b-1)_+ : b\in \bbR\}$, we have that $V(\mathcal{F}) = 2$.
Moreover, the VC index of $\{f_b(-x): b\in \bbR\} = \{ \ramp(x-b): b \in \bbR\}$ is $2$. That is, the VC index of moving-knots ramp functions is $2$.
\end{lem}

\begin{proof}

To show $V(\mathcal{F})=2$, we need to show $m^{G_\mathcal{F}}(1) = 2^1$ and $m^{G_\mathcal{F}}(2) < 2^2$.
The first property is trivially true. For the second property, it suffices to show that
for any two distinct points $\{(x_1, t_1), (x_2, t_2)\}$, there is at least a subset of $\{(x_1, t_1), (x_2, t_2)\}$ that cannot be picked up by $G_{f_b}$ for any $f_b \in \mathcal{F}$.
Without loss of generality, assume that $x_1 \leq x_2$.
We arbitrarily fix $f_b \in \mathcal{F}$ and discuss several cases depending on $t_2 - t_1$ and $x_2 - x_1$.

%First, due to the boundedness of $f_b$ between 0 and 1, the claim can be easily verified if $t_1$ or $t_2$ is outside $[0,1]$.
%In fact, if, for example, $t_1 >1$, then it is impossible for the singleton $\{(x_1, t_1)\}$ to be picked up by $G_{f_b}$ for any $f_b \in \mathcal{F}$ because $f_b(x_1)\le 1 <t_1$.
%If $t_1 < 0$, then $f_b(x_1) \ge 0 > t_1$ for all $f_b \in \mathcal{F}$, and the element $(x_1, t_1)$ is always in $G_{f_b}$, so that the singleton $\{(x_2, t_2)\}$ can never be picked up by $G_{f_b}$.

%Next, we assume that $t_1,t_2 \in [0,1]$ and arbitrarily fix $f_b \in \mathcal{F}$. We discuss several situations depending on the ratio between $t_2 - t_1$ and $x_2 - x_1$.
%If $t_2 - t_1 =-(x_2 - x_1)$, then the two points $(x_1, t_1)$ and $(x_2, t_2)$ are always included or excluded together in $G_{f_b}$,
%so that a subset containing just one point can never be picked up by $G_{f_b}$.

If $t_2 - t_1 \le -(x_2 - x_1)$, then if $(x_1, t_1) \in G_{f_b}$, i.e., $f_b(x_1) > t_1$, we have
\begin{align*}
f_b(x_2) - f_b(x_1) &\ge (-1) (x_2 - x_1) \\
&\ge t_2 - t_1.
\end{align*}
This implies that $f_b(x_2) \ge  f_b(x_1) - t_1 + t_2 >  t_2$ and hence $(x_2, t_2) \in G_{f_b}$. As a result, a subset containing just $(x_1, t_1)$ cannot be picked up by $G_{f_b}$.

If $t_2 - t_1 \ge 0$, then if $(x_2, t_2) \in G_{f_b}$, i.e., $f_b(x_2) > t_2$, we have
\begin{align*}
f_b(x_1) \ge f_b(x_2) > t_2 \ge t_1,
\end{align*}
and hence $(x_2, t_2) \in G_{f_b}$. Thus, the subset $\{(x_2, t_2)\}$ can never be picked up by $G_{f_b}$.

If $t_2-t_1 < 0$ and $t_2 - t_1 > -(x_2-x_1)$, then if $f_b(x_2) > t_2$, we have
\begin{align*}
f_b(x_1) - f_b(x_2) &\ge (-1) (x_1 - x_2) \\
&> -(t_2 - t_1).
\end{align*}
This implies that $f_b(x_1) > f_b(x_2) - t_2 + t_1 > t_1$. As a result, the subset $\{(x_2, t_2)\}$ can never be picked up by $G_{f_b}$.

Combining the preceding cases shows that $m^{G_{\mathcal{F}}}(2) < 2^2$ and $V(\mathcal{F}) = 2$.
Moreover, the class of functions $\{f_b(-x):b\in\bbR\}$, denoted as $\tilde{\mathcal{F}}$,
admits a one-to-one correspondence with $\mathcal{F}$.
A subset of $\{(x_1, t_1), (x_2, t_2)\}$ is picked up by $G_{\mathcal{F}}$ if and only if the corresponding subset of $\{(-x_1, t_1), (-x_2, t_2)\}$ is picked up by $G_{\tilde{\mathcal{F}}}$.
Hence $V( \tilde{\mathcal{F}} ) = V(\mathcal{F}) =2$.
\end{proof}

\begin{lem} \label{lem:intercept-VC}
For $\mathcal{F} =\{f(x) \equiv b  : b  \in \bbR\}$, we have that $V(\mathcal{F}) = 2$.
That is, the VC index of constant functions is $2$.
\end{lem}
\begin{proof}
For any two distinct points $(x_1, t_1)$ and $(x_2, t_2)$, assume that with loss of generality $t_1 \le t_2$.
Then the singleton $\{(x_2, t_2)\}$ can never be picked up by $G_f$ for any $f \in \mathcal{F}$, and hence $m^{G_{\mathcal{F}}}(2) < 2^2$ and $V(\mathcal{F}) = 2$.
In fact, if $(x_2, t_2)$ is in the subgraph of $f(x) \equiv b $, then $t_2 < b $. As a result, $t_1 \le t_2 < b $, indicating that $(x_1, t_1)$ is also in the subgraph.
\end{proof}

\subsection{Lipschitz functions of Gaussian vectors}

Say that a function $g: \bbR^p \to \bbR^m$ is $L$-Lipschitz if
$ \| g(x_1) - g(x_2) \|_2 \le L \| x_1 - x_2 \|_2$ for any $x_1, x_2 \in \bbR^p$.

\begin{lem} \label{lem:lip-gaus}
Let $X \sim \N_p (\mu, \Sigma)$, and $g: \bbR^p \to \bbR^m$ be an $L$-Lipschitz function.

(i) For any vector $w \in \bbR^m$ with $\| w \|_2=1$, we have
\begin{align*}
\E\left[ \{ w^\T  ( g(X) - \E g(X) ) \}^2 \right] \le 2 C_{\mathrm{sg,12}}^2 L^2 \| \Sigma \|_{\op},
\end{align*}
where $C_{\mathrm{sg,12}}$ is {the} universal constant from Lemma~\ref{lem:subg}.
Hence we have $$\| \var \, g(X)  \|_{\op} \le 2 C_{\mathrm{sg,12}}^2 L^2 \| \Sigma \|_{\op}. $$

(ii) For any symmetric matrix $A \in \bbR^{m\times m}$ with $ \|A\|_{\fro}= 1$, we have
\begin{align*}
\E \left[ \{ (g(X) -\E g(X) )^\T A ( g(X) - \E g(X) )  \}^2 \right] \le 4 C_{\mathrm{sg,12}}^4 m L^4 \| \Sigma \|_{\op}^2.
\end{align*}
\end{lem}

\begin{proof}
(i) By \citeappend{BLM13}, Theorem 5.6, it can be shown that for any $L_2$ unit vector $w$,
$ w^\T  ( g(X) - \E g(X) ) $ is sub-gaussian with tail parameter $ L \| \Sigma \|_{\op}^{1/2}$. See the proof of Lemma~\ref{lem:spline-L2-upper}(i) for a similar argument.
Then by Lemma~\ref{lem:subg}, $\E [ \{ w^\T  ( g(X) - \E g(X) ) \}^2 ]  \le 2 C_{\mathrm{sg,12}}^2 L^2 \| \Sigma\|_{op}$.

(ii) Consider an eigen-decomposition $A = \sum_{j=1}^m \lambda_j w_j w_j^{\T} $,
where $\lambda_j$'s are eigenvalues and $w_j$'s are the eigenvectors with $\|w_j\|_2 =1$.
Denote $g=g(X)$ and $\tilde g = g- \E g$. Then $\tilde g^\T A \tilde g = \sum_{j=1}^m \lambda_j (w_j^\T \tilde g)^2$ and
\begin{align*}
& \quad ( \tilde g^\T A \tilde g )^2  \le \left( \sum_{j=1}^m \lambda_j^2 \right) \left( \sum_{j=1}^m (w_j^\T \tilde g)^4 \right)
\le  4 C_{\mathrm{sg,12}}^4 m L^4 \| \Sigma \|_{\op}^2 .
\end{align*}
The first step uses the Cauchy--Schwartz inequality, and the second step uses the fact that $\sum_{j=1}^m \lambda_j^2 = \|A\|_{\fro}^2=1$
and $E (w_j^\T \tilde g)^4  \le 4 C_{\mathrm{sg,12}}^4 L^4 \| \Sigma \|_{\op}^2$ by Lemma~\ref{lem:subg}
because $ w_j^\T \tilde g_j $ is sub-gaussian with tail parameter $ L \| \Sigma \|_{\op}^{1/2}$ for each $j$.
\end{proof}

The following result provides a $4$th-moment bound which depends linearly on $L^2 \| \Sigma \|_{\op}$, under a boundedness condition
in addition to the Lipschitz condition.

\begin{lem} \label{lem:lip-gaus2}
Let $X \sim \N_p (\mu, \Sigma)$, and $g: \bbR^p \to  [0,1]^m$ be an $L$-Lipschitz function.

(i) For any matrix $A \in \bbR^{m\times m}$ with $ \|A\|_{\fro}= 1$, we have
\begin{align*}
\E \left[ \{ (g(X) -\E g(X) )^\T A ( g(X) - \E g(X) )  \}^2 \right] \le 2 C_{\mathrm{sg,12}}^2 m L^2 \| \Sigma \|_{\op}.
\end{align*}

(ii) For any matrix $A \in \bbR^{m\times m}$ with $ \|A\|_{\fro}= 1$, we have
\begin{align*}
\E \left[ \{ g^\T(X) A g(X) - \E g^\T(X) A g(X) \}^2 \right] \le 20 C_{\mathrm{sg,12}}^2 m L^2 \| \Sigma \|_{\op}.
\end{align*}
\end{lem}

\begin{proof}
(i) Denote $g=g(X)$ and $\tilde g = g- \E g$. Then each component of $\tilde g$ is contained in $[-1,1]$ by the boundedness of $g$. The variable $( \tilde g^\T A \tilde g )^2 $ can be bounded as follows:
\begin{align}
& \quad  ( \tilde g^\T A \tilde g )^2  = \tr ( \tilde g^\T A \tilde g \tilde g^\T A^\T \tilde g )
= \tr (A \tilde g \tilde g^\T  A^\T \tilde g \tilde g^\T ) \nonumber \\
& \le \tr ( A \tilde g \tilde g^\T A^\T ) \| \tilde g \tilde g^\T \|_{\op} \label{eq:lem-lip-gaus2-prf1} \\
& \le m\, \tr (A \tilde g \tilde g^\T A^\T) =  m\, \tr (A^\T A \tilde g \tilde g^\T) .  \label{eq:lem-lip-gaus2-prf2}
\end{align}
Line (\ref{eq:lem-lip-gaus2-prf1}) follows from von Neumann's trace equality.
Line (\ref{eq:lem-lip-gaus2-prf2}) uses the fact that $\|\tilde g \tilde g^\T \|_{\op} \le m$, because
$ w^\T \tilde g \tilde g^\T u = (w^\T \tilde g)^2 \le \| w\|_2^2 \|\tilde g\|_2^2 \le m \| w\|_2^2$ for any $w \in \bbR^m$,
by the boundedness of $\tilde g$. Then the desired result follows because
\begin{align}
& \quad \E \tr (A^\T A \tilde g \tilde g^\T) =  \tr (A^\T A \, \var ( g) ) \nonumber \\
& \le \tr (A^\T A) \| \var (  g) \|_{\op} \label{eq:lem-lip-gaus2-prf3}  \\
& \le 2 C_{\mathrm{sg,12}}^2 L^2 \| \Sigma\|_{\op}. \label{eq:lem-lip-gaus2-prf4}
\end{align}
Line (\ref{eq:lem-lip-gaus2-prf3}) also follows from von Neumann's trace equality.
Line (\ref{eq:lem-lip-gaus2-prf4}) follows because $\tr(A^\T A) =\| A \|_{\fro}^2 = 1$ and
$\| \var ( g) \|_{\op} \le  2 C_{\mathrm{sg,12}}^2 L^2 \| \Sigma\|_{\op}$ by Lemma~\ref{lem:lip-gaus}(i), with $g$ being an $L$-Lipschitz function.

(ii) The difference $g^\T A g - \E g^\T A g$ can be expressed in terms of the centered variables as
$g^\T A g - \E g^\T A g =  (\tilde g^\T A \tilde g  -\E \tilde g^\T A \tilde g) + 2 (\E g)^\T A \tilde g $. Then
\begin{align}
& \quad (g^\T A g - \E g^\T A g)^2 \le 2 ( \tilde g^\T A \tilde g -\E \tilde g^\T A \tilde g)^2 + 8  \{ (\E g)^\T A \tilde g \}^2 \nonumber \\
& \le 2 ( \tilde g^\T A \tilde g -\E \tilde g^\T A \tilde g)^2 + 8  \| \E g\|_2^2 \| A \tilde g \|_2^2. \label{eq:lem-lip-gaus2-prf5}
\end{align}
The expectation of the first term on (\ref{eq:lem-lip-gaus2-prf5}) can be bounded using (i) as
\begin{align*}
& \quad 2 \E \{ ( \tilde g^\T A \tilde g -\E \tilde g^\T A \tilde g)^2  \}  = 2   \E \{ ( \tilde g^\T A \tilde g )^2 \} - 2  ( \E \tilde g^\T A \tilde g )^2 \\
& \le 2 \E \{ ( \tilde g^\T A \tilde g )^2 \} \le 4 m  C_{\mathrm{sg,12}}^2 L^2 \| \Sigma\|_{\op}.
\end{align*}
The expectation of the second term on (\ref{eq:lem-lip-gaus2-prf5}) can be bounded as
\begin{align*}
& \quad 8 \| \E g\|_2^2 \, \E \| A \tilde g \|_2^2 = 8 \| \E g\|_2^2 \; \E \tr (A^\T A \tilde g \tilde g^\T) \\
& \le 16 m  C_{\mathrm{sg,12}}^2 L^2 \| \Sigma\|_{\op}.
\end{align*}
by inequality (\ref{eq:lem-lip-gaus2-prf4}) and the fact that $\| \E g\|_2^2 \le m$.
Combining the preceding two bounds yields the desired result.
\end{proof}

\subsection{Moment matching for Lipschitz functions}

{The following result gives an upper bound on moment matching of quadratic forms under a Lipschitz condition.}

\begin{lem} \label{lem:lip-matching}
Let $g: \bbR^p \to \bbR^m$ be an $L$-Lipschitz function,
and let $X_1 = \mu_1 + D_1  Z$ and $X_2 = \mu_2 + D_2 Z$, where
$Z\in \bbR^p$ is a random vector in which the second moments of all components are 1,
$\mu_1, \mu_2 \in \bbR^p$, and $D_1 = \diag (d_1)$ and $D_2 = \diag(d_2)$ with $d_1 , d_2 \in \bbR_+^p$. Then
for any matrix $A \in \bbR^{m\times m}$ with $\| A \|_\fro=1$,
\begin{align*}
& \quad \left| \E g^\T(X_1) A g (X_1) - \E g^\T (X_2) A g(X_2) \right| \\
& \le 2 \sqrt{m} L^2 \Delta^2 + 2 \sqrt{2} L \Delta ( \E \| g_2\|_2^2)^{1/2} ,
\end{align*}
where $\Delta^2 =   \| \mu_1-\mu_2\|_2^2 + \| d_1 - d_2 \|_2^2$.
\end{lem}

\begin{proof}
Denote $g_1 = g (X_1)$ and $g_2 = g(X_2)$. The difference  $g_1^\T A g_1 - g_2^\T A g_2 $ can be decomposed as
\begin{align}
& \quad g_1^\T A g_1 - g_2^\T A g_2 \nonumber \\
& = ( g_1 - g_2)^\T  A (g_1 - g_2)  + 2 (g_1 - g_2) ^\T A g_2 \label{eq:lem-lip-matching-prf1}.
\end{align}
The expectation of the first term on (\ref{eq:lem-lip-matching-prf1}) can be bounded as
\begin{align}
& \quad \left| \E ( g_1 - g_2)^\T  A (g_1 - g_2) \right|
= \left| \tr \left( A V \right) \right| \nonumber \\
& \le  \sum_{j=1}^m s_j(A) \| V \|_\op \label{eq:lem-lip-matching-prf2} \\
& \le 2 \sqrt{m} L^2 \left( \| \mu_1 - \mu_2\|_2^2 + \|d_1 - d_2\|_2^2  \right), \label{eq:lem-lip-matching-prf3}
\end{align}
where $s_1(A), \ldots, s_m(A)$ are the singular values of $A$, and $V= \E \{ ( g_1 - g_2) (g_1 - g_2)^\T \}$.
Line (\ref{eq:lem-lip-matching-prf2}) follows from von Neumann's trace inequality.
Line (\ref{eq:lem-lip-matching-prf3}) follows because
$\sum_{j=1}^m s_j(A)  \le \sqrt{m} \{\sum_{j=1}^m s_j^2 (A) \}^{1/2} = \sqrt{m} \| A \|_\fro = \sqrt{m}$ with $ \|A \|_\fro=1$
and $ \| V \|_\op \le 2 L^2 ( \| \mu_1 - \mu_2\|_2^2 + \|d_1 - d_2\|_2^2  )$,
which can be shown as follows. For any $L_2$ unit vector $w$, we have
\begin{align}
& \quad w^\T V w =   \E \left[ \{ w^\T ( g_1 - g_2) \}^2  \right] \le \E \| g_1 - g_2 \|_2 ^2 \nonumber \\
& \le L^2 \E \| \mu_1 + D_1 Z - (\mu_2 + D_2 Z) \|_2^2 \nonumber \\
& \le 2 L^2 \left\{ \| \mu_1 - \mu_2\|_2^2 + \E  \| (D_1 - D_2) Z \|_2^2  \right\} \nonumber \\
& \le 2 L^2 \left( \| \mu_1 - \mu_2\|_2^2 + \|d_1 - d_2\|_2^2  \right), \label{eq:lem-lip-matching-prf4}
\end{align}
using the fact that $g(\cdot)$ is $L$-Lipschitz and the marginal variances of $Z$ are 1.
The expectation of the second term on (\ref{eq:lem-lip-matching-prf1}) can be bounded as
\begin{align}
& \quad \left| \E  (g_1 - g_2) ^\T A g_2 \right| \le  \E  \left| (g_1 - g_2) ^\T A g_2 \right| \nonumber \\
& \le \E  \|g_1 - g_2 \|_2 \| A g_2 \|_2  \nonumber \\
& \le \E^{1/2} (\|g_1 - g_2 \|_2^2) \E^{1/2} ( \| A g_2 \|_2^2 ) \nonumber\\
& \le \sqrt{2} L \left( \| \mu_1 - \mu_2\|_2^2 + \|d_1 - d_2\|_2^2  \right)^{1/2} ( \E \| g_2\|_2^2)^{1/2} . \label{eq:lem-lip-matching-prf5}
\end{align}
Line (\ref{eq:lem-lip-matching-prf5}) uses the fact that
$\E \{ \|g_1 - g_2 \|_2^2\} \le   2 L^2 ( \| \mu_1 - \mu_2\|_2^2 + \|d_1 - d_2\|_2^2  )$ based on (\ref{eq:lem-lip-matching-prf4})
and the following argument:
\begin{align*}
\E \|A g_2\|_2^2  \le \E (\|A\|_\op^2 \|g_2\|_2^2 ) \le \E \| g_2\|_2^2 ,
\end{align*}
where the last step follows because $ \|A\|_\op \le \|A\|_\fro =1$.
Combining (\ref{eq:lem-lip-matching-prf1}), (\ref{eq:lem-lip-matching-prf3}), and (\ref{eq:lem-lip-matching-prf5}) yields the desired result.
\end{proof}

The following result gives a tighter bound than in Lemma~\ref{lem:lip-matching} under a boundedness condition
in addition to the Lipschitz condition.

\begin{lem} \label{lem:lip-matching-bd}
In the setting of Lemma~\ref{lem:lip-matching}, suppose that each component of $g_1(x)$ and $g_2(x)$ is bounded in $[-1,1]$.
Then for any matrix $A \in \bbR^{m\times m}$ with $\| A \|_\fro=1$,
\begin{align*}
\left| \E g^\T(X_1) A g (X_1) - \E g^\T (X_2) A g(X_2) \right|
\le 2 \sqrt{2 m} L \Delta ,
\end{align*}
where $\Delta^2 =   \| \mu_1-\mu_2\|_2^2 + \| d_1 - d_2 \|_2^2$.
\end{lem}

\begin{proof}
The difference $g_1^\T A g_1 - g_2^\T A g_2 $ can also be decomposed as
\begin{align*}
g_1^\T A g_1 - g_2^\T A g_2
=(g_1 - g_2) ^\T A g_1 + (g_1 - g_2) ^\T A g_2 .
\end{align*}
By (\ref{eq:lem-lip-matching-prf5}), both of the two terms on the right-hand side can be bounded in absolute values by
$\sqrt{2} L \Delta ( \E \| g_2\|_2^2)^{1/2} $
and hence by $\sqrt{2 m} L \Delta $, because $\E \| g_2\|_2^2 \le m$ by the componentwise boundedness of $g_1$ and $g_2$.
\end{proof}

%%% -----------------------------------------------------------------------------------------------------------------------------------------  technical tools

\section{Technical tools}

\subsection{von Neumann's trace inequality}

% Von Neumann, J., Some matrix-inequalities and metrization of matrix-space, Tomsk Univ. Rev. 1, 286-300 (1937). Reprinted in Collected Works (Pergamon Press, 1962), iv, 205-219.

\begin{lem} (\citeappend{Von62})
For any $m\times m$ matrices $A$ and $B$ with singular values $\alpha_1 \ge \cdots \ge \alpha_m \ge 0 $ and $\beta_1 \ge \cdots \ge \beta_m \ge 0$ respectively,
\begin{align*}
| \tr ( A B ) | \le \sum_{j=1}^m \alpha_j \beta_j.
\end{align*}
\end{lem}

As a direct consequence, if $A$ is symmetric and non-negative definite, then
\begin{align*}
|\tr (A B)| \le \tr(A) \| B \|_\op
\end{align*}
This follows because the singular values $\alpha_i$'s are also the eigenvalues of $A$ and hence $\tr(A) = \sum_{j=1}^m \alpha_j$ for a symmetric and nonnegative definite matrix $A$,
and $\| B\|_\op = \max_{i=j,\ldots,m} \beta_j$ by the definition of $\|B\|_\op$.

\subsection{Sub-gassisan and sub-exponential properties}

The following results can be obtained from \citeappend{Ver18}, Proposition 2.5.2.

\begin{lem} \label{lem:subg}
For a random variable $Y$, the following properties are equivalent: there exist universal constants $C_{\mathrm{sg},ij}>0$ such that
the $C_{\mathrm{sg},ij}^{-1} K_j \le K_i \le C_{\mathrm{sg},ij} K_j$ for all $1\le i \not=j \le 4$,
where $K_i$ is the parameter appearing in property (i).
\begin{itemize}
\item[(i)] $\pr ( |Y| > t) \le 2 \exp  ( - \frac{t^2}{2 K_1^2} )$ for any $t>0$.

\item[(ii)] $\E^{1/p} ( |Y|^p ) \le K_2 \sqrt{p}$ for any $p \ge 1$.

\item[(iii)] $\E \exp ( Y^2 / K_3^2 ) \le 2$.
\end{itemize}
If $\E Y=0$, then properties (1)--(3) are also equivalent to the following one.
\begin{itemize}
\item[(iv)] $\E \exp( s Y ) \le \exp (\frac{K_4^2 s^2 }{2} )$, for any $s \in \bbR$.
\end{itemize}
\end{lem}

Say that $Y$ is a sub-gaussian random variable with tail parameter $K$ if property (1) holds in Lemma~\ref{lem:subg} with $K_1=K$.
The following result shows that being sub-gaussian depends only on tail probabilities of a random variable.

\begin{lem} \label{lem:subg-tail}
Suppose that for some $b, K>0$, a random variable $Y$ satisfies that
\begin{align*}
\pr ( |Y| > b+t ) \le 2 \me^{- \frac{t^2}{ 2K^2} } \quad \text{for any $t >0$}.
\end{align*}
Then $Y$ is sub-gaussian with tail parameter $K +b$.
\end{lem}

\begin{proof}
We distinguish two cases of $y>0$.
First, if $y > K+b$, then $ (y-b)/ K > y/(K+b) > 1$ and
$$
\pr ( | Y| > y ) \le 2 \exp\left\{ - \frac{(y-b)^2}{2 K^2} \right\} \le 2 \exp\left\{ - \frac{y^2}{2 (K+b)^2} \right\} .
$$
Second, if $y \le K+b$, then
$$
\pr ( | Y| > y )  \le 1 \le 2 \me^{-\frac{1}{2}} \le 2 \exp\left\{ -\frac{y^2}{2 (K+b)^2} \right\}.
$$
Hence the desired result holds.
\end{proof}

The following result follows directly from Chernoff's inequality.

\begin{lem} \label{lem:subg-concentration}
Suppose that $(Y_1\ldots,Y_n)$ are independent such that $E Y_i=0$ and $Y_i$ is sub-gaussian with tail parameter $K$ for $i=1,\ldots,n$. Then
\begin{align*}
\pr \left( \left|  \frac{1}{n} \sum_{i=1}^n Y_i \right|  >  C_{\mathrm{sg}5} t \right) \le 2 \exp \left( -\frac{n t^2}{2K^2} \right) \quad \text{for any $t>0$}.
\end{align*}
where $C_{\mathrm{sg5}}=C_{\mathrm{sg,14}}$ as in Lemma~\ref{lem:subg}.
\end{lem}

The following result can be obtained from \citeappend{Ver18}, Exercise 2.5.10.

\begin{lem} \label{lem:subg-max}
Let $(Y_1,\ldots,Y_n)$ be random variables such that $Y_i$ is sub-gaussian with tail parameter $K$ for $i=1,\ldots,n$. Then
\begin{align*}
\E \max_{i=1,\ldots,n} |Y_i| \le C_{\mathrm{sg6}} K \sqrt{\log(2n)},
\end{align*}
where $C_{\mathrm{sg6}}>0$ is a universal constant.
\end{lem}

Say that $Y \in \bbR^d$ is a sub-gaussian random vector with tail parameter $K$ if $w^\T Y$ is sub-gaussian with tail parameter $K$ for any $w \in \bbR^d$ with $\| w \|_2=1$.
The following result can be obtained from \citeappend{HKZ12}, Theorem 2.1.

\begin{lem} \label{lem:subg-vec-norm}
Suppose that $Y \in \bbR^d$ with $\E Y=0$ is a sub-gaussian random vector with tail parameter $K$. Then for any $t>0$, we have that with probability at least $1-\me^{-t}$,
\begin{align*}
\| Y \|_2 \le C_{\mathrm{sg7}} K ( \sqrt{ d} + \sqrt{t}) ,
\end{align*}
where $C_{\mathrm{sg7}}>0$ is a universal constant.
\end{lem}

The following result can be obtained from \citeappend{Ver10}, Theorem 5.39 and Remark 5.40(1).
Formally, this is different from \citeappend{Ver18}, Theorem 4.7.1 and Exercise 4.7.3, due to assumption (4.24) used in the latter result.

\begin{lem} \label{lem:subg-vec-var}
Suppose that $Y_1,\ldots,Y_n$ are independent and identically distributed as $Y \in \bbR^d$, where
$Y$ is a sub-gaussian random vector with tail parameter $K$. Then for any $t>0$, we have that with probability at least $1-2 \me^{-t}$,
\begin{align*}
\left\| \frac{1}{n} \sum_{i=1}^n Y_i Y_i^\T - \Sigma \right\|_{\op} \le C_{\mathrm{sg8}} K^2 \left( \sqrt{ \frac{d+t}{n} } + \frac{d+t}{n} \right) ,
\end{align*}
where $\Sigma=\E(YY^\T)$ and $C_{\mathrm{sg8}}$ is a universal constant.
\end{lem}

The following result can be obtained from \citeappend{Ver18}, Proposition 2.5.2.

\begin{lem} \label{lem:subexp}
For a random variable $Y$, the following properties are equivalent: there exist universal constants $C_{\mathrm{sx},ij}>0$ such that
the $C_{\mathrm{sx},ij}^{-1} K_j \le K_i \le C_{\mathrm{sx,ij}} K_j$ for all $1\le i \not=j \le 4$,
where $K_i$ is the parameter appearing in property (i).
\begin{itemize}
\item[(i)] $\pr ( |Y| > t) \le 2 \exp  ( - \frac{t}{K_1} )$ for any $t>0$.

\item[(ii)] $\E^{1/p} ( |Y|^p ) \le K_2 p$ for any $p \ge 1$.

\item[(iii)] $\E \exp ( |Y|/ K_3 ) \le 2$.
\end{itemize}
If $\E Y=0$, then properties (i)--(iii) are also equivalent to the following one.
\begin{itemize}
\item[(iv)] $\E \exp( s Y ) \le \exp (\frac{K_4^2 s^2 }{2} )$, for any $s \in \bbR$ satisfying $|s| \le K_4^{-1}$.
\end{itemize}
\end{lem}

Say that $Y$ is a sub-exponential random variable with tail parameter $K$ if property (1) holds in Lemma~\ref{lem:subexp} with $K_1=K$.
The following result, from \citeappend{Ver18}, Lemma 2.7.7, provides a link from sub-gaussian to sub-exponential random variables.

\begin{lem} \label{lem:subexp-prod}
Suppose that $Y_1$ and $Y_2$ are sub-gaussian random variables with tail parameters $K_1$ and $K_2$ respectively. Then
$Y_1 Y_2$ is sub-exponential with tail parameter $C_{\mathrm{sx5}} K_1K_2$, where
$C_{\mathrm{sx5}} >0$ is a universal constant.
\end{lem}

The following result about centering can be obtained from \citeappend{Ver18}, Exercise 2.7.10.

\begin{lem} \label{lem:subexp-centering}
Suppose that $Y$ is sub-exponential random variable with tail parameter $K$. Then $Y - \E Y $ is sub-exponential random variable with tail parameter $C_{\mathrm{sx6}} K$,
where $C_{\mathrm{sx6}}>0$ is a universal constant.
\end{lem}

The following result can be obtained from \citeappend{Ver18}, Corollary 2.8.3.

\begin{lem} \label{lem:subexp-concentration}
Suppose that $(Y_1\ldots,Y_n)$ are independent such that $E Y_i=0$ and $Y_i$ is sub-exponential with tail parameter $K$ for $i=1,\ldots,n$. Then
\begin{align*}
\pr \left\{ \left|  \frac{1}{n} \sum_{i=1}^n Y_i \right|  >  C_{\mathrm{sx7}} K \left( \sqrt{\frac{t}{n}} \vee \frac{t}{n} \right) \right\} \le 2 \me^{-t} \quad \text{for any $t>0$}.
\end{align*}
where $C_{\mathrm{sx7}} >0$ is a universal constant.
\end{lem}

\subsection{Symmetrization and contraction}
The following result can be obtained from the symmetrization inequality (Section 2.3.1 in \citeappend{VW}) and Theorem 7 in \citeappend{MZ03}.

\begin{lem} \label{lem:symm}
Let $X_1, \dots, X_n$ be i.i.d. random vectors and $\mathcal{F}$ be a class of real-valued functions such that $\E f(X_1) < \infty$ for all $f \in \mathcal{F}$. Then we have
$$
\E\sup_{f \in \mathcal{F}}\left\{\frac{1}{n}\sum_{i=1}^n f(X_i) - \E f(X_i) \right\} \le 2\E\sup_{f\in\mathcal{F}}\left\{\frac{1}{n}\sum_{i=1}^n \epsilon_i f(X_i)\right\},
$$
where $\epsilon_1, \dots, \epsilon_n$ are i.i.d. Rademacher random variables that are independent of $X_1, \dots, X_n$.
The above inequality also holds with the left-hand side replaced by
$$\E\sup_{f \in \mathcal{F}}\left\{\E f(X_i) - \frac{1}{n}\sum_{i=1}^n f(X_i) \right\}.$$
\end{lem}

\begin{proof}
{For completeness, we give a direct proof.} Let $(X_1^\prime, \ldots, X_n^\prime)$ be i.i.d. copies of $(X_1, \ldots, X_n)$. Then we have
\begin{align}
& \quad \E \sup_{f \in \mathcal{F}} \left\{\frac{1}{n}\sum_{i=1}^{n} f(X_i) -  \E f(X)    \right\} \nonumber \\
& = \E_{X_i} \left[\sup_{f \in \mathcal{F}} \E_{X_i^\prime} \left\{\frac{1}{n}\sum_{i=1}^{n} f(X_i) -  f(X_i^\prime)    \right\} \right] \nonumber \\
& \le  \E_{X_i, X_i^\prime} \left[\sup_{f \in \mathcal{F}} \left\{\frac{1}{n}\sum_{i=1}^{n} f(X_i) - \frac{1}{n}\sum_{i=1}^{n} f(X_i^\prime)\right\}\right] \label{eq:symm-1} \\
& = \E_{X_i, X_i^\prime, \epsilon_i} \left[\sup_{f \in \mathcal{F}} \epsilon_i \left\{\frac{1}{n}\sum_{i=1}^{n} f(X_i) - \frac{1}{n}\sum_{i=1}^{n} f(X_i^\prime)\right\}\right] \label{eq:symm-2} \\
& \le \E_{X_i, \epsilon_i} \left\{\sup_{f \in \mathcal{F}}  \frac{1}{n}\sum_{i=1}^{n} \epsilon_i f(X_i) \right\}
 + \E_{X_i^{\prime}, \epsilon_i} \left\{\sup_{f \in \mathcal{F}}  \frac{1}{n}\sum_{i=1}^{n} -\epsilon_i f(X_i^\prime) \right\} \nonumber \\
& = 2\E_{X_i, \epsilon_i} \left\{\sup_{f \in \mathcal{F}}  \frac{1}{n}\sum_{i=1}^{n} \epsilon_i f(X_i) \right\}.  \nonumber
\end{align}
Line (\ref{eq:symm-1}) follows from Jensen's inequality.
Line (\ref{eq:symm-2}) follows because for a pair of i.i.d. random variables, their difference is a symmetric random variable about 0
and its distribution remains the same when multiplied by an independent Rademacher random variable.
A similar argument is also applicable for upper bounding $\E\sup_{f \in \mathcal{F}}\left\{\E f(X_i) - \frac{1}{n}\sum_{i=1}^n f(X_i) \right\}$.
\end{proof}

\begin{lem} \label{lem:contraction}
    Let $\phi$ be a function with a Lipschitz constant $R$. Then in the setting of Lemma~\ref{lem:symm}, we have
    $$
    \E\sup_{f\in\mathcal{F}}\left\{\frac{1}{n}\sum_{i=1}^n \epsilon_i \phi(f(X_i))\right\} \le \E\sup_{f\in\mathcal{F}}\left\{\frac{1}{n}\sum_{i=1}^n \epsilon_i R f(X_i)\right\}.
    $$
\end{lem}
\subsection{Entropy and maximal inequality}

For a function class $\mathcal F$ in a metric space endowed with norm $\|\cdot\|$, the covering number
$\mathcal{N} (\delta, \mathcal{F}, \|\cdot\|)$ is defined as the smallest number of balls of radius $\delta$ in the $\|\cdot\|$-metric needed to cover $\mathcal{F}$.
The entropy, $H( \delta, \mathcal{F}, \|\cdot\|)$ is defined as $\log \mathcal{N} (\delta, \mathcal{F}, \|\cdot\|)$.
The following maximal inequality can be obtained from Dudley's inequality for sub-gaussian variables (e.g., \citeappend{Van2000}, Corollary 8.3; \citeappend{BLT18}, Proposition 9.2) including Rademacher variables.

\begin{lem} \label{lem:entropy}
Let $\mathcal{F}$ be a class of functions $f:\mathcal X \to \bbR$, and $(\epsilon_1,\ldots,\epsilon_n)$ be independent Rademacher random variables.
For a fixed set of points $\{x_i\in \mathcal X: i=1,\ldots,n\}$, define the random variable
\begin{align*}
Z_n (\mathcal{F}) = \sup_{f\in\mathcal{F}} \left| \frac{1}{n} \sum_{i=1}^{n}\epsilon_i f(x_i) \right|.
\end{align*}
Suppose that $\sup_{f \in \mathcal{F}} \| f\|_n \le 1$ and
$\int_0^1 H^{1/2} (u, \mathcal{F}, \|\cdot\|_n )\,\dif u \le \Psi_n (\mathcal{F})$,
where $\|\cdot\|_n$ is the empirical $L_2$ norm, $\|f\|_n = \{ n^{-1} \sum_{i=1}^n f^2(x_i) \}^{1/2}$.
Then for any $t>0$,
\begin{align}
\pr \left\{ Z_n (\mathcal{F}) /C_{\mathrm{rad}} >  n^{-1/2} (\Psi_n (\mathcal{F}) + t) \right\} \le 2 \me^{- \frac{t^2}{2}},  \label{eq:entropy-tail}
\end{align}
where $C_{\mathrm{rad}}>0$ is a universal constant.
\end{lem}

The following result, taken from \citeappend{VW}, Theorem 2.6.7, provides an upper bound on the entropy of a function class in terms of the VC index.
For any $r\ge 1$ and probability measure $Q$, the $L_r(Q)$ norm is defined as
$ \| f \|_{r,Q} = (\int |f|^r \,\dif Q )^{1/r}$.

\begin{lem} \label{lem:VC-dim}
Let $\mathcal{F}$ be a VC class of functions such that $\sup_{f \in \mathcal{F}} |f| \le 1$.
Then for any $r \ge 1$ and probability measure $Q$, we have
\begin{align*}
\mathcal{N} ( u , \mathcal{F}, \|\cdot\|_{r,Q} ) \le C_{\mathrm{vc}} V (\mathcal{F}) (16\me)^{V(\mathcal{F})} u ^{ - r ( V(\mathcal{F})-1 ) } \quad \text{for any $u \in (0,1)$},
\end{align*}
where $V (\mathcal{F}) $ denotes the VC index of $\mathcal{F}$ and $C_{\mathrm{vc}} \ge 1$ is a universal constant.
\end{lem}

We deduce the following implications of the preceding results, which can be used in conjunction with Lemmas~\ref{lem:ramp-VC}--\ref{lem:intercept-VC}.

\begin{manualcor}{S1} \label{cor:entropy-sg}
In the setting of Lemma~\ref{lem:entropy}, the random variable $Z_n (\mathcal{F})$ is sub-gaussian with tail parameter $C_{\mathrm{rad}} n^{-1/2} (\Psi_n(\mathcal{F})+1)$.
\end{manualcor}

\begin{proof}
By (\ref{eq:entropy-tail}), the results follows from an application of Lemma~\ref{lem:subg-tail}.
\end{proof}

\begin{manualcor}{S2} \label{cor:entropy-sg2}
In the setting of Lemma~\ref{lem:entropy}, the following results hold.

(i) If $\sup_{f \in \mathcal{F}} |f| \le 1$, then
$Z_n (\mathcal{F})$ is sub-gaussian with tail parameter $C_{\mathrm{rad2}} \sqrt{V(\mathcal{F})/n} $,
where $C_{\mathrm{rad2}} = C_{\mathrm{rad}} \{ 1+\sqrt{ 2 +\log( 16C_{\mathrm{vc}}) } + \int_0^1 \sqrt{2 \log (u^{-1}) }\, \dif u \} $.

(ii) Consider another two classes $\mathcal{G}$ and $\mathcal{H}$ of functions from $\mathcal{X}$ to $\bbR$ in addition to $\mathcal{F}$, and
let $\mathcal{F}_{\mathrm{com}} =\{fg + h: f \in \mathcal{F}, g \in \mathcal{G}, h \in \mathcal{H}\}$.
If $\sup_{f \in \mathcal{F} \cup \mathcal{G} \cup \mathcal{H} } |f| \le 1$, then
$Z_n (\mathcal{F}_{\mathrm{com}}) $ is sub-gaussian with tail parameter $C_{\mathrm{rad3}} \sqrt{\{V(\mathcal{F}) + V(\mathcal{G}) + V(\mathcal{H})\}/n} $,
where $C_{\mathrm{rad3}} = C_{\mathrm{rad}} \{ 1+ \sqrt{ 2 +\log( 16C_{\mathrm{vc}}) } + \int_0^1 \sqrt{2 \log (3u^{-1}) }\, \dif u \} $.
\end{manualcor}

\begin{proof}
(i) Take $r=2$ and $Q$ to be the empirical distribution on $\{x_1,\ldots,x_n\}$. By Lemma~\ref{lem:VC-dim},
the entropy integral $\int_0^1 H^{1/2} (u, \mathcal{F}, \|\cdot\|_n )\,\dif u$ can be upper bounded by
\begin{align*}
& \quad \int_0^1 \log^{1/2} \left\{ C_{\mathrm{vc}} V (\mathcal{F}) (16\me)^{V(\mathcal{F})} u ^{ - 2 ( V(\mathcal{F})-1 ) } \right\} \,\dif u \\
& =  \int_0^1  \left\{ \log( C_{\mathrm{vc}}) + \log V(\mathcal{F}) + V (\mathcal{F}) \log (16 \me) + 2(V(\mathcal{F})-1) \log (u^{-1}) \right\}^{1/2} \,\dif u \\
& \le \sqrt{V(\mathcal{F})} \int_0^1 \left\{ \sqrt{ 2 +\log( 16C_{\mathrm{vc}}) } + \sqrt{2 \log (u^{-1}) } \right\} \, \dif u ,
\end{align*}
using $\log V(\mathcal{F}) \le V (\mathcal{F})$ for $V (\mathcal{F}) \ge 1$ and $\sqrt{u_1+u_2} \le \sqrt{u_1}+\sqrt{u_2}$.
Taking $\Psi_n(\mathcal{F})$ in (i) to be the right-hand side of the preceding display yields the desired result by Corollary \ref{cor:entropy-sg}.

(ii) First, we show that the covering number $\mathcal{N}(u, \mathcal{F}_{\mathrm{com}}, \|\cdot\|_n)$
is upper bounded by the product $\mathcal{N}(u/3, \mathcal{F}, \|\cdot\|_n)\mathcal{N}(u/3, \mathcal{G}, \|\cdot\|_n)\mathcal{N}(u/3, \mathcal{H}, \|\cdot\|_n)$.
Denote as $\hat{\mathcal{F}}$ a $(u/3)$-net of $\mathcal{F}$ with the cardinality $\mathcal{N}(u/3, \mathcal{F}, \|\cdot\|_n)$.
Similarly, denote as $\hat{\mathcal{G}}$ and $\hat{\mathcal{H}}$ those of $\mathcal{G}$, $\mathcal{H}$ with the cardinality
$\mathcal{N}(u/3, \mathcal{G}, \|\cdot\|_n)$ and $\mathcal{N}(u/3, \mathcal{H}, \|\cdot\|_n)$ respectively.
For any $f \in \mathcal{F}$, $g \in \mathcal{G}$ and $h \in \mathcal{H}$, there exist $\hat{f} \in \hat{\mathcal{F}}$, $\hat{g} \in \hat{\mathcal{G}}$ and $\hat{h} \in \hat{\mathcal{H}}$ such that
\begin{align*}
&\|\hat{f} - f \|_n \le u/3,\quad
\|\hat{g} - g \|_n \le u/3,\quad
\|\hat{h} - h \|_n \le u/3.
\end{align*}
By the triangle inequality and $\sup_{f \in \mathcal{F} \cup \mathcal{G}} |f| \le 1$, we have
\begin{align*}
& \quad \|\hat{f}\hat{g} + \hat{h} - fg - h\|_n  \\
&\le \|(\hat{f} - f)\hat{g}\|_n +  \|f(\hat{g} - g)\|_n + \|\hat{h} - h\|_n\\
&\le \|\hat{f} - f\|_n + \|\hat{g} - g\|_n + \|\hat{h} - h\|_n \le u.
\end{align*}
This shows that $\hat{\mathcal{F}}_{\mathrm{com}}= \{\hat{f}\hat{g} + \hat{h}: \hat{f} \in \hat{\mathcal{F}}, \hat{g} \in \hat{\mathcal{G}}, \hat{h} \in \hat{\mathcal{H}}\}$
is a $u$-net of $\mathcal{F}_{\mathrm{com}}$ with respect to $\|\cdot\|_n$.
Hence the covering number
$\mathcal{N}(u, \mathcal{F}_{\mathrm{com}}, \|\cdot\|_n)$
is upper bounded by the cardinality of $\hat{\mathcal{F}}_{\mathrm{com}}$,
that is, $\mathcal{N}(u/3, \mathcal{F}, \|\cdot\|_n)\mathcal{N}(u/3, \mathcal{G}, \|\cdot\|_n)\mathcal{N}(u/3, \mathcal{H}, \|\cdot\|_n)$.

Next, by Lemma~\ref{lem:VC-dim} applied to $\mathcal{F}$, $\mathcal{G}$, and $\mathcal{H}$ and similar calulation as in (i),
the entropy integral $\int_0^1 H^{1/2} (u, \mathcal{F}_{\mathrm{com}}, \|\cdot\|_n )\,\dif u$ can be upper bounded by
\begin{align*}
%& \quad \int_0^1 \log^{1/2} \left\{ \mathcal{N}(u/3, \mathcal{F}, \|\cdot\|_n)\mathcal{N}(u/3, \mathcal{G}, \|\cdot\|_n) \mathcal{N}(u/3, \mathcal{H}, \|\cdot\|_n)\right\} \,\dif u \\
& \quad \int_0^1 \left\{ \log  \mathcal{N}(u/3, \mathcal{F}, \|\cdot\|_n) + \log \mathcal{N}(u/3, \mathcal{G}, \|\cdot\|_n)+ \log \mathcal{N}(u/3, \mathcal{H}, \|\cdot\|_n) \right\}^{1/2} \,\dif u \\
& \le \sqrt{V(\mathcal{F}) + V(\mathcal{G}) + V(\mathcal{H})} \int_0^1 \left\{ \sqrt{ 2 +\log( 16C_{\mathrm{vc}}) } +  \sqrt{2 \log (3 u^{-1}) } \right\} \, \dif u .
\end{align*}
The desired result follows by applying Corollary \ref{cor:entropy-sg} to the class $\mathcal{F}_{\mathrm{com}}$, with
$\Psi_n(\mathcal{F}_{\mathrm{com}})$ taken to be the right-hand side of the preceding display.
\end{proof}
\bibliographystyleappend{imsart-nameyear} % Style BST file (imsart-number.bst or imsart-nameyear.bst)
\bibliographyappend{final}       % Bibliography file (usually '*.bib')

\end{document}